\documentclass[12pt]{article}
\usepackage{amsmath}
\usepackage{graphicx}
\usepackage{enumerate}
\usepackage[numbers]{natbib}
\usepackage{url} % not crucial - just used below for the URL 
\usepackage{amsthm,amsmath,amsfonts,amssymb}
\usepackage{algorithm}
\usepackage{algpseudocode}
\usepackage{multirow}
\usepackage{enumitem}

\newcommand{\blind}{0}

\numberwithin{equation}{section}
\theoremstyle{plain}
\newtheorem{thm}{Theorem}[section] %(If you want theorem numbered
\newtheorem{lemma}{Lemma}[section] %%    with section number.
\newtheorem{cor}{Corollary}[section]

\newtheorem{cond}{Condition}[section]
\newtheorem{definition}{Definition}[section]
\newtheorem{defi}{Definition}[section]
\newcommand{\bed}{\begin{defi}}
\newcommand{\eed}{\end{defi}}
\renewcommand{\b}{\mathbf b}

\newcommand{\eps}{\epsilon}

\newcommand{\bitem}{\begin{itemize}}
\newcommand{\eitem}{\end{itemize}}

\newcommand{\goto}{\rightarrow}

\newcommand{\beqn}{\begin{equation}}
\newcommand{\eeqn}{\end{equation}}
\newcommand{\balign}{\begin{align}}
\newcommand{\ealign}{\end{align}}

\newcommand{\s}{\sigma}

\newcommand{\tr}{\mathrm{tr}}
\usepackage{amssymb}
\newcommand{\beq}{\begin{equation}}
\newcommand{\eeq}{\end{equation}}

\newcommand{\diag}{\mathrm{diag}}

\newtheorem*{remark}{Remark}
\usepackage{hyperref}
\hypersetup{
    colorlinks=true,
    linkcolor=blue,
    filecolor=magenta,      
    urlcolor=cyan,
}
\allowdisplaybreaks

\begin{document}

\def\spacingset#1{\renewcommand{\baselinestretch}%
{#1}\small\normalsize} \spacingset{1}

%%%%%%%%%%%%%%%%%%%%%%%%%%%%%%%%%%%%%%%%%%%%%%%%%%%%%%%%%%%%%%%%%%%%%%%%%%%%%%

\if0\blind
{
  \title{\bf Network Goodness-of-Fit for the block-model family}
  \author{Jiashun Jin\thanks{
    JJ gratefully acknowledge the support of the NSF grant \textit{DMS-2015469}. ZK gratefully acknowledges the support of the NSF CAREER grant \textit{DMS-1943902}.}\hspace{.2cm}\\
    Department of Statistics, Carnegie Mellon University\\
    and \\
    Zheng Tracy Ke \\
    Department of Statistics, Harvard University\\
    and\\
    Jiajun Tang \\
    Department of Statistics, Harvard University\\
    and \\
    Jingming Wang\\
    Department of Statistics, University of Virginia}
  \maketitle
} \fi

\if1\blind
{
  \bigskip
  \bigskip
  \bigskip
  \begin{center}
    {\LARGE\bf Network Goodness-of-Fit for the block-model family}
\end{center}
  \medskip
} \fi

\bigskip
\begin{abstract}
 The block-model family has four popular network models (SBM, DCBM, MMSBM, and DCMM).  A fundamental problem is, how well each of these models fits with real networks.  We propose GoF-MSCORE as a new Goodness-of-Fit (GoF) metric for DCMM (the broadest one among the four), with two main ideas. The first is to use cycle count statistics as a general recipe for GoF. The second is  
 a novel network fitting scheme.    GoF-MSCORE is a flexible GoF approach, and we further extend it to SBM, DCBM, and MMSBM. This gives rise to a series of GoF metrics covering each of the four models in the block-model family. 

We show that for each of the four models, if the assumed model is correct, then the corresponding GoF metric  
converges to $N(0, 1)$  as the network sizes diverge.    We also analyze the powers and show that these metrics are optimal in many settings. In comparison, many other GoF ideas face challenges: they may lack a parameter-free limiting null, or are non-optimal in power,  or  
face an analytical hurdle. Note that a parameter-free limiting null is especially desirable 
as many network models  have a large number of unknown parameters. 
The limiting nulls of our GoF metrics are always $N(0,1)$, which are parameter-free as desired.

For 12 frequently-used real networks,  we use the proposed GoF metrics to show that DCMM fits well with almost all of them.   We also show that SBM, DCBM, and MMSBM do not fit well with many of these networks, especially when the networks are relatively large.     To complement with our study on GoF, we also show that the DCMM is nearly as broad as the  rank-$K$ network model. Based on these results,  we recommend the DCMM as a promising model for undirected networks.

\end{abstract}

\noindent%
\vfill

%\newpage
\spacingset{1.75} % DON'T change the spacing!

 \tableofcontents
%%%%%%%%%%%%%%%
%%%%%%%%%%%%%%%
%%%%%%%%%%%%%%%
\section{Introduction} 
\label{sec:intro}    
Network modeling is a fundamental problem.  In recent years, a long list of  network models  were proposed (e.g., Section \ref{subsec:model}),  and each has motivated a long line of research. However, eventually,  we wish to understand 
how these models overlap with each other, and to identify a few representative models. 
Therefore, a fundamental problem is: Out of many existing  models,   
which one achieves a better balance between practical feasibility/interpretability and mathematical tractability (i.e., what is 
the sweet spot of network modeling)?  

We study this problem focusing on undirected networks, but the gained insights are useful for  other networks (e.g., multi-layer  networks \cite{FengDy} and dynamic networks \citep{xue2020time,zhang2020mixed,jiang2023autoregressive}).   
Let $A$ be the adjacency matrix of an undirected network with $n$ nodes, where $A_{ij}  = 1$ if there is an edge between nodes $i$ and $j$, and $A_{ij}  = 0$ otherwise (we do not count self-edges, so $A_{ii} = 0$ for all $i$).  
We assume the network has $K$ different communities, ${\cal C}_1, {\cal C}_2, \ldots, {\cal C}_K$ (communities are groups  of nodes that have more edges within than across; e.g., see Example 1, Section \ref{sec:numeric}, and \cite{JiZhuMM, SCORE-Review, bhattacharya2023inferences}).  The table below presents $12$ frequently seen networks,  
where $(n, K)$ are as above, and $d_{min}, d_{max}, \bar{d}$ are the minimum, maximum, and average degrees, respectively. These networks are not hand-picked for our favor and  
provide a solid ground for fair comparison.

\spacingset{1.1}
\begin{table}[htb!]\centering
\scalebox{0.7}{
\begin{tabular}{l|ccccc|l|ccccc}\hline
Dataset &$n$ &$K$ &$d_{\min}$ &$d_{\max}$ &$\overline{d}$ &Dataset &$n$ &$K$ &$d_{\min}$ &$d_{\max}$ &$\overline{d}$ \\ \hline
Karate &34 &2 &1 &17 &4.59 &Polbooks &105 &2 &2 &25 &8.40 \\ 
Football &115 &11 &7 &12 &10.7 &Weblogs &1222 &2 &1 &351 &27.4 \\ 
Dolphin &62 &2 &1 &12 &5.13 &Citee2016 &1790 &3 &1 &977 &115 \\ 
Fan &79 &2 &1 &78 &4.48 &Caltech &590 &8 &1 &179 &43.5 \\ 
CoAuthor &236 &2 &1 &21 &2.51 &Simmons &1137 &4 &1 &293 &42.7 \\ 
UKfaculty &81 &3 &2 &41 &14.2 & LastFM & 7624 & NA & 1& 216& 7.29\\ 
\hline
\end{tabular}
} 
\end{table}

\spacingset{1.75}

\vspace{-2 em}  

\subsection{The rank-$K$ network models and the block-model family} 
\label{subsec:model} 
Many popular network models are rank-$K$ models. Following the convention,  we assume the upper triangular entries of $A$ are independent Bernoulli variables with $ \mathbb{P}(A_{ij}  = 1)=\Omega(i,j)$, for $i\neq j$ and a matrix $\Omega \in \mathbb{R}^{n\times n}$. Let $\diag(\Omega) \in \mathbb{R}^{n\times n}$ be the $n \times n$ diagonal matrix where the $i$-th diagonal entry is $\Omega(i, i)$, and write $W=A - \mathbb{E}[A]\in\mathbb{R}^{n\times n}$. It follows that 
\begin{equation} 
\label{matrixform} 
A = \Omega - \diag(\Omega) + W, \qquad (\mbox{we call $\Omega$ the Bernoulli probability matrix}).
\end{equation}  
\bed \label{def:Rank-K-model}
We say the network model (\ref{matrixform}) is a rank-$K$ model if $\mathrm{rank}(\Omega) = K$. 
\eed  
Many network models are rank-$K$ models, and below are some examples.   {\it (1) Random dot product graph (RDPG) models} \citep{rubin2017statistical}. In this model, $\Omega = BB'$ for a matrix $B \in \mathbb{R}^{n\times K}$ (so $\Omega$ must be positive semi-definite; this makes the model relatively  restrictive). 
{\it (2) Generalized RDPG models (GRDPG)} \citep{rubin2017statistical}. In this model, $\Omega = B J B'$, where $J = \diag(a_1, \ldots, a_K) \in\mathbb{R}^{K\times K}$ and  $a_k = \pm 1$.   
This is broader than RDPG as $\Omega$ does not have to be positive semi-definite. 
{\it (3) The latent space models and the  graphon models}. These models assume $\Omega(i,j)=f(z_i, z_j)$, where $z_i\in [0,1]$ are independent and identically distributed and $f(\cdot,\cdot)$ is a smooth function. The graphon model is not exactly a rank-$K$ model, but it can be approximated by a rank-$K$ model, as shown by \citep{Chatterjee2015Matrix}.  See Figure \ref{fig:model} (top left).

These models are quite popular, but they are not the {\it modeling sweet spot} we seek for.  
Take GRDPG for example.  First, an eligible pair of $(B, J)$ is the pair of matrices where 
all entries of $B J B'$ fall in $[0, 1]$,    but how to figure out the 
eligible set of $(B, J)$ is an unsolved hard problem.  
Second, $B$ often lacks an intuitive interpretation in real applications; especially,  it remains unclear how to relate $B$ to the network community structure aforementioned.

The DCMM and the block-model family (block-family for short) are also rank-$K$ models. 
Natural networks usually have severe degree heterogeneity and mixed membership 
(e.g., see the table above and \cite{SCORE-Review}).  To model these features, Jin {\it et al}  \cite{JKL2019} (see  also \cite{JiZhuMM, MSCORE, SCORE-Review}) proposed the Degree-Corrected Mixed-Membership (DCMM) model.  
For each node $1\leq i\leq n$, let $\theta_i > 0$ be the degree heterogeneity parameter of  node $i$, and 
let $\pi_i \in \mathbb{R}^K$ be the membership vector of node $i$, where $\pi_i(k) = \mbox{the weight node $i$ puts on community $k$}$, $1 \leq k \leq K$. 
Also, for a symmetrical non-negative matrix $P \in \mathbb{R}^{K, K}$, let  
$P(k, \ell)$ be  the baseline connecting probability between communities $k$ and $\ell$,  $1 \leq k, \ell \leq K$. DCMM is a special rank-$K$ model where we assume $\Omega(i,j) = \theta_i \theta_j \pi_i' P \pi_j$, 
$1 \leq i, j \leq n$. 
Write $\theta = (\theta_1,\ldots,\theta_n)'$ and $\Pi=[\pi_1,\ldots,\pi_n]'$, and let $\Theta \in \mathbb{R}^{n\times n}$ be the diagonal matrix where $\Theta(i,i) = \theta_i$.  
With these notations,  in DCMM, 
\begin{equation}  \label{modelmatrixform}
A=\Omega-\diag(\Omega)+W, \qquad \mbox{and} \qquad \Omega = \Theta \Pi P \Pi' \Theta. 
\end{equation}
This model belongs to the block-family. The other three models in this family are the Stochastic Block Model (SBM), the Degree-Corrected 
Block Model (DCBM) \citep{karrer2011stochastic}, and the Mixed-Membership Stochastic Block Model (MMSBM)  \citep{airoldi2008mixed}.  If we call node $i$ a pure node when  
$\pi_i$ is degenerate (i.e., one entry is $1$, all other entries are $0$),    
then  (a) DCMM reduces to DCBM if all nodes are pure (i.e., each $\pi_i$ must be degenerate),
(b) DCMM reduces to MMSBM 
if all $\theta_i$ are equal, and (c) DCBM reduces to SBM if all nodes are pure and all $\theta_i$ in the same community are equal.  See \cite{JKL2019, EstK} and Figure~\ref{fig:model} (left).  
Regarding model identifiability,  SBM is always identifiable, MMSBM is 
identifiable if each community has at least one pure node, 
DCBM is identifiable if $P$ has unit-diagonal entries, and DCMM is 
identifiable if $P$ has unit-diagonal entries and each community has at least one pure node (e.g., \cite{MSCORE, SCORE-Review}).

\spacingset{1}
\begin{figure}[tb!]
\centering
\includegraphics[height = .21\textwidth, width=.4\textwidth]{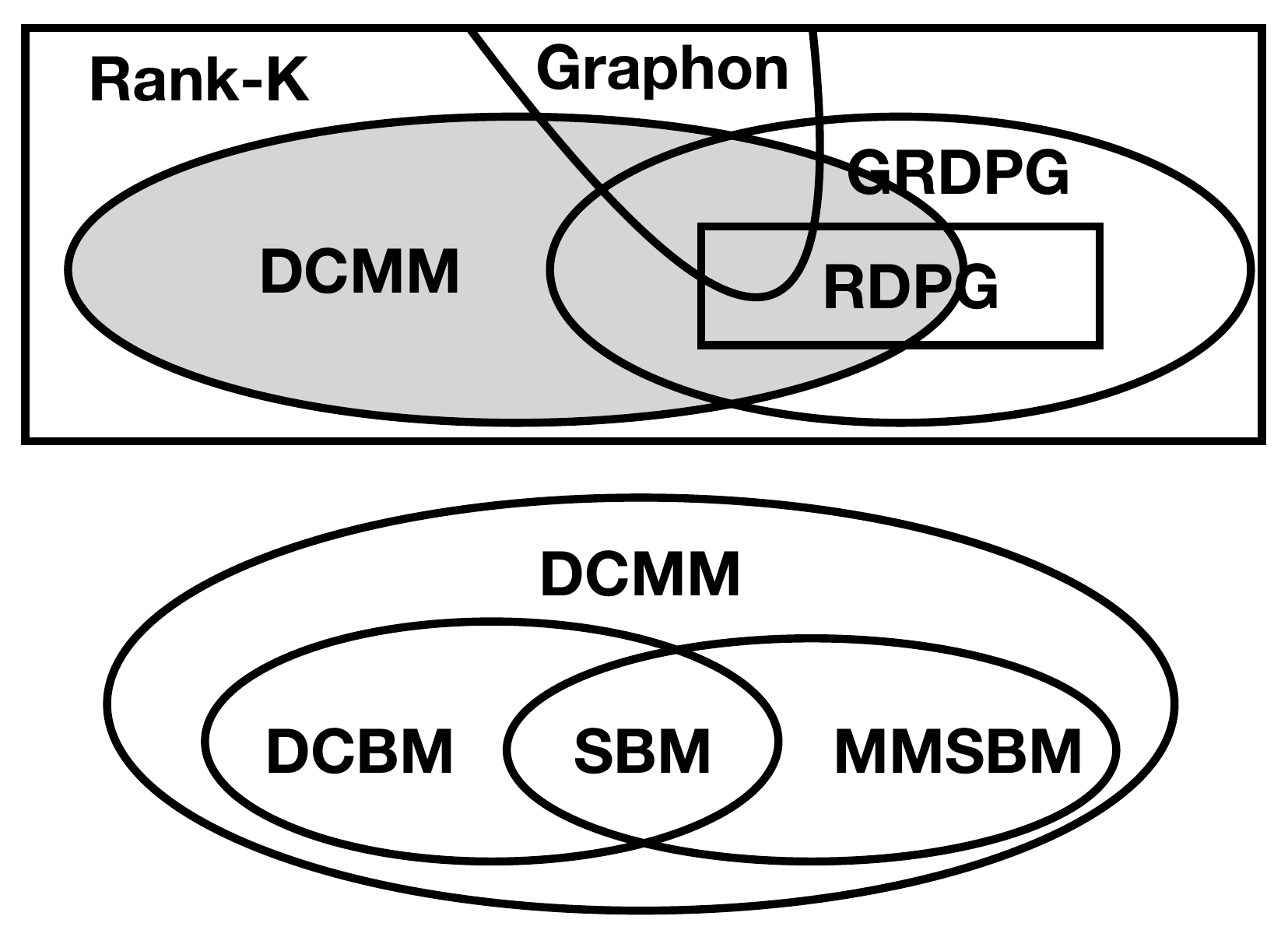} \hspace{.5cm}
\includegraphics[height = .23\textwidth, width=.45\textwidth, trim=0 9 0 6, clip=true]{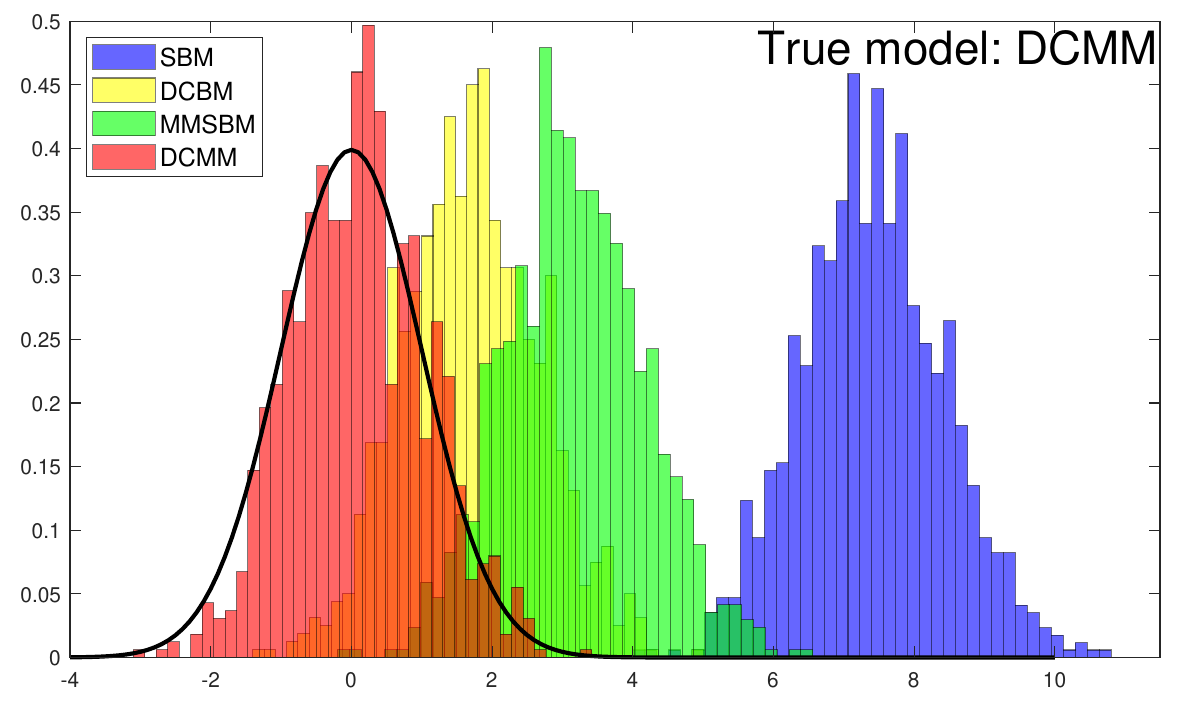} 
\caption{{\small Left: Comparison of models (Section~\ref{subsec:model}). Right:  Histograms of four GoF metrics (based on 1000 networks generated from DCMM; black curve: density of $N(0,1)$; see Section~\ref{subsec:contribution}).   The results suggest that DCMM fits well with the networks, but SBM, DCBM, and MMSBM do not.}  
\label{fig:model}
}
\end{figure}
\spacingset{1.75}

A DCMM model is always a rank-$K$ model, but when the reverse is true?  This is the problem of non-negative matrix factorization (NMF). Fix an $\Omega$ as in the rank-$K$ model. For simplicity, we suppose $\Omega$ is irreducible  (i.e., $\Omega$ cannot be made block-wise diagonal by a simultaneous row \& column permutation), but this can be relaxed. Let $u_i  = \sum_{j = 1}^n \Omega(i, j)$, 
$\bar{u} = (1/n) \sum_{i=1}^n u_i$, $U = \diag(u_1, \ldots, u_n)$.   Let $(\tau_k, \rho_k)$ be the $k$-th eigen-pair of the Laplacian $L = U^{-1} \Omega$, and let $\omega_k = \bar{u} / (\rho_k' U \rho_k)$.   
Assume that more than $K/2$ eigenvalues of $\Omega$ are positive  and $\sum_{k =2}^K |\tau_k|  \cdot  \omega_k  \cdot  \|\sqrt{n} \rho_k\|_{\infty}^2 \leq 1/(K-1)$. 
Theorem \ref{thm:Laplacian} is proved in the supplement.
%%%%%%%%%%%%%%%%
%%%%%%%%%%%%%%%%
%%%%%%%%%%%%%%%%
\begin{thm}[NMF] \label{thm:Laplacian}  
First,  $\tau_1 = 1$.  Second,   if $K = 2$, or if $K \geq 3$ and  the above conditions hold, then there are matrices $(\Theta, \Pi, P)$ as in the DCMM  such that $\Omega = \Theta \Pi P \Pi' \Theta$. 
\spacingset{1}
\spacingset{1.47}
\end{thm} 

 For many real data sets (see Section~\ref{sec:numeric}), $K=2$; it follows by Theorem~\ref{thm:Laplacian} that a DCMM model is always a rank-$K$ model. When $K\geq 3$, 
we can show that $\omega_k \leq C$ and $\|\sqrt{n} \rho_k\|_{\infty}^2 \leq C$ under a mild community balance condition.  Also, $\tau_1  = 1$, and in the most challenging weak-signal case, 
$\max\{|\tau_2|, \ldots, |\tau_K|\} =o(1)$ \cite{JKL2019, SCORE-Review}.  In such cases, the  conditions of Theorem \ref{thm:Laplacian} hold \cite[Section 3]{NMF}.  
Therefore, the DCMM is nearly as broad as the rank-$K$ model.

{\bf Example 1}. In \cite{JBES}, DCMM was used to model $21$ citee networks of the same 
 $2830$ nodes and three communities ``Bayesian", ``Biostatisitcs", and ``non-parametrics". For each node $i$, $\theta_i$ models the impact of author $i$, 
  $\pi_i$ models the research interests of author $i$ in the three communities,  
 and $P$ models the baseline citation probability between these communities.

In summary,  compared to the rank-$K$ model, DCMM is not only nearly as broad but 
also more useful and interpretable, because all of its parameters $(\Theta, P, \Pi)$ have explicit meanings; see Example 1.  This makes DCMM more appealing than the rank-$K$ model in practice.  
 
 \vspace{-.5 em} 
 
%%%%%%%%%%%
%%%%%%%%%%%
%%%%%%%%%%%  
\subsection{Network Goodness-of-Fit (GoF): review and our contribution}  
\label{subsec:contribution}   
The above suggests that the block-model family (or block-family) (especially the DCMM) are promising models for real  networks. A fundamental problem is therefore as follows. 
\vspace{-.65 em} 
\begin{itemize} 
\item {\it Goodness-of-Fit (GoF) for the block-family}. Given each of the $4$ models in the block  family, 
how well does it fit with real networks?  In particular, is DCMM adequate?  Can we replace it by a simpler  model in the family (e.g.,  SBM, MMSBM, or DCBM)?  
\end{itemize}  
\vspace{-.65 em} 
To test  if a model fits with an observed network with $n$ nodes,  we hope to develop a test statistic $T_n$ and claim a lack-of-fit  when
$|T_n| \geq t_0$ for a threshold $t_0$.   
To use this in practice, it is crucial that the threshold $t_0$ can be specified explicitly.   This requires that we design $T_n$ such that for an {\it explicit distribution $\psi_0$ that does not depend on any unknown parameters},   
\vspace{-.5 em} 
\begin{equation} \label{gof-goal} 
\mbox{$T_n$ converges weakly to $\psi_0$ if the assumed model is true}.  
\end{equation} 
\vspace{-1 em} 
When \eqref{gof-goal} holds, we call $\psi_0$ the ``parameter-free limiting null".  

\vspace{.75 em}

How to construct a GoF metric with a parameter-free limiting null is a challenging problem.   
One may use a likelihood approach, but as the assumed model may have a large number of  parameters,    the limiting null  is hard to derive and may depend on unknown parameters in a complicated way.   
Lei \cite{LeiGoF} constructed a normalized adjacency matrix $\widetilde{A}$ in an SBM setting and showed 
that the first eigenvalue of $\widetilde{A}$ converges to a Tracy-Widom distribution.  Dong et al. \cite{dong2020spectral} used linear spectral statistics of $\widetilde{A}$ and showed that the limiting null is $N(0,1)$.  Hu et al. \cite{hu2021using} used the maximum entry-wise deviation to test the goodness of fit for SBM. 
Unfortunately,  these approaches focused on the SBM and it is hard to extend them to more complicated models such as DCMM where we have {\it far more parameters}. 
In DCMM, it is unclear how to construct $\widetilde{A}$, and it is even harder to 
derive a test statistic with a parameter-free limiting null.    Also, \cite{LeiGoF, dong2020spectral} required that the average degree of the network is $O(n)$ (this is rather restrictive as most networks are sparse) and did not analyze the power.

Our problem is also related to network model selection \citep{han2019universal,  li2020network}.  Given a network from a rank-$K$ model, they are interested in estimating $K$ or testing $K = K_0$ (here, $K_0$ is known but $K$ is not).  Since both DCBM and DCMM are rank-$K$ models,   those approaches 
have no power in telling  whether DCBM or DCMM is more appropriate for a given data set.   
Also,  since a rank-$K$ model is not necessarily a DCMM model,  we cannot use their methods as GoF metrics for DCMM.  See Section \ref{subsec:power} for more comparisons. Our GoF problem is also related to global testing (i.e., testing $K = 1$ vs. $K > 1$) \citep{JKL2019}, but the goals are clearly very different. 

For these reasons, how to find {\it a proper GoF metric for each of the four models}  remains an open problem. We face several challenges. First,  our data has limited information but the models have many unknown parameters.   Second, for many GoF metrics we may have,  the limiting nulls may depend on unknown parameters in a complicated way.  Last, we want an approach that is optimal in power and 
automatically adapts to a wide range of sparsity levels.

Our idea is as follows.   We discover a family of self-normalized cycle count (SCC) statistics, denoted by $\psi_{n, m}(\widehat{\Omega})$, where $m \geq 3$ is the fixed length of the cycles being counted and $\widehat{\Omega}$ is an estimate for $\Omega$ to be determined. 
  We show that  
 if some mild regularity conditions hold, then 
\begin{equation} \label{Introlimitnull} 
\psi_{n, m}(\widehat{\Omega}) \goto N(0,1) \;\;  \mbox{in law},  \qquad  \mbox{in the idealized case where $\widehat{\Omega} = \Omega$}. 
\end{equation}  
This is exactly what we need for a GoF: while $\Omega$ may have numerous unknown parameters,  the limiting null $N(0,1)$ is parameter-free. 
Since (\ref{Introlimitnull}) holds for all $m \geq 3$, we may take $m = 3$ for simplicity   (the analysis for $m\geq 4$ is similar but more lengthy). Let $T_n(\widehat{\Omega}) = \psi_{n, 3}(\widehat{\Omega})$.

It remains to find a good estimate $\widehat{\Omega}$ for $\Omega$ and show that 
$|T_n(\widehat{\Omega}) - T_n(\Omega)| \goto  0$.   
Take the DCMM for example. In DCMM,  $\Omega = \Theta \Pi P \Pi' \Theta$ and $\Pi$ is an $n \times K$ matrix of latent variables.   
Therefore, a reasonable approach is,  
we first estimate $\Pi$ and then use the result to estimate $\Omega$ by refitting.  
There are relatively few approaches to 
estimating $\Pi$,  among which is the recent idea of Mixed-SCORE (MSCORE)  \citep{MSCORE, SCORE-Review}, an efficient spectral approach.

Unfortunately, either we use MSCORE or other spectral approaches, 
we face {\it an analytical hurdle}.   The main reason is that, to show $|T_n(\widehat{\Omega}) - T_n(\Omega)| \goto 0$, the SCC approach dictates that we must choose an $\widehat{\Omega}$ that is a {\it relatively simple (multivariate) function of $A$ with an explicit form} (e.g., Section \ref{subsec:SCChurdle}), but none of the existing approaches has such a property.

We propose GoF-MSCORE as a new method for estimating $\Omega$ under the DCMM model, 
motivated by a simplex structure we discover in an embedded subspace. 
In this approach, we first use MSCORE \citep{MSCORE} to obtain an initial 
estimate for $\Pi$.  We then combine the MSCORE estimate,  
 the simplex structure aforementioned,  and several other techniques (e.g., net-rounding) to re-estimate 
$\Pi$. We then use the result to estimate $\Omega$ by refitting.  This gives rise to a 
new GoF metric $T_n(\widehat{\Omega}^{\mathrm{DCMM}})$ for DCMM, where $\widehat{\Omega}^{\mathrm{DCMM}}$ 
is an estimate for $\Omega$ under DCMM that is not only accurate but also a 
simple function of $A$ with an explicit form, as desired.

The GoF-MSCORE metric for DCMM, $T_n(\widehat{\Omega}^{{\rm DCMM}})$, is a flexible idea: we extend it to   GoF metrics for SBM, DCBM, and MMSBM, denoted by $T_n(\widehat{\Omega}^{{\rm SBM}})$,  $T_n(\widehat{\Omega}^{{\rm DCBM}})$, and 
$T_n(\widehat{\Omega}^{{\rm MMSBM}})$, 
respectively. In all these cases,  we have successfully circumvented the analytical hurdle and 
showed: $T_n(\widehat{\Omega}^*) \goto N(0,1)$ in law, if $*$ is the 
correct model (here $* =$ SBM, DCBM, MMSBM, and DCMM).  
Although the design of these GoF metrics no longer poses any hurdle impossible to overcome,  
 the analysis is still delicate 
and long, especially when $*=$ MMSBM or DCMM.  
We also analyze the power of GoF-MSCORE and show that it is 
optimal in a broad setting.

Figure \ref{fig:model} (right)  compares the four GoF metrics using $1000$ simulated networks from DCMM with parameters as follows: (1)  $(n, K)=(3000,2)$, and $P \in \mathbb{R}^{2, 2}$  with $1$ on the main diagonal and $0.05$ elsewhere, 
(2) $\theta_i$'s are iid from ${\rm Uniform}(0.1,0.3)$, and (3) each community has $n/8$ pure nodes; for all other nodes, the $\pi_i$'s are iid from ${\rm Dirichlet}(0.5,0.5)$. 
The results show that DCMM fits well with the networks, but  SBM, DCBM, and MMSBM 
do not. %Therefore, our approach has power in differentiating different models. 
Thus, our approach is effective in distinguishing between different models.

We also investigate the $12$ real networks (see the table above). Take 
Citee2016 for example. The GoF metrics are $T_n(\widehat{\Omega}^*) = 759.2$, $308.4$,  $405.3$ and $3.687$, for $* =$ SBM, DCBM, MMSBM  and DCMM, respectively.  This strongly suggests that  DCMM 
fits well with the citee network, but SBM, DCBM, and MMSBM yield a poor fit.  
See Section \ref{sec:numeric} for more discussion. 
  
Below in Section \ref{sec:SCC}, we introduce a family of SCC statistics and 
a general recipe for GoF. We also explain why many GoF approaches may pose an analytical challenge. In Section \ref{sec:DCMM}, we propose GoF-MSCORE as new GoF metric for DCMM, and extend it to the other three models in the block-family. We show all four GoF metrics have $N(0,1)$ as the parameter-free limiting null, and have optimal power in a broad alternative setting.  Section \ref{sec:numeric} 
analyzes some simulated networks and the 12 real networks above.  
Section \ref{sec:Discu} is a short discussion. 

%For a vector $x \in \mathbb{R}^n$, $\|x\|_q$ denotes the Euclidean $\ell_q$-norm (the subscript $q$ is dropped when $q = 2$), $x_{min} = \min\{x_1, \ldots, x_n\}$, $x_{max} = \max\{x_1, \ldots, x_n\}$, and 
%$\diag(x)$ is the $n \times n$ diagonal matrix where the $i$-th diagonal entry is $x_i$.  
%For an $n \times n$ matrix $\Omega$, $\diag(\Omega)$ is the $n \times n$ diagonal matrix 
%where the $i$-th diagnoal entry is $\Omega(i, i)$.  

\vspace{-.5 cm}

%%%%%%%%%%%%
%%%%%%%%%%%%
%%%%%%%%%%%%    
\section{A general recipe for GoF and an analytical hurdle} 
\label{sec:SCC} 
To find a  GoF metric with a parameter-free limiting null is a challenging task.  
We tackle this by introducing a family of self-normalized cycle count (SCC) statistics.  
The SCC approach is promising  but has an (unexpected) analytical hurdle.  We discuss our key ideas to overcome the hurdle in Section \ref{subsec:SCChurdle}, with the detailed methods and theory deferred to Section~\ref{sec:DCMM}.  
%%%%%%%%%%%%
%%%%%%%%%%%%
%%%%%%%%%%%%

\vspace{-1 em} 
\subsection{A general GoF recipe by Self-normalized Cycle Counts (SCC)} 
\label{subsec:SCCoracle}

\begin{definition}
Given an integer $m \geq 3$, define the order-$m$ cycle count statistic by $C_{n, m} = \sum_{i_1, i_2, \ldots, i_m (dist)} A_{i_1 i_2} A_{i_2 i_3} \ldots A_{i_m i_1}$, where ``dist" means $i_1,i_2,\ldots,i_m$ are $m$ distinct indices. Given any estimate $\widehat{\Omega}$, define $U_{n, m}(\widehat{\Omega})  = \sum_{i_1, i_2, \ldots, i_m (dist)} \widehat{A}_{i_1 i_2} \widehat{A}_{i_2 i_3} \ldots \widehat{A}_{i_m i_1}$,  with $\widehat A=A-\widehat\Omega$. 
\end{definition}
Note that $A_{i_1 i_2} A_{i_2 i_3} \ldots A_{i_m i_1} = 1$ if $(i_1, i_2, \ldots, i_m)$ forms a length-$m$ (order-$m$) cycle in the graph, and 
$A_{i_1 i_2} A_{i_2 i_3} \ldots A_{i_m i_1} =  0$ otherwise. So, $\frac{1}{m!}C_{n, m}$ is the total number of 
order-$m$ cycles.  We call $\frac{1}{m!}U_{n, m}(\widehat{\Omega})$   
the number of order-$m$ signed-cycles  
\citep{Bubeck2016testing, JKL2019} (e.g., 
$\frac{1}{6}C_{n, 3}$ and $\frac{1}{6}U_{n, 3}(\widehat{\Omega})$ are numbers of triangles and signed triangles;  $\frac{1}{24}C_{n, 4}$ and $\frac{1}{24}U_{n, 4}(\widehat{\Omega})$ are numbers of quadrilaterals and signed quadrilaterals).  
Define the Self-normalized Cycle Count (SCC) statistic by  
$\psi_{n, m}(\widehat{\Omega}) =  U_{n, m}(\widehat{\Omega}) / \sqrt{2 m C_{n, m}}$.     
The following results are proved in the supplement.  
%%%%%%%%%%%%%
%%%%%%%%%%%%
%%%%%%%%%%%%
\begin{thm}[Parameter-free limiting null (oracle case)]\label{thm:SCC} 
Let $u = \Omega {\bf 1}_n = (u_1, \ldots, u_n)'$, $\bar{u} = (1/n) \sum_{i=1}^n u_i$,   and  $u_{max} = \max\{u_1, \ldots, u_n\}$.  Fix $m \geq 3$.  As $n \goto \infty$, if 
$\Omega(i,j) \leq C u_iu_j/(n\bar{u})$,  ${\rm tr} (\Omega^m)\geq C\Vert u\Vert^{2m}/(n\bar{u})^{m}$,  and $
\max\{ n\bar{u} / \|u\|^2,\,  u^2_{\max}/ (n\bar{u}) \}  \goto 0$, 
 then  $\psi_{n, m}({\Omega})  \goto N(0,1)$.   
\end{thm} 
%%%%%%%%%%%%
%%%%%%%%%%%%
%%%%%%%%%%%%
\begin{cor}\label{cor:SCC}
Fix an integer $m \geq 3$ and suppose there are constants $c_1 > c_0 > 0$ and a positive vector $\theta \in \mathbb{R}^n$ such that $c_0 \theta_i \theta_j \leq \Omega(i, j) \leq c_1 \theta_i \theta_j$ for all $1 \leq i, j \leq n$. As $n \goto \infty$, if $\| \theta\|\to \infty$,  $\theta_{\max} \to 0$,   then 
$\psi_{n, m}({\Omega})  \goto N(0,1)$. 
\end{cor}  
The proof of these results is different from the analysis of classical cycle counts (e.g., \citep{zhang2022edgeworth}): The classical cycle counts are non-degenerate U-statstics, but $U_{n,m}(\widehat{\Omega})$ is a degenerate U-statistic.   Our proofs use Hall's works on martingale central limit theorem (CLT) \citep{hall2014martingale}.  To verify the conditions of martingale CLT, we need delicate analysis in moments and combinatorics.   
Network test statistics with asymptotic normality were discovered earlier in \cite{JKL2019, EstK}, but with different forms.  Also, they only considered special cases.  For example,  \cite{JKL2019} only considered $\Omega = \theta \theta'$ and $m = 3, 4$, and 
\cite{EstK} only considered the case where $m =4$ and $\Omega$ satisfies a DCBM.   Our results are for general $(m, \Omega)$ and the analysis is more complicated.  In particular,  we do not require $\mathrm{rank}(\Omega) = K$ so our results may motivate GoF metrics beyond the rank-$K$ model  
(while those in \cite{JKL2019, EstK} are special rank-$K$ models with specific structures).  

The SCC statistics provide  a general recipe for constructing GoF metrics with parameter-free limiting null:    (a) obtain a good estimate $\widehat{\Omega}$ for $\Omega$,   (b) show that $|\psi_{n, m}(\widehat{\Omega}) - \psi_{n, m}(\Omega)| \goto_p 0$, and (c) claim $\psi_{n, m}(\widehat{\Omega}) \goto N(0,1)$.  For simplicity, we focus on the case of $m = 3$, where  
\begin{equation} \label{DefineTn} 
T_n(\widehat{\Omega}) \equiv \psi_{n, 3}(\widehat{\Omega}) = U_{n, 3}(\widehat{\Omega}) / \sqrt{6 C_{n, 3}},  
\end{equation} 
but our results are readily extendable to $m = 4$: using a larger $m$ (e.g., $m = 5, 6$) does not improve the power substantially but makes the analysis more tedious. See also Remark 1.  
  
{\bf Remark 1} {\it (Why we should not take $m \leq 2$)}.    
A natural idea is to 
use a normalized version of $Y_n = \sum_{i \neq j} \widehat{A}(i,j)^2$ for GoF,  
where as above $\widehat{A} = A - \widehat{\Omega}$.  This  is essentially the SCC statistic with $m = 2$.   
It was pointed out in \cite{JKL2019} that the variance of $Y_n$ is much larger than expected, so the power 
of the test is typically much smaller than that of $m = 3$ or $m = 4$. Moreover, when $m=1$, 
the SCC has no definition, but we may enforce $\psi_{n, 1}(\widehat{\Omega})  = \sum_{i \neq j} \widehat{A}(i, j) / \sqrt{\sum_{i \neq j} A(i, j)}$. The power of the statistic hinges on  
${\bf 1}_n' (\Omega - \widetilde{\Omega}) {\bf 1}_n / \sqrt{2  \mathrm{trace}(\Omega)}$, where $\widetilde{\Omega}$ belongs to the assumed model class and is the population counterpart of $\widehat{\Omega}$ (see Section~\ref{subsec:power}).  
For many pairs of $(\Omega, \widetilde{\Omega})$,  the term ${\bf 1}_n' (\Omega - \widetilde{\Omega}) {\bf 1}_n$ can be $0$ (e.g., when $\Omega$ follows SBM with $K=1$ and $\widetilde{\Omega}$ follows a symmetric SBM with $K=2$) or much smaller than $\sqrt{2\mathrm{trace}(\Omega)}$, 
and this statistic easily loses power.   For these reasons, we do not recommend SCC test statistics with $m \leq 2$.

\subsection{Analytical strategies and hurdles} 
\label{subsec:SCChurdle}    
Theorem \ref{thm:SCC} is for the oracle case.  To use the idea for the real case, we need to construct an $\widehat{\Omega}$ and show that $|T_n(\widehat{\Omega}) - T_n(\Omega)| \goto 0$ in probability. 
Write $T_n(\widehat{\Omega}) - T_n(\Omega) = [U_{n, 3}(\widehat{\Omega}) - U_{n, 3}(\Omega)] / \sqrt{6 C_{n, 3}}$.  
To show the claim, the key is to analyze $|U_{n, 3}(\widehat{\Omega}) - U_{n, 3}(\Omega)|$.  A conventional strategy is to decompose $(U_{n, 3}(\widehat{\Omega})  - U_{n, 3}(\Omega))$ into many terms, and analyze them one by one. 
Letting $W_1=W-\diag(\Omega)$ and $\Delta=\widehat{\Omega}-\Omega$,  
the following lemma is proved in the supplement.    
%%%%%%%%%%%%
%%%%%%%%%%%%
%%%%%%%%%%%%
\begin{lemma}\label{lem:diffU} 
We have that $U_{n,3} (\widehat{\Omega}) - U_{n,3} ({\Omega}) =3{\rm tr}(W_1\Delta^2)-3{\rm tr}(W_1^2\Delta)-{\rm tr}(\Delta^3)
+6\tr(W_1\circ W_1\Delta)-3\tr(W_1\circ\Delta^2)+3\tr(\Delta\circ W_1^2)-6\tr(\Delta\circ W_1\Delta)+3\tr(\Delta\circ\Delta^2)+6\tr(W_1\circ\Delta\circ\Delta)-6\tr(W_1\circ W_1\circ\Delta)-2\tr(\Delta\circ\Delta\circ\Delta)$, where $\tr(\cdot)$ is the trace of a matrix  and $\circ$ is the Hadamart product.  
\end{lemma}
We now discuss how to construct $\widehat{\Omega}$ and how to analyze the individual terms in Lemma 
\ref{lem:diffU}. For space reasons,  we only discuss the DCBM case and the DCMM case.

Consider the DCBM case. Recall that $A = \Omega - \diag(\Omega) + W$ and $\Omega = \Theta \Pi P \Pi' \Theta$, where $\Pi$ is the $n \times K$ matrix of latent variables.   
In such a  low-rank latent-variable model, it is conventional to estimate $\Omega$ using a spectral approach.  
For example, one may use the (classical) SVD approach: 
$\widehat{\Omega} = \sum_{k = 1}^K \hat{\lambda}_k \hat{\xi}_k \hat{\xi}_k'$, 
where $(\hat{\lambda}_k, \hat{\xi}_k)$ is the $k$-th eigen-pair of $A$.  
One may also use the GoF-SCORE approach as follows: 
we first use SCORE (a community detection algorithm by \cite{SCORE-Review}) to obtain an 
estimate $\widehat{\Pi}^{{\rm score}}$ for $\Pi$, and then use $(A, \widehat{\Pi}^{{\rm score}})$ 
to estimate $(\Theta, P)$.  

It turns out that the SVD approach poses an analytical hurdle, but the SCORE approach does not. 
To see the point, consider the SVD approach first. For example, suppose we want to analyze the term ${\rm tr}( {W}_1^2 \Delta)$ in Lemma  \ref{lem:diffU}. Write ${\rm tr}( {W}_1^2 \Delta)  = {\rm tr} ( [W-\diag(\Omega)]^2(\widehat{\Omega}-\Omega))$ and 
$\widehat{\Omega} -\Omega = \sum_{k=1}^K \hat{\lambda}_k \hat{\xi}_k \hat{\xi}_k'-\sum_{k=1}^K \lambda_k \xi_k \xi_k' $, where similarly $(\lambda_k, \xi_k)$ is the $k$-th eigen-pair of  $\Omega$. 
Now, first, since $(\hat{\lambda}_k, \hat{\xi}_k)$ are complicated functions of $A$ without an explicit form,  
it is hard to analyze the dependence between $\widehat{\Omega}$ and $W$, and so 
it is hard to analyze the precise mean and variance of ${\rm tr}( {W}_1^2 \Delta)$ directly. 
Second, one may bound the term using the Cauchy-Schwarz inequality: $
|{\rm tr}( {W}_1^2 \Delta)|\leq 2K\|[W-\diag(\Omega)]^2(\Omega-\widehat{\Omega})\|\leq 2K\|\Omega-\widehat{\Omega}\|\cdot\|W-\diag(\Omega)\|^2$. Unfortunately, even with the best 
large-deviation results on $\|\Omega-\widehat{\Omega}\|$,  the bound is too loose for our purpose 
(it is easy to have a random sequence $x_i$ such that $\sum_{i=1}^n x_i \goto 0$ but 
$\sum_{i = 1}^n |x_i| \goto \infty$, in probability).  Such an analytical hurdule was noted in \cite{JKL2019, han2019universal}  
among others. 

Consider the SCORE approach. As DCBM  does not allow mixed membership,   each row of $\Pi$ is a {\it degenerate weight vector}.\footnote{A weight vector is degenerate if one of its entry is $1$ and all other entries are $0$.}  In such a case, previous study (e.g., \cite{EstK}) showed that  $\mathbb{P}(\widehat{\Pi}^{{\rm score}} \neq  \Pi) = o(1)$ under mild conditions.  Therefore,  we can assume $\Pi$ as known, and it is not hard to derive an estimate for $(\Theta, P, \Omega$) that is a {\it relatively simple and explicit (multivariate) function of $A$} (e.g., see \cite{EstK}).  Though the analysis is still delicate and long,  the approach does not face an analytical hurdle.

Consider now the DCMM case, where we have mixed memberships.  
We may extend the GoF-SCORE approach, but 
there is a key difference when mixed membership presents: 
it is merely impossible 
to have an estimate $\widehat{\Pi}$ for $\Pi$ such that $\mathbb{P}(\widehat{\Pi} \neq \Pi) = o(1)$.  This is because: Under DCMM, a row of $\Pi$ is not discrete but a continuous variable which may take any value in the set  $
\{x \in \mathbb{R}^K: \sum_{i = 1}^K x_i = 1, x_i \geq 0\}$.   
For this reason,   we face an analytical hurdle.

To overcome the challenge, we propose GoF-MSCORE, which 
 is related to the recent idea of  Mixed-SCORE (MSCORE) \citep{MSCORE}.  
 MSCORE has three steps. (i) Obtaining $\widehat{\Xi}=[\hat{\xi}_1,\ldots,\hat{\xi}_K]$, where $\hat{\xi}_k$ denotes the $k$th eigenvector of $A$. (ii) Constructing $\widehat{R} = [\hat{\xi}_2 / \hat{\xi}_1, \ldots, \hat{\xi}_K / \hat{\xi}_1] \in \mathbb{R}^{n, K-1}$, 
where for any $\xi, \eta \in \mathbb{R}^n$,   $\xi / \eta \in \mathbb{R}^n$ is the vector of  entry-wise ratios. This step is known as the SCORE normalization \citep{SCORE}, which helps remove the nuisance effects of degree parameters from the empirical eigenvectors. (iii) Exploiting a simplex geometry in $\widehat{R}$ to estimate $\Pi$. In detail, \cite{MSCORE} showed that there is an ideal $K$-vertex simplex ${\cal S}$, such that, subject to noise corruption, each row of $\widehat{R}$ is contained in ${\cal S}$. Therefore, we can apply a vertex hunting algorithm on $\widehat{R}$ to estimate the vertices of ${\cal S}$. \cite{MSCORE} also showed that there exists a direct connection between ${\cal S}$ and the target $\Pi$, which gives rise to an explicit step of constructing $\widehat{\Pi}$ from the estimated simplex vertices. We refer to the supplement for a full description of MSCORE. 

We now discuss GoF-MSCORE: 
We first obtain an initial estimate $\widehat{\Pi}^{{\rm MS}}$ for $\Pi$ by MSCORE and then use a  {\it net-rounding algorithm}   to obtain an $n \times K$ matrix $\widehat{H}$. 
Next, we use   $A \widehat{H}$ to obtain a second estimate for 
$\Pi$, denoted by $\widehat{\Pi}$.  Finally, we use $\widehat{\Pi}$ to estimate $(\Theta, P)$ and derive an estimate $\widehat{\Omega}$ for $\Omega$. The key is,  for a non-stochastic matrix $H$, $\mathbb{P}(\widehat{H} \neq H) = o(1)$, so the GoF-MSCORE estimates of $(\Theta, \Pi, P, \Omega)$ are not only accurate but also have a  relatively simple and explicit form;  this is exactly what we need.  See Figure \ref{fig:flow} and Remark 3.

\spacingset{1}
\begin{figure}[htb!] 
\centering
\includegraphics[width=.6\textwidth]{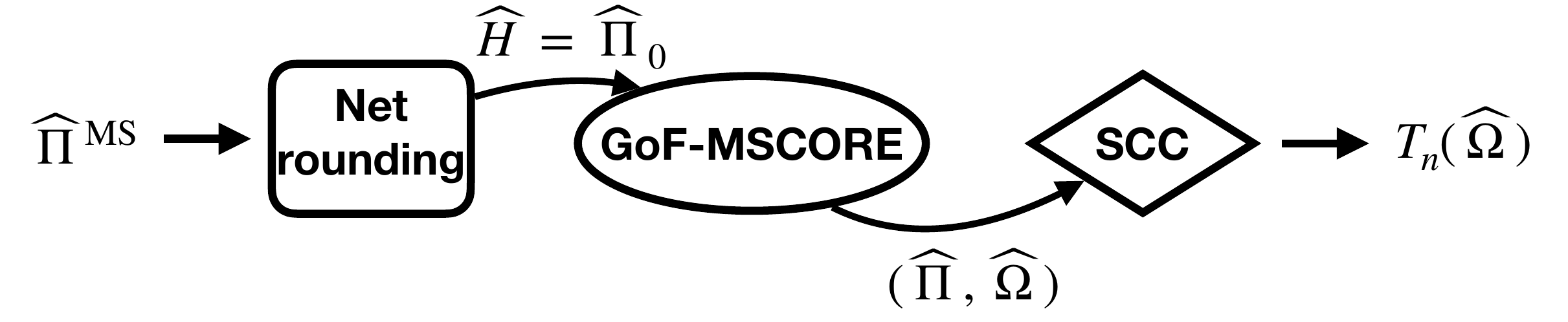}
\caption{\small GoF-MSCORE for DCMM (flow chart).}  \label{fig:flow} 
\end{figure}
\spacingset{1.75}

In summary,   to remove the analytical  hurdle, we need an $\widehat{\Omega}$ that is not only accurate but also has a  relatively simple and explicit form.  
For DCBM,  each row of $\Pi$ is a degenerate weight vector, so we can obtain such an estimate by the GoF-SCORE  above. For the DCMM,  a row of $\Pi$ is a continuous variable which may take any 
weight-vector values. In this case, existing approaches face an analytical hurdle,  and we need a new and more complicated  approach.

%%%%%%%%%%%%
%%%%%%%%%%%%
%%%%%%%%%%%%    
\section{GoF for all models in the block-model family} 
\label{sec:DCMM} 
We first propose GoF-MSCORE as a GoF metric for DCMM in Section \ref{subsec:GoF-MSCORE} and study the asymptotic normality in Section \ref{subsec:normality}.  
We then extend the metrics and theory to MMSBM, DCBM, and SBM, respectively in Sections 
\ref{subsec:MMSBM}-\ref{subsec:DCBM}.  Power is discussed in Section \ref{subsec:power}.  
%%%%%%%%%%%%%
%%%%%%%%%%%%%
%%%%%%%%%%%%%
\subsection{GoF-MSCORE: a new GoF metric for DCMM} 
\label{subsec:GoF-MSCORE} 

GoF-MSCORE is inspired by Lemma \ref{lem:AFMoracle} below.  
Fix  $H \in \mathbb{R}^{n, K}$ such that each row of $H$ is a weight vector. Write
$\eta=\Pi'\Theta{\bf 1}_n\in\mathbb{R}^K$ and $G_H = \Pi'\Theta H\in\mathbb{R}^{K,K}$.  Introduce 
\begin{equation} \label{DefineV} 
R_H=[r_1,   \ldots, r_n]' :=\diag(\Omega {\bf 1}_n)^{-1}\Omega H, \qquad V_H = [v_1,  \ldots, v_K]' := [\diag(P\eta)]^{-1}PG_H.  
\end{equation} 
For $1 \leq i \leq n$, let $w_i =  (V_H^{-1})' r_i$. For simplicity,  we drop the subscript ``$H$" in $v_k$, $r_i$ and $w_i$, but bear in mind that they all depend on $H$. Write $Z_H =V_H (H'\Omega H)^{-1}V_H'$.  
Recall that in the DCMM model, we require $P$ to have unit diagonal entries for identifiability (see Section \ref{subsec:model}). 
Lemma \ref{lem:AFMoracle}  is proved in the supplement  (where $\circ$ is the Hadamard product). 
%%%%%%%%%%%
%%%%%%%%%%%%%
%%%%%%%%%%%%%%
\begin{lemma}[The population quantities] \label{lem:AFMoracle}
Consider a DCMM model where $G_H$ is non-singular. Then, the rows of $R_H$ are in a simplex ${\cal S}\subset\mathbb{R}^K$ whose vertices are denoted by $v_1, v_2, \ldots, v_K$. For each $1 \leq i \leq n$, $r_i$ falls on  one vertex of ${\cal S}$ if node $i$ is pure, and it falls in the interior of ${\cal S}$ otherwise.  
Furthermore,  the following statements are true. First, letting $w_i=\|\pi_i\circ P\eta\|_1^{-1}(\pi_i\circ P\eta)$, it holds that 
$r_i = \sum_{k=1}^K w_i(k)v_k$, $1\leq i\leq n$.   
Second, $Z_H =[\diag(P\eta)]^{-1}P[\diag(P\eta)]^{-1}$, and it follows that $\diag(P \eta) = [\diag(Z_H)]^{-1/2}$ and $P = [\diag(Z_H)]^{-1/2} Z_H [\diag(Z_H)]^{-1/2}$.  
Last, $\theta_i = e_i'\Omega {\bf 1}_n / \|\pi_i\circ P\eta\|_1$, $1 \leq i \leq n$. 
\end{lemma}
We call the simplex in Lemma \ref{lem:AFMoracle} the {\it ideal simplex}. 
Lemma~\ref{lem:AFMoracle} inspires an oracle approach for using $(\Omega, H, K)$ to retrieve $(\Theta, P, \Pi)$ as follows:  (1) Obtain $R_H$ as in (\ref{DefineV}) and write $R_H = [r_1, r_2, \ldots, r_n]'$. Obtain $v_1,\ldots,v_K$ by computing the convex hull of $r_1,  \ldots, r_n$  (this recovers the ideal simplex as long as each community has a pure node) and write $V_H=[v_1,v_2,\ldots,v_K]'$. (2) Retrieve $\Pi = [\pi_1, \pi_2, \ldots, \pi_n]'$:  (2a) Obtain $w_i$ from $r_i$ and $V_H$ using the linear equations $\sum_{k=1}^K w_i(k)v_k=r_i$  and $\sum_{k=1}^K w_i(k)=1$. (2b) Obtain $Z_H =V_H (H'\Omega H)^{-1}V_H'$ and $\diag(P\eta)=[\diag(Z_H)]^{-1/2}$.   
(2c) Obtain $\pi_i$ from $w_i$ and $P\eta$ using the relationship $w_i\propto (\pi_i\circ P\eta)$ and the fact that $\|\pi_i\|_1=1$. 
(3) Obtain $P = [\diag(Z_H)]^{-1/2} Z_H [\diag(Z_H)]^{-1/2}$. (4) Obtain $\theta_i$ from $\Omega$, $\pi_i$ and $P\eta$ using the last item of Lemma~\ref{lem:AFMoracle}, and write $\Theta=\diag(\theta_1,\theta_2,\ldots,\theta_n)$. Combining these with Lemma \ref{lem:AFMoracle} gives the following lemma, the proof of which is elementary so is omitted.  
\begin{lemma} \label{lemma:AFMoracle-Add} 
Fix $(\Omega, H, K)$. Under the condition of Lemma \ref{lem:AFMoracle}, suppose each community has at least one pure node.   
The oracle procedure above retrieves the matrices $(\Theta, \Pi, P)$ exactly. 
\end{lemma}

%%%%%%%%%%%
%%%%%%%%%%%
%%%%%%%%%%%
\spacingset{1}
\begin{algorithm}[tb!]
\caption{The GoF-MSCORE algorithm for GoF of the DCMM model}\label{alg:afmSCORE}
\medskip
\noindent
{\bf Input:} The adjacency matrix $A\in\mathbb{R}^{n,n}$, and the number of communities $K$ 
\begin{enumerate} 
\item {\it (Initial estimate of $\Pi$)}.  Apply the orthodox MSCORE to obtain $\widehat{\Pi}^{\mathrm{MS}} = [\hat{\pi}^{\mathrm{MS}}_1, \ldots, \hat{\pi}^{\mathrm{MS}}_n]'$.   
\item ({\it Obtain $H$ by net-rouding}). For $1 \leq i \leq n$, suppose $\hat{\pi}_i^{\mathrm{MS}}(k)$ is the largest entry in $\hat{\pi}_i^{\mathrm{MS}}$ (if there exists more than one largest entry, pick the smallest $k$). Let $\hat{\pi}_{0i} = e_k$ (the $k$th standard basis of $\mathbb{R}^K$). Output $\widehat{H} =\widehat{\Pi}_0= [\hat{\pi}_{01}, \hat{\pi}_{02}, \ldots, \hat{\pi}_{0n}]'$
\item {\it (Re-Estimate $\Pi$)}. Let $\widehat{\Pi}= [\hat{\pi}_1,\hat{\pi}_2,\ldots,\hat{\pi}_n]'$, where $\hat{\pi}_1, \hat{\pi}_2,  \ldots, \hat{\pi}_n$ are computed as follows: 
\begin{itemize}
\item[(3.1)] {\it (Node embedding and Vertex Hunting)}. Obtain
$\widehat{R}_H  =\diag(A {\bf 1}_n)^{-1}A \widehat{H}$, and denote by $\hat{r}_i'$ the $i$th row of $\widehat{R}_H$, for $1\leq i\leq n$. 
Apply the successive projection algorithm to $\hat{r}_1,  \ldots, \hat{r}_n$ to obtain $\widehat{V}_H = [\hat{v}_1, \hat{v}_2, \ldots,  \hat{v}_K]'$, which contains the estimated vertices. In the rare event that 
$\widehat{V}_H$ is singular, re-set $\widehat{V}_H=I_K$. 
\item[(3.2)] Compute the barycentric coordinates $\hat{w}_i=(\widehat{V}_H^{-1})'\hat{r}_i$. 
\item[(3.3)] Obtain $\widehat{Z}_H = \widehat{V}_H (\widehat{H}'A \widehat{H})^{-1}\widehat{V}_H'$. 
Let $\widehat{P\eta}\in\mathbb{R}^K$ be the vector containing the diagonal entries of $[\diag(\widehat{Z}_H)]^{-1/2}$. In the rare event that $\widehat{Z}_H(k,k)\leq 0$ for some $k$, set the $k$th entry of $\widehat{P\eta}$ to $|\widehat{Z}_{H}(k,k)|^{-1/2}$ if $\widehat{Z}_{H}(k,k)\neq 0$ and $1$ if $\widehat{Z}_H(k,k)=0$.  
\item[(3.4)] Let $\hat{\pi}^*_i=[\diag(\widehat{P\eta})]^{-1}\hat{w}_i$ and $\hat{\pi}_i=\|\hat{\pi}^*_i\|^{-1}_1 \hat{\pi}^*_i$. 
\end{itemize}
\item {\it (Estimate $\Theta$, $P$ and $\Omega$)}. Let $\widehat{P}=[\diag(\widehat{Z}_H)]^{-1/2} \widehat{Z}_H [\diag(\widehat{Z}_H)]^{-1/2}$ and $\widehat{\Theta}=\diag(\hat{\theta}_1,\hat{\theta}_2,\ldots,\hat{\theta}_n)$, where for each $1\leq i\leq n$, $\hat{\theta}_i=e_i'A{\bf 1}_n/\|\hat{\pi}_i\circ\widehat{P\eta}\|_1$. 
Let $\widehat{\Omega}=\widehat{\Theta}\widehat{\Pi}\widehat{P}\widehat{\Pi}'\widehat{\Theta}$.  
\item {\it (GoF for DCMM)}.  Obtain the GoF metric for DCMM by $T_n(\widehat{\Omega})$ as in 
(\ref{DefineTn}). 
\end{enumerate}
{\bf Output:}  $\widehat{\Pi}=[\hat\pi_1,\ldots,\hat\pi_n]'$, $\widehat{\Theta}$, $\widehat{P}$,   $\widehat{\Omega}$, and $T_n(\widehat{\Omega})$ (or $T_n(\widehat{\Omega}^{{\rm DCMM}})$ to be more specific).  
\end{algorithm}
\spacingset{1.75}
%%%%%%%%%%%%%%%%
%%%%%%%%%%%%%%%%%
%%%%%%%%%%%%%%%%%%%%

We now extend the oracle procedure to the real case: In the input, we replace $(\Omega, H, K)$ by 
$(A, \widehat{H}, K)$ for an $\widehat{H}$ to be introduced. The most challenging part is to estimate the vertices $v_1, v_2, \ldots, v_K$ of the ideal simplex (i.e., Vertex Hunting). 
\bed 
Vertex Hunting (VH)  is the process of estimating $v_1, v_2, \ldots, v_K$,  the $K$ vertices of the ideal simplex 
prescribed by Lemma \ref{lem:AFMoracle} that lies on a hyperplane of $\mathbb{R}^K$.  
\eed 
In the oracle case, computing the convex hull of $r_1, r_2, \ldots, r_n$  gives exactly the ideal simplex. In the real version, we must replace $R_H$ by its noisy counterpart $\widehat{R}_H: = \diag(A\widehat{H})^{-1} A \widehat{H}$ (in the subscript, we drop the hat in $\widehat{H}$ for simplicity).  Write
$\widehat{R}_H = [\hat{r}_1, \ldots, \hat{r}_n]'$.  Due to noise  
corruption, the convex-hull approach no longer works, and we need  a different approach.  
There are existing Vertex Hunting algorithms; e.g., 
successive projection (SP) and $K$-nearest neighborhood successive projection (KNN-SP); see \cite[Section 3.4]{SCORE-Review} for a survey. 
%SP is an iterative algorithm with $K$ iterations. In the first iteration, it finds the row of $\widehat{R}_H$ with the largest Euclidean norm and uses it as the first vertex $\hat{v}_1$. In the $k$th iteration, it projects all rows of $\widehat{R}_H$ into the subspace orthogonal to the previous $(k-1)$ vertices, finds the row of $\widehat{R}_H$ with the largest Euclidean norm after projection, and uses this row as $\hat{v}_k$ (see the supplement). The output of SP is a (relatively) explicit function of $\widehat{R}_H$, so 
SP is convenient to analyze  and enjoys nice theoretical properties,   
but SP is vulnerable to outliers;  and numerically, it frequently underperforms KNN-SP.  
For these reasons, we use SP for theoretical study and KNN-SP for numerical study, as recommended by the literature 
(e.g., \cite{MSCORE, SCORE-Review}).

After we obtain an estimate for $v_1, v_2, \ldots,v_K$,   we 
extend the oracle procedure to the real case. What remains is to  choose a data-driven $\widehat{H}$.  
A  good $\widehat{H}$ should satisfy two requirements:  
(i) To remove the analytical hurdle  (see Section \ref{subsec:SCChurdle}), 
we need to have a non-stochastic matrix $H$ such that 
$\mathbb{P}(\widehat{H} \neq H)= o(1)$. (ii)  
The absolute eigenvalues of $G_H=\Pi'\Theta H$ are properly large (as $\Theta\Pi$ is a basis for the column span of $\Omega$, this requires the sin-theta distance between the column span of $H$ and that  of $\Omega$ to be properly small).\spacingset{1}\footnote{We can relax (i)-(ii) to that of $\mathbb{P}(\widehat{H} \notin \{H_1, \ldots, H_N\}) =   o(1)$ for a finite number of non-stochastic matrices $H_1, \ldots, H_N$, where the absolute eigenvalues of 
$\Pi' \Theta H_m$ are properly large for all $1 \leq m \leq N$.}\spacingset{1.75}
One may take $\widehat{H} = \widehat{\Pi}^{\mathrm{MS}}$, where 
$\widehat{\Pi}^{\mathrm{MS}}= [\hat{\pi}^{\mathrm{MS}}_1, \ldots, \hat{\pi}^{\mathrm{MS}}_n]'$ is from MSCORE \citep{MSCORE}. However, in DCMM, $\pi_i$ is a continuous variable and  $\hat{\pi}_i^{\mathrm{MS}}$ may take infinitely many values. 
Hence,  $\widehat{H} = \widehat{\Pi}^{\mathrm{MS}}$ does not satisfy Requirement (i).

Fortunately, we can use $\widehat{\Pi}^{\mathrm{MS}}$ to construct a desirable $\widehat{H}$ by {\it net-rounding}.   Let $S_0 = \{x \in \mathbb{R}_+^K: \sum_{k = 1}^K x_k = 1\}$ be the standard simplex in $\mathbb{R}^K$.  A net on $S_0$ is a finite-size subset ${\cal N}$ of $S_0$.  
One example is the $\epsilon$-net ${\cal N}_{\eps}$, 
where for any $a \in S_0$,  
there is $b \in {\cal N}_{\eps}$ so that $\|a - b\| \leq \eps$. 
Another example is ${\cal N}^*=\{e_1,e_2,\ldots,e_K\}$ ($e_k$: $k$-th standard Euclidean basis vector of $\mathbb{R}^K$).  To apply net-rounding to $\widehat{\Pi}^{\mathrm{MS}}$, we fix a net ${\cal N}$, and for each $1 \leq i \leq n$,  
we replace $\hat{\pi}_i^{\mathrm{MS}}$ by the closest point in the net, denoted by $\hat{\pi}_{i0}$ 
(we break ties lexicographically).  We use the resultant matrix as 
$\widehat{H}$.
The purpose of net-rounding is to discretize $\widehat{\Pi}^{\mathrm{MS}}$ to achieve $\mathbb{P}(\widehat{H}\neq H)=o(1)$ for some non-stochastic $H$.
We may use other discretization ideas (e.g., \cite{shen2022one} applied k-means iteratively to discretize graph embeddings), but other ideas lack theory to guarantee $\mathbb{P}(\widehat{H}\neq H)=o(1)$.

The choice of the net ${\cal N}$ is not unique.  To avoid over-fitting, we prefer to select a relatively simple ${\cal N}$. In light of this, we recommend to take  
${\cal N} = {\cal N}^*  = \{e_1, e_2, \ldots, e_K\}$   (net-rounding reduces to clustering nodes into $K$ groups, depending on which of the $K$ entries of $\hat{\pi}^{\mathrm{MS}}_i$ is the largest one; see Step 2 of Algorithm~\ref{alg:afmSCORE}). 
For numerical study with either real or simulated data,   we do not tune our algorithm over different choices of ${\cal N}$ to avoid over-fitting. Fortunately, it turns out such a choice works well  both for our theoretical and numerical study.

Finally, combining these ideas gives rise to a new GoF metric $T_n(\widehat{\Omega})$ for DCMM which we call GoF-MSCORE; see Algorithm \ref{alg:afmSCORE} and  Figure \ref{fig:flow}.  
We can view GoF-MSCORE as a generic algorithm: for 
a carefully chosen $\widehat{H}$, we use Algorithm \ref{alg:afmSCORE}
with $(A, \widehat{H}, K)$ as input but with Steps 1-2 of Algorithm \ref{alg:afmSCORE} skipped;   
let $T_n(\widehat{\Omega}; \widehat{H})$ be the resultant GoF metric. 
In this sense, Algorithm \ref{alg:afmSCORE} is a special case of generic GoF-MSCORE where $\widehat{H}$ is computed by Steps 1-2.

{\bf Remark 2} {\it (Comparison with MSCORE \citep{MSCORE})}.  GoF-MSCORE 
is not simply using MSCORE twice: Step 3 in Algorithm \ref{alg:afmSCORE} is different from the MSCORE in \citep{MSCORE}. 
Although both Step 3 and MSCORE use a low-dimensional node embedding and explore a simplex structure, the  embedding in MSCORE is based on eigenvectors, while the one in Step 3 is by projecting 
$A$ to the column space of $\widehat{H}$ (and many choices of $\widehat{H}$ may work).  
The simplex structures are not identical either.   The new approach is flexible in choosing an analytically friendly $\widehat{H}$.

In Algorithm \ref{alg:afmSCORE}, we use Vertex Hunting (VH)  in both Step 1 and Step~3.1.   For VH, following the recommendation by \cite{MSCORE, SCORE-Review}, we use SP for theoretical study   and KNN-SP for real-data analysis; our numerical study confirms that SP may significantly underperform KNN-SP.   Note that, first, SP is tuning-free, and second, aside from the VH steps, GoF-MSCORE is tuning-free.
To apply KNN-SP to data points $x_1, \ldots, x_n$,  we fix $\alpha > 0$ and an integer $N \geq 1$.   For each $x_i$, let $S_i$ be the set of $N$ nearest neighbors of $x_i$ falling within a distance of $(\max_{j,k}\|x_i-x_k\|) / \alpha$ to $x_i$ 
(including $x_i$ itself).  If $|S_i| \leq 2$, we prune $x_i$ out; otherwise, we replace $x_i$ by the average of points in $S_i$. We then apply SP  \citep{araujo2001successive} to the resultant set of points.   For real networks,  we set $\alpha = 20$ if $[n/K]>20$ and $\alpha = 5$ if $[n/K] \leq 20$ (in this case, $n$ is small, so we need to choose a smaller $\alpha$, otherwise, we may have too many empty sets $S_i$).
We set $N$ as the integer closest to $(m_0+1) \min\{10, n/10\}$,  where $m_0$ is the integer closest to $[K (\bar{d} / d_{min}) / 250]^2$ 
($\bar{d}$ and $d_{min}$ denote the average and minimum degrees of the network, respectively).  Intuitively, $m_0$ should be a super-linear function of $K (\bar{d} / d_{min})$, 
which explains the square here (for all data sets in Table \ref{tab:realdata0}, 
$m_0 \leq 3$).   Our results are relatively insensitive to different choices of $(N, \alpha)$: see Table~\ref{tab:tuning}.  For all real networks,  
we use the same KNN-SP algorithm with the same $(N, \alpha)$ specified above without further tuning. 
See Section \ref{sec:numeric} and \cite{SCORE-Review} for details.

The computational cost of GoF-MSCORE is from (a) initializing by MSCORE and obtaining $\widehat{H}$ from net-rounding, (b) constructing $\widehat{\Omega}$, and (c) computing $T_n(\widehat{\Omega})$.  
In (a)-(b), except for the two VH steps (in MSCORE and in Step~3.1 of Algorithm~\ref{alg:afmSCORE}), 
the other steps (including obtaining the leading eigenvectors in MSCORE and other explicit operations) have a complexity of $O(n^2K^2)$. 
For the VH steps, 
the complexity depends on which algorithm is used. In this paper, we use either SP or KNN-SP, 
whose complexity is $O(nK^3)$. Therefore, the complexity of part (a)-(b) is polynomial. 
For part (c), it was known (\cite[Theorem 1.1]{JKL2019} or Section~\ref{subsec:proof-diffU} of the supplement) that $T_n(\widehat{\Omega})$ can be written as a matrix function of $(A, \widehat{\Omega})$ that only involves matrix multiplication, trace, and Hadamard product; in addition, $A$ has many zeros, and $\widehat{\Omega}$ has a given rank-$K$ factorization $\widehat{\Omega}=\widehat{\Theta}\widehat{\Pi}\widehat{P}\widehat{\Pi}'\widehat{\Theta}$. Consequently, the complexity of computing $T_n(\widehat{\Omega})$ does not exceed $O(n^2\bar{d}+n^2K^2)$. This shows that  part (c) is also polynomial. 

%{\bf Remark 3}.   By Algorithm \ref{alg:afmSCORE}, each of $\widehat{\Pi}, \widehat{P}, \widehat{\Theta}$ is a relatively simple and explicit function of $A$, 
%$\widehat{H}$ and $\widehat{V}$, so $\widehat{\Omega}$ is also such a function of $A, \widehat{H}, \widehat{V}$.  
%Moreover, for a non-stochastic matrix $H$, we can show that $\mathbb{P}(\widehat{H} \neq H) = o(1)$, 
%so for our purpose, we can pretend $\widehat{H} = H$, and so $\widehat{\Omega}$ is also an explicit function of $A, H, \widehat{V}$. Last, note that $\widehat{V}$ is obtained by applying SP  to the matrix 
%$\widehat{R}_H = (\mathrm{diag}(A {\bf 1}_n)^{-1}  A H$, where each column of $\widehat{V}$ equals to a specific row of $\widehat{R}_H$, so $\widehat{V}$ is also an explicit function of $A, H$. Therefore, 
%approximately, $\widehat{\Omega}$ is a relatively simple and explicit function of $A$ and $H$. 
%This overcomes the analytical hurdle mentioned in Section \ref{subsec:SCCoracle}. 

{\bf Remark 3} {\it (How we overcome the analytical hurdle)}.  By Algorithm~\ref{alg:afmSCORE}, $(\widehat{\Pi}, \widehat{P}, \widehat{\Theta})$ are  simple functions of $(A, \widehat{H}, \widehat{V})$, so $\widehat{\Omega}$ is an explicit function of $(A, \widehat{H}, \widehat{V})$. We recall that $\widehat{V}$ is obtained by applying successive projection to 
$\widehat{R}_H = [\mathrm{diag}(A {\bf 1}_n)]^{-1} A \widehat{H}$. This algorithm ensures that each column of $\widehat{V}$ is a specific row of 
$\widehat{R}_H$; hence, $\widehat{V}$ is an explicit form of $(A, \widehat{H})$. It follows that $\widehat{\Omega}$ is an explicit function of $(A, \widehat{H})$. 
Also, the $\widehat{H}$ from net rounding is such that $\mathbb{P}(\widehat{H} \neq H) = o(1)$ for a non-stochastic $H$. 
To show weak convergence of $T_n(\widehat{\Omega})$, we can replace $\widehat{H}$ by $H$ and treat $\widehat{\Omega}$ as an explicit function of $(A, H)$. 
%approximately, $\widehat{\Omega}$ is a relatively simple and explicit function of $A$ and $H$. 
This overcomes the analytical hurdle in Section~\ref{subsec:SCChurdle}.

%%%%%%%%%%
%%%%%%%%%%
%%%%%%%%%%%%
\subsection{Asymptotic normality of the GoF-MSCORE} \label{subsec:normality}
Fix $K\geq 1$. Consider the DCMM model \eqref{modelmatrixform}. As in Section~\ref{subsec:model}, we assume $P$ has unit diagonal (for identifiability). 
 For any $\eta\in\mathbb{R}^n$, let $\|\eta\|$, $\|\eta\|_1$, $\eta_{\max}$ and $\eta_{\min}$ denote its Euclidean norm,  $\ell^1$-norm (absolute sum of entries), maximum entry, and minimum entry, respectively.  For any $M\in\mathbb{R}^{K\times K}$, let  $\|M\|$ and $\|M\|_{\max}$ be its spectral norm and entry-wise maximum norm, respectively. 
Define $G=\|\theta\|^{-2}\Pi'\Theta^2\Pi\in\mathbb{R}^{K\times K}$. Let $\lambda_k(PG)$ be the $k$th largest (in magnitude) right eigenvalue of $PG$, $1\leq k\leq K$. We assume the following regularity condition. 
%%%%%%%%%%%
%%%%%%%%%%%
\begin{cond}\label{cond:afmSCORE} 
Let $\beta_n\in (0,1]$ be a sequence indexed by $n$. Let $c_1$-$c_3$ be positive constants. 
\begin{itemize} 
\setlength \itemsep{-.5 em} 
\item[(a)]  $\|P\|_{\max}\leq c_1$, $\|G^{-1}\|\leq  c_1$,  $\theta_{\max}\sqrt{\log(n)}\leq c_3$, and $(\|\theta\|_5^{5}\|\theta\|_1 / \|\theta\|^6) \log(n)\to 0$.  
\item[(b)] $|\lambda_K(PG)|\geq c_2\beta_n$ and $\beta_n\|\theta\|/\sqrt{\log(n)}\to\infty$. 
\end{itemize}
\end{cond} 

In (a), the first two items are about $(P, G)$.  Since $P$ has unit diagonals and $G$ is properly scaled ($\|G\|$ is always bounded), these conditions are mild.  
The next two items are about $\theta$, where the first one is mild:  since most real networks are sparse,   $\theta_{\max}$ is usually much smaller than $1$.\spacingset{1}\footnote{In the dense case of $\theta_{\max}\asymp 1$,  if we use a variant of SCC where we replace $C_{n, 3}$ in the denominator by $V_{n,3}=\sum_{i_1,i_2,i_3 (dist)} (\widehat{\Omega}_{i_1i_2}-\widehat{\Omega}^2_{i_1i_2})(\widehat{\Omega}_{i_2i_3}-\widehat{\Omega}^2_{i_2i_3})(\widehat{\Omega}_{i_3i_1}-\widehat{\Omega}^2_{i_3i_1})$, then all our results can be extended.}\spacingset{1.75}
The last one is also mild. 
To see the point, take the special case of $\theta_{\max}\asymp\theta_{\min}$. In this case, $(\|\theta\|_5^{5}\|\theta\|_1 /  \|\theta\|^6) \log(n) \asymp n^{-1} \log(n)$, and the condition reduces to $n^{-1} \log(n) \goto 0$.  In the literature (e.g., \cite{JKL2019, EstK}), these conditions are known as standard. 

Condition (b) is our main assumption, which is known to be nearly necessary and hard to relax. 
To see the point, denote by $\lambda_k(\Omega)$ the $k$th largest eigenvalue of $\Omega$ (in magnitude). 
It can be shown that $\lambda_k(\Omega)=\|\theta\|^2\lambda_k(PG)$ (in the supplement).  
Using this, Condition (b) boils down to $|\lambda_K(\Omega)|/\sqrt{\lambda_1(\Omega)\log(n)}\to\infty$.  It was noted by \cite{JKL2019, EstK} that $|\lambda_K(\Omega)|/\sqrt{\lambda_1(\Omega)}$ 
is the effective Signal-to-Noise Ratio (SNR) of network testing. In particular, it was shown in 
\cite{EstK} that when $|\lambda_K(\Omega)|/\sqrt{\lambda_1(\Omega)}\to 0$, a 
consistent estimate for $K$ does not exist. In such a case, 
the problem of  GoF is not well-posed. 
This suggests that Condition (b) is nearly necessary.

Recall that in Algorithm~\ref{alg:afmSCORE} we use MSCORE and net-rounding to obtain $\widehat{H}$. We need the following condition to ensure the good performance of $\widehat{H}$.  For each $i$,  let $g_i \geq 0$ be the gap between the largest and second largest entries of $\pi_i$, and let $k_i^*$ be index of  the largest entry of $\pi_i$ ($k_i^*$ is uniquely defined if $g_i > 0$). Define $\Pi_0 =   
[\pi_{01}, \pi_{02},\ldots,\pi_{0n}]'$, where  $\pi_{0i} =e_{k_i^*}$  and $e_k$ is the 
$k$-th standard Euclidean basis vector of $\mathbb{R}^K$.  
It is seen that if we apply net-rounding to $\Pi$ using the net ${\cal N}^*=\{e_1,\ldots,e_K\}$, then $\Pi_0$ is the output. 
Same as before, let $\lambda_k=\lambda_k(PG)$ be the $k$th right eigenvalue of $PG$, and let $\mu_k\in\mathbb{R}^K$ be the corresponding right eigenvector. Let  $\kappa(\cdot)$ be the conditioning number of a matrix. Write $\delta_n= \theta_{\min}\|\theta\|/(\theta_{\max}\sqrt{\theta_{\max}\|\theta\|_1})$.

\begin{cond}\label{cond:MSCORE} 
Let $c_4\in (0,1)$ be a constant. We assume:
(a) Each community has at least one pure node. (b) $|\lambda_2|\leq (1-c_4)\lambda_1$, and $\mu_1$ is a positive vector with $\frac{\max_{1\leq k\leq K}\mu_1(k)}{\min_{1\leq k\leq K}\mu_1(k)}\leq c_4^{-1}$. (c) $\delta_n \beta_n\|\theta\|/\log(n)\to\infty$. (d) $\kappa(\Pi'\Theta\Pi_0)\leq c_4^{-1}$. (e) For each $1\leq i\leq n$, $g_i \geq (\delta_n\beta_n\|\theta\|)^{-1}\log(n)$.   
\end{cond} 

Here, (a)-(c) are standard mild conditions required for MSCORE; see \cite{MSCORE} for justification.  
In these requirements, $\delta_n\beta_n\|\theta\|/\log(n)\to\infty$ is slightly stronger than $\beta_n\|\theta\|/\sqrt{\log(n)}\to\infty$ in Condition~\ref{cond:afmSCORE}, as we need   to guarantee that $\widehat{H}$ is concentrated at $\Pi_0$. 
Moreover,  (d) is also a mild regularity condition, which excludes the cases where the sine-theta distance between the column spaces of $\Theta^{1/2} \Pi$ and $\Theta^{1/2} \Pi_0$ is large. The condition (e) guarantees that the output of net-rounding is unique: Define $\hat{g}_i$ similarly as $g_i$ where we replace $\pi_i$ by $\hat{\pi}_i^{\mathrm{MS}}$. Under this condition, it can be shown that $\hat{g}_i>0$ with high probability, ensuring that there is only a unique point in ${\cal N}^*$ that is closest to $\hat{\pi}_i^{\mathrm{MS}}$.
Since $(\delta_n\beta_n\|\theta\|)^{-1}\log(n)=o(1)$ (as implied by (c)), this condition excludes only a tiny subset from the probability simplex, which is mild.

The next theorem studies Steps 1-2 of Algorithm~\ref{alg:afmSCORE} and is proved in the supplement. 
%%%%%%%%%%
%%%%%%%%%%
%%%%%%%%%%
\begin{thm}[MSCORE and net-rounding]\label{thm:MSCORE}
Fix $K\geq 1$. Consider the DCMM model \eqref{modelmatrixform}, where Conditions~\ref{cond:afmSCORE}-\ref{cond:MSCORE} hold. Denote by $\widehat{\Pi}^{\mathrm{MS}} = [\hat{\pi}^{\mathrm{MS}}_1, \ldots, \hat{\pi}^{\mathrm{MS}}_n]'$ the output of MSOCRE. As $n \goto \infty$, with probability $1 - O(n^{-3})$, up to a column permutation of $\widehat{\Pi}^{\mathrm{MS}}$, 
$\max_{1 \leq i \leq n} \{\|\hat{\pi}^{\mathrm{MS}}_i - \pi_i\|_1\} \leq C(\delta_n\beta_n\|\theta\|)^{-1}\sqrt{\log(n)}$. 
Moreover, let $\widehat{H}$ be the output from Steps 1-2 of Algorithm~\ref{alg:afmSCORE}.  As $n\to\infty$, 
$\mathbb{P}(\widehat{H}=\Pi_0)=1-O(n^{-3})$.  
\end{thm}

 Here, the first claim  contains node-wise error bounds for MSCORE. It improves the theory in \cite{MSCORE} by removing some unnecessary conditions (e.g., \cite{MSCORE} assumes that all except the first eigenvalues of $PG$ are at the same order, which is not needed here). 
Our proof uses a more refined entry-wise eigenvector analysis of $A$ (see Section~\ref{sec:MSCORE-supp}).  The second claim says that  the $\widehat{H}$ produced by MSCORE and net-rounding concentrates at a non-stochastic matrix $\Pi_0$.  

%%%%%%%%%%%%%%%%%%%%%%%%%%%%%%%%%%%%%
%%%%%%%%%%
%%%%%%%%%%
%%%%%%%%%%
\begin{thm}[Asymptotic normality of GoF-MSCORE (DCMM)]\label{thm:main}
Fix $K\geq 1$. Consider the DCMM model \eqref{modelmatrixform}, where Condition~\ref{cond:afmSCORE} holds. Let $T_n(\widehat{\Omega})$ be the output of Algorithm~\ref{alg:afmSCORE}, and let  $T_n(\widehat{\Omega}; \widehat{H})$ be the output of the generic GoF-MSCORE with a given $\widehat{H}$.  As $n \goto \infty$,  if there is a non-stochastic $H$ such that $\kappa(\Pi'\Theta H)\leq C$ and $\mathbb{P}(\widehat{H}\neq H)=o(1)$, then $T_n(\widehat{\Omega};  \widehat{H}) \goto N(0, 1)$. 
Moreover, if Condition~\ref{cond:MSCORE} also holds, then $T_n(\widehat{\Omega}) \goto N(0, 1)$. 
\end{thm}

Theorem \ref{thm:main} analyzes Steps 3-5 of Algorithm~\ref{alg:afmSCORE} and contains our main result for GoF. Here, to show the asymptotic normality of the generic 
GoF-MSCORE, we only need Condition~\ref{cond:afmSCORE} and that the $\widehat{H}$ satisfies $\mathbb{P}(\widehat{H} \neq H) = o(1)$; Condition~\ref{cond:MSCORE} is not needed. 
Therefore,  the generic GoF-MSCORE is broadly applicable, not tied to the specific choice of $\widehat{H}$ in Algorithm~\ref{alg:afmSCORE}. Investigating other choices of $\widehat{H}$ is an interesting future research direction. 
 
The proof of Theorem~\ref{thm:main} is  technically involved. The key is to 
show $|T_n(\widehat{\Omega}) - T_n(\Omega)| \goto0$   in probability. To show this, we must decompose $T_n(\widehat{\Omega}) - T_n(\Omega)$ as the sum of {\it many terms}  (see Lemma~\ref{lem:diffU}) and analyze each of them separately. Especially, there does not exist a framework where we can analyze these terms uniformly (e.g., \cite{JKL2019}), and for each term, we need to carefully account for the dependence between $W$ and $(\widehat{\Omega}-\Omega)$ and deal with complicated combinatorics. 
%All these make the proof hard, delicate, and long. 
%See the supplement for details. 
%

%%%%%%%%%%%%
%%%%%%%%%%%%
%%%%%%%%%%%%
\subsection{Extension of GoF-MSCORE from DCMM to MMSBM} 
\label{subsec:MMSBM} 
MMSBM is a special case of DCMM where $\Omega =  \Pi P \Pi'$. 
Recall that in DCMM (e.g.,  (\ref{modelmatrixform})), for identifiability, we require $P$ to have 
unit-diagonal.  For MMSBM, we remove the constraint for it is not only inappropriate but 
also not required for identifiability. Instead, letting $\alpha_n = (1/K) \mathrm{trace}(P)$, we  
 write $\Omega = \alpha_n \Pi (\alpha_n^{-1} P) \Pi'$, where $\mathrm{trace}(\alpha_n^{-1} P) = K$, so 
$\alpha_n^{-1} P$ is on the same scale as the $P$ in DCMM. 
In light of this, we  
use $\alpha_n^{-1}P$ as the new $P$ and
rewrite MMSBM as  
\begin{equation}  \label{MMSBM}
A=\Omega-\diag(\Omega)+W, \qquad \mbox{and} \qquad \Omega = \alpha_n \Pi P \Pi', \qquad \mathrm{trace}(P) = K. 
\end{equation}

In such a setting, the main idea of GoF-MSCORE continues to work.   
In detail, let $(\lambda_k, \xi_k)$ be the $k$-th eigen-pair of
$\Omega$ as before.  First, similar to \cite{MSCORE}, we can show that all the $n$ rows of $\Xi = [\xi_1, \xi_2, \ldots, \xi_K]$ are contained in a simplex (Ideal Simplex 1) with $K$ vertices in (a hyperplane of)  $\mathbb{R}^K$, 
where each vertex is a row of $\Xi$. Therefore, 
we can extend MSCORE to estimate $\Pi$ in the current setting. 
Moreover, fix a matrix $H$ as in Lemma \ref{lem:AFMoracle} and let $R_H = \Omega H$ (note that the form of $R_H$ is simpler than that in Lemma \ref{lem:AFMoracle}). 
Similar to Lemma~\ref{lem:AFMoracle}, the $n$ rows of $R_H$ are contained in a simplex (Ideal Simplex 2) with $K$ vertices in $\mathbb{R}^K$, where each vertex is a row of $R_H$.  With these simplex structures, it is not hard to extend the GoF-MSCORE to the current setting. 
Since the idea is similar, we omit the details.

To this end, we propose {\it a revised (and simpler) version of GoF-MSCORE}, denoted by 
GoF-MSCORE-rev, as a GoF metric for MMSBM. Let $(\hat{\lambda}_k, \hat{\xi}_k)$ be the 
$k$-th eigen-pair of $A$ and let $\widehat{\Xi} = [\hat{\xi}_1, \hat{\xi}_2, \ldots, \hat{\xi}_K]$. GoF-MSCORE-rev runs as follows.  Input: $(A, K)$. 
\begin{itemize} 
\setlength \itemsep{-.25 em} 
\item ({\it Initial estimate of $\Pi$}). Apply vertex hunting to the $n$ rows of $\widehat{\Xi}$ (e.g., using successive projection), 
%using successive projection (SP) as before 
and let $\hat{v}_1^*, \ldots, \hat{v}_K^*$ be the estimated vertices. Let  $\widehat{V}^*=[\hat{v}_1^*,\ldots,\hat{v}_K^*]$.  For $1\leq i\leq n$, let $\tilde{\pi}_i=(\widehat{V}^*)^{-1}\widehat{\Xi}'e_i$, and obtain $\hat{\pi}_i^{\mathrm{MS}_0}$ by setting negative entries of $\tilde{\pi}_i$ to zero and then re-normalizing it to have a unit $\ell^1$-norm. Write $\widehat{\Pi}^{\mathrm{MS}_0}=[\hat{\pi}_1^{\mathrm{MS}_0},  \ldots, \hat{\pi}_n^{\mathrm{MS}_0}]'$.  
\item ({\it Net-rounding}).  Obtain  $\widehat{H} \in \mathbb{R}^{n, K}$ by applying net-rounding to 
$\widehat{\Pi}^{\mathrm{MS}_0}$ with ${\cal N} = {\cal N}^*$.  
\item ({\it Re-estimate of $\Pi$}).  Obtain $\widehat{R}_H=A \widehat{H}$ and denote by $\hat{r}_i'$ the $i$th row of  $\widehat{R}_H$, $1\leq i\leq n$. Apply vertex hunting to the $n$ rows of $\widehat R_H$ and let $\hat{v}_1,\ldots,\hat{v}_K$ be the estimated vertices. 
Let $\widehat{V}_H=[\hat{v}_1,\ldots,\hat{v}_K]'$. In the rare event that $\widehat{V}_H$ is singular, re-set $\widehat{V}_H=I_K$.  Estimate $\pi_i$ by  $\hat{w}_i=(\widehat{V}_H^{-1})'\hat{r}_i$ and write $\widehat{\Pi}=[\hat{w}_1,\ldots,\hat{w}_n]'$. 
\item Estimate $\alpha_nP$ by $\widehat{\alpha_nP}=(\widehat\Pi'\widehat \Pi)^{-1}\widehat\Pi'A\widehat\Pi(\widehat\Pi'\widehat \Pi)^{-1}$. Let $\widehat{\Omega}  =\widehat{\Pi}\cdot\widehat{\alpha_nP}\cdot\widehat{\Pi}'$ and output $T_n(\widehat{\Omega})$.  
\end{itemize}

%Similarly as before, we may use GoF-MSCORE-rev as it is, or as a generic 
%algorithm where we input a $\widehat{H}$ but skip the first two steps. 
%In addition, 
For vertex hunting, similarly as before, we use SP for theoretical study and KNN-SP (where tuning parameters $(N, \alpha)$ are set in the same way as in GoF-MSCORE) for real data analysis.  
Since MMSBM is a special DCMM, we adopt similar regularity conditions as in Section \ref{subsec:normality} but 
replace the matrix $\Theta$ there (which is general) by  $\Theta = \sqrt{\alpha_n}  I_n$. Let $\Pi_0$ be defined in the same way by applying net-rounding to $\Pi$. 
Theorem \ref{thm:MMSBM}  is proved in the supplement.

%%%%%%%%%%
%%%%%%%%%%
%%%%%%%%%%
\begin{thm}\label{thm:MMSBM}
Fix $K\geq 1$. Consider the MMSBM model \eqref{MMSBM}, where Condition~\ref{cond:afmSCORE} and (a) and (c)-(e) of Condition~\ref{cond:MSCORE} hold with $\Theta=\sqrt{\alpha_n}I_n$. Let $\widehat{H}$ and $T_n(\widehat{\Omega})$ be as in GoF-MSCORE-rev. As $n\to\infty$, $\mathbb{P}(\widehat{H}=\Pi_0)=1-O(n^{-3})$, and $T_n(\widehat{\Omega})\to N(0,1)$. 
\end{thm}

%Theorem~\ref{thm:MMSBM} can not deduced from Theorems~\ref{thm:MSCORE}-\ref{thm:main} for we use 
%different steps to compute $\widehat{H}$ and $\widehat{\Omega}$ here. The regularity conditions are also different.  

\subsection{GoF-SCORE for DCBM and SBM}
\label{subsec:DCBM}
The SBM and DCBM are special cases of MMSBM and DCMM, respectively.  
In MMSBM and DCMM,  a row of $\Pi$ is a continuous variable and may take any $K$-dimensional weight-vector value, so it is impossible 
to have an estimate $\widehat{\Pi}$ such that $\mathbb{P}(\widehat{\Pi} \neq \Pi) = o(1)$. 
As explained in Section \ref{subsec:SCChurdle}, this poses an analytical hurdle. To overcome this hurdle, we need to estimate $\Pi$ twice and develop 
new ideas so that the final  estimate not only is accurate but also has an explicit and 
relatively simple form;  see Algorithm \ref{alg:afmSCORE} and Section \ref{subsec:SCChurdle}. 
In SBM and DCBM, each row of $\Pi$ is a degenerate 
weight vector, taking values in $\{e_1, e_2, \ldots, e_K\}$ ($e_k$: 
$k$-th standard basis vector in $\mathbb{R}^K$).  
In such a case, there is an estimate $\widehat{\Pi}$ so that $\mathbb{P}(\widehat{\Pi} \neq \Pi) = o(1)$ (e.g., by  SCORE \citep{SCORE-Review, EstK}), and we do not have the analytical hurdle.  As a result, for SBM and DCBM, we can use a GoF algorithm simpler than those for MMSBM and DCMM. 

Consider the DCBM first. In this model, $\Omega = \Theta \Pi P \Pi' \Theta$, but each row of $\Pi$ is a degenerate weight vector. We propose GoF-SCORE as a new GoF metric as follows. 
%%%%%%%%%%
%%%%%%%%%%
%%%%%%%%%%
\begin{itemize}
\setlength \itemsep{-.35 em} 
\item {\it (Clustering)}. Apply the SCORE algorithm \citep{SCORE-Review} to cluster nodes into $K$ groups. Let $\hat{\pi}_i = e_k$ if node $i$ is clustered to group $k$ and write $\widehat{\Pi} = [\hat{\pi}_1, \hat{\pi}_2, \ldots, \hat{\pi}_n]'$. 
\item {\it (Estimation of $\Omega$)}. 
Let $M=\widehat{\Pi}'A\widehat{\Pi}$. 
Estimate $P$ by $\widehat{P}=[\diag(M)]^{-1/2}M[\diag(M)]^{-1/2}$. Estimate $\Theta$ by $\widehat{\Theta} = \diag(\hat{\theta}_1,\hat{\theta}_2, \ldots,\hat{\theta}_n)$, with $\hat{\theta}_i =  e_i'A  {\bf 1}_n\sqrt{\hat{\pi}_i'M\hat{\pi}_i} /(\hat{\pi}_i'M{\bf 1}_K)$, for $1\leq i\leq n$.   
Output $\widehat{\Omega}^{\mathrm{DCBM}}  = \widehat{\Theta} \widehat{\Pi} \widehat{P} \widehat{\Pi}'  \widehat{\Theta}$ and $T_n(\widehat{\Omega}^{\mathrm{DCBM}})$.  
\end{itemize} 
%%%%%%%%%%%%%%%%%%%%%%%%%%%%%%%%%%%%%%%%%%%%%%%%%%%%%
%%%%%%%%%%
%%%%%%%%%%
%%%%%%%%%%
\begin{thm}\label{thm:DCBM}
Fix $K\geq 1$. Consider the DCBM model in \eqref{modelmatrixform}, where $\lambda_{\min}(\Pi'\Theta \Pi)\geq c_5\Vert \theta\Vert_1$ for a constant $c_5>0$, and Condition~\ref{cond:afmSCORE} and (b)-(c) of Condition~\ref{cond:MSCORE} are satisfied. 
As $n \goto \infty$, 
$T_n(\widehat{\Omega}^{\mathrm{DCBM}}) \goto N(0, 1)$.   
\end{thm}

Our approach can be used for estimating $K$, an interesting problem in network analysis.  
Consider a DCMM with $K$ communities but $K$ is unknown. Fix $K_0 \geq 1$ and we test $K=K_0$ versus $K=K_0+1$. We can use $T_n(\widehat{\Omega}^{\mathrm{DCBM}})$ for $K=K_0$ as the test statistic.
If we conduct this test sequentially for $K_0=1,2,\ldots,$, it provides an estimator of $K$ (see \cite{EstK}). 
In Section~\ref{subsec:power}, we show that GoF-SCORE can tell between $K=K_0$ and $K=K_0+1$ with minimal requirement on the signal-to-noise ratio. This shows that GoF-SCORE is optimal.   
See also Section~\ref{subsec:power} for comparison with existing methods of estimating $K$.

{\bf Remark 4} {\it (Comparison with \cite{EstK})}.  GoF-SCORE is connected to the method in \cite{EstK}, but their focus is on estimating $K$ via a statistic similar to $\psi_{n, 4}(\widehat{\Omega}^{\mathrm{DCBM}})$. In comparison, Theorem \ref{thm:DCBM} considers $T_n=\psi_{n, 3}(\widehat{\Omega}^{\mathrm{DCBM}})$, which was not studied in \cite{EstK}. Moreover, in Section~\ref{subsec:power}, 
we study the power of GoF-SCORE in detecting a mis-specified $K$ or mixed membership. These are new results that cannot be deduced from \cite{EstK}. 
Last, the focus of this paper is using SCC as a general recipe for GoF of different models (not limited to DCBM). To achieve this broad scope, we need methods and analysis more sophisticated  than those in \cite{EstK} (see Section~\ref{sec:SCC}).

We now consider SBM.  This is a special DCBM, where $\Omega = \alpha_n \Pi P \Pi'$. It can also be viewed a special MMSBM where each row of $\Pi$ is a degenerate weight vector. Similarly as in \eqref{MMSBM}, we assume $\mathrm{trace}(P)=K$ so that $\alpha_n$ and  $P$ are identifiable and that $P$ is on the same scale as in DCBM.   
We propose GoF-SCORE-rev as a simpler version of GoF-SCORE as follows. 
\begin{itemize} 
\setlength \itemsep{-.35 em} 
\item {\it (Spectral clustering)}. Let $\hat{\xi}_k\in\mathbb{R}^n$ be the eigenvector of $A$ associated with the $k$th largest eigenvalue (in magnitude). Apply the k-means algorithm to the rows of $\widehat{\Xi}=[\hat{\xi}_1,\ldots,\hat{\xi}_K]$, assuming $K$ clusters. Let $\hat{\pi}_i = e_k$ if row $i$ is clustered to class $k$. Write 
$\widehat{\Pi} = [\hat{\pi}_1,\ldots,\hat{\pi}_n]'$. 
\item {\it (Estimate $\Omega$)}. Let $\widehat{\alpha_n P} = (\widehat\Pi'\widehat\Pi)^{-1}\widehat\Pi' A\widehat\Pi (\widehat\Pi'\widehat\Pi)^{-1}$. Output  $\widehat{\Omega}^{\mathrm{SBM}} = \widehat{\Pi} \widehat{\alpha_n P}  \widehat{\Pi}'$ and $T_n(\widehat{\Omega}^{\mathrm{SBM}})$.      
\end{itemize} 
%%%%%%%%%%
%%%%%%%%%%
%%%%%%%%%%
%%%%%%%%%%
\begin{thm}\label{thm:SBM}
Fix $K\geq 1$. Consider the SBM model \eqref{MMSBM} (where each $\pi_i$ ranges in $\{e_1,\ldots,e_K\}$). For $1\leq k\leq K$, let $n_k$ denote the number of nodes with $\pi_i=e_k$, and let $\lambda_k$ be the $k$th largest right eigenvalue  (in magnitude) of $n^{-1}P\Pi'\Pi\in\mathbb{R}^{K,K}$. Suppose $\|P\|_{\max}\leq C$, $\alpha_n\log(n)\leq C$,  $|\lambda_K|\sqrt{n\alpha_n/\log(n)}\to\infty$, and $\max_k\{n_k\}\leq C\min_k \{n_k\}$. 
As $n \goto \infty$, 
$T_n(\widehat{\Omega}^{\mathrm{SBM}}) \goto N(0, 1)$.   
\end{thm}

 {\bf Remark 5}.  \cite{LeiGoF} proposed a Tracy-Widom approach for GoF, but it is for the SBM case 
with $\alpha_n\geq C$ (so the networks are non-sparse), so it is unclear how to extend it to a broader setting. 
For DCBM, another GoF idea is to use degree-based $\chi^2$-tests \citep{Karwa2022}.  In particular,  \cite{Zhang2022} proposed an adjusted $\chi^2$-test by applying an initial grouping of nodes and then adjusting node degrees by the grouping.    However,  \cite{JKL2019} pointed out that using Sinkhorn's theorem, for a DCBM   with 
    $K > 1$ communities,  we can  pair it with a DCBM with $K = 1$, such that for each node, 
    the expected degrees under two models are approximately the same.  Extending this, we can find settings where the adjusted node degrees \citep{Zhang2022} are also approximately the same between two models.   For this reason, the $\chi^2$-test may not be an appropriate GoF metric. 

%%%%%%%%%%%%
%%%%%%%%%%%%
%%%%%%%%%%%%  
\subsection{Power analysis and optimality} \label{subsec:power} 
If we cast GoF as a hypothesis test setting,  then we have an explicit 
null hypothesis. Consider now an alternative hypothesis, where similarly, the upper triangle of $A$ contains independent Bernoulli   with $\mathbb{P}(A_{ij} =1)=\Omega(i,j)$, but $\Omega$ does not satisfy the assumed model. 
Recall that $\psi_{n, m}(\widehat{\Omega}) = U_{n, m}(\widehat{\Omega}) / \sqrt{2 m C_{n, m}}$, where $\widehat{\Omega}$ is fitted under the assumed  model. 
By Section~\ref{sec:DCMM}, if the assumed model is one of the $4$ block-models, then  we can write 
$\widehat{\Omega} =  \mathbb{M}(A)$ for a mapping $\mathbb{M}$. Let $
\widetilde{\Omega} =\mathbb{M}(\Omega)$   (e.g., if the assumed model is DCMM and we run Steps 1-4
of Algorithm~\ref{alg:afmSCORE} with $A = \Omega$ as the input, then $\widetilde{\Omega}$ is the  output). 
Define the Signal-to-Noise Ratio (SNR):  
\beq \label{SNR}
\mathrm{SNR}_{n,m}(\Omega) = \mathrm{trace}\bigl([\Omega - \widetilde{\Omega}]^{m}\bigr) / \sqrt{2m \cdot\mathrm{trace}(\Omega^m)}, \qquad\mbox{where}\quad \widetilde{\Omega}=\mathbb{M}(\Omega).
\eeq
%%%%%%%%%%%%
%%%%%%%%%%%%
%%%%%%%%%%%%
\begin{thm}[Power (oracle case)]\label{thm:SCC-power-general}
Let $u=\Omega{\bf 1}_n$. 
Suppose $\Omega$ satisfies the same conditions in Theorem~\ref{thm:SCC}. 
We further assume there exists $\alpha_n\in (0,1)$ such that $|\Omega_{ij}-\widetilde{\Omega}_{ij}|\leq C\alpha_nu_iu_j/(n\bar{u})$ for all $1\leq i,j\leq n$, and $|\tr((\Omega-\widetilde\Omega)^m)|\asymp  C\alpha^m_n \|u\|^{2m}/(n\bar{u})^m$. For $m\in \{3,4\}$, as $n\to\infty$,  
if $\mathrm{SNR}_{n,m}(\Omega)\to\infty$, then $\psi_{n, m}(\widetilde{\Omega})  \goto \infty$ in probability.
\spacingset{1}\footnote{A similar argument holds for a general $m$, but the proof is much more tedious due to combinatorics.}\spacingset{1.75}  
\end{thm} 

Theorem~\ref{thm:SCC-power-general} studies the power of the oracle SCC metric $\psi_{n,m}(\widetilde{\Omega})$, but two questions remain: (i) For each specific null-alternative hypothesis pair, when does the SNR tend to infinity? (ii) Can we show that the real SCC metric also satisfies that $\psi_{n,m}(\widehat{\Omega})\to\infty$ in probability?

First, we answer (i). For simplicity, we focus on $m=3$ and calculate $\mathrm{SNR}_{n,3}(\Omega)$ for three different  null-alternative hypothesis pairs or cases (the calculation for other $m$ is similar).    

In the first case, we assume that $\Omega$ is positive semi-definite and has a rank larger than $K$. For each assumed model in the block-model family, the algorithm of estimating $\Omega$ in Section~\ref{sec:DCMM} outputs a matrix with rank at most $K$; so, $\mathrm{rank}(\widetilde{\Omega})\leq K$. We then apply the Weyl's inequality \citep[Theorem 4.3.1]{HornJohnson} to obtain: $\lambda_j(\widetilde{\Omega}-\Omega)\geq \lambda_{K+j}(\Omega)$, for $1\leq j\leq n-K$, where $\lambda_k(\cdot)$ denotes the $k$th largest eigenvalue of a symmetric matrix. This yields a lower bound for the SNR:  
%%%%%%%%
%%%%%%%%%%%%
\begin{lemma}[A higher-rank alternative]\label{lem:SNR-higher-rank}
Suppose the assumed model is in the block-model family with $K$ communities, and the true $\Omega$ is positive semi-definite and has a rank $K_0>K$. 
Suppose the absolute sum of cubes of $K$ smallest negative eigenvalues of $\Omega-\widetilde{\Omega}$  is upper bounded by $(1-c_0)|\tr((\Omega-\widetilde{\Omega})^3)|$, for a constant $c_0\in (0,1)$. 
Let $\lambda_k$ denote the $k$th largest eigenvalue of $\Omega$. 
 Then, $\mathrm{SNR}_{n,3}(\Omega)\geq \bigl(\sum_{j=1}^{K_0-K}\lambda_{K+j}^3\bigr)/\sqrt{6\sum_{j=1}^{K_0}\lambda_j^3}$. Furthermore, as $n\to\infty$, if $K_0$ is bounded and $\lambda_{K+1}/\sqrt{\lambda_1}\to\infty$, then $\mathrm{SNR}_{n,3}(\Omega)\to\infty$. 
\spacingset{1}\footnote{When $m$ is even, we do not need the condition that $\Omega$ is positive definite. For example, for $m=4$, we can show that $\mathrm{SNR}_{n,4}(\Omega)\to\infty$, as long as $|\lambda_{K+1}|/\sqrt{\lambda_1}\to\infty$. See \cite{JKL2019, EstK}  for related discussions.}
\spacingset{1.75}
\end{lemma}

The setting in Lemma~\ref{lem:SNR-higher-rank} includes two special lack-of-fit scenarios:
\begin{itemize} 
\setlength \itemsep{-.35 em} 
\item {\it Misspecification of $K$}: Letting $*$ denote any model in the block-model family, the true model is a $K_0$-community $*$ model, but the assumed model is a $K$-community $*$ model for $K<K_0$.  
For instance, consider an SBM where $\Omega=\alpha \Pi P\Pi'$, $P=(1-b)I_{K_0}+b{\bf 1}_{K_0}{\bf 1}_{K_0}'$, and the community sizes are equal. Then, $\lambda_1\asymp n\alpha$, and $\lambda_j\asymp (1-b)n\alpha$ for $j\geq 2$. If we compute the GoF metric for SBM with $K<K_0$, the SNR tends to $\infty$ if $(1-b)\sqrt{n\alpha}\to\infty$.  
\item {\it Nonlinearity}: The true model is $\Omega=f(\Omega^*)$, with $\Omega^*=\Theta\Pi P\Pi'\Theta$ as in a $K$-community DCMM and $f(\cdot)$ being a nonlinear function applied entry-wise to $\Omega^*$ (one such example is the beta-model \citep{chatterjee2011random}). The assumed model is a $K$-community DCMM. Then, the power comes from the fact that the rank of $\Omega$ is often strictly larger than $K$.  
\end{itemize}   

In the second case, fixing $K\geq 1$, we assume that the true model is a $K$-community DCBM, but the assumed model is a $K$-community SBM. According to the GoF-SCORE-rev algorithm in Section~\ref{subsec:DCBM}, $\widetilde{\Omega}$ is a blockwise constant matrix. However, the true $\Omega$ follows a DCBM and cannot be accurately approximated by any blockwise constant matrix. This explains why our GoF metric has power. The following lemma considers a special case with $K=1$ and $\Omega=\theta\theta'$. The quantity $v(\theta)$ defined below measures the level of degree heterogeneity. It shows that our GoF metric for SBM has power in detecting degree heterogeneity.

%%%%%%%%%%%%%%%%%%%%%%%%%%%%%%%%%%%%%%%%%
\begin{lemma}[SBM versus DCBM]\label{lem:SBMvsDCBM}
Suppose that the true model is a 1-community DCMM with $\Omega=\theta\theta'$, and the assumed model is a  1-community SBM. Then, $\mathrm{SNR}_{n,3}(\Omega)\geq \|\theta\|^3\cdot v^2(\theta)/\sqrt{6}$, where 
$v(\theta):=[\sum_{i=1}^n(\theta_i-\bar{\theta})^2]/(\sum_{i=1}^n\theta_i^2)$, with $\bar{\theta}$ being the average of $\theta_i$'s. Therefore, as $n\to\infty$ if $\|\theta\|^3 v^2(\theta)\to\infty$, then $\mathrm{SNR}_{n,3}(\Omega)\to\infty$. 
\end{lemma}

In the third case, fixing $K\geq 1$, we assume that the true model is a $K$-community DCMM, but the assumed model is  a $K$-community DCBM. According to the GoF-SCORE algorithm in Section~\ref{subsec:DCBM}, $\widetilde{\Omega}=\widehat{\Theta}\widehat{B}\widehat{\Theta}$, where $\widehat{\Theta}$ is a diagonal matrix, and $\widehat{B}$ is  blockwise constant. However, $\Omega=\Theta\Pi P\Pi'\Theta$ under the true model, where $\Pi P\Pi'$ is not  blockwise constant. Therefore, even when $\widehat{\Theta}=\Theta$, it is impossible to have $\widetilde{\Omega}=\Omega$. This is where the power arises. The following lemma considers an example with $K=2$, where the signal strength is captured by a parameter $a$. It shows that our GoF metric for DCBM has power in detecting mixed membership.

\begin{lemma}[DCBM versus DCMM] \label{lem:DCBMvsDCMM}
Suppose that the true model is a 2-community DCMM, where $P=(1-b)I_2+b{\bf 1}_2{\bf 1}_2'$, and each $\pi_i$ takes value in $\bigl\{(1,0)', (1-a,a)', (0,1)', (a, 1-a)'\bigr\}$, for some $b\in [0,1)$ and $a\in (0,1/2)$. 
By the value of $\pi_i$, nodes divide into 4 groups $G_1,G_2,G_3,G_4$; we assume that $\sum_{i\in G_k}\theta_i$ is the same for 4 groups; and the same holds for $\sum_{i\in G_k}\theta_i^2$. Suppose that the assumed model is a 2-community DCBM. Then, $\mathrm{SNR}_{n,3}(\Omega) \geq c_0 a^4(1-b)^3\|\theta\|^3$, where $c_0>0$ is a universal constant. As $n\to\infty$, if $a^4(1-b)^3\|\theta\|^3\to\infty$, then $\mathrm{SNR}_{n,3}(\Omega)\to\infty$.  
\end{lemma}

Next, we answer (ii). We still fix $m=3$ and write $T_n(\widehat{\Omega})=\psi_{n,3}(\widehat{\Omega})$. The key is showing that $|T_n(\widehat{\Omega})-T_n(\widetilde{\Omega})|=O_{\mathbb{P}}(1)$ under the alternative hypothesis. Since our algorithms in Section~\ref{sec:DCMM} are designed such that $\widehat{\Omega}$ has an analytically tractable form, the analysis of $|T_n(\widehat{\Omega})-T_n(\widetilde{\Omega})|$ follows a similar strategy as in the null hypothesis, except that we still need additional technical tools, such as the Non-Splitting Property \citep{EstK} of SCORE. Since this part is technical, 
we put it in Section~\ref{sec:power} of the supplement, where we show that $T_n(\widehat{\Omega})\to\infty$ in probability, for several of the aforementioned examples.

Finally, we show the optimality of our GoF metric for one case: 
The true model is a DCBM with $K_0+1$ communities, but the assumed model is a DCBM with $K_0$ communities.
\spacingset{1}\footnote{The results are extendable to the case that the true model has $K_0+j$ communities, for $j> 1$. In fact, by Lemma~\ref{lem:SNR-higher-rank}, regardless of $j$, the SNR tends to infinity as long as the $(K_0+1)$th eigenvalue of $\Omega$ is large.}
\spacingset{1.75}
Given any integers $1\leq L\leq K$, let ${\cal U}_L(K)$ be the collection of $\Omega$ from a DCBM with $K$ communities satisfying the regularity conditions in Section~\ref{sec:power}.  

%%%%%%%%%%%%%%
%%%%%%%%%%%%%%
%%%%%%%%%%%%%%
\begin{thm}\label{thm:power1-mainpaper}
Fix $K_0\geq 1$. Let ${\cal M}_{K_0}(\gamma_n)$ be the collection of $(\Omega_0, \Omega_1)\in {\cal U}_{K_0}(K_0)\times {\cal U}_{K_0}(K_0+1)$ such that 
$\Omega_1$ satisfies $\|P-\widetilde{P}\|\ll \Vert \theta\Vert^2\cdot \mathrm{trace}\bigl([(P-\widetilde{P})\diag(\tilde{s})]^3\bigr)$ and 
$\mathrm{SNR}_{n,3}(\Omega_1)\geq \gamma_n$,
where $\widetilde{P}$ is as in Lemma~\ref{lem:power-DCBM} of the supplement and $\mathrm{SNR}_{n,3}$ is as in \eqref{SNR}. 
As $n\goto\infty$, if $\gamma_n\gg \sqrt{\log(n)}$, then there exists a proper $\alpha_n$ such that for any $(\Omega_0,\Omega_1)\in {\cal M}_{K_0}(\gamma_n)$, the sum of type-I and type-II errors of the level-$\alpha_n$ GoF-SCORE test tends to $0$. If $\gamma_n\to 0$, then there is a constant $c_0\in (0,1)$ such that for any given test we can find a pair $(\Omega_0,\Omega_1)\in {\cal M}_{K_0}(\gamma_n)$ for which the sum of type-I and type-II errors of this test is lower bounded by $c_0$.    
\end{thm}

Theorem~\ref{thm:power1-mainpaper} reveals an interesting phase transition:
In the Region of Possibility (defined by $\mathrm{SNR}_{n,3}(\Omega) / \sqrt{\log(n)} \goto \infty$),  asymptotically, the  GoF-based test can perfectly separate the two hypotheses.   In the Region of Impossibility (defined by $\mathrm{SNR}_{n,3}(\Omega)\to 0$),  asymptotically, the two hypotheses are partially indistinguishable and no test can have full power.

{\bf Remark 6} {\it (Local power analysis)}. Such analysis considers a sequence of local hypotheses under which a test has constant power. For example, consider the setting of Lemma~\ref{lem:SNR-higher-rank}, where $\Omega=\theta\theta'$. For a sequence $\alpha_n=o(1)$ and a constant $c$, let $\theta=\sqrt{\alpha_n} {\bf 1}_n$ in the null hypothesis and $\theta=\sqrt{\alpha_n}\bigl({\bf 1}_n+ c\cdot n^{-5/4}\alpha_n^{-3/4}\eta\bigr)$ in the alternative hypothesis, where $\eta\in\mathbb{R}^n$ is such that its first half entries are $1$ and second half entries are $-1$. Using the proof of Lemma~\ref{lem:SNR-higher-rank}, we can show that the SNR is equal to a constant $\omega=\omega(c)$. We further conjecture that $T_n(\widehat{\Omega})\to N(\omega(c), 1)$ in law, so that the local power can be characterized. We leave this to future work. 

 {\bf Remark 7} {\it (Alternative models where edges are dependent)}. Our analysis in this subsection assumes that the edges are independent in the alternative model.  
To study the dependent-edge case, we did a simulation where the alternative model is a 
triadic closure model \citep{di2022hierarchy}, and we computed the GoF metric by assuming a 2-community DCMM. The results are in Figure~\ref{Fig:simu-power} of the supplement. It suggests that our test still has power in this case.

%%%%%%%%%%%%%%%
%%%%%%%%%%%%%%%
%%%%%%%%%%%%%%% 
 %%%%%%%%%%%%%%%
%%%%%%%%%%%%%%%
\section{Simulations and analysis of the $12$ real networks}
\label{sec:numeric} 

{\it Owing to space limit, the simulation results are contained in Section~\ref{subsec:Simu} of the supplement.} 
 We now investigate our approaches with the $12$ real networks aforementioned. Among them, 
CoAuthor was studied in \cite[Section 4.2]{JaJ2016},  Citee2016 was studied in \cite[Section 5.2]{MSCORE}, 
and Fan is the ego network of   ``Jianqing Fan" in the coauthor network of  \cite[Sections 3.1, 3.3]{JBES} (i.e.,  the subnetwork consisting of Fan and all nodes having an edge with Fan),  LastFM is from SNAP (\url{https://snap.stanford.edu/data}), and the other 8 are the same as those used in \cite{SCOREplus}.  
For all except LastFM,  previous works \citep{JaJ2016, MSCORE, JBES} carefully studied what the most appropriate $K$ is (for example, \cite{JaJ2016} argued that the Fan network has two communities: ``nonparametric statistics'' and ``high-dimensional statistics'', so $K = 2$), so we 
assume $K$ as known;  see Table \ref{tab:realdata0}.  For the LastFM network, we compute the GoF metrics for different $K$ and find that $K=3$ yields the best fit. Due to space constraint, we only report in Table~\ref{tab:realdata0} the results on LastFM with $K=3$; the results of other $K$ are in Table~\ref{tab:SNAPv1} of the supplement.

We test whether the SBM, DCBM, MMSBM, and DCMM models are  appropriate, respectively.   
Consider $T_n(\widehat\Omega^{\rm DCMM})$ first. To use it for real data analysis, we use Algorithm~\ref{alg:afmSCORE}  with $3$ minor regularization steps: (a) after obtaining $\widehat{P}$ in Algorithm~\ref{alg:afmSCORE}, we set its negative entries to zero; (b) after obtaining $\widehat{\Omega}$ in Algorithm~\ref{alg:afmSCORE}, we add a regularization step where we 
set $\widehat\Omega(i,j)=1$ if $\widehat\Omega(i,j)>1$ and similarly set $\widehat\Omega(i,j)=0$ if $\widehat\Omega(i,j)<0$; (c) in Step~(2.1) in Algorithm~\ref{alg:afmSCORE}, after obtaining $\hat{w}_i$ (the barycentric coordinate of $\hat{r}_i$ in the estimated simplex),  we set all of its negative entries to $0$ and then rescale it to have a unit-$\ell_1$-norm. These regularizations are reasonable, as $P$ and $w_i$ are non-negative and $\Omega(i,j) \in [0, 1]$.  Next, we consider $T_n(\widehat\Omega^{\rm MMSBM})$. We use the algorithm in Section \ref{subsec:MMSBM} with the  same regularization steps (b)-(c) above, but (a) is skipped as 
 a negative entry  of $\widehat{P}$ rarely exists in the analysis.   Third, consider $T_n(\widehat\Omega^{\rm SBM})$ and  $T_n(\widehat\Omega^{\rm DCBM})$.   
We use the two algorithms in Section \ref{subsec:DCBM}.   
Compared to $T_n(\widehat{\Omega}^{{\rm DCMM}})$, these algorithms are much simpler: 
There is no $\hat w_i$, and 
the entries of $\widehat P$ are always non-negative. Hence,  
regularization steps (a) and (c) are  skipped.  
Last, recall that MMSBM and DCMM model mixed-memberships and SBM and DCBM do not, 
so $T_n(\widehat\Omega^{\rm MMSBM})$ and $T_n(\widehat\Omega^{\rm DCMM})$ use the Vertex Hunting (VH) algorithm, and $T_n(\widehat\Omega^{\rm SBM})$ and $T_n(\widehat\Omega^{\rm DCBM})$ do not.  
It was argued by \cite{SCORE-Review} 
that, compared with the other $10$ networks, the Caltech and Simmons networks have weak signals, and are harder to analyze. Therefore, when we 
apply $T_n(\widehat\Omega^{\rm MMSBM})$ and $T_n(\widehat\Omega^{\rm DCMM})$ to these two networks, we make a small change: 
In Step 1 of both algorithms, 
before the VH steps, we remove 1\% outlying rows of $\widehat R_H$, using the R package \texttt{isotree} \citep{liu2008isolation}.   The outlier removal stabilizes our estimate $\widehat{\Omega}$  a bit, so there is no need to use the regularization step (b).  For fair comparison, we also remove regularization step (b) in  $T_n(\widehat\Omega^{\rm SBM})$ and  $T_n(\widehat\Omega^{\rm DCBM})$ when applied to the two networks (the results are similar  if we include step (b)).  
Among these algorithms,  
$T_n(\widehat\Omega^{\rm SBM})$ is tuning free, 
and $T_n(\widehat\Omega^{\rm SBM})$ has only one tuning parameter, the threshold $t$ in the SCORE step. We set $t = \log(n)$ as recommended in \cite{MSCORE, SCORE-Review}.  For $T_n(\widehat\Omega^{\rm MMSBM})$ and $T_n(\widehat\Omega^{\rm DCMM})$, the only tuning parameters come from 
the Vertex Hunting algorithm (recall that two other algorithms do not use VH). The VH algorithm has tuning parameters $(N, \alpha)$, 
which are set as in Section \ref{subsec:GoF-MSCORE} (last two paragraphs). 
To avoid over-fitting and  for a fair comparison,  we fix $(N, \alpha, t)$ as recommended above and do not change them from setting to setting (e.g., data set, algorithm).  
The results are relatively insensitive to tuning parameters (see 
Table \ref{tab:tuning} in the supplement for a study of how our results depend on different tuning parameters).

Our main results are in Table \ref{tab:realdata0}. 
Recall that when an assumed model (say, SBM) is true, then the corresponding 
GoF metric (say, $T_n(\widehat{\Omega}^{\mathrm{SBM}})$) converges to $N(0,1)$ as $n \goto \infty$.   
We think we have  {\it a reasonably good fit} if the (absolute value of the; same below) GoF metric is smaller than $5$: 
this may seem a bit less critical, but many of these networks have a relatively small $n$,  and some of them have secondary effects (e.g., outliers)  which are not fully captured by our models.     
Also,  we think {\it a moderate lack of fit}
happens if a GoF metric falls between $5$ and $7.5$, and {\it a significant lack-of-fit} happens if 
we have a GoF metric bigger than $7.5$.

 \spacingset{1.1}
\begin{table}[]\centering
\caption{\small The GoF metrics for SBM, MMSBM, DCBM and DCMM, respectively (each has a limiting null of $N(0,1)$  if the corresponding model is true), for the $12$ real networks.  Boldface: GoF metric bigger than $7.5$  
(significant lack of fit);  italic: GoF metric between $5$ and $7.5$ (moderately lack of fit); others: GoF metric smaller than $5$ (reasonably good fit).}
\label{tab:realdata0}
\scalebox{0.88}{
\begin{tabular}{l|rrrr|l|rrrr}
\hline
Dataset &SBM &DCBM &MMSBM &DCMM &Dataset &SBM &DCBM &MMSBM &DCMM \\ \hline
Karate &$-1.744$ &$0.061$ &$-1.483$ &0.198 &Polbooks &{\bf 8.146} &{\it 5.229} &{\it 6.061} &{\it 5.039} \\ 
Football &0.625 &0.542 &0.824 &$0.752$ &Weblogs &{\bf 117.1} &${\bf 9.033}$ &{\bf 48.74} &{\bf 7.763} \\ 
Dolphin &4.312 &3.165 &3.977 &$3.124$ &Citee2016 &{\bf 759.2} &{\bf 308.4} &{\bf 405.3} &3.687 \\ 
Fan &2.180 &1.789 &$-4.614$ &$3.500$ &Caltech &{\bf 54.88} &{\bf 38.95} &{\bf 42.76} &$4.559$ \\ 
CoAuthor &4.166 &{\it 6.635} &{\it 6.904} &{\it 6.695} &Simmons &{\bf 115.5} &{\bf 85.14} &{\bf 108.4} &$-4.949$ \\ 
UKfaculty &{\it 5.915} &2.540 &2.730 &1.767 & LastFM & {\bf 140.0} & {\bf 141.7} & {\bf 193.1} &  2.204 \\ 
\hline
\end{tabular}
} 
\end{table}
\spacingset{1.75}

We divide the $12$ networks into $3$ groups: (A) Karate, Football, Dolphin, Fan (relatively small and easy data sets),  (B) the last $7$ networks in Table \ref{tab:realdata0}, and (C) the CoAuthor network.  

First, for the $4$ networks in group (A), 
SBM provides a reasonably good fit, and as expected, the other three models (each is broader than the SBM) also provide a reasonably good fit.   

Second, among the $8$ networks in group (B), the last four (Citee2016, Caltech, Simmons, LastFM) are relatively large and difficult networks (e.g., \cite{SCORE-Review}). 
For them, the three models SBM, DCBM, and MMSBM provide a significantly poor fit (with a GoF metrics larger than $42$), but DCMM provides a good fit (with GoF metrics smaller than $5$).  Also, the GoF metrics corresponding to DCBM and MMSBM are significantly smaller than that of SBM.  
These results suggest a very interesting point in network modeling: 
{\it many real networks have two noteworthy features: 
severe degree heterogeneity and mixed-memberships.  
To have a good fit, we must model both features; modeling only one of them would still give a significantly poor fit}.   The discussion for Weblog is similar, except for that even the broadest DCMM model has a moderate lack-of-fit for this network; see more discussion below.     
Among the $12$ networks, Weblog has the largest ratio of $d_{max} / d_{min}$ and $3rd$ largest ratio of $\bar{d}/d_{min}$.  From Table \ref{tab:realdata0}, replacing SBM by DCBM (i.e., adding the layer of degree heterogeneity modeling to SBM) reduces the GoF metric from $117,1$ to $9.033$, which is a remarkable 
improvement.  This illustrates how valuable it is to properly model 
severe degree heterogeneity. The discussion for Polbook and UKfaculty is similar, except for the differences between the GoF metrics for different models  are not as significant as those for Weblog,  Citee2016, Caltech, and Simmons. 
The moderate lack-of-fit of DCMM for Weblog can be 
explained as follows.  In its original form, Weblog is a {\it directed network} with $1,494$ nodes \cite{Adamic2005}, where each node is a blog, and 
each directional edge is a hyperlink.  As first suggested by \cite{karrer2011stochastic}, 
we may use {\it brute-force symmetrization} to construct an undirected network as follows: define an {\it undirected edge} between $i$ and $j$ as long as 
there is at least one directed edged between them. This gives rise to an undirected network (the network in Table \ref{tab:realdata0} is its giant component).  
Fix two nodes $i$ and $j$.  The reciprocal effect between them refers to the effect that, 
when there is a directed edge from $i$ to $j$, it is likely that there is also 
an edge from $j$ to $i$.  The main problem of the above symmetrization is that it forcefully increases the reciprocal effects between 
many pairs of nodes and generate some undesirable artifacts. 
This partially explains why DCMM has a moderate lack-of-fit for  Weblog,  
and suggests that we must be cautious when  
symmetrizing a directed network.

Finally, we discuss CoAuthor network (the only one in group (C)).  
The GoF metrics for SBM, DCBM, MMSBM and DCBM are 
$4.166, 6.635, 6.904,  6.695$, respectively.  
It may seem that SBM (the most idealized one  
among the $4$)  fits best with the network.  This seeming contradiction  
may be due to that CoAuthor not only has a relatively small $n$ but also has the smallest $\bar{d}$ among all $12$ networks (i.e., $(n, \bar{d}) = (236, 2.51)$).  In such a difficult case, 
the GoF metric may not converge fast enough to the limiting null,  so 
it is hard to tell  the best model. 

We argue that DCMM is the best model for CoAuthor, for the following reasons.  
First, $(d_{min}, \bar{d}, d_{max}) = (1, 2.51, 21)$ for this network, 
indicating severe degree heterogeneity. Therefore, out of the four block models, 
DCBM and DCMM are the more appropriate choices. 
Second, recent study suggests the network has significant mixed-memberships, 
and so DCMM is the beset choice.  Recall that CoAuthor is a co-authorship network  constructed using the MADStat data set \citep{JaJ2016, JBES, MSCORE}.    
It was argued by \cite{JaJ2016, MSCORE} that the network has two communities (so $K = 2$): a Carroll-Hall (CH) community on non-parametrics/semi-parametric and North-Carolina (NC) community.   When they first analyzed the network,  Ji and Jin  \cite{JaJ2016} {\it assumed a DCBM with $K = 2$} and considered the problem of community detection (i.e., clustering all $236$ nodes into $2$ groups). 
They applied $4$ community detection methods: 
SCORE, NSC,  APL, 
and BCPL. 
These methods produce similar clustering results in several other networks in Table \ref{tab:realdata0} (e.g., Karate, Weblog), but for CoAuthor, they produced drastically different results \citep{JaJ2016}. 
Why did this happen? To solve the puzzle, \cite{MSCORE} pointed out that one of the nodes, Jianqing Fan, is a leading figure in non-parametric/semi-parametric statistics and was also
on the faculty of UNC-Chapel Hill in the 1990s. Therefore, Fan and many of his collaborators may have significant mixed-memberships in both CH and NC.  These suggest that the DCMM is a more appropriate model than DCBM for CoAuthor. Using MSCORE,  \cite[Table 3]{MSCORE} estimated the mixed membership vectors $\pi_1, \pi_2,  \ldots, \pi_n$. Since $K = 2$, we write each $\hat{\pi}_i = (1 - \hat{w}_i, \hat{w}_i)$, with $\hat{w}_i$ being the estimated weight in the NC community. 
Figure \ref{fig:Fan} presents $\hat{w}_i$ for a few selected nodes $i$. 
These results confirm that the DCMM is most appropriate model for CoAuthor.  
\spacingset{1}
\begin{figure}[tb!] 
\label{fig:Fan} 
\centering
\includegraphics[width=.6\textwidth]{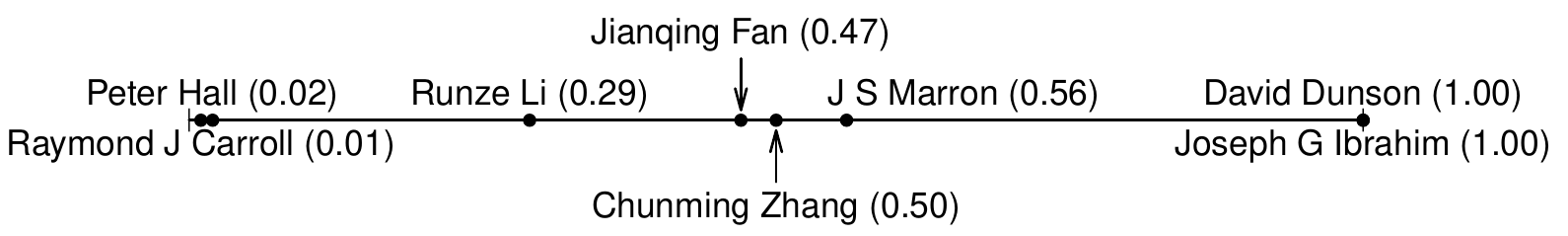}
\caption{Estimated weight $\hat{w}$ in NC for some authors ($(1 - \hat{w})$ is the weight in CH).} 
\end{figure}
\spacingset{1.75}

We summarize our findings.  
First, the block-family provides adequate and interpretable models for real networks. In particular, DCMM is adequate for all $12$ networks (except for Weblog, 
but it may come from the artificial reciprocal effects caused by forceful symmetrization).  
Real networks typically have two 
features: severe degree heterogeneity and significant mixed-memberships. 
DCMM models both features and is preferred.  SBM, MMSBM, and DCBM incorporate at most one of the two features and are inadequate for many networks.  Second, we must be cautious in constructing 
undirected networks from directed ones by brute-force symmetrization: 
it may create artificial reciprocal 
effects (e.g., see the discussion on Weblog) that we may not find in natural networks. 
Third, GoF is not the only way to assess network modeling. 
Our findings on the CoAuthor network suggest we can try different ways to understand 
the important aspects of network modeling. 
Last, combining our results with Theorem~\ref{thm:Laplacian} on NMF, we recommend DCMM as a sweet spot of modeling undirected networks.

\vspace{-1em}
%%%%%%%%%%%%
%%%%%%%%%%%%
%%%%%%%%%%%%
\section{Discussion} 
\label{sec:Discu} 
We focus on GoF of  four block models for undirected binary networks, but the idea is readily extendable to other settings,  
such as weighted networks,   directed networks,  bipartite networks, 
and hierarchical networks (e.g., \cite{JaJ2016, JBES}).  
For example, if we put a tree structure on the mixed mbership 
vectors $\pi_1, \pi_2, \ldots, \pi_n$, then we have a model for 
hierarchical network.  
Once the model is set up, 
many of our ideas on GoF are readily extendable.  In addition, we assume the number of communities, $K$, is finite for the current setting, but our analysis can be easily extended to the case where $K = \big(\log(n)\big)^c $ for some constant $c>0$, by slightly modifying the regularity conditions to include extra logarithmic factors. 
 We also impose some regularity conditions: e.g., the average node degree is much larger than $\log(n)$, and the degree heterogeneity follows some requirement as implied by Conditions~\ref{cond:afmSCORE}-\ref{cond:MSCORE}. When these conditions are not satisfied, we may improve our GoF metric and/or the theory to adapt to such more challenging settings. %We leave this to future work.  

Natural networks have many noteworthy features, some are common (e.g., sparsity) and some are more specific (e.g., cliques). The block-model family models sparsity, degree heterogeneity, 
and mixed-memberships, but they do not specifically model 
outliers  and  cliques. It is unclear whether we should incorporate 
these features into the block models, and if so, how.  It is also unclear how to test whether a 
network has outliers and small-size cliques. These are interesting questions and 
we leave them to the future. 
The cycle count statistics are essentially U-statistics \citep{zhang2022edgeworth} and we can adapt them 
to settings beyond networks, such as
testing and GoF for hypergraph analysis \citep{yuan2022testing}, text analysis,  factor models \citep{wang2017asymptotics}, 
and cancer clustering models \citep{peng2016more,zhong2017tests}. We leave these for the future research.

 \appendix

%%%%%%%%%%
%%%%%%%%%%
%%%%%%%%%%
%\section{Proofs of properties of the RQ test statistic}  \label{sec:proof}
%In this section, we prove Theorems \ref{thm:CLB} and \ref{thm:CUB}.  
%Corollaries~\ref{cor:CLB}-\ref{cor:CUB} follow directly from Theorems \ref{thm:CLB} and \ref{thm:CUB}, 
%respectively, so the proofs are omitted.    

\section{Additional numerical details} \label{supp:Simu}

In this section, we provide details of numerical results, including the descriptions of auxiliary algorithms, three simulation experiments, and additional real-data results. 

%%%%%%%%%%%%%%%%%%%
\subsection{Descriptions of MSCORE, SP, and KNN-SP}
%Description of auxiliary algorithms used in GOF-MSCORE

In Algorithm~\ref{alg:afmSCORE}, we use MSCORE to get an initial estimate of $\Pi$ and apply the successive projection (SP) for vertex hunting. We also recommend the K-nearest neighborhood successive projection (KNN-SP) algorithm \citep{SCORE-Review} as an alternative of SP for practical use. There we treat MSCORE, SP and KNN-SP as auxiliary algorithms and plug them into our main algorithm.  We now provide descriptions of MSCORE, SP and KNN-SP. 

% are not included in the main text. We provide them here in  Algorithm~\ref{alg:MSCORE}, Algorithm~\ref{alg:SP}, and Algorithm \ref{alg:PPSPA}.  

%\spacingset{1.5}

MSCORE \citep{MSCORE} is a spectral approach for estimating $\Pi$. The details are in Algorithm~\ref{alg:MSCORE}. 
%%%%%%%%%%%%%%%%%%%MSCORE%%%%%%%%

\spacingset{1}

\begin{algorithm}[H]
\caption{The MSCORE algorithm}\label{alg:MSCORE}
\medskip
\noindent
{\bf Input:} The adjacency matrix $A\in\mathbb{R}^{n,n}$, and the number of communities $K$. 
\begin{enumerate} 
\item ({\it SCORE normalization}). Obtain the eigen-pairs $(\hat \lambda_1, \hat \xi_1), (\hat \lambda_2, \hat \xi_2), \cdots, (\hat \lambda_K, \hat \xi_K)$ of  $A$.  Compute $\widehat R =[\hat r_1, \hat r_2, \cdots, \hat r_n]'= {\rm diag}(\hat \xi_1)^{-1} [\hat \xi_2, \cdots, \hat \xi_K] \in \mathbb R^{n, K-1}$.
%\begin{itemize}
%\item Compute the eigenvalue-eigenvector pairs $(\hat \lambda_1, \hat \xi_1), (\hat \lambda_2, \hat \xi_2), \cdots, (\hat \lambda_K, \hat \xi_K)$ by performing PCA on $A$.
%
%\item Compute $\widehat R =[\hat r_1, \hat r_2, \cdots, \hat r_n]'= {\rm diag}(\hat \xi_1)^{-1} [\hat \xi_2, \cdots, \hat \xi_K] \in \mathbb R^{n, K-1}$
%\end{itemize}

%\[
%\widehat R =[\hat r_1, \hat r_2, \cdots, \hat r_n]'= {\rm diag}(\hat \xi_1)^{-1} (\hat \xi_2, \cdots, \hat \xi_K) \in \mathbb R^{n, K-1}\, .
%\]
\item ({\it Vertex hunting}). Obtain the simplex vertices $\hat v_1, \hat v_2, \ldots, \hat v_K$ via KNN-SP algorithm in Algorithm~\ref{alg:KNN-SP} with input $\hat r_1, \ldots, \hat r_n$.

%Using rows in $\hat R$, i.e., $\hat r_1, \cdots, \hat r_n$,  to estimate the vertices of the Ideal Simplex, denoted by $\hat v_1, \hat v_2, \cdots, \hat v_K$. 
\item ({\it Membership estimation}). 
\begin{itemize}
\item Compute $\hat b_1\in \mathbb R^{K}$, with entries $\hat b(k)= [\hat \lambda_1 + \hat v_k' {\rm diag}(\hat \lambda_2, \ldots, \hat \lambda_K) \hat v_k]^{-1/2}$.

\item For each $1\leq i \leq n$, solve the barycentric coordinate $\hat w_i\in \mathbb R^{K}$ from the equations: 
$\hat r_i = \sum_{k=1}^K \hat{w}_i(k) \hat{v}_k$ and $\sum_{k=1}^K \hat w_i(k) = 1$.

\item Obtain $\hat \pi_i^* \in \mathbb R^{K}$, where $\hat \pi_i^*(k)  = \max\{\frac{\hat{w}_i(k)}{\hat b_1(k)}, 0\}$. Let $\hat \pi_i = \hat \pi_i^*/ \|\hat \pi_i^*\|_1$ for $1\leq i \leq n$.

%\item Compute $\hat \pi_i = \hat \pi_i^*/ \|\hat \pi_i^*\|_1$ for $1\leq i \leq n$.
\end{itemize}

\end{enumerate}
{\bf Output:} The estimated membership vectors $\hat\pi_1,\ldots,\hat\pi_n$. 
\medskip
\end{algorithm}

\spacingset{1.5}

SP \citep{araujo2001successive} is a popular algorithm for vertex hunting. Given $K\geq 2$, SP takes $n$ data points in  the input and outputs the estimated vertices $\hat{v}_1,\hat{v}_2,\ldots,\hat{v}_K$. 
In the first iteration of SP, it searches for the data point with the largest Euclidean norm and uses it as $\hat{v}_1$. In the $k$th iteration, it projects all data points into the subspace orthogonal to the previous $(k-1)$ vertices and uses the data point with the largest (post-projection) Euclidean norm as $\hat{v}_k$. 
A nice property of SP is that each $\hat{v}_k$ equals to one of the data points. This makes  theoretical analysis convenient and explains why SP is often recommended in theoretical study. 
Note that the error rate of SP has been given explicitly in \cite{jinimproved}.

The input data and output vertices can be in any dimension $d\geq K-1$. In Algorithm~\ref{alg:afmSCORE}, we apply SP to the rows of $\widehat{R}_H$, which corresponds to $d=K$. Additionally, MSCORE also requires plugging a vertex hunting algorithm, in which the data points are rows of $\widehat{R}$, so that $d=K-1$. We provide the description of SP in Algorithm~\ref{alg:SP} for a general $d$. 
%(Note that for $d=K-1$, the last two iterations of SP are slightly different.) 

\spacingset{1}
%%%%%%%%%%SP%%%%%%%%%%%

\begin{algorithm}[htbh]
\caption{The SP Algorithm }\label{alg:SP}
\medskip
\noindent
{\bf Input:} The number of vertices $K$, and data points $\hat r_1, \hat r_2, \cdots, \hat r_n \in \mathbb R^{d}$ ($d\geq K-1$). 
\begin{enumerate}
\item Initialize by taking $P$ to be the $d\times d$ zero matrix.

\item For each $k=1,2,\ldots,K$: If $k\leq d-1$ or $k=d\geq K$, do the following: 
\begin{itemize} 
\item Obtain $i_k=\mathrm{argmax}_{1\leq i\leq n} \| (I_K -   P)\hat r_i\|$, and let 
$ \hat v_k = \hat r_{i_k}$. 
\item Update 
$P =\widehat V_{k}(\widehat V_{k}'\widehat V_{k})^{-1}\widehat  V_{k}'$, where $\widehat V_{k} = [\hat v_1, \cdots, \hat v_{k}]$. 
\end{itemize}
If $k= d=K-1$, do the following: 
\begin{itemize} 
\item Note that $\{(I_K -   P)\hat r_i\}_{1\leq i\leq n}$ are now located in a line. Let  $i_{k}$ and $i_{k+1}$ be the indices of the two end points on this line. 
\item Let $ \hat v_k = \hat r_{i_k}$ and $ \hat v_{k+1} = \hat r_{i_{k+1}}$. End the loop. 
\end{itemize}

\end{enumerate}

{\bf Output:} The estimated vertices $\hat v_1, \hat v_2, \ldots,\hat v_K$. 

\medskip
\end{algorithm}
\spacingset{1.5}

%%%%%%%%%%%%%%KNN-SP%%%%%%%%%%%
\spacingset{1}

\begin{algorithm}[htb!]
\caption{ The KNN-SP Algorithm} \label{alg:KNN-SP}
\medskip

{\bf Input}: $X_1, X_2, \ldots, X_n$, the number of vertices $K$, and tuning parameters $(\alpha,N)$.

\begin{enumerate}

\item {\it (KNN denoise)}. For each $i$, let $B_\alpha(X_i)$ be the ball of radius $s_{\max}/\alpha$ and center $X_i$, where $s_{\max}=\max_{1\leq j\neq k\leq n}\|X_j-X_k\|$. 
 If there are fewer than $N$ points (including $X_i$ itself) in $B_{\alpha}(X_i)$, delete $X_i$; Otherwise, replace $X_i$ by $X_i^*$, which is the average of all points in $B_{\alpha}(X_i)$. 

\item {\it (Vertex Hunting)}. Let ${\cal J}\subset\{1,\ldots,n\}$ denote the set of retained points in Step~1. 
Apply successive projection (SP) algorithm in Algorithm~\ref{alg:SP} to $\{X_i^*: i\in {\cal J}\}$ to get $\hat{v}_1,\ldots,\hat{v}_K$. 

\end{enumerate}
{\bf Output}: The estimated vertices $\hat{v}_1,\ldots,\hat{v}_K$. 

\medskip
\end{algorithm}
\spacingset{1.5}

Although SP is convenient for theoretical analysis, it is not robust to outliers and strong noise. KNN-SP \cite[Section 3.4]{SCORE-Review} modifies SP by adding a denoise step via k-nearest neighbor smoothing; see Algorithm~\ref{alg:KNN-SP}. KNN-SP has tuning parameters $(N, \alpha)$. How to choose $(N, \alpha)$ has been described in Section~\ref{subsec:GoF-MSCORE}. \cite{SCORE-Review} reported that KNN-SP frequently outperforms SP, so we recommend KNN-SP in practical. The error rate of KNN-SP has been studied in \cite{jinimproved}.

\subsection{Simulations} \label{subsec:Simu}

We investigate the numerical performance of the proposed GoF metrics via simulated networks. 
We consider three experiments.  Experiment~1 shows the histograms of GoF metrics.   In Experiment~2, we examine the control of type-I errors when the model is correctly specified. 
In Experiments~3, we study the power of detecting model misspecification.  
%In Experiment~4, we investigate the power of our GoF metric when the true model is not in the block-model family.

\paragraph{Experiment 1: Histograms of GoF metrics.} Fix $(n, K)=(3000,2)$. Let $P\in\mathbb{R}^{2\times 2}$ be such that its diagonal entries are $1$ and off-diagonal entries are $0.05$.  We simulate $1000$ networks from each of the four models:
{\it Experiment 1.1 (DCMM)}: In this sub-experiment, $\theta_i$'s are independently drawn from ${\rm Unif}(0.1,0.3)$; each community has $n/8$ pure nodes, and $\pi_i$'s of the remaining nodes are independently drawn from ${\rm Dirichlet}(0.5,0.5)$. 
{\it Experiment 1.2 (MMSBM)}: $\theta_i$'s are all equal to $\sqrt{\alpha}_n$, where $\alpha_n=0.3$;  and $\pi_i$'s  are drawn in the same way as in Experiment 1.1.
{\it Experiment 1.3 (DCBM)}: $\theta_i$'s are drawn in the same way as in Experiment 1.1, and each community has $n/2$ nodes. 
{\it Experiment 1.4 (SBM)}: $\theta_i$'s are all equal to $\sqrt{\alpha}_n$, where $\alpha_n=0.3$; and each community has $n/2$ nodes. For each true model, we plot the histograms of the proposed GoF metrics $T_n(\widehat\Omega^{\rm DCMM})$, $T_n(\widehat\Omega^{\rm MMSBM})$, $T_n(\widehat\Omega^{\rm DCBM})$, and $T_n(\widehat\Omega^{\rm SBM})$. The results are shown in Figure~\ref{Fig:1}.

\spacingset{1}
\begin{figure}[tb!]
\centering
\includegraphics[width=.42\textwidth, height=.24\textwidth]{PDF/DCMM_new.pdf}%dcmm_hist_0818.m
\includegraphics[width=.42\textwidth, height=.24\textwidth]{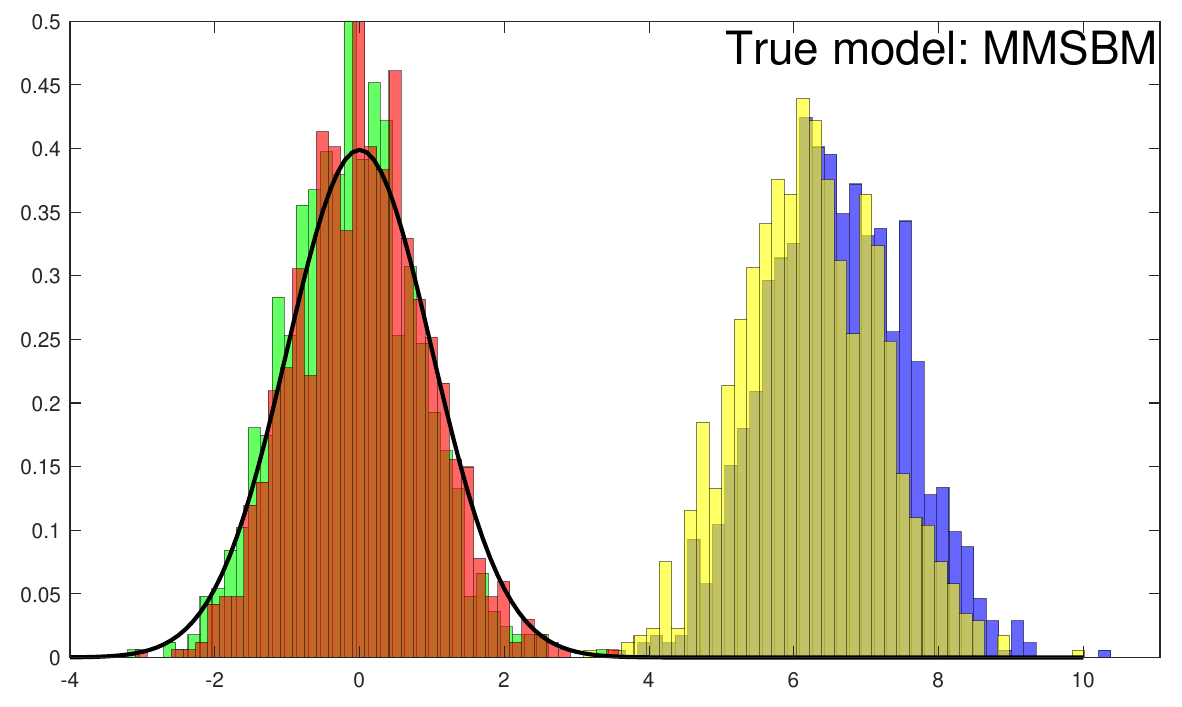}\\%mmsbm_hist_0818.m
\includegraphics[width=.42\textwidth, height=.24\textwidth]{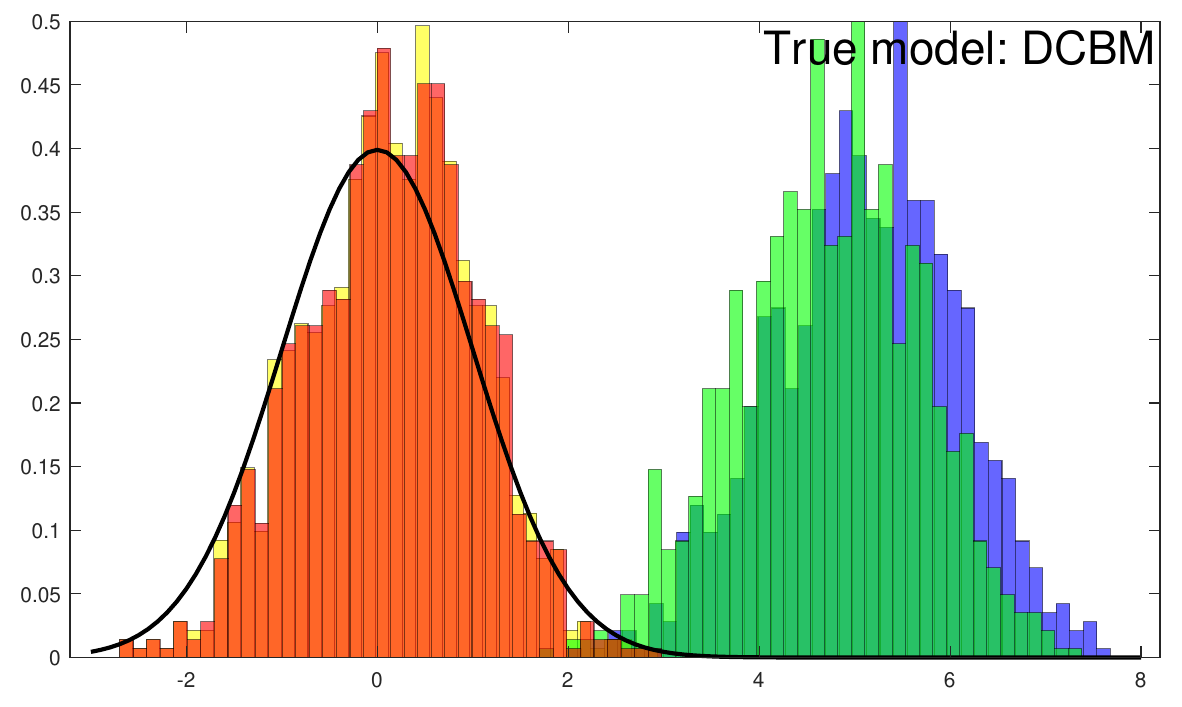}%dcbm_hist_0818.m
\includegraphics[width=.42\textwidth, height=.24\textwidth]{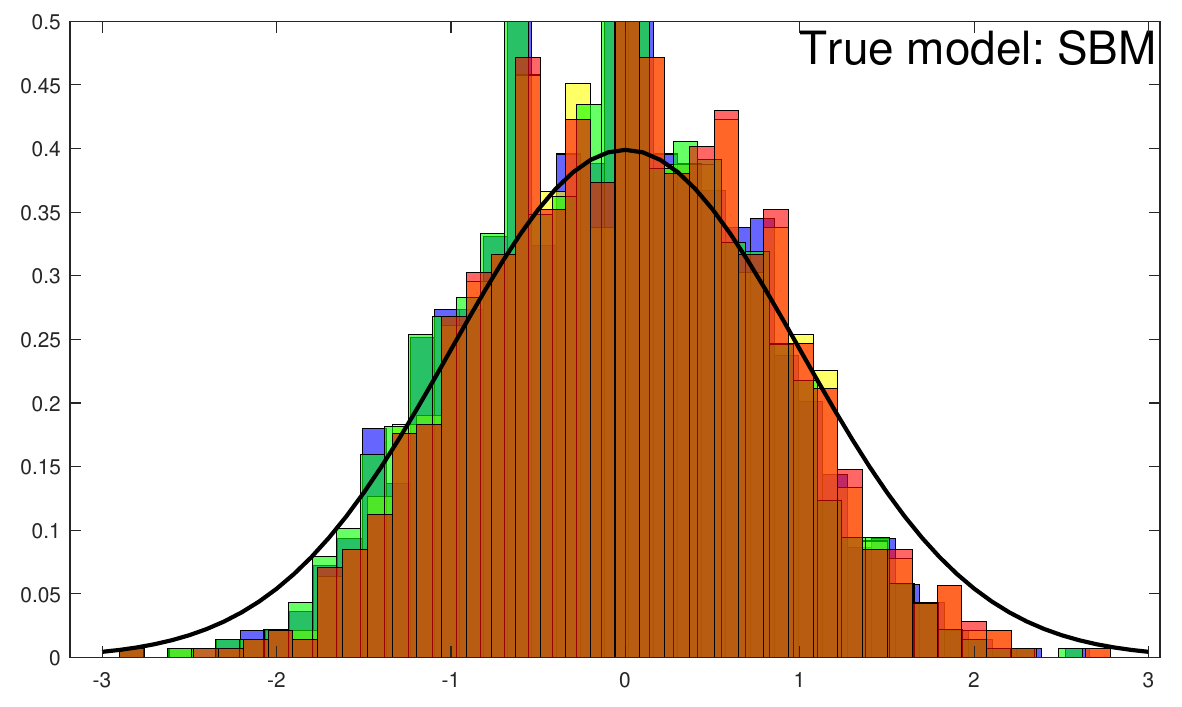}%sbm_hist_0818.m
\caption{Histograms of the GoF metrics in Experiment 1.1-1.4, where the off-diagonal entries of $P$ are equal to $0.05$. In each panel, the networks are generated from a true model in the block-model family, the four colors correspond to four GoF metrics, and the black curve is the density of $N(0,1)$.\label{Fig:1}}
\end{figure}
\spacingset{1.5}

We have some notable observations. First, when the assumed model is   true,   the histogram of the GoF metric is a good approximation to the density of $N(0,1)$ (e.g., Figure~\ref{Fig:1} (top left);  when the true model is DCMM, the histogram of $T_n(\widehat\Omega^{\rm DCMM})$ fits well the black curve; similar for Figure \ref{Fig:1}). This verifies the asymptotic normality shown in Sections~\ref{subsec:normality}-\ref{subsec:DCBM}. Second, when the assumed model includes the true model as a special case, the histogram of the GoF metric is still a good approximation to the standard normal density (e.g., bottom left panel of Figure~\ref{Fig:1}, when the true model is DCBM, the histogram of $T_n(\widehat\Omega^{\rm DCMM})$ still fits well the black curve). This is also consistent with theory. Last, when the assumed model does not include the true model, the GoF metric has a significant departure from the standard normal density, showing a good power (e.g., in the bottom left panel of Figure~\ref{Fig:1}, the histograms of $T_n(\widehat\Omega^{\rm SBM})$ and $T_n(\widehat\Omega^{\rm MMSBM})$ both have a considerable shift from the black curve). These observations suggest that our proposed GoF metrics indeed have both parameter-free limiting nulls and good powers. 

In Experiments 1.5-1.8, we mimic Experiments 1.1-1.4 but change the off-diagonal entries of $P$ from $0.05$ to $0.2$. 
This makes the communities more ``similar" in the true model.  
The results are in Figure~\ref{Fig:2}, where we have similar observations as above.

\spacingset{1}
\begin{figure}[tbh!]
\centering
\includegraphics[width=.42\textwidth, height=.24\textwidth]{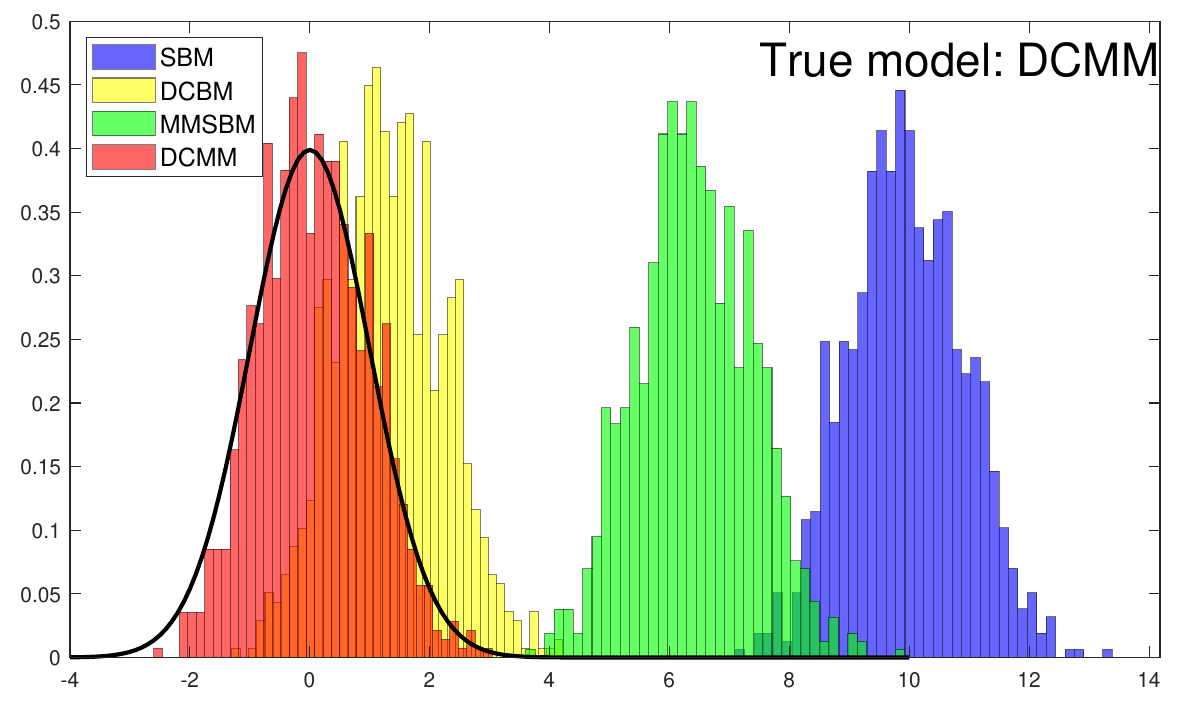}%dcmm_hist_0818.m
\includegraphics[width=.42\textwidth, height=.24\textwidth]{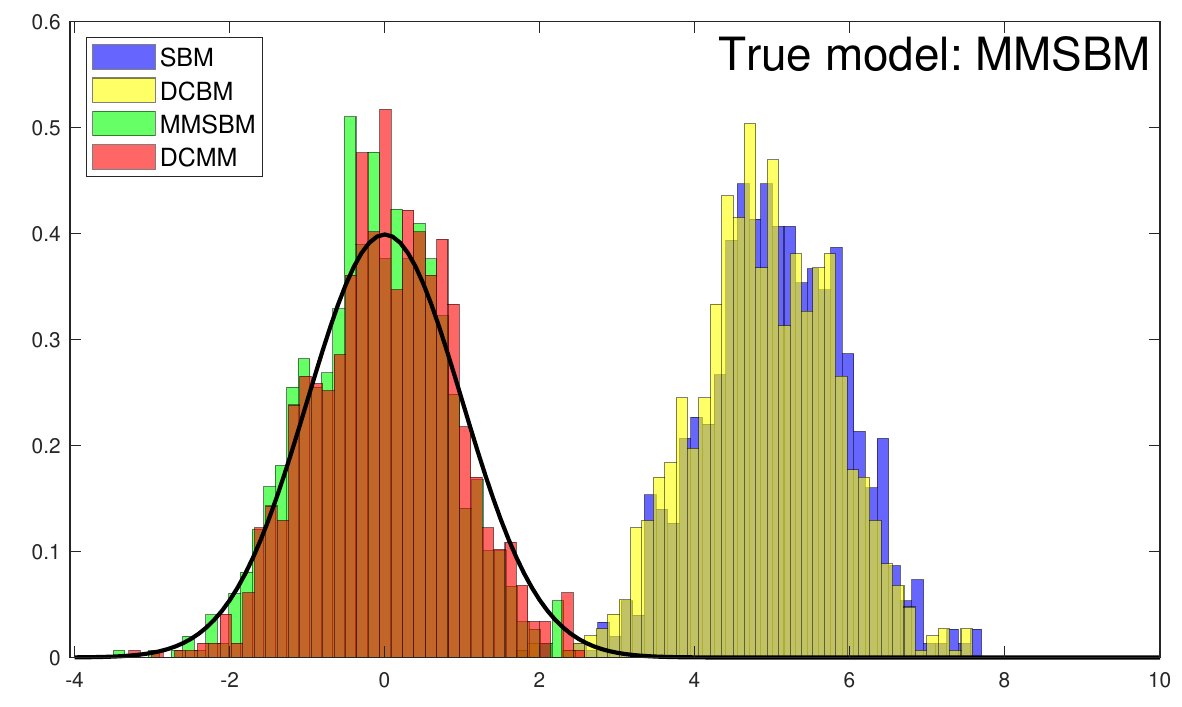}\\%mmsbm_hist_0818.m
\includegraphics[width=.42\textwidth, height=.24\textwidth]{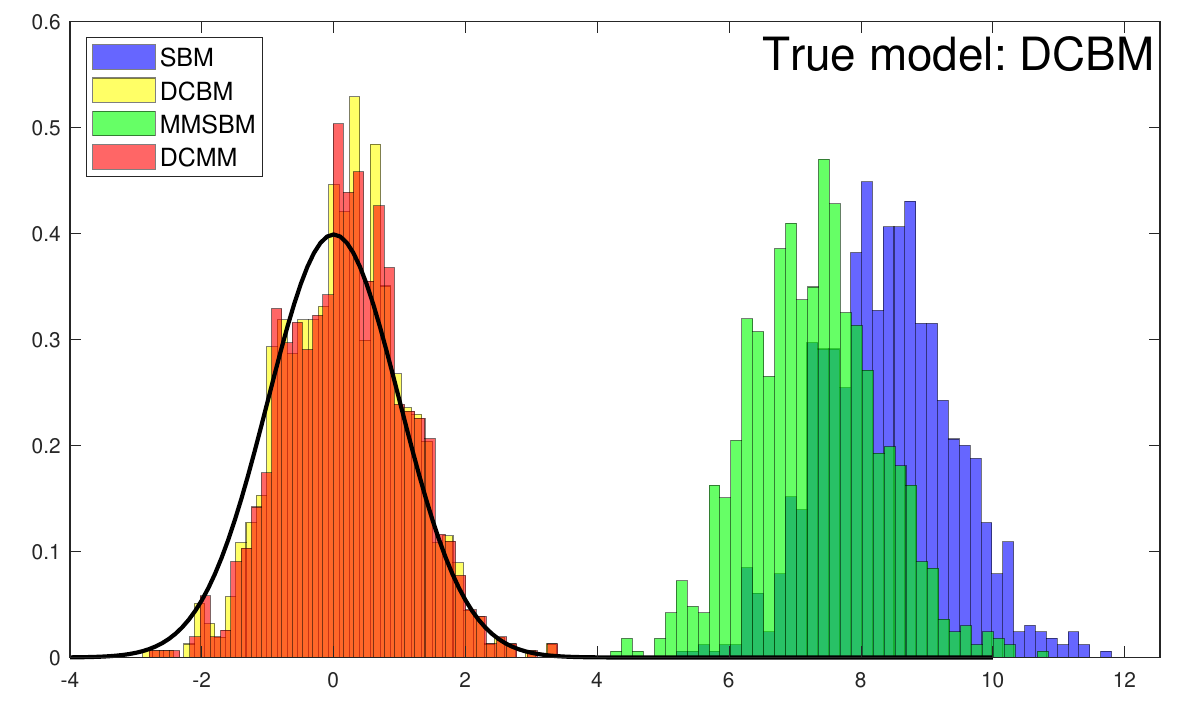}%dcbm_hist_0818.m
\includegraphics[width=.42\textwidth, height=.24\textwidth]{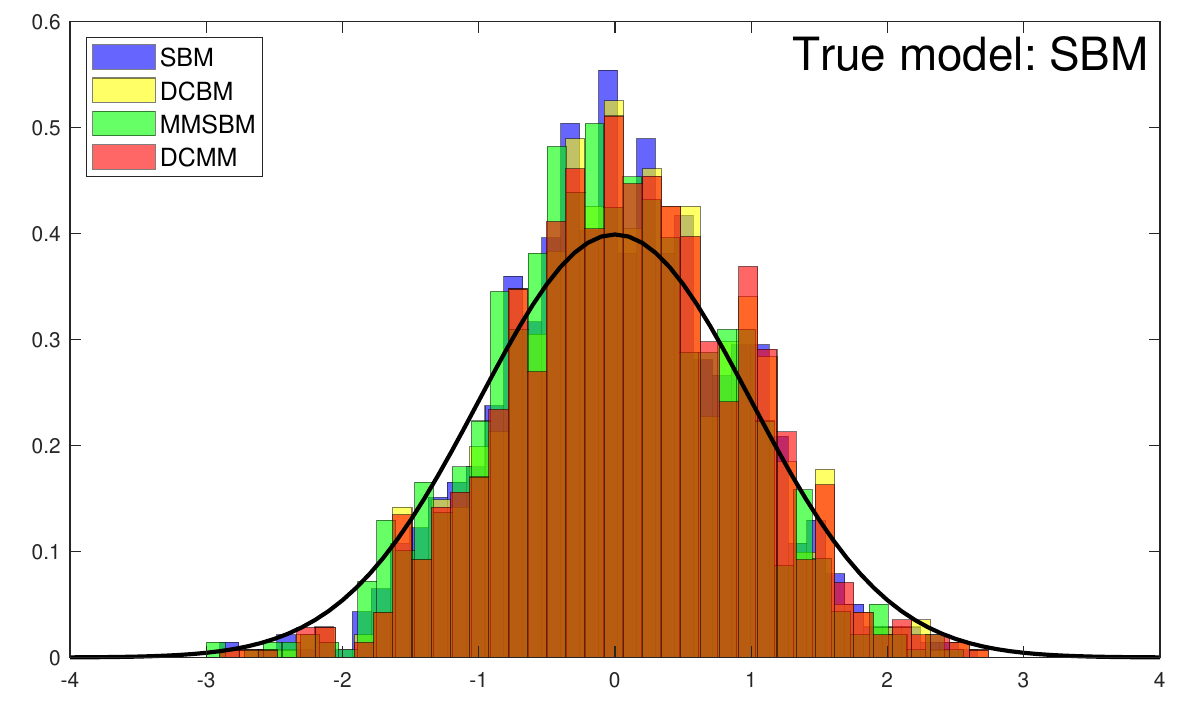}%sbm_hist_0818.m
\vspace{-1em}
\caption{Histograms of the GoF metrics in Experiment 1.1-1.4, where the off-diagonal entries of $P$ are equal to $0.2$. In each panel, the networks are generated from a true model in the block-model family, the four colors correspond to four GoF metrics, and the black curve is the density of $N(0,1)$.\label{Fig:2}}
\end{figure}
\spacingset{1.5}

\paragraph{Experiment 2: Type-I error control.} We study the type-I error control with the GoF metrics when the model is correctly specified. We focus on the setting where the true model is DCMM. The type-I error control for the other three models is similar and omitted here.  Fix $K=2$ and let $P$ have diagonals equal to $1$ and off-diagonals equal to $b$, for some $b\in (0,1)$. Let $\Pi$ be such that each community as $n/8$ pure nodes and that for the remaining nodes, half have $\pi_i=(x, 1-x)'$ and half have $\pi=(1-x, 1)'$, for some $x\in [0, 1/2]$. We use $(n, b, x)=(2000, 0.2, 0.25)$ as the basic setting and vary one parameter  each time: In setting (a), $n$ varies in $\{1200, 1600, \ldots, 3600\}$; in setting (b), 
$b$ varies in $\{0.1, 0.15,\ldots, 0.4\}$; in setting (c), $x$ varies in $\{0.05, 0.1, \ldots, 0.35\}$. 
For each setting, we  consider cases, with $\theta_i\overset{iid}{\sim}{\rm Unif}(0.1,0.3)$ (moderate degree heterogeneity) and  $\theta_i^{-1}\overset{iid}{\sim}{\rm Unif}(2,10)$ (severe degree heterogeneity), respectively (the average of $\theta_i$'s is approximately $0.2$ in both cases). For each setting, we generate $400$ networks, compute $T_n(\widehat\Omega^{\rm DCMM})$ and reject the null model DCMM when $|T_n(\widehat\Omega^{\rm DCMM})|$ exceeds the 97.5\% quantile of a standard normal (hence, this test targets to control the type-I error at the nominal level of $5\%$). 
Table~\ref{tab:1} displays the rejection rate, suggesting that the test achieves a good type-I error control.

\spacingset{1.2}
\begin{table}[tbh!]\centering
\scalebox{1}{
\begin{tabular}{l|c|rrrrrrrr}\hline
\multirow{3}{*}{Setting~(a)} &$n$ &1200 &1600 &2000 &2400 &2800 &3200 &3600 \\
\cline{2-9}
&$\theta_i\sim {\rm Unif}(0.1,0.3)$ &4.25 &5.25 &4.00 &4.00 &4.00 &3.25 &4.00 \\
&$\theta_i^{-1}\sim {\rm Unif}(2,10)$ & 3.50 &4.25 &2.75 &1.75 &4.00 &2.50 &3.25 \\\hline
\multirow{3}{*}{Setting~(b)} &$b$ &0.1 &0.15 &0.2 &0.25 &0.3 &0.35 &0.4 \\ 
\cline{2-9}
&$\theta_i\sim {\rm Unif}(0.1,0.3)$ &2.25 &3.25 &4.00 &4.25 &4.25 &3.00 &4.25 \\ 
&$\theta_i^{-1}\sim {\rm Unif}(2,10)$ &1.75 &3.00 &3.00 &2.50 &5.25 &4.25 &3.75 \\\hline
\multirow{3}{*}{Setting~(c)} &$x$ &0.05 &0.1 &0.15 &0.2 &0.25 &0.3 &0.35 \\ 
\cline{2-9}
&$\theta_i\sim {\rm Unif}(0.1,0.3)$ &3.25 &3.50 &4.75 &4.75 &3.00 &4.75 &4.75 \\ 
&$\theta_i^{-1}\sim {\rm Unif}(2,10)$ &3.50 &3.50 &1.25 &3.00 &3.25 &4.00 &6.00 \\\hline
\end{tabular}}
\caption{\small Empirical rejection rate of the level-$5\%$ GoF test for DCMM (based on 400 repetitions).\label{tab:1}}\label{tab: }
\end{table}

\spacingset{1.5}

\paragraph{Experiment 3: Power in detecting model mis-specification.} 
When the mis-specified model is still in the block model family and has the same $K$, Experiment~1 already shows that our proposed GoF metrics have good power (see Figures~\ref{Fig:1}-\ref{Fig:2}). 
In this experiment, we consider three more cases in which the mis-specified model either has a larger $K$ or is outside the block model family. 

{\it Experiment 3.1: Misspecification of $K$}. In this setting, the true model is a 3-community DCMM, but the assumed model is a 2-community DCMM. Fix $(n, K_0)=(3000, 3)$. Let $P\in\mathbb{R}^{3\times 3}$ be such that its diagonal and off-diagonal entries are 1 and $0.2$, respectively. We generate $\theta_i$'s independently generated from $\mathrm{Unif}(0.1, 0.3)$. Similarly as in Experiment~1, each community have $n/8$ pure nodes, and $\pi_i$'s of the remaining nodes are drawn independently from Dirichlet$(1/K_0, \cdots, 1/K_0)$. 
We compute the GoF-DCMM metric assuming $K=2$. The histogram of $T_n(\widehat{\Omega}^{\mathrm{DCMM}})$ based on 1000 repetitions is displayed in Figure~\ref{Fig:simu-power} (left panel). 
It shows a large departure from the standard normal density, and so the GoF metric has good power in detecting  misspecification of $K$. This setting has been covered in Lemma~\ref{lem:SNR-higher-rank}. The simulation results here are consistent with the statement of Lemma~\ref{lem:SNR-higher-rank}.  

{\it Experiment 3.2: A nonlinear DCMM model}. Fix $(n,K) = (3000, 2)$ and $f(x)=x^2+0.2$. Let $\Omega=\Theta\Pi P\Pi'\Theta$, where $(\Theta,\Pi, P)$ are generated in the same manner as in Experiment~3.1.  We still generate the networks so that the edges are independent, but the edge probabilities are $\mathbb{P}(A_{ij}=1)=f(\Omega_{ij})$,  for $1\leq i<j\leq n$. We call it the  nonlinear DCMM model. In Figure~\ref{Fig:simu-power} (middle panel), we plot the histogram of $T_n(\widehat{\Omega}^{\text{DCMM}})$ based on 1000 repetitions. 
It deviates from the standard normal density, suggesting a non-trivial power of our proposed GoF metric. 
This setting has also been covered in Lemma~\ref{lem:SNR-higher-rank}, where the signal-to-noise ratio (SNR) is hinged on $\{\lambda_{K+1}(f(\Omega))\}_{j\geq K+1}$, where $f(\Omega)$ is the element-wise transformation of $\Omega$. 
In general, when $f(\cdot)$ deviates more from a linear function, there is a more significant departure from the normal density.

{\it Experiment 3.3: A model with dependent edges}. All the misspecified models we have considered so far assume independent edges. We now consider a model that has dependent edges. 
Triadic closure is the property among three nodes A, B, and C, that if the connections A-B and A-C exist, there is a tendency for the new connection B-C to be formed. We follow the hierarchy triadic closure model  \citep{di2022hierarchy} to generate networks with $n=1000$ nodes, using the Matlab code available at \url{https://www.maths.ed.ac.uk/~dhigham/bistability.html}. We compute $T_n(\widehat{\Omega}^{\text{DCMM}})$ with $K=1$ and plot the histogram using 1000 repetitions. The result is in Figure~\ref{Fig:simu-power} (right panel). 
It suggests a non-trivial power.

\spacingset{1}
\begin{figure}[tb!]
\centering
\includegraphics[width=.32\textwidth, height=.28\textwidth]{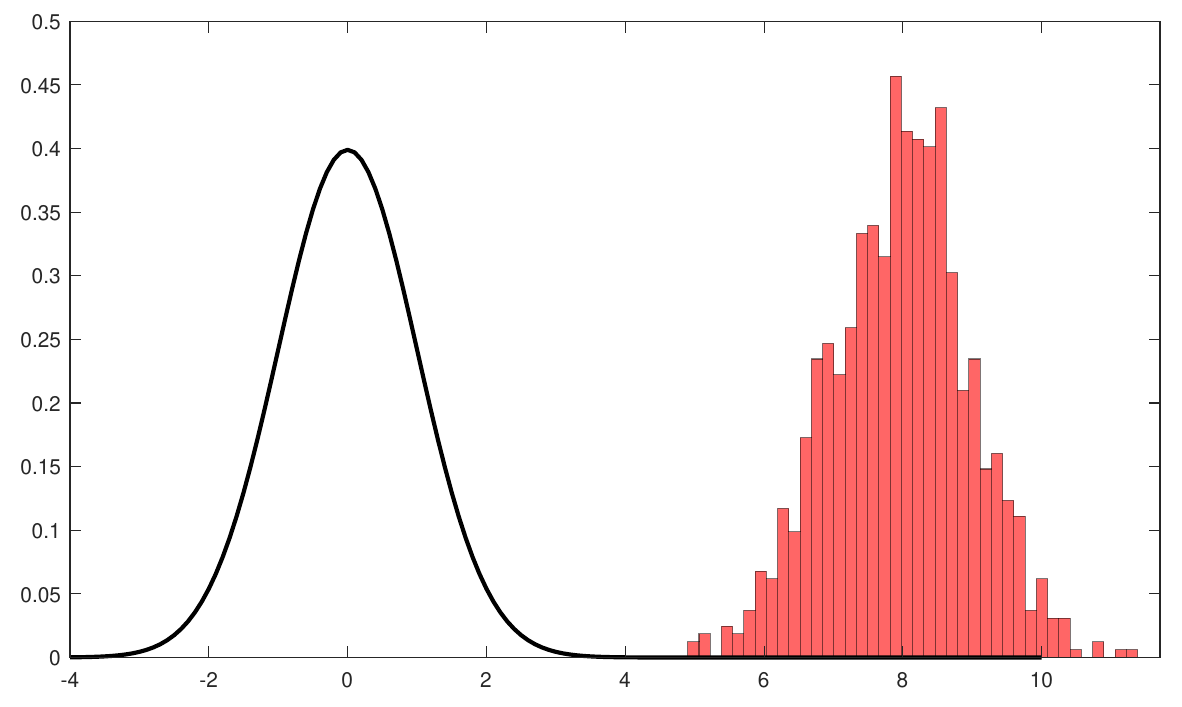}%dcmm_hist_0818.m
\includegraphics[width=.32\textwidth, height=.28\textwidth]{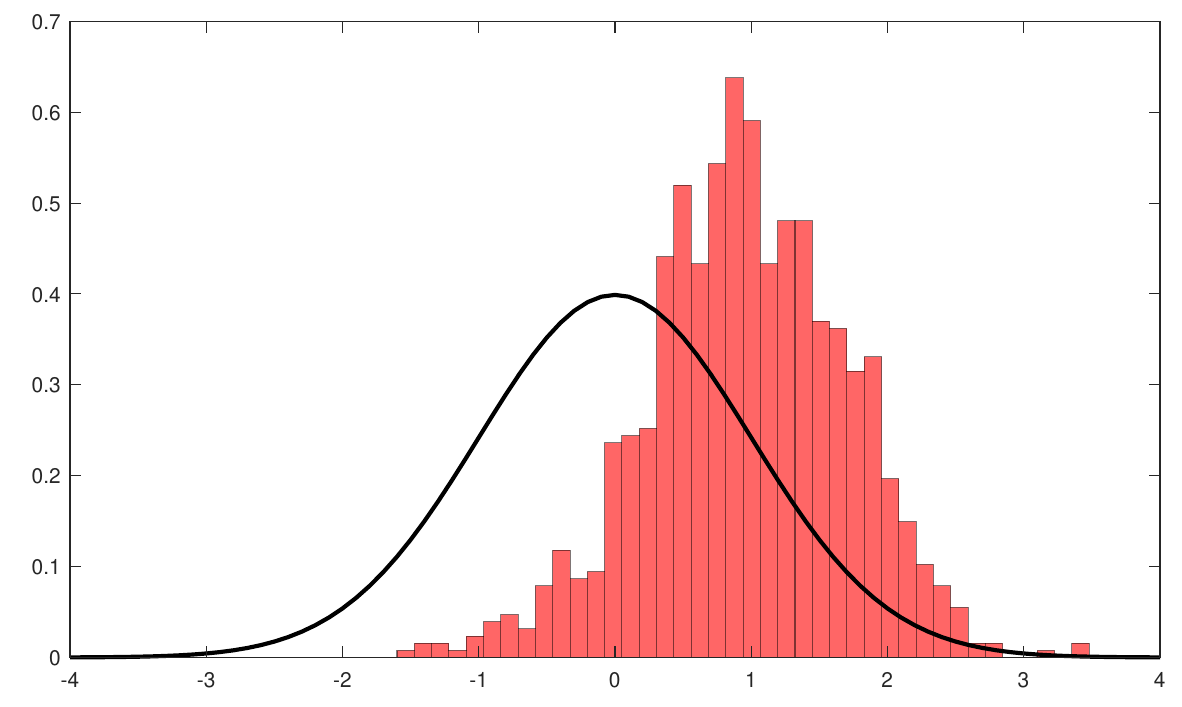}%mmsbm_hist_0818.m
\includegraphics[width=.32\textwidth, height=.28\textwidth]{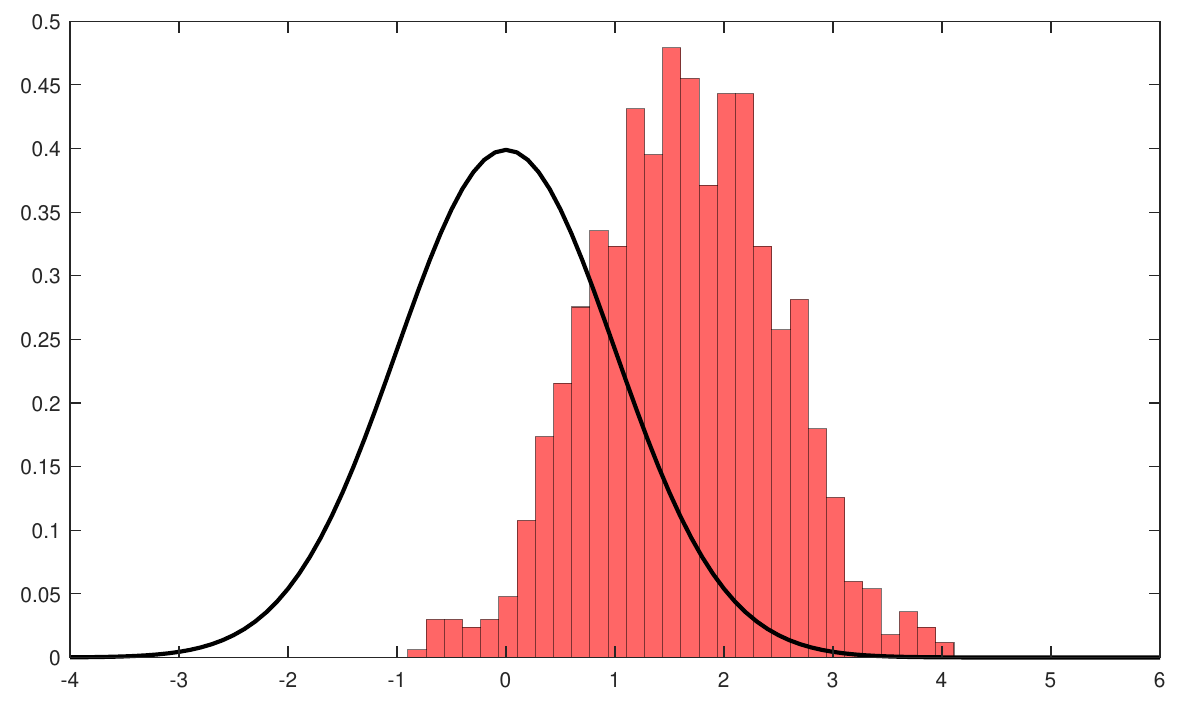}
\caption{Histograms of the GoF-DCMM metrics Experiment 3, where the true models are DCMM with misspecified $K$ (left), nonlinear DCMM (middle), and hierarchical triadic closure (right), respectively. In all panels, the black curve is probability density of $N(0,1)$.\label{Fig:simu-power}}
\end{figure}
\spacingset{1.5}

%number of possible edges $N=100$;
%birth rate $c_1=0.5$;
%death rate $c_2=0.5$;
%triadic closure effect $c_3=0.5$;
%size dependent rate $\hat c_3=0.5$;
%number of existing edges at time 0 is set as 50;
%final time of simulations $T=10$;
%stepsize of the simulation $D_t=$;

%\spacingset{1}
%\begin{figure}[htbh!]
%\centering
%\includegraphics[width=.5\textwidth]{pdf/DCMM_triadic_closure_2}
%\caption{Histogram of  our GoF  metric under DCMM  in 1000 repetitions  in Experiment 3,  under setting (S3).  The true model is the hierarchy triadic closure model. Black curve: probability density function of the standard normal distribution.\label{fig:triadic}}
%\end{figure}
%
%\spacingset{1.5}

\subsection{Additional results on the LastFM network}

Among the 12 networks analyzed in our paper, LastFM is the only one that does not have a ground truth (or partial ground truth) of $K$. Therefore, we compute the GoF metrics of the four models (DCMM, DCBM, MMSBM, SBM) for each value of $K$ from $2$ to $13$. 
The results are presented in Table~\ref{tab:SNAPv1}. It suggests that the best fit is DCMM with $K=3$, where the corresponding GoF metric is much smaller than other values in the table. This is why we report the results associated with $K=3$ in the main paper.

\spacingset{1.2}
\begin{table}[!tb]\centering
\caption{GoF metrics on the LastFM network, for the four models with different $K$.\label{tab:SNAPv1}}
\begin{tabular}{l|rrrrrr}\hline
$K$ &2 &3 &4 &5 &6 &7 \\\hline
DCMM &162.6 &\textbf{2.204} &63.05 &68.53 &62.09 &-32.59 \\
DCBM &165.0 &141.7 &123.6 &109.1 &101.2 &106.8 \\
MMSBM &196.3 &193.1 &189.1 &185.2 &183.4 &181.9 \\
SBM &160.1 &140.0 &128.7 &133.3 &145.8 &114.5 \\
\hline\hline
$K$ &8 &9 &10 &11 &12 &13 \\\hline
DCMM &24.39 &-215.8 &-131.4 &42.73 &-213.5 &-556.9 \\
DCBM &106.6 &86.00 &81.45 &85.60 &76.09 &71.25 \\
MMSBM &180.2 &178.7 &176.7 &173.8 &172.7 &173.6 \\
SBM &138.2 &100.8 &149.6 &87.21 &109.7 &97.03 \\\hline
\end{tabular}
\end{table}
\spacingset{1.5}

At first glance, it is surprising that choosing $K=3$ is enough for fitting such a large network with over 7000 nodes. In fact, this is because DCMM allows for mixed membership, so that $\pi_i$ can take infinitely many values even for a small $K$. This also partially explains why DCBM and SBM fit poorly, even when $K$ is large. In these models, $\pi_i$'s can only take $K$ distinct values, so that it may require a large $K$ to fit a large network. 
Another interesting observation is that $T_n(\widehat{\Omega}^{\text{DCMM}})$ can detect not only under-specification but also over-specification of $K$. For example, the GoF-DCMM metrics with $K>3$ (i.e., $K$ is likely over-specified) 
are all large in magnitude. We observe a similar phenomenon in simulations. 
This case is however not covered by Lemma~\ref{lem:SNR-higher-rank}, and we leave it to future exploration.

\subsection{Robustness  to tuning parameters} \label{supp:Real}

We recall that aside from the vertex hunting steps, GoF-MSCORE (Algorithm~\ref{alg:afmSCORE}) is tuning-free. For real-data analysis, we use KNN-SP for vertex hunting. KNN-SP has two tuning parameters $(N, \alpha)$. In Section~\ref{subsec:GoF-MSCORE}, we have explained how we choose $(N, \alpha)$ in real-data results.  
In this subsection, we investigate the robustness our method to the choice of $(N, \alpha)$, using the UKfaculty network for example. By the default choice (see Section~\ref{subsec:GoF-MSCORE}), $N=8$ and $\alpha=20$ for UKfaculty.
We let $N$ range in $\{4, 5, 6,\ldots, 22\}$ and  $\alpha$ range in $\{5, 10, 15, 20\}$, and report the values of $T_n(\widehat{\Omega}^{\text{DCMM}})$ in Table~\ref{tab:tuning}. It suggests that our results are relatively insensitive to the choice of $(N, \alpha)$, especially the choice of $\alpha$.

%In Section~\ref{sec:numeric}, we implement our GoF metrics on $11$ real networks. In Table~\ref{tab:tuning} below, we study how the results may depend on different choices of tuning parameters. We use $T_n(\widehat\Omega^{\rm DCMM})$ and the UKfaculty network as an example (the UKfaculty network has a large $K$, so the results are more likely to be unstable with tuning). Table~\ref{tab:tuning} suggests that our results are relatively insensitive to tuning parameters. 

\spacingset{1.2}

\begin{table}[!htp]\centering
\caption{The GoF metrics for DCMM (UKfaculty) for different tuning parameters $(\alpha,N)$ in the Vertex Hunting step.   The results are relatively insensitive different choices of $(\alpha,N)$.}\label{tab:tuning}
\scalebox{1}{
\begin{tabular}{c|rrrrrrrrrr}
\hline
$\alpha\backslash N$&$4$ &6 &8 &10 &12 &14 &16 &18 &20 & 22\\\hline
$5$ &1.700 &1.752 &1.767 &1.760 &1.753 &1.721 &1.688 &1.634 &1.578 & 1.552  \\
$10$ &1.732 &1.752 &1.767 &1.760 &1.753 &1.721 &1.688 &1.634 &1.578& 1.552  \\
$15$ &1.732 &1.752 &1.767 &1.760 &1.749 &1.721 &1.688 &1.634 &1.578 & 1.552  \\
$20$ &1.732 &1.752 &1.767 &1.763 &1.749 &1.721 &1.688 &1.634 &1.578 & 1.552  \\
\hline
\end{tabular}
} 
\end{table}

\spacingset{1.5}

\section{Proof sketch of the theoretical analysis} \label{sec:proof_sketch}

Given the theoretical analysis of normality for our GoF metrics is quite lengthy and involved, we provide a roadmap/proof sketch in this section. 

%\subsection{Main ideas of proving weak convergence of our GoF metric}

Recall from Section~\ref{subsec:SCCoracle} of the main article that our GoF metric is given by
\[
T_n(\widehat \Omega)  = U_{n,3} (\widehat \Omega) /\sqrt{6C_{n,3}},
\]
where
\[U_{n,3} (\widehat \Omega) = \sum_{i_1, i_2, i_3\, (dist)} (A - \widehat \Omega)_{i_1i_2} (A - \widehat \Omega)_{i_2i_3} (A - \widehat \Omega)_{i_3i_1}, \qquad C_{n,3} =  \sum_{i_1, i_2, i_3\, (dist)} A_{i_1i_2} A_{i_2i_3} A_{i_3i_1}\,. 
\] 
The goal is to show $T_n(\widehat \Omega) \to N(0,1)$ for different estimators $\widehat \Omega$ under different model assumptions. This is achieved by the following two steps:
\begin{itemize}[leftmargin=3.5em]
\item [Step 1.] For the ``oracle'' case where $\widehat \Omega = \Omega$, we prove that $T_n(\Omega)\to N(0,1)$ in law. This result is stated in Theorem~\ref{thm:SCC} of the main article. The main technical tool used in its proof is the martingale CLT \cite{hall2014martingale} via a careful construction of the martingale difference sequence and analyses of the conditional variance and moments using techniques in combinatorics. The detailed proof is given in Section~\ref{supp:SCC}. 

\item [Step 2.] Under different models (DCMM, SBM, MMSBM, DCBM), we prove that 
\begin{align}\label{new5}
\big| U_{n,3} (\widehat \Omega) -  U_{n,3} (\Omega)\big|/\sqrt{6C_{n,3}} = o_{p} (1),\qquad\text{as } n\to\infty.
\end{align}
To be more specific, we first claim $$C_{n,3}/\mathbb E (C_{n,3}) \to 1,\qquad\mathbb E (C_{n,3}) = {\rm tr} (\Omega^3 ) \{1 +o(1)\}.$$
This is achieved by examining the asymptotic order of $C_{n,3}$ and ${\rm tr} (\Omega^3 )$. The details are provided in  Lemma~\ref{lem:SCC} and the detailed proof in Section~\ref{supp:proof:SCC}. Consequently, it reduces to show that under the four modes in the block-model family
\begin{align}\label{new}
\big| U_{n,3} (\widehat \Omega) -  U_{n,3} (\Omega)\big| = o_{p}\{ {\rm tr} (\Omega^3 ) \},\qquad n\to\infty.
\end{align}
This is achieved by the following sub-steps:
\begin{enumerate}[leftmargin=2.8em]
\item [Step 2.1.]
%To see the role of the deviation $\widehat \Omega - \Omega$ in $U_{n,3} (\widehat \Omega) -  U_{n,3} (\Omega)$, 

Firstly, we decompose $U_{n,3} (\widehat \Omega) -  U_{n,3} (\Omega)$ into several terms, where each term is the trace of a matrix that is a functional of $\Delta = \widehat \Omega - \Omega$ and $W_1 = A - \Omega$. This is stated in Lemma~\ref{lem:diffU} of the main paper and proved in
Section~\ref{subsec:proof-diffU}
\item [Step 2.2.]
%What remains is to examine each term with the different GoF-SCORE estimator $\widehat \Omega$ substituted in.

Secondly, as a direct consequence of Step 2.1, we bound 
\begin{align}\label{new3}
\big| U_{n,3} (\widehat \Omega) -  U_{n,3} (\Omega)\big| &\lesssim \Vert \Delta \Vert^3 + \Vert {\rm diag} (\Omega)\Vert \Vert \Delta\Vert^2 + \Vert {\rm diag} (\Omega)\Vert ^2 \Vert \Delta \Vert+ \Vert W_1 \Delta \Vert \Vert \Delta \Vert  \notag\\
&\quad + \Vert W_1^2 \Delta \Vert  +  \Vert {\rm diag} (W_1^2) \Delta\Vert + \Vert {\rm diag} (\Omega)\Vert \Vert W_1\Delta\Vert, 
\end{align}
where $\Delta = \widehat \Omega - \Omega$, the estimation error, varies across the four models in the block-model family.

\item[Step 2.3.] Finally, it suffices to investigate the high-probability upper bounds of the following four terms in equation~\eqref{new3}:
\begin{align}\label{new4}
\|\Delta \|,\quad \|W_1 \Delta\|,\quad \|W_1^2 \Delta \|,\quad \|{\rm diag}(W_1^2) \Delta \|,
\end{align}
under each of the four models in the block-model family. This is the most delicate and technical part of our analysis, where we employ some tricks in combinatorics. 

\end{enumerate} 
\end{itemize}

In the sequel, we use the GoF-SCORE for DCMM as an example to demonstrate how we prove the high-probability upper bounds for the terms in (\ref{new4}).  The analysis for SBM, DCBM, and MMSBM is similar but simpler, with different representations of $\widehat \Omega$, and the details are relegated in Section~\ref{supp:other3}. 

A key step in deriving the high-probability upper bounds for the four terms in \eqref{new4} is to represent $\widehat \Omega$ as a tractable function of $A$. By doing so, the deviation $\Delta = \widehat \Omega - \Omega$ can be reduced to some simple function of $W_1 = A- \Omega$, where $W_1$ is the only random factor. Notice that $\Delta$ is of low rank (at most $2K$), and the four key terms all involve $\Delta$. Consequently, upper bounding these terms is ultimately  simplified to bounding quadratic forms of certain functions of $W_1$. Since $W_1$ is a symmetric random matrix with independent upper triangular entries,  the analysis becomes feasible by calculating the related moments. More specifically, the proof is streamlined into the following steps:
\begin{itemize}[leftmargin=3.9em]
\item [Step A.] We prove that the output $\widehat H$ of the net-rounding step in Algorithm~\ref{alg:afmSCORE} satisfies $\widehat H = \Pi_0$ with probability $1- o(n^{-3})$, where $\Pi_0$ is the (deterministic) net-rounded version of the ground truth $\Pi$. To prove this, we first derive a sharp rate for the initial estimate $\widehat\Pi^{\rm MS}$ in Algorithm~\ref{alg:afmSCORE} by an entry-wise eigenvector analysis of MSCORE. Next, according to the net rounding procedure, under the assumption of  the ground truth $\Pi$, we obtain the concentration of $\widehat H$. More details are provided in Theorem~\ref{susec:pf-MSCORE} and its proof in Section~\ref{sec:MSCORE-supp}.

This step enables us to analyze our GoF-MSCORE metric on the event that $\widehat H = \Pi_0$. This greatly simplifies many intermediate quantities. For example, in the remaining steps, we only need to consider $\widehat R_H = {\rm diag} (A{\bf 1}_n)^{-1}AH$ in Algorithm~\ref{alg:afmSCORE} for a deterministic $H$.

\item[Step B.] 

In view of the definitions of the estimators of $\Pi$, $\Theta$, and $P$ in Algorithm~\ref{alg:afmSCORE}, we first represent $\widehat{\Omega}$ as an explicit function of $(A, H, \widehat{V})$. We then study the connection between $\widehat{V}$ and $\widehat R_H = {\rm diag} (A{\bf 1}_n)^{-1}AH$. Eventually, we obtain $\widehat{\Omega}$ as an explicit function of $(A, H)$. 
%derive an alternative expression for the output $\widehat\Omega$ in Algorithm~\ref{alg:afmSCORE} that is easier to handle:
%\[
%\widehat{\Omega} =  AH(H'AH)^{-1}H'A. 
%\]
See more details in Lemma~\ref{lem:TvsH} and Section~\ref{subsec:TtoH} for its proof.

%In light of the algorithm procedures for re-estimating $\Pi$ and further estimating $\Theta, P$, we re-express the output $\widehat \Omega$ as 
%\[
%\widehat{\Omega} =  AH(H'AH)^{-1}H'A. 
%\]
%The lemma and details are provided in Section~\ref{subsec:TtoH}. 

\item [Step C.] Using the alternative expression from Step B and applying Taylor expansion, we derive the dominating terms of $\Delta = \widehat \Omega - \Omega$, which are linear and quadratic functionals of $W_1 = A - \Omega$ and have rank $K$. Consequently, the upper bounds of the four terms in \eqref{new4} are reduced to bounding operator norms of matrices that are simple functionals of $W_1$ and $H$. These matrices are also of rank $K$. Therefore, it suffices to prove the entry-wise bounds by computing the asymptotic order of the mean and variance of their entries, utilizing some combinatorial techniques. We refer to Lemma~\ref{lem:Delta} and Section~\ref{supp:main-auxiliary} for a complete proof.

 \end{itemize}

\section{Proof of Theorem \ref{thm:Laplacian}}\label{supp:section6} 

For preparations, we first show that 
\begin{itemize} 
\item(a).  If exactly $m$ of the $K$ nonzero eigenvalues of $\Omega$ are negative, then 
exactly $m$ of the $K$ nonzero eigenvalues of $U^{-1/2} \Omega U^{-1/2}$ are negative.  
\item(b). $\tau_1 = 1$ and $\rho_1 = n^{-1/2} {\bf 1}_n$.  
\end{itemize} 
Part (a) follows directly by Sylvester’s law of inertia \cite{HornJohnson}.  
Consider (b). Recall that $U = \diag(u_1, u_2, \ldots, u_n)$ and 
$u_i = \sum_{j = 1}^n \Omega(i, j)$, $1 \leq i \leq n$.   
Let $u = (u_1, u_2, \ldots, u_n)'$.  
 By direct calculations,  $L {\bf 1}_n = U^{-1} \Omega {\bf 1}_n = U^{-1} u = {\bf 1}_n$, so $(1, n^{-1/2} {\bf 1}_n)$ is an eigen-pair of $L$. By Perron's theorem \cite{HornJohnson}, 
 $(1, n^{-1/2} {\bf 1}_n)$ must be the first eigen-pair of $L$ and the claim follows.

We now prove Theorem \ref{thm:Laplacian}.   By the arguments above,  all remains to show is the second claim. Also, by the arguments above,  the $k$-th eigen-pair of 
$U^{-1/2} \Omega U^{1/2}$ is $(\tau_k,  c_k U^{1/2} \rho_k)$, where $c_k  = (\rho_k' U  \rho_k)^{-1/2}$.  
Write $\eta_k = c_k U^{1/2} \rho_k$ for short. 
By \cite[Section 2]{NMF}, the NMF is solvable for $U^{-1/2} \Omega U^{-1/2}$ if either $K = 2$ or $K \geq 3$ but 
\[
\sum_{k = 2}^K \frac{\tau_k}{\tau_1} (\frac{\eta_k(i)}{\eta_1(i)})^2 \leq \frac{1}{K-1}. 
\] 
Since $\eta_k = c_k U^{1/2} \rho_k$ and $U = \diag(u_1, u_2, \ldots, u_n)$ is a diagonal matrix, 
it is seen that 
\[
\frac{\eta_k(i)}{\eta_1(i)} = \frac{c_k \sqrt{u_i}  \rho_k(i)}{c_1 \sqrt{u_i} \rho_1(i)} =  \frac{c_k \rho_k(i)}{c_1 \rho_1(i)} = \sqrt{n} (c_k / c_1)  \rho_k(i), 
\] 
where we have used $\rho_1 = n^{-1/2} {\bf 1}_n$. 
Therefore, the LHS reduces to 
\[
\sum_{k = 2}^K  \frac{|\tau_k|}{\tau_1}  \biggl(\frac{\rho_1' U  \rho_1}{\rho_k' U \rho_k}\biggr)  n   \rho_k^2(i)   =  
\sum_{k = 2}^K  |\tau_k|  \cdot  \frac{\bar{u}}{\rho_k' U  \rho_k}   \cdot   n   \rho_k^2(i) \equiv  \sum_{k = 2}^K  |\tau_k|  \cdot \omega_k   \cdot    (\sqrt{n}   \rho_k(i))^2, 
\] 
where we have used $\tau_1 = 1$,  $1/ c_1^2  =  \rho_1' U \rho_1 =  \bar{u}$,  $1/c_k^2 = \rho_k U \rho_k$, and $\omega_k = \bar{u}/(\rho_k' U \rho_k)$. 
The claim now follows from the assumption of $\sum_{k = 2}^K  |\tau_k| \cdot  \omega_k \cdot \|\sqrt{n} \rho_k\|_{\infty}^2 \leq \frac{1}{K-1}$.

\section{Analysis of the SCC statistic under a general model}\label{supp:SCC}

 In this section,  we prove Theorem~\ref{thm:SCC} and Corollary~\ref{cor:SCC}.
The proof of Theorem~\ref{thm:SCC}  relies on martingale central limit theorem and some combinatorics tricks. Corollary~\ref{cor:SCC} can be regarded as a special case of Theorem~\ref{thm:SCC}. To show it, one only need to verify the conditions in  Theorem~\ref{thm:SCC} hold for the setting in Corollary~\ref{cor:SCC}. Before the proofs, we provide a useful lemma below which will be employed in the proof of  Theorem~\ref{thm:SCC}.

We introduce some notations for simplicity. These notations will be used not only throughout this section but also throughout the subsequent sections. For any two sequences $a_n$ and $b_n$, $ a_n \asymp b_n$ means  $a_n \leq C b_n $ and $b_n\leq C' a_n$ for some constants $C, C' >0$;  $a_n \lesssim b_n$ means $a_n\leq C b_n$ for some constant $C>0$. We use $C, C', c, c'$ to represent some generic positive constants independent of dimension $n$, which may vary from line to line.  For any matrix $M$ of dimension $m \times n$, we use either $M_{ij}$ or $M(i,j)$ to denote its $(i,j)$-th entry.

%{\color{red} If we choose the conditions in Theorem~\ref{thm:SCC}, very likely, there are particular $d_i, d_j \sim d_{\max}\gg {\overline{d}}$. And if $d_i\sim n\overline{\theta}\theta_i$, then $\overline{d}\sim n\overline{\theta}^2$ and $d_id_j/\overline{d}^2\sim \theta_i\theta_j/\overline{\theta}^2$ instead of $\theta_i\theta_j$. I think the appropriate normalizer should be $\overline{d}:=\sqrt{\sum_{i=1}^n d_i} =\Vert d\Vert_1^{1/2} $. And the conditions for $d_i$'s should be $\Vert d\Vert^2 \gg \max\{\Vert d\Vert_1,  d^2_{\max}\}$
%}

\begin{lemma}\label{lem:SCC}
Fix an integer $m\geq 3$. Let $\Omega$ satisfy the conditions in  Theorem~\ref{thm:SCC}, then as $n\to \infty$,
\begin{align*}
{\rm tr}(\Omega^m) \to \infty, \qquad  \mathbb{E} C_{n,m} ={\rm tr}(\Omega^m) \big(1+ o(1) \big)\,
% \quad  \text{and }\quad  \prod_{t=1}^s\|\Omega^{p_t}\|_{\max}=o\Big({\rm tr}(\Omega^m)\Big)
\end{align*}
%for any integer $s\geq 1$ and  sequence of integers $2\leq p_1, \cdots, p_s \leq m-1$ such that $\sum_{t=1}^s p_t\leq m $. 
Moreover, let $\bar{\Omega} = \Omega \circ (\mathbf{1}_n\mathbf{1}_n' -\Omega )$ where $\circ$ represents the Hadamard product and $\mathbf{1}_n$ is the all-one vector in $\mathbb{R}^n$. The results also hold for $\bar{\Omega}$.
\end{lemma}

\begin{proof}[Proof of Lemma~\ref{lem:SCC}]
First by the conditions in Theorem~\ref{thm:SCC}, we trivially have 
\begin{align*}
{\rm tr} (\Omega^m)\geq C\Vert u\Vert^{2m}/(n\bar{u})^{m}\to \infty\, . 
\end{align*}

Second, we prove that $\mathbb{E} C_{n,m} ={\rm tr}(\Omega^m) \big(1+ o(1) \big)$. We will show that ${\rm tr}(\Omega^m) - \mathbb{E} C_{n,m} = o(\Vert u\Vert^{2m}/(n\bar{u})^{m}) $. To do this, we first notice that ${\rm tr}(\Omega^m) = \sum_{i_1, \cdots, i_m=1}^n \Omega_{i_1i_2}\cdots  \Omega_{i_mi_1}$ and we can decompose $ \sum_{i_1, \cdots, i_m=1}^n$ into a combination of sums over $\widetilde{m}$ distinct indices, where $1\leq \widetilde{m}\leq m $, for instance, $\sum_{i_1, \cdots, i_{m-1} \, dist, i_m=i_1}$ as one sum. For simplicity, we write $\mathbf{I}_m : = \{i_1, \cdots, i_m\}$ and denote by $|\mathbf{I}_m|$ the cardinality of $\mathbf{I}_m$. We use the notation $[\![1,n]\!] : = \{1, \cdots, n\}$.  As a consequence, 
\begin{align*}
{\rm tr}(\Omega^m) - \mathbb{E} C_{n,m} & \leq  \sum_{\tilde{m} =1}^{m-1}\, \sum_{\substack{\mathbf{I}_m\in [\![1,n]\!]^m\\s.t. \, |\mathbf{I}_m| = \tilde{m} }}\Omega_{i_1i_2}\cdots  \Omega_{i_mi_1} \notag\\
& \lesssim  \sum_{\tilde{m} =1}^{m-1}\, \sum_{\substack{\mathbf{I}_m\in [\![1,n]\!]^m\\s.t. \, |\mathbf{I}_m| = \tilde{m} }} \Big( \frac{u_{i_1}}{\sqrt{n\bar{u}}} \Big)^2 \Big( \frac{u_{i_2}}{\sqrt{n\bar{u}}} \Big)^2 \cdots \Big( \frac{u_{i_m}}{\sqrt{n\bar{u}}} \Big)^2 
\end{align*}
where we used the condition $\Omega(i,j) \leq C u_iu_j/(n\bar{u})$. Furthermore, for a fixed $1\leq \tilde{m}\leq m-1$, 
\begin{align*}
&\sum_{\substack{\mathbf{I}_m\in [\![1,n]\!]^m\\s.t. \, |\mathbf{I}_m| = \tilde{m} }} \Big( \frac{u_{i_1}}{\sqrt{n\bar{u}}} \Big)^2\Big( \frac{u_{i_2}}{\sqrt{n\bar{u}}} \Big)^2 \cdots \Big( \frac{u_{i_m}}{\sqrt{n\bar{u}}} \Big)^2 
\\
&\lesssim \sum_{ \substack{a_1, \cdots, a_{\tilde{m}}\geq 1 \\ a_1+\cdots+ a_{\tilde{m}}= m } } \sum_{i_1, \cdots, i_{\tilde{m}}} \Big( \frac{u_{i_1}}{\sqrt{n\bar{u}}} \Big)^{2a_1} \Big( \frac{u_{i_2}}{\sqrt{n\bar{u}}} \Big)^{2a_2} \cdots \Big( \frac{u_{i_{\tilde{m}}}}{\sqrt{n\bar{u}}} \Big)^{2a_{\tilde{m}}} \notag\\
& \lesssim \sum_{ \substack{a_1, \cdots, a_{\tilde{m}}\geq 1 \\ a_1+\cdots+ a_{\tilde{m}}= m } } \frac{ \Vert u\Vert^{2\tilde{m}} } {(n\bar{u})^{\tilde{m}}} \cdot \Big(\frac{u_{\max}}{\sqrt{n\bar{u}}}\Big)^{2(a_1+ \cdots + a_{\tilde{m}}- \tilde{m})} \notag\\
& \ll \frac{ \Vert u\Vert^{2\tilde{m}} } {(n\bar{u})^{\tilde{m}}}  \ll \frac{ \Vert u\Vert^{2 {m}} } {(n\bar{u})^{{m}}} 
\end{align*}
following from the conditions that $u_{\max}^2/ (n\bar{u}) = o(1)$ and $n\bar{u}/\Vert u\Vert^2 = o(1) $. We therefore conclude that ${\rm tr}(\Omega^m) - \mathbb{E} C_{n,m}  = o\big( {\rm tr}(\Omega^m) \big)$, which is equivalent to $\mathbb{E} C_{n,m}   = {\rm tr}(\Omega^m)  (1+ o(1))$.

In the end, we claim the results also hold for $\bar{\Omega}$. It is worthy noting that $\bar{\Omega}(i,j) = \Omega(i,j) \big(1-   \Omega(i,j)\big)$. By the conditions in Theorem~\ref{thm:SCC}, we observe that $\Omega(i,j)= o(1)$ for all $1\leq i, j \leq n$. It follows that $\bar{\Omega}(i,j) = {\Omega}(i,j) (1+ o(1)) $, which implies that 
\begin{align*}
{\rm tr} \big(\bar{\Omega}^m  \big) &= \sum_{i_1, \cdots, i_m} \bar{\Omega}(i_1, i_2)\cdots\bar{\Omega}(i_m, i_1) 
\\
&= \sum_{i_1, \cdots, i_m}  \Omega (i_1, i_2)\cdots \Omega(i_m, i_1) (1+o(1)) = {\rm tr} \big(\Omega^m  \big) (1+ o(1))\, .
\end{align*}
%and 
%\begin{align*}
%\Vert \bar{\Omega}^k\Vert_{\max} =  \max_{i,j} \bar{\Omega}^k (i,j) &= \max_{i,j} \sum_{i_1, \cdots, i_{k-1}}\bar{\Omega}(i,i_1)\bar{\Omega}(i_1,i_2)\cdots \bar{\Omega}(i_{k-1}, j) \notag\\
%&= \max_{i,j} {\Omega}^k (i,j) (1+o(1)) = \Vert {\Omega}^k\Vert_{\max}  (1+o(1))\,. 
%\end{align*} 
Therefore, the results follows directly for $\bar{\Omega}$.
\end{proof}

\subsection{Proof of Theorem~\ref{thm:SCC}}\label{supp:proof:SCC}
Since the condition that $\widehat{\Omega} = \Omega$ with probability $1$, it suffices to show $\psi_{n, m}(\Omega)  \goto N(0,1)$. Recall the definition $\psi_{n, m}(\Omega) =  U_{n, m}({\Omega}) / \sqrt{2 m C_{n, m}}\, .$ To achieve the goal, we split the proof into two parts:
\begin{align} \label{SCC:2goal}
(1)\quad  \frac{U_{n,m}   - \mathbb{E} U_{n,m} }{\sqrt{{\rm var}(U_{n,m})}} \overset{d}{\longrightarrow} N(0,1), \qquad \text{and }  \qquad  (2)\quad  \frac{C_{n,m}}{\mathbb{E}C_{n,m}} \overset{p}{\longrightarrow} 1\,,
\end{align}
where we write $U_{n,m}\equiv U_{n, m}({\Omega})$. We denote $CC[1,k]\subset\{(i_1,i_2,\ldots,i_m):1\leq i_1,\ldots,i_m\leq k\}$ the set of $m$-cycles taking distinct values in $[\![1, k]\!]$. Therefore,  by definition, 
\begin{align*}
U_{n,m} = 2m \sum_{CC[1,n]} W_{i_1 i_2}{W}_{i_2 i_3}\ldots {W}_{i_m i_1}
\end{align*}
since each unique cycle corresponds to $2m$ different representations which can be construct by rotating and/or flipping it.  The two statements in (\ref{SCC:2goal}) indeed imply $\psi_{n, m}(\Omega)  \goto N(0,1)$ due to the  facts that $\mathbb{E}U_{n,m} = 0$ and the derivations that
\begin{align*}
 {\rm var}(U_{n,m}) &=(2m)^2 \sum_{CC[1,n]}\mathbb{E} {W}^2_{i_1 i_2} \mathbb{E}{W}^2_{i_2 i_3} \ldots \mathbb{E}{W}^2_{i_m i_1}  \notag\\
 &= (2m)^2 \sum_{CC[1,n]}\Omega_{i_1i_2}(1- \Omega_{i_1i_2})  \Omega_{i_2i_3}\ldots\Omega_{i_mi_1}(1-\Omega_{i_mi_1}) \notag\\
 & =  (2m)^2 \big[\sum_{CC[1,n]}\Omega_{i_1i_2}\Omega_{i_2i_3}\ldots\Omega_{i_mi_1}\big] (1+ o(1))= 2m\mathbb{E} (C_{n,m}) (1+o(1))
\end{align*}
Here  the third step is due to $\max_{i,j} \Omega(i,j) = o(1)$ by the assumptions in Theorem~\ref{thm:SCC}.

In the sequel, we first prove (1) in (\ref{SCC:2goal}). Recall the notation $\bar{\Omega} = \Omega \circ (\mathbf{1}_n\mathbf{1}_n' -\Omega )$ where $\circ$ stands for the Hadamard product and $\mathbf{1}_n$ is the all-one vector in $\mathbb{R}^n$. It is straightforward to observe from this notation that $\bar{\Omega}(i,j) = \Omega(i,j) (1-\Omega(i,j)) = {\rm var}(A_{ij}) = \mathbb{E} W_{ij}^2$ for all $1\leq i\neq j \leq n$. Hence, we can write $ {\rm var}(U_{n,m})  = (2m)^2 \sum_{CC[1,n]}\bar{\Omega}_{i_1i_2}\bar{\Omega}_{i_2i_3}\ldots\bar{\Omega}_{i_mi_1}$.  We follow the proof in \cite{GC}
and construct the $\sigma$-algebra $\mathcal{F}_{n,k} = \sigma\big( \{A_{ij}\}_{1\leq i <j \leq k}\big)$ and the martingale difference sequence 
\begin{align}\label{notes1:xnk}
X_{n,k} = S_{n,k} - S_{n,k-1}, \quad {\rm where } \quad S_{n,k} : = \frac{\sum_{CC[1, k]} W_{i_1i_2}\cdots W_{i_mi_1}}{\sqrt{\sum_{CC[1,n]} \bar{\Omega}_{i_1i_2}\cdots \bar{\Omega}_{i_mi_1}}}.
\end{align}
In order to show $S_{n,n} = U_{n,m}/\sqrt{{\rm var}(U_{n,m})} \overset{d}{\longrightarrow} N(0,1)$, we apply the martingale central limit theorem in
\cite{hall2014martingale} and verify that
\begin{align}
&\sum_{k=1}^n \mathbb{E} (X_{n,k}^2 | \mathcal{F}_{n,k-1}) \overset{p}{\to} 1,\label{notes1:1}\\
&\sum_{k=1}^n \mathbb{E}(X_{n, k}^2 1_{\{|X_{n, k}|>\epsilon\}} | \mathcal{F}_{n, k-1}) \stackrel{p}{\rightarrow} 0,\quad \text{for any } \epsilon>0.\label{notes1:2}
\end{align}
We first verify \eqref{notes1:1}. 
For simplicity, we denote $[\![a,b]\!]=\{a,a+1,\ldots,b\}$ for integers $a,b\in\mathbb N_+$ and we write for short $M_n: = \sum_{CC[1,n]} \bar{\Omega}_{i_1i_2}\cdots \bar{\Omega}_{i_mi_1}$. 
It follows from the definition of $X_{n,k}$ in \eqref{notes1:xnk} and direct calculations that $X_{n,k} = 0$ for $k<m$, if $k\geq m$, 
\begin{align*}
X_{n,k} & = \frac{1}{\sqrt{M_n} } \sum_{CC[1, k]\setminus CC[1,k-1] } W_{i_1i_2}\cdots W_{i_mi_1}= \frac{1}{\sqrt{M_n} } \sum_{1\leq i <j\leq k-1} W_{ki}W_{kj} Y_{kij},
\end{align*}
where
\begin{align}\label{def:Y_kij}
Y_{kij}= \sum_{\substack{1\leq\ell_1, \cdots, \ell_{m-3}\leq k-1\\
\ell_1, \cdots, \ell_{m-3}(dist)\\ \ell_1, \cdots, \ell_{m-3}\notin\{i,j\}}} W_{i\ell_1} W_{\ell_1\ell_2} \cdots W_{\ell_{m-3}j}.
\end{align}
Observe the fact that
\begin{align*}
\mathbb{E} (X_{n,k}^2 | \mathcal{F}_{n,k-1}) = \frac{1}{M_n} \sum_{1\leq i <j \leq k-1}  Y_{kij}^2 \bar{\Omega}_{ki} \bar{\Omega}_{kj};\quad \mathbb{E} Y_{kij}^2 = \sum_{\substack{1\leq\ell_1, \cdots, \ell_{m-3}\leq k-1\\
\ell_1, \cdots, \ell_{m-3}(dist)\\ \ell_1, \cdots, \ell_{m-3}\notin\{i,j\}}} \bar{\Omega}_{i\ell_1}\bar{\Omega}_{\ell_1\ell_2} \cdots \bar{\Omega}_{\ell_{m-3}j}.
\end{align*}
By direct calculations, we obtain from the above equation that
\begin{align*}
\mathbb{E}\bigg(\sum_{k=1}^n\mathbb{E} (X_{n,k}^2 | \mathcal{F}_{n,k-1})\bigg) & = \frac{1}{M_n} \sum_{k=m}^n \,  \sum_{CC[1, k]\setminus CC[1,k-1] }\bar{\Omega}_{i_1i_2}\bar{\Omega}_{i_2i_3}\ldots\bar{\Omega}_{i_mi_1}\notag\\
& = \frac{1}{M_n}\sum_{CC[1, n] }\bar{\Omega}_{i_1i_2}\bar{\Omega}_{i_2i_3}\ldots\bar{\Omega}_{i_mi_1} = 1\, .
\end{align*}
%Since the proof for the mean equals $1$ is straightforward as the computations in ICML paper by applying the identity 
%\begin{align*}
%\mathbb{E} Y_{kij}^2 = \sum_{\substack{1\leq\ell_1, \cdots, \ell_{m-3}\leq k-1\\
%\ell_1, \cdots, \ell_{m-3}(dist)\\ \ell_1, \cdots, \ell_{m-3}\notin\{i,j\}}} \Omega_{i\ell_1}^*\Omega_{\ell_1\ell_2}^* \cdots \Omega_{\ell_{m-3}j}^*,
%\end{align*} 
Then, in order to verify \eqref{notes1:1}, we will prove that
\begin{align*}
{\rm var} \Big( \sum_{k=1}^n \mathbb{E} (X_{n,k}^2 | \mathcal{F}_{n,k-1})\Big) \to 0,\qquad\text{as }n\to\infty. 
\end{align*}
For the sake of simplicity, we define the index set 
\begin{align*}
\mathcal{L}(k,i,j) : = \{\ell:= (\ell_1, \cdots \ell_{m-3}): \ell_1, \cdots \ell_{m-3} \in [\![1, k-1]\!]\setminus\{i,j\}  \text{ are distinct}\}.
\end{align*}
And we also introduce a shorthand notation $M_{i\to \ell\to j} := M_{i\ell_1} M_{\ell_1\ell_2}\cdots M_{\ell_{m-3} j_1}$ for a matrix $M\in\mathbb R^{n,n}$. Based on these notations, we can write $Y_{kij} = \sum_{\ell\in \mathcal{L}(k,i,j)}W_{i\to \ell\to j}$.
Then by some delicate but elementary derivations, we have the  representation 
\begin{align}
 \sum_{k=m}^n\Big( \mathbb{E} (X_{n,k}^2 | \mathcal{F}_{n,k-1})  - \mathbb{E} (X_{n,k}^2)\Big) &= \frac{1}{M_n} \sum_{k= m}^n \sum_{1\leq i<j<k } \, \sum_{\ell\in \mathcal{L}(k,i,j)} \big((W_{i\to \ell\to j})^2 - \bar{\Omega}_{i\to \ell\to j} \big)\bar{\Omega}_{ki} \bar{\Omega}_{kj} \notag\\
 &+ \frac{1}{M_n} \sum_{k= m}^n \sum_{1\leq i<j<k }\,  \sum_{\substack{\ell^1, \ell^2\in \mathcal{L}(k,i,j) \\ dist}} W_{i\to \ell^1\to j}W_{i\to \ell^2\to j} \bar{\Omega}_{ki} \bar{\Omega}_{kj}
 \notag\\
&=: I_a + I_b,
\label{202305190-1}
\end{align}
Since 
\begin{align}
{\rm var} \Big( \sum_{k=1}^n \mathbb{E} (X_{n,k}^2 | \mathcal{F}_{n,k-1})\Big) \leq 2{\rm var} (I_a)+ 2{\rm var} (I_b),
\end{align}
we need to verify that ${\rm var} (I_a)\to 0$ and $ {\rm var} (I_b)\to 0$ respectively. 

For the first term $I_a$ in \eqref{202305190-1}, we have
\begin{align*}
{\rm var} (I_a) &= \frac{1}{M_n^2} \sum_{k_1, k_2 = m}^n \,\sum_{\substack{1\leq i_1<j_1\leq k_1-1 \\1\leq i_2<j_2\leq k_2-1 }} \, \sum_{\ell^1, \ell^2\in \mathcal{L}(k,i,j) }\notag\\
& \mathbb{E}\Big[(W_{i_1\to \ell^1\to j_1})^2 (W_{i_2\to \ell^2\to j_2})^2 -\bar{\Omega}_{i_1\to \ell^1\to j_1}  \bar{\Omega}_{i_2\to \ell^2\to j_2} \Big] \bar{\Omega}_{k_1i_1}\bar{\Omega}_{k_1j_1}\bar{\Omega}_{k_2i_2}\bar{\Omega}_{k_2j_2}\Big\}.
\end{align*}

We can analyze the composition of each summand on the RHS by considering its associated graph structure, as determined by the indices present in the subscripts. 
It is important to note that each summand on the RHS of the above equation is nonzero if and only if the two paths, $i_1\to \ell^1\to j_1$ and $i_2\to \ell^2\to j_2$, share at least one edge. These two paths, along with the edges corresponding to $\bar{\Omega}_{k_1i_1}\bar{\Omega}_{k_1j_1}\bar{\Omega}_{k_2i_2}\bar{\Omega}_{k_2j_2}$, give rise to two intersecting cycles denoted as $\mathcal{L}_1$ and $\mathcal{L}_2$. As a consequence, the number of distinct nodes present in $\mathcal{L}_1$ and $\mathcal{L}_2$, or the effective indices engaged in the summation,  is bounded by $2m-2$. Every unique edge within the graph contributes to a corresponding entry in $\Omega$, for instance, $(i,j)$ corresponds to $\Omega(i,j)$.  
By applying the inequality $ \bar{\Omega}_{ij} \lesssim \Omega_{ij} \lesssim u_iu_j/(n\bar{u})$,  we can infer  that each node $i$ in $\mathcal{L}_1\cup \mathcal{L}_2$ entails a factor of $u_i/\sqrt{n\bar{u}}$ in the final pattern of the corresponding nonzero summand, and the exponent of this factor is determined by the total count of unique edges connecting to node $i$. Based on these arguments,  each nonzero summand is bounded by certain terms of the form 
\begin{align}\label{form:varI_a}
 \prod_{a=1}^{\tilde{m} } \Big( \frac{u_{\ell_a}}{\sqrt{n\bar{u}}}\Big)^{2+\beta_a}  
\end{align} 
for some $m\leq \tilde{m}\leq 2m-2$ and $\beta_\alpha\in \{0, 1, 2\}$. With a little ambiguity of notations, we do not adopt the original index notations in $\mathcal{L}_1\cup \mathcal{L}_2$,   and instead employ generic indices $\{\ell_a\}_{a=1}^{\tilde{m}}$. This prevents us from specifying which node in $\mathcal{L}_1$ aligns with a particular node in $\mathcal{L}_2$. Subsequently, we investigate the summation of the form (\ref{form:varI_a}). By considering the summation over unrestricted indices, it yields that 
\begin{align*}
{\rm var} (I_a) \lesssim M_n^{-2} \sum_{\tilde{m}=m}^{2m-2} \sum_{\ell_1, \cdots, \ell_{\tilde{m}}} \prod_{a=1}^{\tilde{m} } \Big( \frac{u_{\ell_a}}{\sqrt{n\bar{u}}}\Big)^{2+\beta_a}  &\lesssim M_n^{-2} \sum_{\tilde{m}=m}^{2m-2}  \bigg( \frac{\Vert u\Vert^2 }{n\bar{u}}\bigg)^{\tilde{m}} \cdot \Big(\frac{u_{\max}}{\sqrt{n\bar{u}}}\Big)^{\sum_{a=1}^{\tilde{m}} \beta_\alpha} \notag\\
& \lesssim \bigg( \frac{\Vert u\Vert^2 }{n\bar{u}}\bigg)^{-2}  = o(1)\, . 
\end{align*}
Here we used $M_n \asymp \mathbb{E}C_{n,m} \asymp {\rm tr} (\Omega^m) \geq C\big(\Vert u\Vert^2/(n\bar{u})\big)^m $ and  the conditions $u_{\max}^2/ (n\bar{u}) = o(1)$, $n\bar{u}/\Vert u\Vert^2 = o(1) $.

Next, for the second term $I_b$ in \eqref{202305190-1}, we have
\begin{align*}
{\rm var}(I_b) =  \frac{1}{M_n^2} \sum_{k_1, k_2 = m}^n& \sum_{\substack{1\leq i_1<j_1\leq k_1-1 \\1\leq i_2<j_2\leq k_2-1 }} \,\sum_{\substack{\ell^1, \widetilde{\ell}^1 \in \mathcal{L}(k_1, i_1, j_1) \, dist\\\ell^2,\widetilde{\ell}^2 \in \mathcal{L}(k_2, i_2, j_2) \, dist}} \notag\\
& \mathbb{E} (W_{i_1\to \ell^1\to j_1}W_{i_1\to \widetilde{\ell}^1\to j_1}W_{i_2\to \ell^2\to j_2}W_{i_2\to \widetilde{\ell}^2\to j_2})  \bar{\Omega}_{k_1i_1} \bar{\Omega}_{k_1j_1} \bar{\Omega}_{k_2i_2} \bar{\Omega}_{k_2j_2}.
\end{align*}
The summand on the RHS is nonzero if and only if the four paths eventually merge into two paths: $i_1\to j_1$ and $i_2\to j_2$, in which all edges are repeated. Consequently, the sum over $4m-6$ indices can be reduced to a sum over $2m$ indices. Furthermore, the two paths $i_1\to j_1$ and $i_2\to j_2$, including their endpoints, must share at least two nodes and may have overlapping edges. To see this, if $\{i_1, j_1\}= \{i_2, j_2\}$,  it is trivial to identify the two overlapping nodes.  Otherwise,  since $\ell_1$ and $\widetilde{\ell}_1$ are distinct, there must be an edge present in $i_1\to \ell_1\to  j_1$ that does not appear in $i_1\to \widetilde{\ell}_1\to  j_1$.  This edge must be present in either $i_2\to \ell_2\to j_2$ or $i_2\to \widetilde{\ell}_2\to j_2$; otherwise, the expectation on this single edge will be zero.  Thus, this edge results in the overlapped two nodes in the final two paths  $i_1\to j_1$ and $i_2\to j_2$. 
% Therefore, the total number of  distinct indices reduces to $\leq 2m-2$. 

Based on the aforementioned arguments and the inclusion of $k_1, k_2$, the graph structure associated with the nonzero summand forms the two cycles which must share at least two nodes and may have overlapping edges. Moreover, the number of effective indices is bounded by $\leq 2m-2$. We can mimic the arguments in bounding the terms in ${\rm var}(I_a)$ by leveraging the condition $\Omega_{ij} \lesssim u_iu_j/(n\bar{u})$.  It yields that 
\begin{align*}
{\rm var} (I_b) \lesssim M_n^{-2} \bigg( \frac{\Vert u\Vert^2}{n\bar{u}}\bigg)^{2m-2} \Big(\frac{u_{\max}}{\sqrt{n\bar{u}}}\Big)^{4}  \lesssim  \bigg( \frac{\Vert u\Vert^2}{n\bar{u}}\bigg)^{-2} \Big(\frac{u_{\max}}{\sqrt{n\bar{u}}}\Big)^{4} = o(1)\, . 
\end{align*}

%We remark that we do not provide rigorous proof of the arguments from the overlap between paths to the final trace form of the upper bound  terms. But this is not too hard to obtain form some combinatorics and representation tricks in trace form. The heuristics behind it is the shared paths or edges contribute to $\Omega^c$ in the trace; while if two paths without overlaps in edges share the same two end points, they lead to  $\Omega^c\circ \Omega^c$.  For reader's convenience, we provide a simple example for the summand in ${\rm var}(I_b)$ for the case $m=4$. It then simplifies to 
%\begin{align*}
%\mathbb{E}\Big( W_{i_1\ell_1}W_{\ell_1j_1}W_{i_1\widetilde{\ell}_1}W_{\widetilde{\ell}_1j_1}W_{i_2\ell_2}W_{\ell_2j_2}W_{i_2\widetilde{\ell}_2}W_{\widetilde{\ell}_2j_2}\Omega_{k_1i_1}\Omega_{k_1j_1}\Omega_{k_2i_2}\Omega_{k_2j_2}\Big)\, .
%\end{align*}
%Since $\ell_1\neq \widetilde{\ell}_1$, $i_1<j_1$, the only way to make this term nonzero is $(i_1, j_1)= (i_2, j_2)$ and $(\ell_1, \widetilde{\ell}_1) = (\ell_2, \widetilde{\ell}_2) \text{ or } ( \widetilde{\ell}_2, \ell_2)$, or equivalently  two paths $i\to \ell \to j$ and $i\to \widetilde{\ell}\to j$ have two overlapping nodes $i,j$. This gives rise to a upper bound by freeing more indices 
%\begin{align*}
%\sum_{k_1, k_2} \sum_{i,j} \sum_{\ell, \widetilde{\ell} }\Omega_{i\ell} \Omega_{\ell j}\Omega_{i\widetilde{\ell}} \Omega_{\widetilde{\ell} j} \Omega_{k_1i}\Omega_{k_1j} \Omega_{k_2i}\Omega_{k_2j}  =\sum_{i,j} (\Omega^2(i,j) )^4 = {\rm tr} (\Omega^2 \circ \Omega^2)^2\, . 
%\end{align*}

In order to finish the proof of (1) in \eqref{SCC:2goal}, we shall verify \eqref{notes1:2}.
%\begin{align*}
%\sum_{k=1}^n \mathbb{E} (X_{n,k}^2 \mathbf{1}(|X_{n,k}|>\epsilon) | \mathcal{F}_{n,k-1}) \to 0 \quad \text{in probability.}
%\end{align*}
Similarly to the arguments in the proof in \cite{GC}, it suffices to show that 
\begin{align*}
\sum_{k=1}^n \mathbb{E} (X_{n,k}^4)  = o(1)\, . 
\end{align*}
Towards that, we derive 
\begin{align*}
 \mathbb{E} (X_{n,k}^4|\mathcal{F}_{n, k-1}) \leq  & \frac{1}{M_n^2} \sum_{1\leq i<j \leq k-1} Y_{kij}^4 \bar{\Omega}_{ki} \bar{\Omega}_{kj} + \frac{C}{M_n^2} \sum_{ i<j , i'<j' (dist)} ^kY_{kij}^2Y_{ki'j'}^2 \bar{\Omega}_{ki} \bar{\Omega}_{kj}  \bar{\Omega}_{ki'} \bar{\Omega}_{kj'}\notag\\
 &  + \frac{C}{M_n^2} \sum_{i,j,j'(dist)}Y^2_{kij} Y^2_{kij'} \bar{\Omega}_{ki} \bar{\Omega}_{kj} \bar{\Omega}_{kj'}\,,
\end{align*}
which indicates 
\begin{align*}
\sum_{k=1}^n \mathbb{E} (X_{n,k}^4)\leq & \frac{1}{M_n^2} \sum_{k=m}^n\sum_{1\leq i<j \leq k-1} \mathbb{E}Y_{kij}^4  \bar{\Omega}_{ki} \bar{\Omega}_{kj} + \frac{C}{M_n^2}\sum_{k=m}^n \sum_{ i<j , i'<j'  (dist)} ^k \mathbb{E}Y_{kij}^2Y_{ki'j'}^2 \bar{\Omega}_{ki} \bar{\Omega}_{kj}  \bar{\Omega}_{ki'} \bar{\Omega}_{kj'}\notag\\
 &  + \frac{C}{M_n^2} \sum_{k=m}^n\sum_{i,j,j'(dist)} \mathbb{E}Y^2_{kij} Y^2_{kij'} \bar{\Omega}_{ki} \bar{\Omega}_{kj} \bar{\Omega}_{kj'}\, . 
\end{align*}
Recall the definition of $Y_{kij}$ in \eqref{def:Y_kij}. Each term on the RHS above can be represented as a sum with  a similar structure to ${\rm var} (I_b)$. Essentially, $Y_{kij}$ corresponds to a summation over a path connecting $i$ and $j$ of length $m-2$. The four paths constructed from $\mathbb{E}Y_{kij}^4$ or $\mathbb{E}Y_{kij}^2Y_{ki'j'}^2$ or $ \mathbb{E}Y^2_{kij} Y^2_{kij'}$ have to ultimately merge into two paths where all edges are repeated so that the expectation of the corresponding summand is nonzero. Furthermore, the final two cycles, which include the index $k$,  must share some nodes and possibly share some edges. This implies that the numbers of effective indices in the three summations are bounded by  $2m - 3$, $2m-1$ and $2m-2$ respectively.  By employing similar arguments used to bound ${\rm var}(I_a)$ and ${\rm var}(I_b)$ and the condition $\Omega(i,j) \lesssim u_iu_j/(n\bar{u})$, we can deduce that
\begin{align*}
\sum_{k=1}^n \mathbb{E} (X_{n,k}^4) \lesssim  \bigg( \frac{\Vert u\Vert^2}{n\bar{u}}\bigg)^{-2m} \bigg[ \bigg( \frac{\Vert u\Vert^2}{n\bar{u}}\bigg)^{2m-3}  + \bigg( \frac{\Vert u\Vert^2}{n\bar{u}}\bigg)^{2m-1}  + \bigg( \frac{\Vert u\Vert^2}{n\bar{u}}\bigg)^{2m-2} \bigg] = o(1)
\end{align*}
under the assumptions in Theorem~\ref{thm:SCC}.
We omit the details and finish the proof.

In the sequel, we verify (2) in \eqref{SCC:2goal}
%In the sequel, we claim that  $\frac{C_{n,m}}{\mathbb{E}C_{n,m}} \to 1$ in probability 
under our conditions. Employing Lemma~\ref{lem:SCC},  it suffices to show that 
\begin{align*}
{\rm var} (C_{n, m}) \ll (\mathbb{E} C_{n,m} )^2 =  ({\rm tr} \Omega^m)^2 (1+ o(1))\, . 
\end{align*}
To proceed, we first have 
\begin{align*}
{\rm var} (C_{n, m}) &= (2m)^2 {\rm var} \Big( \sum_{CC[1,n]} A_{i_1i_2} \cdots A_{i_mi_1}\Big) \notag\\
& = (2m)^2 \sum_{CC[1,n]^2} \mathbb{E}\Big(A_{i_1i_2} \cdots A_{i_mi_1}A_{i'_1i'_2} \cdots A_{i'_mi'_1}\Big) -\Omega_{i_1i_2} \cdots \Omega_{i_mi_1}\Omega_{i'_1i'_2} \cdots \Omega_{i'_mi'_1}
\end{align*}
Each summand on the RHS above is nonzero if and only if the two cycle $(i_1, \cdots, i_m)$ and $(i'_1, \cdots, i'_m)$ share at least one edge. This again  implies the number of effective indices in each summand bounded by $2m-2$. In the same manner to previous analysis, we can get 
\begin{align*}
{\rm var} (C_{n, m}) \lesssim  \bigg( \frac{\Vert u\Vert^2}{n\bar{u}}\bigg)^{2m-2 } \ll ({\rm tr} \Omega^m)^2
\end{align*}
under the assumptions that $u_{\max}^2/ (n\bar{u}) = o(1)$, $n\bar{u}/\Vert u\Vert^2 = o(1) $. 
 We therefore finish the proof. 

\subsection{Proof of Corollary~\ref{cor:SCC}}\label{supp:cor:SCC}
We only need to check the conditions in Theorem~\ref{thm:SCC} hold under the assumptions in Corollary~\ref{cor:SCC}. First, by the assumption that $c_0 \theta_i \theta_j \leq \Omega(i, j) \leq c_1 \theta_i \theta_j $, it is easy to see that 
\begin{align*}
u_i = \sum_{j=1}^n \Omega(i, j) \asymp \theta_i \sum_{j=1}^n \theta_j = n\overline{\theta} \theta_i, \qquad \bar{u}  = \frac{1}{n}\sum_{i,j=1}^n \Omega(i, j)\asymp n\overline{\theta}^2\, .
\end{align*} 
Here we adopt the notation $a_n\asymp b_n$ to indicate that there is a constant $C>1$ such that $C^{-1} b_n\leq a_n \leq Cb_n$ for any  two sequences $a_n$ and $b_n$. Consequently,  $u_i/\sqrt{n\bar{u}}\asymp  \theta_i$ for all $1\leq i \leq n$. As a result, 
\begin{align*}
 \Omega(i, j)\leq c_1 \theta_i\theta_j \leq C u_iu_j/(n\bar{u})\,.
\end{align*}
Next, by definition, 
\begin{align*}
{\rm tr} (\Omega^m) = \sum_{i_1, \cdots, i_m} \Omega(i_1, i_2)\ldots \Omega(i_m, i_1)\geq c_1^m \big(\sum_{i}\theta_i^2\big)^m \asymp \Big( \sum_{i=1}^ n u_i^2/ (n\bar{u})\Big)^m =\frac{ \Vert u\Vert^{2m}}{(n\bar{u})^m}\,.
\end{align*}
Lastly, we have 
\begin{align*}
\frac{n\bar{u}}{\Vert u\Vert^2}\asymp \frac{(n\overline{\theta})^2}{(n\overline{\theta})^2 \sum_{i=1}^n \theta_i^2} = 1/\Vert \theta\Vert^2 \to 0, \qquad \frac{u^2_{\max}}{n\bar{u}}  \asymp {\theta_{\max}^2} \to 0\,  
\end{align*}
following from the conditions in Corollary~\ref{cor:SCC}. We thus finish the proof.

\subsection{Proof of Theorem~\ref{thm:SCC} for the special case of $m=3$}

In order to increase readability of our proof and provide better intuition and visualization of various terms in the proof of Theorem~\ref{thm:SCC} in Section~\ref{supp:proof:SCC}, in this section we present the proof of  Theorem~\ref{thm:SCC} for the special case of $m=3$. In this special case, we state and prove  the following corollary.

\begin{cor}[Parameter-free limiting null (oracle case, $m=3$)]\label{cor:SCC2}
As $n \goto \infty$, assume 
\begin{align*}
\Omega(i,j) \leq C u_iu_j/(n\bar{u}),\qquad {\rm tr} (\Omega^3)\geq C\Vert u\Vert^{6}/(n\bar{u})^{3},\qquad \max\{ n\bar{u} / \|u\|^2,\,  u^2_{\max}/ (n\bar{u}) \}  \goto 0.
\end{align*}
Then, $\psi_{n, 3}({\Omega})  \goto N(0,1)$.   

\end{cor}

\begin{proof}[\bf Proof of Corollary~\ref{cor:SCC2}]

The proof of Corollary follows from verifying the following two conditions
\begin{align}
&(1)\quad  \frac{U_{n,3}   - \mathbb{E} U_{n,3} }{\sqrt{{\rm var}(U_{n,3})}} \overset{d}{\longrightarrow} N(0,1);\label{(1)} \\
&(2)\quad  \frac{C_{n,3}}{\mathbb{E}C_{n,3}} \overset{p}{\longrightarrow} 1,\label{(2)}
\end{align}
where
\begin{align*}
U_{n,3}= \sum_{i_1, i_2, i_3 (dist)} W_{i_1 i_2} W_{i_2 i_3} W_{i_3i_1}, \quad C_{n,3} =  \sum_{i_1, i_2, i_3 (dist)} A_{i_1 i_2} A_{i_2 i_3} A_{i_3i_1}\,. 
\end{align*}

We first prove \eqref{(1)}. Note that $\mathbb E(U_{n,3})=0$ and $\mathbb E(C_{n,3}) = \sum_{i_1, i_2, i_3 (dist)} \Omega_{i_1 i_2}  \Omega_{i_2 i_3}  \Omega_{i_3i_1}$. We also have
\begin{align*}
{\rm var}(U_{n,3})&=\sum_{i_1, i_2, i_3 (dist)}\mathbb{E} (W_{i_1 i_2}^2 W_{i_2 i_3}^2 W_{i_3i_1}^2)
\\
&=\{1+o(1)\}\sum_{i_1, i_2, i_3 (dist)} \Omega_{i_1 i_2} \Omega_{i_2 i_3} \Omega_{i_3i_1}
=\{1+o(1)\}\mathbb E(C_{n,3}).
\end{align*}
With an abuse of notation, let $CC[1,k]=\{(i_1,i_2,i_3):1\leq i_1,i_2,i_3\leq k; i_1,i_2,i_3 \text{ distinct}\}$ denote the set of $3$-cycles taking distinct values in $[\![1, k]\!]$.
Consider the $\sigma$-algebra $\mathcal{F}_{n,k} = \sigma\big( \{A_{ij}\}_{1\leq i <j \leq k}\big)$ and define the martingale difference sequence
\begin{align*}
X_{n,k} = S_{n,k} - S_{n,k-1}, \quad {\rm where } \quad S_{n,k} : = \frac{\sum_{CC[1,k]} W_{i_1 i_2} W_{i_2 i_3} W_{i_3i_1}}{\sqrt{\sum_{CC[1,n]} \bar{\Omega}_{i_1 i_2} \bar{\Omega}_{i_2 i_3} \bar{\Omega}_{i_3i_1}}},
\end{align*}
where $\bar\Omega_{ij}=\Omega_{ij}(1-\Omega_{ij})$.
In order to show $S_{n,n} = U_{n,3}/\sqrt{{\rm var}(U_{n,3})} \overset{d}{\longrightarrow} N(0,1)$, we apply the martingale central limit theorem in
\cite{hall2014martingale} and verify that
\begin{align}
&\sum_{k=1}^n \mathbb{E} (X_{n,k}^2 | \mathcal{F}_{n,k-1}) \overset{p}{\to} 1,\label{notes1:12}\\
&\sum_{k=1}^n \mathbb{E}(X_{n, k}^2 1_{\{|X_{n, k}|>\epsilon\}} | \mathcal{F}_{n, k-1}) \stackrel{p}{\rightarrow} 0,\quad \text{for any } \epsilon>0.\label{notes1:22}
\end{align}
We first verify \eqref{notes1:12}.
For simplicity, we write
\begin{align*}
M_n=\sum_{CC[1,n]} \bar{\Omega}_{i_1 i_2} \bar{\Omega}_{i_2 i_3} \bar{\Omega}_{i_3i_1}\asymp {\rm tr} (\Omega^3) \geq C\big(\Vert u\Vert^2/(n\bar{u})\big)^3 .
\end{align*}

 Note that $X_{n,k} = 0$ for $k<3$, and if $k\geq 3$,
\begin{align*}
X_{n,k}=\frac{\sum_{CC[1, k]\setminus CC[1,k-1]} W_{i_1 i_2} W_{i_2 i_3} W_{i_3i_1}}{\sqrt{M_n}}
=\frac{\sum_{1\leq i <j\leq k-1} W_{ki}W_{kj} W_{ij}}{\sqrt{M_n} }.
\end{align*}
Note that $Y_{kij}=W_{ij}$ in \eqref{def:Y_kij} in the case of $m=3$. Observe the fact that
\begin{align}\label{n1}
\mathbb{E} (X_{n,k}^2 | \mathcal{F}_{n,k-1}) = \frac{\sum_{1\leq i <j \leq k-1}  W_{ij}^2 \bar{\Omega}_{ki} \bar{\Omega}_{kj}}{M_n}.
\end{align}
By direct calculations, we obtain from the above equation that
\begin{align*}
\mathbb{E}\bigg(\sum_{k=1}^n\mathbb{E} (X_{n,k}^2 | \mathcal{F}_{n,k-1})\bigg) & = \frac{\sum_{k=3}^n\sum_{1\leq i <j \leq k-1}  \bar\Omega_{ij} \bar{\Omega}_{ki} \bar{\Omega}_{kj}}{M_n}
%\\
%&= \frac{\sum_{CC[1, n] }\bar{\Omega}_{i_1i_2}\bar{\Omega}_{i_2i_3}\ldots\bar{\Omega}_{i_mi_1}}{M_n} 
= 1\, .
\end{align*}
Then, in order to verify \eqref{notes1:1}, we will prove that
\begin{align}\label{n22}
{\rm var} \Big( \sum_{k=1}^n \mathbb{E} (X_{n,k}^2 | \mathcal{F}_{n,k-1})\Big) \to 0,\qquad\text{as }n\to\infty. 
\end{align}
Note that \eqref{n1} implies that
\begin{align*}
&{\rm var} \Big( \sum_{k=1}^n \mathbb{E} (X_{n,k}^2 | \mathcal{F}_{n,k-1})\Big)
= M_n^{-2}{\rm var}\Big( \sum_{i,j,k(dist)}  W_{ij}^2 \bar{\Omega}_{ki} \bar{\Omega}_{kj}\Big)
\\
&\qquad= M_n^{-2} \sum_{i,j,k(dist);i',j',k'(dist)}  \big[\mathbb{E}(W_{ij}^2W_{i'j'}^2)-\bar\Omega_{ij}\bar\Omega_{i'j'}\big] \bar{\Omega}_{ki} \bar{\Omega}_{kj}\bar{\Omega}_{k'i'} \bar{\Omega}_{k'j'}
\\
&\qquad= 2M_n^{-2} \sum_{i,j,k(dist);i,j,k'(dist)}  \big[\mathbb{E}(W_{ij}^4)-\bar\Omega_{ij}^2\big] \bar{\Omega}_{ki} \bar{\Omega}_{kj}\bar{\Omega}_{k'i} \bar{\Omega}_{k'j}
\\
&\qquad\lesssim \bigg( \frac{\Vert u\Vert^2}{n\bar{u}}\bigg)^{-6} \frac{\Vert u \Vert^4 \Vert u \Vert_3^6}{(n\bar u)^5} \leq \bigg( \frac{\Vert u\Vert^2}{n\bar{u}}\bigg)^{-6} \frac{\Vert u \Vert^8  u_{\max}^2}{(n\bar u)^5}  =  \bigg( \frac{\Vert u\Vert^2}{n\bar{u}}\bigg)^{-2} \frac{  u_{\max}^2}{n\bar u} =o(1),
\end{align*}
under the assumptions in Corollary~\ref{cor:SCC2}. Here in the third step, we used the fact that $\mathbb{E}(W_{ij}^2W_{i'j'}^2)-\bar\Omega_{ij}\bar\Omega_{i'j'} =0$ if $(i,j) \neq (i',j')$.
Thus, this proves \eqref{n22}.

In order to finish the proof of \eqref{(1)}, we shall verify \eqref{notes1:22}.
Similarly to the arguments in the proof in \cite{GC}, it suffices to show that 
\begin{align}\label{n2}
\sum_{k=1}^n \mathbb{E} (X_{n,k}^4)  = o(1)\, . 
\end{align}
Towards that, we derive 
\begin{align*}
 \mathbb{E} (X_{n,k}^4|\mathcal{F}_{n, k-1}) \leq  & \frac{1}{M_n^2} \sum_{1\leq i<j \leq k-1} W_{ij}^4 \bar{\Omega}_{ki} \bar{\Omega}_{kj} + \frac{C}{M_n^2} \sum_{ i<j , i'<j' (dist)} ^kW_{ij}^2W_{i'j'}^2 \bar{\Omega}_{ki} \bar{\Omega}_{kj}  \bar{\Omega}_{ki'} \bar{\Omega}_{kj'}\notag\\
 &  + \frac{C}{M_n^2} \sum_{i,j,j'(dist)}W^2_{ij} W^2_{ij'} \bar{\Omega}_{ki} \bar{\Omega}_{kj} \bar{\Omega}_{kj'}\,,
\end{align*}
which implies that, under assumptions,
\begin{align*}
\sum_{k=1}^n \mathbb{E} (X_{n,k}^4)&\leq  \frac{1}{M_n^2} \sum_{k=3}^n\sum_{1\leq i<j \leq k-1} \mathbb{E}W_{ij}^4  \bar{\Omega}_{ki} \bar{\Omega}_{kj} + \frac{C}{M_n^2}\sum_{k=3}^n \sum_{ i<j , i'<j'  (dist)} ^k \mathbb{E}W_{ij}^2W_{i'j'}^2 \bar{\Omega}_{ki} \bar{\Omega}_{kj}  \bar{\Omega}_{ki'} \bar{\Omega}_{kj'}\notag\\
 &  \qquad+ \frac{C}{M_n^2} \sum_{k=3}^n\sum_{i,j,j'(dist)} \mathbb{E}W^2_{ij} W^2_{ij'} \bar{\Omega}_{ki} \bar{\Omega}_{kj} \bar{\Omega}_{kj'}
\\
&\lesssim
\bigg( \frac{\Vert u\Vert^2}{n\bar{u}}\bigg)^{-6} \bigg[ \bigg( \frac{\Vert u\Vert^2}{n\bar{u}}\bigg)^{3}  + \bigg( \frac{\Vert u\Vert^2}{n\bar{u}}\bigg)^{5}  + \bigg( \frac{\Vert u\Vert^2}{n\bar{u}}\bigg)^{4} \bigg] = o(1).
\end{align*}
This proves \eqref{n2}, which concludes the proof of \eqref{(1)}.

Finally, we verify \eqref{(2)}. Applying Lemma~\ref{lem:SCC},  it suffices to show that 
\begin{align*}
{\rm var} (C_{n, 3}) \ll (\mathbb{E} C_{n,3} )^2 =  ({\rm tr} (\Omega^3))^2 (1+ o(1))\, . 
\end{align*}
To proceed, we first have 
\begin{align*}
&{\rm var} (C_{n, 3}) = 6^2 {\rm var} \Big( \sum_{i_1,i_2,i_3 (dist.)} A_{i_1i_2} A_{i_2i_3} A_{i_3i_1}\Big) \notag\\
& = 6^2 \sum_{i_1,i_2,i_3 (dist.)}\sum_{i_1',i_2',i_3' (dist.)} \Big[\mathbb{E}\Big(A_{i_1i_2} A_{i_2i_3} A_{i_3i_1}A_{i_1'i_2'} A_{i_2'i_3'} A_{i_3'i_1'}\Big) -\bar\Omega_{i_1i_2} \bar\Omega_{i_2i_3} \bar\Omega_{i_3i_1}\bar\Omega_{i_1'i_2'} \bar\Omega_{i_2'i_3'} \bar\Omega_{i_3'i_1'}\Big].
\end{align*}
Each summand on the RHS above is nonzero if and only if the two cycle $(i_1, i_2, i_3)$ and $(i'_1, i_2', i'_3)$ share at least one edge. Therefore, we further derive 
\begin{align*}
{\rm var} (C_{n, 3}) &  \lesssim \sum_{i_1,i_2,i_3, i_4 (dist)}\mathbb{E}\Big(A_{i_1i_2}^2 A_{i_2i_3} A_{i_3i_1}A_{i_1i_4} A_{i_2i_4} \Big)+ \sum_{i_1,i_2,i_3 (dist)}\mathbb{E}\Big(A_{i_1i_2}^2 A_{i_2i_3}^2 A_{i_3i_1}^2  \Big)\notag\\
&=  \sum_{i_1,i_2,i_3, i_4 (dist)}\mathbb{E}\Big(\bar \Omega_{i_1i_2} \bar \Omega_{i_2i_3} \bar \Omega_{i_3i_1}\bar \Omega_{i_1i_4} \bar \Omega_{i_2i_4} \Big)+ \sum_{i_1,i_2,i_3 (dist)}\mathbb{E}\Big(\bar \Omega_{i_1i_2} \bar \Omega_{i_2i_3} \bar \Omega_{i_3i_1}  \Big)\notag\\
& \lesssim  \bigg( \frac{\Vert u\Vert^2}{n\bar{u}}\bigg)^{4}  \frac{u_{\max}^2}{n\bar u} + \bigg( \frac{\Vert u\Vert^2}{n\bar{u}}\bigg)^{3}
\end{align*}
It follows from the conditions $u_{\max}^2/n\bar u = o(1)$ and $n\bar u/ \| u\|^2 = o(1)$ that
\begin{align*}
{\rm var} (C_{n, 3}) \ll  \bigg( \frac{\Vert u\Vert^2}{n\bar{u}}\bigg)^{4} \ll ({\rm tr} (\Omega^3))^2
\end{align*}
The proof is complete.

\end{proof}

\subsection{Proof of Corollary~\ref{cor:SCC} for the special case of $m=3$}
The proof is straightforward by substituting $m=3$ into the arguments in Section~\ref{supp:cor:SCC} and  calculating directly that
\begin{align*}
{\rm tr} (\Omega^3) = \sum_{i_1, i_2, i_3} \Omega(i_1, i_2)\Omega(i_2, i_3) \Omega(i_3, i_1)\geq c_1^3 \bigg(\sum_{i=1}^n\theta_i^2\bigg)^3 \asymp \Big( \sum_{i=1}^ n u_i^2/ (n\bar{u})\Big)^3 =\frac{ \Vert u\Vert^{6}}{(n\bar{u})^3}\,.
\end{align*}
Here we used $u_i/\sqrt{n\bar{u}}\asymp  \theta_i$ for all $1\leq i \leq n$, which follows from \begin{align*}
u_i = \sum_{j=1}^n \Omega(i, j) \asymp \theta_i \sum_{j=1}^n \theta_j = n\overline{\theta} \theta_i, \qquad \bar{u}  = \frac{1}{n}\sum_{i,j=1}^n \Omega(i, j)\asymp n\overline{\theta}^2\, .
\end{align*} 
We omit the remaining details as they are identical to those in Section~\ref{supp:cor:SCC}, thereby concluding the proof.

\section{Analysis of GoF-SCORE with a given $H$}\label{supp:GoF}
Consider the SCC statistic $T_n(\widehat{\Omega})$ as in \eqref{DefineTn}, where $\widehat{\Omega}$ is from GoF-SCORE (Algorithm~\ref{alg:afmSCORE}). In this section, we aim to prove that $T_n(\widehat{\Omega})\to N(0,1)$. 

Recall that $T_n(\widehat{\Omega})=U_{n,3}(\widehat{\Omega})/\sqrt{6C_{n,3}}$. The key is establishing the asymptotic normality of $U_{n,3}(\widehat{\Omega})$. Section~\ref{supp:SCC} has shown the asymptotic normality of $U_{n,3}(\Omega)$. It remains to bound
\[
\delta:=U_{n,3}(\widehat{\Omega})-U_{n,3}(\Omega), \qquad\mbox{where $\widehat{\Omega}$ is from GoF-SCORE with a given $H$}. 
\]
The analysis of $\delta$ can be roughly divided into two parts:
\begin{itemize}
\item {\it Re-expressing $\delta$ as a tractable function of $(\Omega, W, H)$}. 
First, in Section~\ref{subsec:proof-diffU}, we prove Lemma~\ref{lem:diffU}, which gives an expression of $\delta$ in terms of $(\Omega, W, \widehat{\Omega})$.   
%However, as explained in Section~\ref{subsec:challenge}, this is insufficient for deriving a tight large-deviation bound for $\delta$. In Section~\ref{subsec:TtoH}, 
Next, in Section~\ref{subsec:TtoH}, we study the GoF-SCORE procedure and further write $\widehat{\Omega}$ as an explicit function of $(A, H)$. Combining these results gives the desirable expression of $\delta$. 

\item {\it Large-deviation analysis of $\delta$}. %The upper triangle of $W$ contains independent Bernoulli variables.  
When the $H$ in GoF-SCORE concentrates at a non-stochastic $H_0$, each term in the expression of $\delta$ is a polynomial of the upper triangular entries of $W$ (which are independent Bernoulli's). 
In Section~\ref{supp:main}, we conduct large-deviation analysis for each term in $\delta$ and prove the main result, Theorem~\ref{thm:main}.  

%We then conduct large-deviation analysis of $\delta$, using the fact that the upper triangle of $W$ consists of independent Bernoulli variables and the assumption that $H$ concentrates at a non-stochastic $H_0$. This is the longest part of the proof and will be given in Section~\ref{supp:main}.  
\end{itemize}

%In the remaining of this section, we first study GoF-SCORE in Sections~\ref{subsec:GoForacle}-\ref{subsec:TtoH}, where Section~\ref{subsec:GoForacle} proves Lemma~\ref{lem:GoForacle} (which is about the population quantities related to GoF-SCORE), and Section~\ref{subsec:TtoH} gives an alternative expression of the $\widehat{\Omega}$ from GoF-SCORE, so as to connect it to $(A, H)$ in a more explicit way. 

\subsection{Proof of Lemma~\ref{lem:diffU}} \label{subsec:proof-diffU}

Define $f(M)=\sum_{i_1,i_2,i_3 \text{(dist)}}M(i_1,i_2)M(i_2,i_3)M(i_3,i_1)$ for a symmetric matrix $M\in\mathbb{R}^{n,n}$. Using this notation, we have
\[
U_{n,3}(\widehat{\Omega}) - U_{n,3}(\Omega)= f(A-\widehat{\Omega}) - f(A-\Omega) = f(W_1+\Delta)-f(W_1), 
\] 
where the first equality is by definition, and the second equality is from $A=\Omega+W-\diag(\Omega)$. An equality about $f(\cdot)$ was given in Theorem 1.1 of  \cite{JKL2019}: $f(M) = \tr(M^3) - 3\tr(M\circ M^2) + 2\tr(M\circ M\circ M)$. 
Plugging it into the above equation, we can write 
\beq \label{proof-diffU-0}
U_{n,3}(\widehat{\Omega}) - U_{n,3}(\Omega) = I_{n,1}-3I_{n,2}+2I_{n,3},
\eeq
where 
\begin{align*} 
I_{n,1}&={\rm tr}\{(W_1-\Delta)^3\}-{\rm tr}(W_1^3), \cr
I_{n,2}&={\rm tr}  \{(W_1-\Delta)\circ (W_1-\Delta)^2  \}-{\rm tr}(W_1\circ W_1^2),\cr
I_{n,3}&={\rm tr} \{(W_1-\Delta)\circ (W_1-\Delta)\circ (W_1-\Delta)\}-{\rm tr}(W_1\circ W_1\circ W_1).
\end{align*}
We further simplify each term above. For any two square matrices $B$ and $D$, $\tr((B-D)^3)=\tr(B^3)-3\tr(B^2D)+3\tr(BD^2)-\tr(D^3)$ (this can be shown by expanding $(B-D)^3$ into 8 terms and recursively applying the circular property of the trace operator). It follows that
\beq \label{proof-diffU-1}
I_{n,1} = 3{\rm tr}(W_1\Delta^2)-3{\rm tr}(W_1^2\Delta)-{\rm tr}(\Delta^3). 
\eeq
For $I_{n,2}$, note that $(W_1-\Delta)^2=W_1^2-W_1\Delta-\Delta W_1+\Delta^2$. As a result, $(W_1-\Delta)\circ(W_1-\Delta)^2$ decomposes into 8 terms, each with the form $\tr(M_1\circ M_2M_3)$, for some symmetric matrices $M_k\in\{W_1,\Delta\}$. We also note that $\tr(M_1\circ M_2M_3)=\tr(M_1' \circ (M_2M_3)')=\tr(M_1\circ M_3M_2)$. This allows us to combine equal terms. The calculations yield  
\beq \label{proof-diffU-2}
I_{n,2}=2\tr(\Delta\circ W_1\Delta) -2\tr(W_1\circ W_1\Delta)+\tr(W_1\circ\Delta^2)-\tr(\Delta\circ W_1^2)-\tr(\Delta\circ\Delta^2). 
\eeq
For $I_{n,3}$, we note that the order of three matrices in the Hadamard product $M_1\circ M_2\circ M_3$ can be arbitrarily re-shuffled. It follows that  
\beq \label{proof-diffU-3}
I_{n,3}=3\tr(W_1\circ\Delta\circ\Delta)-3\tr(W_1\circ W_1\circ\Delta)-\tr(\Delta\circ\Delta\circ\Delta).
\eeq
The claim follows by plugging \eqref{proof-diffU-1}-\eqref{proof-diffU-3} into \eqref{proof-diffU-0}.\qed

\subsection{Proof of Lemma~\ref{lem:AFMoracle}} \label{subsec:AFMoracle}
First, we show the simplex structure and the claim about $R_H$. According to the definition, $R_H=[\diag(\Omega{\bf 1}_n)]^{-1}\Omega H$. By the DCMM model, $\Omega=\Theta\Pi P\Pi'\Theta$. It follows that
\beq \label{GoForacle-1}
R_H = [\diag(\Theta \Pi P\eta)]^{-1}\Theta\Pi PG_H, \qquad \mbox{where}\quad  \eta=\Pi'\Theta{\bf 1}_n,  \quad G_H=\Pi'\Theta H. 
\eeq
Recalling that $V_H=[\diag(P\eta)]^{-1}PG_H$ and noticing that $\Theta$ is a diagonal matrix, we deduce from the above equation that
\[
R_H = [\diag(\Pi P\eta)]^{-1}\Pi PG_H = \underbrace{[\diag(\Pi P\eta)]^{-1}\Pi [\diag(P\eta)]}_{:=W}V_H. 
\]
Write $V_H=[v_1,\ldots,v_K]'$ and $W=[w_1,\ldots,w_n]'$. The above equation $R_H=WV_H$ implies that $r_i=\sum_{k=1}^K w_i(k)v_k$, $1\leq i\leq n$.  Moreover, by definition of $W$, $w_i=\frac{1}{\pi_i'P\eta}\, [\diag(P\eta)]\pi_i$. Note that $[\diag(P\eta)]\pi_i =\pi_i\circ P\eta$, and $\pi_i'P\eta=(\pi_i\circ P\eta)'{\bf 1}_K=\|\pi_i\circ P\eta\|_1$ (the second equality is because $\pi_i\circ P\eta$ is a nonnegative vector). Hence, we can re-write $w_i=\frac{1}{\|\pi_i\circ P\eta\|_1}(\pi_i\circ P\eta)$. Combining these arguments gives 
\beq \label{GoForacle-2}
r_i=\sum_{k=1}^K w_i(k)v_k, \qquad w_i = \frac{1}{\|\pi_i\circ P\eta\|_1}(\pi_i\circ P\eta). 
\eeq
This shows that each $r_i$ is a convex combination of $v_1,v_2,\ldots,v_K$ (i.e., $r_i$ is contained in the simplex with $v_1,\ldots,v_K$ as the vertices), where $w_i$ is the vector of convex combination coefficients. 
In addition, we note that the $k$th entry of $P\eta$ is $\geq P(k,k)\eta_k=\sum_{i=1}^n\theta_i\pi_i(k)>0$, implying that $P\eta$ is a strictly positive vector. Therefore, $w_i=e_k$ if and only if $\pi_i=e_k$; in other words, $r_i$ falls at the vertex $v_k$ if and only if $i$ is a pure node of community $k$.  
This proves the simplex structure and the claim about $R_H$. 

Next, we show the claims about $P$ and $\Theta$. Using the definition of $\eta$ in \eqref{GoForacle-1}, we have 
\beq \label{GoForacle-3}
\Omega{\bf 1}_n = \Theta\Pi P\Pi'\Theta{\bf 1}_n = \Theta\Pi P\eta\qquad\Longleftrightarrow\qquad e_i'\Omega{\bf 1}_n = \theta_i\cdot \pi_i'P\eta = \theta_i\|\pi_i\circ P\eta\|_1. 
\eeq
This proves the claim about $\Theta$. By definition of $G_H$ in \eqref{GoForacle-1}, $H'\Omega H=H'\Theta \Pi P\Pi'\Theta H=G_H'PG_H$. 
Additionally, $V_H=[\diag(P\eta)]^{-1}PG_H$, or equivalently, $PG_H=[\diag(P\eta)] V_H$. It follows that
\[
H'\Omega H = G_H'PP^{-1}PG_H =V_H'[\diag(P\eta)]P^{-1} [\diag(P\eta)]V_H. 
\]
We plug it into the definition of $Z_H$ to get 
\beq \label{GoForacle-4}
Z_H = V_H(H'\Omega H)^{-1}V'_H = [\diag(P\eta)]^{-1}P[\diag(P\eta)]^{-1}. 
\eeq
It yields that $\diag(Z_H)=[\diag(P\eta)]^{-2}\diag(P)=[\diag(P\eta)]^{-2}$ (note that we have assumed the identifiability condition that $P$ has unit diagonal entries). Consequently, $\diag(P\eta)= [\diag(Z_H)]^{-1/2}$, and 
\[
P =  \diag(P\eta)\, Z_H \, \diag(P\eta) = [\diag(Z_H)]^{-1/2}Z_H[\diag(Z_H)]^{-1/2}. 
\]
This proves the claim about $P$. \qed

\subsection{An alternative expression of GoF-SCORE} \label{subsec:TtoH}

The next lemma re-expresses the $\widehat{\Omega}$ from GoF-SCORE as an explicit function of $(A, H)$. 

\begin{lemma}\label{lem:TvsH}
If $\widehat{\Omega}$ is from the GoF-SCORE algorithm, 
%if the $\widehat{V}_H$ from vertex hunting has a full rank, the step of setting negative entries in $\widehat{w}_i$ to zero is skipped, and $\widehat{Z}_H=\widehat{V}_H(H'AH)^{-1}\widehat{V}_H'$ has positive diagonal entries,
 then
\beq \label{T-form}
\widehat{\Omega} =  [\diag(\widehat{W}{\bf 1}_K)]^{-1}AH(H'AH)^{-1}H'A [\diag(\widehat{W}{\bf 1}_K)]^{-1}. 
\eeq
In addition, if each estimated vertex in $\widehat{V}$ is a linear combination of $\widehat{r}_1,\ldots,\widehat{r}_n$, then 
\beq \label{H-form}
\widehat{\Omega} =  AH(H'AH)^{-1}H'A. 
\eeq
\end{lemma}

\begin{proof}[Proof of Lemma~\ref{lem:TvsH}]
First, we prove the expression in \eqref{T-form}. We start from a step-by-step expression of the GoF-SCORE algorithm in matrix form:
\begin{itemize}
\item[(a)] The node embedding obtains $\widehat{R}_H=[\diag(A{\bf 1}_n)]^{-1}AH$. Let $\widehat{V}_H=[\widehat{v}_1,\ldots,\widehat{v}_K]'$ be the output of applying vertex hunting to rows of $\widehat{R}_H$.  If $\widehat{V}_H$ is singular, re-set $\widehat{V}_H = I_K$. 
\item[(b)] For each $1\leq i\leq n$, obtain $\widehat{w}_i=(\widehat{V}_H^{-1})'\widehat{r}_i$. Write $\widehat{W}=[\widehat{w}_1,\ldots,\widehat{w}_n]'$. It is equivalent to the matrix form: $\widehat{W} = \widehat{R}_H\widehat{V}_H^{-1}$. 
\item[(c)] Let $\widehat{Z}_H = \widehat{V}(H'AH)^{-1}\widehat{V}'$. Obtain $\widehat{P\eta}(k) = \big|\widehat{Z}_H (k,k) \big|^{-1/2}  $ if $ \big|\widehat{Z}_H (k,k) \big|\neq 0$ and $\widehat{P\eta}(k) = 1$ otherwise. 
%$\diag(\widehat{P\eta}) = [ \diag(\widehat{Z}_H)]^{-1/2}$,  
% (by our assumption, $\widehat{Z}_H$ has positive diagonal entries; hence, taking absolute value of $\diag(\widehat{Z}_H)$ is not needed). 
\item[(d)] For each $1\leq i\leq n$, obtain $\widehat{\pi}_i^*=[\diag(\widehat{P\eta})]^{-1}\widehat{w}_i$ (note that the step of setting negative entries to zero is skipped; hence, $\widetilde{w}_i=\widehat{w}_i$). Write $\widehat{\Pi}^*=[\widehat{\pi}^*_1, \ldots,\widehat{\pi}^*_n]'$. The equivalent matrix form is $\widehat{\Pi}^* = \widehat{W}[\diag(\widehat{P\eta})]^{-1}$. 
\item[(e)] For each $1\leq i\leq n$, let $\widehat{\pi}_i=\|\widehat{\pi}_i^*\|_1^{-1}\widehat{\pi}_i^*$. Equivalently, $\widehat{\Pi}=[\diag(\widehat{\Pi}^*{\bf 1}_K)]^{-1}\widehat{\Pi}^*$.
\item[(f)] Obtain $\widehat{P}=|\diag(\widehat{Z}_H)|^{-1/2}\widehat{Z}_H|\diag(\widehat{Z}_H)|^{-1/2}$, in which if $|\diag(\widehat{Z}_H) (k,k)|=0$, re-set $|\diag(\widehat{Z}_H)(k,k) |= 1$. 
\item[(g)] For $1\leq i\leq n$, obtain $\theta_i$ by $e_i'A{\bf 1}_n/\|\widehat{\pi}_i\circ \widehat{P\eta}\|_1$. In matrix form, $\widehat{\theta} = [\diag(\widehat{\Pi}\widehat{P\eta})]^{-1} A{\bf 1}_n$. 
\item[(h)] Obtain $\widehat{\Omega}=\widehat{\Theta}\widehat{\Pi}\widehat{P}\widehat{\Pi}'\widehat{\Theta}$, where $\widehat{\Theta}=\diag(\widehat{\theta})$. 
\end{itemize}

We now use (a)-(h) to show \eqref{T-form}. Define $\widetilde{Z}_H = \widehat{Z}_H + {\rm diag} (\delta_{\widehat{Z}_H(1,1)}, \cdots,\delta_{\widehat{Z}_H(K,K)} )$ where $\delta_a$ takes value $1$ if $a= 0$ and $0$ otherwise. Based on this notation, it is easy to observe that ${\rm diag} (\widehat{P\eta} )=|\diag(\widetilde{Z}_H)|^{-1/2} $ in (c) and $\widehat{P}=|\diag(\widetilde{Z}_H)|^{-1/2}\widehat{Z}_H|\diag(\widetilde{Z}_H)|^{-1/2}$. 
Introduce a notation $\widehat{\gamma}= \widehat{W}|\diag(\widetilde{Z}_H)|^{1/2} {\bf 1}_K$. It follows from (a) that $\widehat{V}_H$ has a full rank so that $\widehat{W}$ is well-defined.    Using (c) and (d), we have $\widehat{\Pi}^*=\widehat{W}|\diag(\widetilde{Z}_H)|^{1/2}$, and $\widehat{\Pi}^*{\bf 1}_K = \widehat{\gamma}$. 
We plug them into (e) to get
\beq \label{TtoH-1}
\widehat{\Pi} =[\diag(\widehat{\Pi}^*{\bf 1}_K)]^{-1}\widehat{\Pi}^* = [\diag(\widehat{\gamma})]^{-1}\widehat{W}|\diag(\widetilde{Z}_H)|^{1/2}. 
\eeq
We combine \eqref{TtoH-1} with the expression that $\diag(\widehat{P\eta})=|\diag(\widetilde{Z}_H)|^{-1/2}$. 
It follows that
\beq \label{TtoH-2}
\widehat{\Pi}\widehat{P\eta} = \widehat{\Pi} \cdot \diag(\widehat{P\eta}){\bf 1}_K =   [\diag(\widehat{\gamma})]^{-1}\widehat{W}{\bf 1}_K. 
\eeq
By (g), $\widehat{\theta} = [\diag(\widehat{\Pi}\widehat{P\eta})]^{-1} A{\bf 1}_n$. Using \eqref{TtoH-2}, we have
\beq \label{TtoH-3}
\widehat{\Theta} =[\diag(\widehat{\Pi}\widehat{P\eta})]^{-1}\diag(A{\bf 1}_n) = \diag(A{\bf 1}_n)\, [\diag(\widehat{W}{\bf 1}_K)]^{-1}\diag(\widehat{\gamma}). 
\eeq
Combining the expression of $\widehat{\Pi}$ in \eqref{TtoH-1} and the expression of $\widehat{\Theta}$ in \eqref{TtoH-3}, we have
\beq \label{TtoH-4}
\widehat{\Theta}\widehat{\Pi}  = \diag(A{\bf 1}_n)[\diag(\widehat{W}{\bf 1}_K)]^{-1}\widehat{W}|\diag(\widetilde{Z}_H)|^{1/2}. 
\eeq 
Note that 
\[
\widehat{\Omega} =\widehat{\Theta}\widehat{\Pi}\widehat{P}\widehat{\Pi}'\widehat{\Theta}. 
\] 
We plug in the expression $\widehat{P}=|\diag(\widetilde{Z}_H)|^{-1/2}\widehat{Z}_H|\diag(\widetilde{Z}_H)|^{-1/2}$ from (f) and the expression of $\widehat{\Theta}\widehat{\Pi}$ from \eqref{TtoH-4}. 
It yields    
\beq \label{TtoH-5}
\widehat{\Omega}  =\diag(A{\bf 1}_n)[\diag(\widehat{W}{\bf 1}_K)]^{-1}\widehat{W}\widehat{Z}_H\widehat{W}'[\diag(\widehat{W}{\bf 1}_K)]^{-1}\diag(A{\bf 1}_n). 
\eeq
By (a) and (c), $\widehat{Z}_H=\widehat{V}_H(H'AH)^{-1}\widehat{V}_H'$ and $\widehat{W}=\widehat{R}_H\widehat{V}_H^{-1}=[\diag(A{\bf 1}_n)]^{-1}AH\widehat{V}_H^{-1}$. Hence, 
\beq \label{TtoH-6}
\widehat{W}\widehat{Z}_H\widehat{W}' =[\diag(A{\bf 1}_n)]^{-1}AH(H'AH)^{-1}H'A[\diag(A{\bf 1}_n)]^{-1}. 
\eeq
The claim \eqref{T-form} follows immediately by plugging \eqref{TtoH-6} into \eqref{TtoH-5} and noting that any two diagonal matrices are exchangeable in the product. 

Next, we prove the expression in \eqref{H-form} when each estimated vertex is a linear combination of $\widehat{r}_i$'s. In GoF-SCORE, the input matrix $H$ is such that each row is a weight vector. It follows that  $H{\bf 1}_K={\bf 1}_n$ and 
\beq \label{TtoH-7}
\widehat{R}_H{\bf 1}_K=[\diag(A{\bf 1}_n)]^{-1}AH{\bf 1}_K =[\diag(A{\bf 1}_n)]^{-1}A {\bf 1}_n={\bf 1}_n. 
\eeq
This means each $\widehat{r}_i$ satisfies $\widehat{r}_i'{\bf 1}_K=1$. When an estimated vertex $\widehat{v}_k$ is a linear combination of $\widehat{r}_i$'s, it also satisfies $\widehat{v}_k'{\bf 1}_K=1$. Consequently, $\widehat{V}_H{\bf 1}_K={\bf 1}_K$ for $\widehat{V}_H$ obtained from vertex hunting. If  $\widehat{V}_H$ is singular, re-set $\widehat{V}_H = I_K$ and it holds again that $\widehat{V}_H{\bf 1}_K = I_K {\bf 1}_K={\bf 1}_K$. This  further  yields that  ${\bf 1}_K=\widehat{V}_H^{-1}{\bf 1}_K$ for $\widehat{V}_H$ from GoF-SCORE. We plug this into \eqref{TtoH-7} and use the expression of $\widehat{W}$ in (b) to obtain
\beq \label{TtoH-8}
\widehat{W}{\bf 1}_K = \widehat{R}_H\widehat{V}_H^{-1}{\bf 1}_K =\widehat{R}_H{\bf 1}_K = {\bf 1}_n. 
\eeq
Combining \eqref{TtoH-8} with \eqref{H-form} gives \eqref{T-form}. 
\end{proof}

By Lemma~\ref{lem:TvsH}, we stick to the following simplified expression of $\widehat{\Omega}$ in the remaining proofs:
\[
\widehat{\Omega}=AH(H'AH)^{-1}H'A. 
\]

\subsection{Proof of Theorem~\ref{thm:main}}\label{supp:main}

Write $U_{n}(\cdot)=U_{n,3}(\cdot)$ and $C_n=C_{n,3}$ for short. Note that $T_n(\widehat{\Omega})=U_n(\widehat{\Omega})/\sqrt{6C_n}$. To show $T_n(\widehat{\Omega})\to N(0,1)$, it suffices to show that
\beq \label{main-proof-goal1}
T_n(\widehat{\Omega})/\sqrt{6\mathrm{tr}(\Omega^3)}\quad \overset{d}{\to}\quad N(0,1), 
\eeq
and
\beq \label{main-proof-goal2}
C_n/{\mathrm{tr}(\Omega^3)}\quad \overset{p}{\to}\quad 1.
\eeq
Given \eqref{main-proof-goal1}-\eqref{main-proof-goal2}, the claim follows immediately by applying the Slutsky's theorem. 

The proof of \eqref{main-proof-goal2} is shorter, hence, we consider it first. We hope to apply Lemma~\ref{lem:SCC} and the second claim in \eqref{SCC:2goal} with $m=3$. This requires verification of the conditions of Theorem~\ref{thm:SCC} for $m=3$, which reduce to the following statements: 
\[
\Omega(i,j) \leq C u_iu_j/(n\overline{u}),  \qquad  {\rm tr} (\Omega^3)\geq C\Vert u\Vert^{6}/(n\overline{u})^{3}, \qquad \max\{ n\overline{u} / \|u\|^2  ,\,  u^2_{\max}/n\overline{u} \} \goto 0. 
\]
In Theorem~\ref{thm:SCC}, $u=\Omega{\bf 1}_n$. In fact, the proof of this theorem only requires that the above holds for an arbitrary positive vector $u\in\mathbb{R}^n$. We choose $u=\|\theta\|_1\theta$ (as a result, $n\overline{u}=\|\theta\|_1^2$, $\theta_i=u_i/\sqrt{n\overline{u}}$, and $\|\theta\|^2=\|u\|^2/(n\overline{u})$). The above requirements become $\Omega(i,j)\leq C\theta_i\theta_j$, $\mathrm{tr}(\Omega^3)\geq C\|\theta\|^6$, and $\max\{\|\theta\|^{-2},\, \theta^2_{\max}\}\to 0$, which are guaranteed by  Condition~\ref{cond:afmSCORE}. In particular, to claim $\mathrm{tr}(\Omega^3)\geq C\|\theta\|^6$, we first notice that
$\tr(\Omega^3) = \|\theta\|^6 {\rm tr} \big((PG)^3\big)$. Since $P$ and $G$ are nonnegative matrices and ${\rm diag}(P)= I_K$, we can derive 
\begin{align}\label{trace-LB}
\tr(\Omega^3) \geq  \|\theta\|^6 \sum_{k=1}^K P_{kk}^3G_{kk}^3  = \|\theta\|^6 \sum_{k=1}^K G_{kk}^3 \geq K  \|\theta\|^6 \lambda_{\min} (G)^3 =  K  \|\theta\|^6 \Vert G^{-1}\Vert^{-3}
\geq C\|\theta\|^6\,. 
\end{align}
where the last inequality in \eqref{trace-LB} is based on (a) of Condition~\ref{cond:afmSCORE}. 
We can now quote Lemma~\ref{lem:SCC} and the second claim in \eqref{SCC:2goal} to get
\beq \label{main-proof-1}
\tr(\Omega^3)\to\infty, \qquad\mathbb{E}[C_n]=\mathrm{tr}(\Omega^3)[1+o(1)], \qquad C_n/\mathbb{E}[C_n]\overset{p}{\to} 1. 
\eeq
Then, \eqref{main-proof-goal2} follows immediately. 

We then consider \eqref{main-proof-goal1}. We quote the first claim in \eqref{SCC:2goal} to get
\beq \label{main-proof-2}
U_n(\Omega)/\sqrt{\mathrm{Var}(U_n(\Omega))}\quad \overset{d}{\to}\quad N(0,1). 
\eeq
By direct calculations (e.g., see \cite{JKL2019}), $\mathrm{Var}(U_n(\Omega))=6\sum_{i,j,k \text{(dist)}}\overline{\Omega}(i,j)\overline{\Omega}(j,k)\overline{\Omega}(k,i)$, where $\overline{\Omega}(i,j):=\Omega(i,j)[1-\Omega(i,j)]$. At the same time,
$\mathbb{E}[C_n]=\sum_{i,j,k \text{(dist)}}\Omega(i,j)\Omega(j,k)\Omega(k,i)$. Combining it with $\Omega(i,j)\leq C\theta_i\theta_j$, we have
\begin{align} \label{main-proof-3}
\Bigl|\mathrm{Var}&(U_n(\Omega))-6\mathbb{E}[C_n]\Bigr|  \leq C\sum_{i,j,k}[\Omega(i,j)]^2 \Omega(j,k)\Omega(k,\ell)\cr
&\leq C\sum_{i,j,k}\theta_i^3\theta_j^3\theta_k^2\leq C\|\theta\|_3^6\|\theta\|^2\leq C\theta_{\max}^2\|\theta\|^6 = o(\|\theta\|^6),  
\end{align}
where the last equality is from the assumption on $\theta_{\max}$ (see (a) of Condition~\ref{cond:afmSCORE}). 
%Note that the eigenvalues of $\Omega=\Theta\Pi P\Pi'\Theta$ are the same as the right eigenvalues of $P\Pi'\Theta^2\Pi=\|\theta\|^2PG$. 
Using \eqref{main-proof-3}, \eqref{trace-LB} and the fact that $\mathbb{E}[C_n]=\mathrm{tr}(\Omega^3)[1+o(1)]$ (see \eqref{main-proof-1}), we immediately have
\[
\mathrm{Var}(U_n(\Omega)) = 6\mathbb{E}[C_n] + o(1)\cdot\tr(\Omega^3) = 6\mathrm{tr}(\Omega^3)[1+o(1)]. 
\]
Combining it with \eqref{main-proof-2} gives
\[
U_n(\Omega)/\sqrt{6\tr(\Omega^3)}\quad\overset{d}{\to}\quad N(0,1). 
\]
Therefore, to show \eqref{main-proof-goal1}, it suffices to show $|U_n(\widehat{\Omega})-U_n(\Omega)|/\sqrt{\tr(\Omega^3)}\to 0$. In light of \eqref{trace-LB}, we only need to show that 
\beq \label{eq:finalgoal}
|U_n(\widehat{\Omega})-U_n(\Omega)|=o_{\mathbb{P}}(\|\theta\|^3). 
\eeq

It remains to show \eqref{eq:finalgoal}. 
We need some preparations. In Section~\ref{subsec:TtoH}, we have shown that $\widehat{\Omega}=AH(H'AH)^{-1}H'A$ with probability $1-o(1)$. Additionally,  using either Theorem~\ref{thm:MSCORE} or the assumption in  Theorem~\ref{thm:main} gives that $H=H_0 \text{ or } \Pi_0$, with probability $1-o(1)$. Hence, without loss of generality, we consider
\beq\label{main-proof-4}
\widehat{\Omega} = AH_0(H_0'AH_0)^{-1}H_0'A. 
\eeq
Let $W_1=W-\diag(\Omega)$ and $\Delta=\widehat{\Omega}-\Omega$. 
Lemma~\ref{lem:diffU} gives a decomposition: $U_n(\widehat{\Omega})-U_n(\Omega)=I_1+I_2+I_3$, with
\begin{align} \label{main-proof-5}
I_1 &= - {\rm tr} (\Delta^3) + 3 {\rm tr}( \Delta\circ  \Delta ^2) - 2 {\rm tr} (\Delta \circ \Delta \circ \Delta ), \cr
I_2 &= 3{\rm tr} ({W}_1\Delta^2) - 6{\rm tr} (\Delta \circ \Delta {W}_1)  -3 {\rm tr}( {W}_1^2 \Delta), \cr
I_3 &= - 3{\rm tr}( {W}_1 \circ \Delta^2)  + 6 {\rm tr}( {W}_1 \circ \Delta \circ \Delta)+ 3{\rm tr} ({W}_1^2 \circ \Delta ) \notag\\
& \quad + 6{\rm tr} ({W}_1\circ {W}_1\Delta) - 6{\rm tr} ({W}_1 \circ {W}_1 \circ \Delta ).  
\end{align}

%$U_{n,3} (\widehat{\Omega}) - U_{n,3} ({\Omega}) =3{\rm tr}(W_1\Delta^2)-3{\rm tr}(W_1^2\Delta)-{\rm tr}(\Delta^3)
%+3\tr(W_1\circ W_1\Delta)+3\tr(W_1\circ\Delta W_1)-3\tr(W_1\circ\Delta^2)+3\tr(\Delta\circ W_1^2)-3\tr(\Delta\circ W_1\Delta)-3\tr(\Delta\circ \Delta W_1)+3\tr(\Delta\circ\Delta^2)+6\tr(W_1\circ\Delta\circ\Delta)-6\tr(W_1\circ W_1\circ\Delta)-2\tr(\Delta\circ\Delta\circ\Delta)$

We shall study each of the three terms. The analysis will rely on the following key technical lemma, which is proved in Section~\ref{proof:lem:Delta}: 
%The following two lemmas proved in Sections~\ref{proof:lem:Delta} and \ref{supp:lem:hada} are useful for proving \eqref{eq:finalgoal}.
% We decompose $\Delta $ as
% \begin{align*}
% \Delta = \sum_{\substack{\ell_1, \ell_2, \ell_3 \in \{0,1\}^3 \\(\ell_1, \ell_2, \ell_3)\neq (0,0,0)}} U_{\ell_1}V_{\ell_2} U_{\ell_3}'
% \end{align*}
% where 
% \begin{align*}
% U_0 = \Omega H, \quad U_1 = W_1 H, \quad V_0 = (H'\Omega H)^{-1}, \quad V_1 =  (H'A H)^{-1} -  (H'\Omega H)^{-1} 
% \end{align*}

\begin{lemma}\label{lem:Delta}
%Assume $\theta_{\max}\leq c$ and $\beta_n=\lambda_{\min}(P)\gg\|\theta\|^{-1}\sqrt{\log(n)}$.
Suppose the conditions of Theorem~\ref{thm:main} hold. Let $\widehat{\Omega}=AH_0(H_0'AH_0)^{-1}H_0'A$, where $H_0$ is a non-stochastic matrix satisfying $\kappa(\Pi'\Theta H_0)\leq C$.  Write $W_1=W-\diag(\Omega)$ and $\Delta=\widehat{\Omega}-\Omega$. Then, the following statements are true. 
\begin{itemize}
\item[(a)] $\Vert \Delta \Vert =o_{\mathbb{P}}( \Vert \theta\Vert )$;
\item[(b)] $ \Vert W_1\Delta\Vert =O_{\mathbb{P}}(\Vert \theta\Vert^2)$;
\item[(c)] $\Vert W_1^2 \Delta \Vert  = o_{\mathbb{P}}(\Vert \theta\Vert^3)$;
\item[(d)] $\Vert {\rm diag} (W_1^2) \Delta \Vert = o_{\mathbb{P}}(\Vert\theta\Vert^3)$.
\end{itemize}
 \end{lemma}
 
\noindent
In addition, we will frequently quote some linear algebra arguments, which are summarized in the following lemma and proved in Section~\ref{supp:lem:hada}.

\begin{lemma}\label{lem:hada}
For any $B,D\in\mathbb R^{n,n}$, it holds that ${\rm tr} (B\circ D) = {\rm tr}\{{\rm diag}(B) D\}$. In addition,
\begin{align} \label{hada}
&{\rm rank} (B\circ D) \leq  {\rm rank} (B)\times {\rm rank} ( D),\cr
&\Vert B\circ D\Vert \leq \min\{  {\rm rank} (B), {\rm rank} ( D)\}  \Vert B\Vert \Vert D\Vert\, .
\end{align}
\end{lemma}

We now show \eqref{eq:finalgoal}. Consider $I_1$ in \eqref{main-proof-5}. Note that $\mathrm{rank}(\Delta)\leq\mathrm{rank}(\widehat{\Omega})+\mathrm{rank}(\Omega)\leq 2K$, and $\mathrm{rank}(\Delta^2)=\mathrm{rank}(\Delta)\leq 2K$. Moreover, by applying the first inequality in \eqref{hada}, $\mathrm{rank}(\Delta\circ\Delta^2)\leq 4K^2$ and $\mathrm{rank}(\Delta\circ\Delta\circ\Delta)\leq 8K^3$. Combining these observations with the inequality $|\tr(B)|\leq \mathrm{rank}(B)\|B\|$, we have
\[
|I_1|\leq C\bigl(\|\Delta\|^3 + \|\Delta \circ \Delta^2\| +\|\Delta\circ\Delta\circ\Delta\| \bigr)\leq C\|\Delta\|^3, 
\]
where the last inequality is from the second line of \eqref{hada}. By (a) of Lemma~\ref{lem:Delta}, i.e., $\|\Delta\|=o_{\mathbb{P}}(\Vert \theta\Vert)$, we immediately have 
%Our assumption \eqref{Main-Cond-2} says that $\beta^{-1}_n\ll \|\theta\|/\sqrt{\log(n)}$. It follows that
\beq \label{main-proof-6}
|I_1|
%=O_{\mathbb{P}}(\beta_n^{-3})
=o_{\mathbb{P}}(\|\theta\|^3). 
\eeq

Consider $I_2$ in \eqref{main-proof-5}. 
Similarly, we can use the inequalities such as $|\tr(B)|\leq\mathrm{rank}(B)\|B\|$, $\mathrm{rank}(BD)\leq\min\{\mathrm{rank}(B),\mathrm{rank}(D)\}$, and those in Lemma~\ref{lem:hada} to obtain: 
\begin{align*}
|I_2|&\leq C\Bigl(|{\rm tr} ({W}_1\Delta^2)|  + |{\rm tr} (\Delta \circ \Delta {W}_1) |
+ |{\rm tr}( {W}_1 \Delta \circ \Delta) | + | {\rm tr}( {W}_1^2 \Delta)|\Bigr)\cr
&\leq C \bigl(\Vert W_1 \Delta \Vert \Vert \Delta \Vert + \Vert W_1^2 \Delta \Vert\bigr).
\end{align*}
We plug in the statements (a)-(c) in Lemma~\ref{lem:Delta} to get 
%(note that $\beta_n^{-1}\ll \|\theta\|/\sqrt{\log(n)}$): 
\beq \label{main-proof-7}
|I_2|
%=O_{\mathbb{P}}\bigl(\beta_n^{-1}\|\theta\|^2\bigr) + o_{\mathbb{P}}(\beta_n^{-3})
=o_{\mathbb{P}}(\|\theta\|^3). 
\eeq

Consider $I_3$ in \eqref{main-proof-5}. For $B, D\in\mathbb{R}^{n,n}$ and $m\geq 1$, 
\[
\Bigl|\tr\Bigl(\underbrace{B\circ \cdots \circ B}_m\circ D\Bigr)\Bigr|=\Bigl|\tr\Bigl(\underbrace{\diag(B)\circ \cdots \circ \diag(B)}_m\circ D\Bigr)\Bigr|\leq \mathrm{rank}(D)\cdot \|\diag(B)\|^m \|D\|.
\]
Using this inequality, we can obtain:
\begin{align*}
|I_3| &\leq C\Bigl(\Vert {\rm diag} (W_1)\Vert \Vert \Delta\Vert^2 +  \Vert {\rm diag} (W_1^2) \Delta\Vert + \Vert {\rm diag} (W_1)\Vert \Vert W_1\Delta\Vert  + \Vert {\rm diag} (W_1)\Vert ^2 \Vert \Delta \Vert \Bigr)\cr
&\leq C\Bigl(\Vert {\rm diag} (\Omega)\Vert \Vert \Delta\Vert^2 +  \Vert {\rm diag} (W_1^2) \Delta\Vert + \Vert {\rm diag} (\Omega)\Vert \Vert W_1\Delta\Vert  + \Vert {\rm diag} (\Omega)\Vert ^2 \Vert \Delta \Vert \Bigr), 
\end{align*}
where the second line is because $\diag(W_1)=-\diag(\Omega)$. We then apply (a)-(d) and use the naive bound $\|\diag(\Omega)\|\leq C\theta_{\max}^2$. It follows that
\begin{align} \label{main-proof-8}
|I_3| &= O_{\mathbb{P}}\bigl(\theta_{\max}^2\|\theta\|^{2}\bigr)+o_{\mathbb{P}}(\|\theta\|^3) + O_{\mathbb{P}}\bigl(\theta_{\max}^2\|\theta\|^{2}\bigr) + O_{\mathbb{P}}\bigl(\theta_{\max}^2\|\theta\|^{-1}\bigr)= o_{\mathbb{P}}(\|\theta\|^3)\,. 
\end{align}
%where in the second line we have used the assumptions about $\beta_n$ and $\theta_{\max}$ in \eqref{Main-Cond-1}. 
Now, \eqref{eq:finalgoal} follows from \eqref{main-proof-6}-\eqref{main-proof-8}. The proof is complete. \qed

%\subsection{Proof of auxiliary lemmas in the proof of Theorem~\ref{thm:main}} \label{supp:main-auxiliary}

\subsection{Proof of auxiliary lemmas} \label{supp:main-auxiliary}

\subsubsection{Proof of Lemma \ref{lem:Delta}}\label{proof:lem:Delta}

The proof of Lemma \ref{lem:Delta} relies on the following lemma, which collects useful operator norm bounds, and is proved in Section~\ref{proof:lem:1.7}.
\begin{lemma}\label{lem:1.7}
Under the conditions of Lemma~\ref{lem:Delta}, for $\alpha=1, 2, 3$ and $\beta= 1,2$, let $H= H_0$,  we have
% \begin{align*}
% \Vert M_1' \Theta W_1 \Theta M_1\Vert \lesssim  \Vert \theta\Vert_1^{1+ \mathbf{1}_{(a>1)}} 
% \end{align*}
 \begin{align*}
&\Vert H' W_1 H\Vert =O_{\mathbb P}( \Vert \theta\Vert_1 \sqrt{\log(n)})\, ,  \qquad ~~~~~~~
\quad \Vert \Pi'\Theta W_1H\Vert =O_{\mathbb P}( \Vert \theta\Vert_3^{3/2} \Vert \theta\Vert_1^{1/2} \sqrt{\log(n)})\,  , \notag\\
  &\Vert \Pi'\Theta W_1^{2} H\Vert=O_{\mathbb P}( \Vert \theta\Vert^2 \Vert\theta\Vert_1),  \qquad ~~~~~~~~~~~~~\Vert \Pi'\Theta W_1^{3} H\Vert=o_{\mathbb P}( \|\theta\|^3\|\theta\|_1),\\
&\Vert\Pi'\Theta W_1^\beta \Theta H \Vert=O_{\mathbb P}(\Vert \theta\Vert^2 \Vert\theta\Vert_1^{\beta-1}),\qquad~~~~~~~~  \Vert H' W_1^{1+\alpha} H\Vert =O_{\mathbb P}( \Vert \theta\Vert_1^2 \Vert \theta\Vert^{2\max\{\alpha-2,0\}}),\notag\\
&\Vert \Pi'\Theta {\rm diag}(W_1^2) \Theta \Pi'\Vert=O_{\mathbb P}(\Vert \theta\Vert^2 \Vert \theta\Vert_1),\qquad~~~~\Vert \Pi'\Theta {\rm diag}(W_1^2) W_1H\Vert=o_{\mathbb P}( \|\theta\|^3\|\theta\|_1),
% \|\theta\|_3^{3/2}\|\theta\|_1^{3/2},  
\notag\\
& \Vert H'W_1 {\rm diag}(W_1^2) W_1H\Vert=O_{\mathbb P}( \Vert \theta\Vert^2 \Vert \theta\Vert_1^2).
 \end{align*}
%As a consequence, with probability $1-o(1)$, it holds that
%\begin{align*}
%&\Vert W_1\Theta \Pi\Vert=\|\Pi'\Theta W_1^2\Theta\Pi\|^{1/2} \lesssim \|\theta\|\|\theta\|_1^{1/2},\\
%&\Vert W_1^{\beta} H\Vert =\|H'W_1^{2\beta}H\|^{1/2}\lesssim \Vert \theta\Vert_1 \Vert \theta\Vert^{\max\{2\beta-3,0\}}.
%\end{align*}

\end{lemma}

Now, we prove Lemma~\ref{lem:Delta} using Lemma~\ref{lem:1.7}. To facilitate the presentation, with a little abuse of notation, we use $H$ in place of $H_0$. Nonetheless, it's important to note that throughout our analysis, $H$ is equivalent to $H_0$, maintaining its non-stochastic nature.  Recall from \eqref{GoForacle-1} the definition of $G_H=\Pi'\Theta H$. Observe the fact that $(H'\Omega H)^{-1}=G_H^{-1}P^{-1}(G_H')^{-1}$ and $(H'\Omega H)^{-1} H' \Omega = G_H^{-1}\Pi'\Theta$.
By Corollary~5.6.16 in \cite{HornJohnson}, it holds that
 \begin{align}\label{new1}
 &(H'AH)^{-1} =(H'\Omega H)^{-1}(H'\Omega H)(H'AH)^{-1}\notag\\
&=(H'\Omega H)^{-1}\Big[I_K+\sum_{j=1}^\infty\{I_K-(H'A H)(H'\Omega H)^{-1}\}^j\Big]=(H'\Omega H)^{-1}(I_K+\mathcal T),
 \end{align}
where
\begin{align}\label{new2}
\mathcal T&=\sum_{j=1}^\infty\{I_K-(H'A H)(H'\Omega H)^{-1}\}^j=\sum_{j=1}^\infty\{-H'W_1 H(H'\Omega H)^{-1}\}^j.
\end{align}
In addition, $(H'AH)^{-1} - (H'\Omega H)^{-1}= - (H'\Omega H)^{-1} (I_K + \mathcal{T} ) H'W_1 H (H'\Omega H)^{-1}.$
%It follows from Lemma~B.4 in \cite{mixedscore} that $|\lambda_k|\asymp\|\theta\|^2|\lambda_k(P)|$ for $1\leq k\leq K$, which together with Assumption~\ref{alambdaK} imply that
%\begin{align}\label{temp1}
%\beta_n\|\theta\|=|\lambda_{\min}(P)|\|\theta\|\asymp|\lambda_K|/\sqrt{\lambda_1}\to\infty\,,\qquad \text{as }n\to\infty\,.
%\end{align}

It can be claimed that $\|G_H^{-1}\|\lesssim \|\theta\|_1^{-1}$ based on $\kappa(\Pi'\Theta H)\leq C$. To elaborate, since $\kappa(\cdot)$ denotes the conditioning number of a matrix,  we have $s_{\min}(\Pi'\Theta H)/s_{max}(\Pi'\Theta H)\geq c$. It remains to show that $s_{max}(\Pi'\Theta H)\asymp \Vert \theta\Vert_1$. By direct computations and the fact that $H_0$ is weight matrix whose rows are weight vectors, 
\begin{align*}
\Vert \Pi'\Theta H\Vert_{\max} = \max_{k, \ell} |e_k'  \Pi\Theta H e_{\ell}| \leq {\bf 1}_n' \Theta {\bf 1}_n = \Vert \theta\Vert_1\, . 
\end{align*}
 It follows that 
\begin{align*}
s_{\max}(\Pi'\Theta H) \lesssim  \Vert \Pi'\Theta H \Vert_{F}\leq K \Vert \Pi'\Theta H \Vert_{\max} \lesssim \Vert \theta\Vert_1\, . 
\end{align*}
since $K$ is a fixed integer.
Further by definition of largest singular value, we can deduce that
\begin{align*}
s^2_{\max}(\Pi'\Theta H) & \geq K^{-1} {\bf 1}_K' \Pi' \Theta H_0 H_0' \Theta \Pi {\bf 1}_K = K^{-1} \Vert\theta' H_0 \Vert^2  \notag\\
&= K^{-2} \Vert\theta' H_0 \Vert^2  \Vert {\bf 1}_K\Vert^2 \geq K^{-2}( \theta' H_0 {\bf 1}_K)^2 = K^{-2} \Vert\theta\Vert_1^2\, . 
\end{align*}
Here we used the identity $\Pi{\bf 1}_K = {\bf 1}_n = H_0{\bf 1}_K$ and Cauchy-Schwarz inequality. Thus, we proved $s_{max}(\Pi'\Theta H)\asymp \Vert \theta\Vert_1$.  

Recalling the definition of $\beta_n$ in Condition~\ref{cond:afmSCORE}, in view of Lemma~\ref{lem:1.7} and $\|G_H^{-1}\|\lesssim \|\theta\|_1^{-1}$, we obtain
\begin{align*}
&\|H'W_1 H\|\|(H'\Omega H)^{-1}\|\leq\|H'W_1 H\|\|P^{-1}\|\|G_H^{-1}\|^{-2}=O_{\mathbb P} (\|\theta\|_1^{-1}\beta_n^{-1}\sqrt{\log(n)}).
\end{align*}
Therefore, due to the assumption that 
$\beta_n \Vert \theta \Vert \gg \sqrt{\log n} $ in (b) of Condition~\ref{cond:afmSCORE}, we deduce that
\begin{align*}
\|H'W_1 H\|\|(H'\Omega H)^{-1}\|=O_{\mathbb P}(\|\theta\|_1^{-1}\beta_n^{-1}\sqrt{\log(n)})=o_{\mathbb P}(1).
\end{align*}
The above equation implies that
\begin{align}\label{boundt}
\|\mathcal T\|\leq \sum_{j=1}^\infty\|H'W_1 H\|^j\|(H'\Omega H)^{-1}\|^j=o_{\mathbb P}(1).
\end{align}

Now, in order to bound $\|\Delta\|$, observing the definition of $\Delta$ and in view of \eqref{new1} and \eqref{new2}, we deduce from Lemma~\ref{lem:TvsH} that
% \begin{align}\label{eq:decomDelta}
% \Delta = & W_1HG^{-1} \Pi' \Theta + \Theta \Pi (G')^{-1} H' W_1 - \Theta \Pi G^{-1} (I_k + \mathcal{T} ) H'W_1 H(G')^{-1} \Pi' \Theta \notag\\
% & -  W_1H (H'\Omega H)^{-1} \Big(I_K + \mathcal{T} \Big) H'W_1 H  G^{-1} \Pi' \Theta \notag\\
% & - \Theta\Pi (G')^{-1}  H'W_1 H (I_K + \mathcal{T}')(H'\Omega H)^{-1} H' W_1 \notag\\
% & + W_1HG^{-1} P^{-1} (G')^{-1} H'W_1 - W_1 H(H'\Omega H)^{-1} (I_K+ \mathcal{T}) H'W_1H (H'\Omega H)^{-1} H'W_1 
% \end{align}
\begin{align}\label{Delta1230}
\Delta&=\Omega H(H'A H)^{-1}H'W_1+W_1 H(H'A H)^{-1}H'\Omega\notag\\
&\quad-\Omega H\{(H'\Omega H)^{-1}-(H'A H)^{-1}\}H'\Omega+W_1H(H'A H)^{-1}H'W_1\notag\\
&=\Omega H(H'\Omega H)^{-1}(I_K+\mathcal T)H'W_1+W_1 H(I_K+\mathcal T)(H'\Omega H)^{-1}H'\Omega\notag\\
&\quad-\Omega H(H'\Omega H)^{-1} (I_K + \mathcal{T} ) H'W_1 H (H'\Omega H)^{-1}H'\Omega+W_1H(H'\Omega H)^{-1}(I_K+\mathcal T)H'W_1\notag\\
%
%&=\Theta\Pi(G')^{-1}(I_K+\mathcal T)H'W_1+W_1 H(I_K+\mathcal T)G^{-1}\Pi'\Theta\notag\\
%&\quad-\Theta\Pi(G')^{-1} (I_K + \mathcal{T} ) H'W_1 H G^{-1}\Pi'\Theta+W_1H(G')^{-1}P^{-1}G^{-1}(I_K+\mathcal T)H'W_1\\
%
&=\Delta_1+\Delta_1'+\Delta_2+\Delta_3,
\end{align}
where
\begin{align}\label{Delta123}
&\Delta_1=\Theta\Pi (G_H')^{-1}(I_K+\mathcal T)H'W_1,\qquad\Delta_2=-\Theta\Pi (G_H')^{-1} (I_K + \mathcal{T} ) H'W_1 H G_H^{-1}\Pi'\Theta,\notag\\
&\Delta_3=W_1HG_H^{-1}P^{-1}(G_H')^{-1}(I_K+\mathcal T)H'W_1.
\end{align}
In view of \eqref{boundt} and Lemma~\ref{lem:1.7},  with probability $1-o(1)$,
\begin{align*}
& \Vert \Delta_1\Vert \leq \Vert H'W_1 \Theta \Pi \Vert \Vert G_H^{-1}\Vert \Vert I_K + \mathcal{T}\Vert \lesssim \Vert \theta\Vert_3^{3/2} \Vert\theta\Vert_1^{-1/2} \sqrt{\log n}\lesssim 1 ,\notag\\
& \Vert \Delta_2\Vert \leq \Vert \Pi'\Theta^2 \Pi\Vert \Vert H'W_1H\Vert \Vert G_H^{-1}\Vert^2 \Vert I_K + \mathcal{T}\Vert \lesssim \Vert \theta\Vert^2 \Vert\theta\Vert_1^{-1} \sqrt{\log n}\lesssim 1, \notag\\
& \Vert \Delta_3\Vert \leq \Vert H'W_1^2H  \Vert \Vert G_H^{-1}\Vert^2 \Vert I_K + \mathcal{T}\Vert \Vert P^{-1}\Vert \lesssim \beta_n^{-1} \ll \Vert \theta\Vert\,,
\end{align*}
where we used the facts $\Vert \theta\Vert_3^{3/2} \leq \Vert \theta\Vert_1^{1/2} \theta_{\max}$ and $\Vert \theta\Vert^2 \leq \Vert \theta\Vert_1 \theta_{\max}$, together with condition that $\beta_n \Vert \theta \Vert \gg \sqrt{\log n} $.
Using the above bounds, by the trivial inequality $\Vert \Delta \Vert\leq \Vert \Delta_1\Vert + \Vert \Delta_1'\Vert  + \Vert \Delta_2\Vert + \Vert \Delta_3\Vert$,  we conclude the proof of the first inequality (a) in Lemma \ref{lem:Delta}. 

For $\|W_1\Delta\|$ in (b), observe that  from \eqref{Delta1230} that $W_1 \Delta=W_1 \Delta_1+W_1 \Delta_1'+W_1 \Delta_2+W_1 \Delta_3$, where
\begin{align*}
&W_1 \Delta_1=W_1\Theta\Pi(G_H')^{-1}(I_K+\mathcal T)H'W_1,
\\
&W_1\Delta_2=-W_1\Theta\Pi(G_H')^{-1} (I_K + \mathcal{T} ) H'W_1 H G_H^{-1}\Pi'\Theta,
\\
&W_1\Delta_3=W_1^2H(G_H')^{-1}P^{-1}G_H^{-1}(I_K+\mathcal T)H'W_1.
\end{align*}
%\begin{align*}
%W_1 \Delta&=W_1 \Delta_1+W_1 \Delta_1'+W_1 \Delta_2+W_1 \Delta_3\\
%%
%&=W_1\Theta\Pi(G')^{-1}(I_K+\mathcal T)H'W_1+W_1^2H(I_K+\mathcal T)G^{-1}\Pi'\Theta\\
%&\quad-W_1\Theta\Pi(G')^{-1} (I_K + \mathcal{T} ) H'W_1 H G^{-1}\Pi'\Theta+W_1^2H(G')^{-1}P^{-1}G^{-1}(I_K+\mathcal T)H'W_1.
%\end{align*}
Similarly to the derivations above, by Lemma \ref{lem:1.7}, we obtain 
\begin{align*}
\Vert W_1 \Delta \Vert &\lesssim \Vert H'W_1^2 \Theta \Pi\Vert \Vert G_H^{-1}\Vert  + \Vert \Pi'\Theta W_1\Theta \Pi\Vert \Vert G_H^{-1}\Vert^2 \Vert H'W_1H\Vert + \Vert H'W_1^3H\Vert \Vert G_H^{-1}\Vert^2 \Vert P^{-1}\Vert \notag\\
& = O_{\mathbb P}(\Vert \theta\Vert^2) + O_{\mathbb P}(\Vert \theta\Vert^2\Vert \theta\Vert_1^{-1} \sqrt{\log n})  +O_{\mathbb P}( \beta_n^{-1})=O_{\mathbb P} (\Vert \theta\Vert^2).
\end{align*}
Here, the last step is due to $\Vert \theta\Vert_1>\Vert \theta\Vert^2\gg \log n$ and $\beta_n\Vert \theta\Vert \gg\sqrt{\log n} $ in view of Condition~\ref{cond:afmSCORE}. 

Last, for $\|W_1^2 \Delta\|$ and $\|{\rm diag} (W_1^2) \Delta\|$ in (c) and (d), by applying Lemma \ref{lem:1.7} and the condition that $\beta_n\gg \Vert\theta\Vert^{-1} \sqrt{\log n}$, we obtain
\begin{align*}
&\Vert W_1^2 \Delta \Vert \lesssim \Vert H'W_1^3 \Theta \Pi\Vert \Vert G_H^{-1}\Vert  + \Vert \Pi'\Theta W_1^2\Theta \Pi\Vert \Vert G_H^{-1}\Vert^2 \Vert H'W_1H\Vert + \Vert H'W_1^4H\Vert \Vert G_H^{-1}\Vert^2 \Vert P^{-1}\Vert \notag\\
&\qquad=o_{\mathbb P}(\|\theta\|^3)+O_{\mathbb P}\{\Vert \theta\Vert^2 (\beta_n^{-1} + \sqrt{\log n}\, ) \}=o_{\mathbb P} (\Vert\theta\Vert^3) ,\\
&\Vert {\rm diag} (W_1^2) \Delta \Vert 
\lesssim \Vert H'W_1{\rm diag} (W_1^2) \Theta \Pi\Vert \Vert G_H^{-1}\Vert  + \Vert \Pi'\Theta {\rm diag} (W_1^2)\Theta \Pi\Vert \Vert G_H^{-1}\Vert^2 \Vert H'W_1H\Vert \notag\\
&
+ \Vert H'W_1{\rm diag} (W_1^2)W_1 H\Vert \Vert G_H^{-1}\Vert^2 \Vert P^{-1}\Vert =o_{\mathbb P}(\|\theta\|^3)+O_{\mathbb P}\{ \Vert \theta\Vert^2 (\beta_n^{-1} + \sqrt{\log n})\} =o_{\mathbb P}( \Vert\theta\Vert^3) .
\end{align*}
The proof is therefore complete.

\subsubsection{Proof of Lemma~\ref{lem:hada}}\label{supp:lem:hada}

First, for $B,D\in\mathbb R^{n,n}$, it holds that ${\rm tr} (B\circ D)=\sum_{i=1}^n B_{ii}D_{ii} = {\rm tr}\{{\rm diag}(B) D\}$. Second, let $r_B={\rm rank}(B)$, $r_D={\rm rank}(D)$. Suppose $B=\sum_{i=1}^{r_B}\sigma_iu_iv_i'$ and $D=\sum_{j=1}^{r_D}\widetilde\sigma_j\widetilde u_j\widetilde v_j'$ are the SVD of $B$ and $D$, respectively. Then, 
\begin{align*}
B\circ D=\sum_{i=1}^{r_B}\sum_{j=1}^{r_D}\sigma_i\widetilde\sigma_j (u_i\circ\widetilde u_j)(v_i\circ\widetilde v_j)',
\end{align*}
so that ${\rm rank}(B\circ D)\leq r_Br_D={\rm rank}(B)\times{\rm rank}(D)$. Third, we have
\begin{align*}
\Vert B\circ D\Vert&=\sup_{\substack{\xi,\eta\in\mathbb R^n;\\ \|\xi\|=\|\eta\|=1}}|\xi'(B\circ D)\eta|=\sup_{\substack{\xi,\eta\in\mathbb R^n;\\ \|\xi\|=\|\eta\|=1}}\bigg|\sum_{i,j=1}^n\xi(i) B_{ij}D_{ij}\eta(j)\bigg|\\
&=\sup_{\substack{\xi,\eta\in\mathbb R^n;\\ \|\xi\|=\|\eta\|=1}}\big|\tr\big\{\diag(\xi)B\diag(\eta)D'\big\}\big|\\
&\leq \sup_{\substack{\xi,\eta\in\mathbb R^n;\\ \|\xi\|=\|\eta\|=1}}\Big[{\rm rank}\big\{\diag(\xi)B\diag(\eta)D'\big\}\times\|\diag(\xi)B\diag(\eta)D'\|\Big]\\
&\leq \min\{{\rm rank}(B),{\rm rank}(D)\}\times\|B\|\times\|D\|,
\end{align*}
where $\diag(\xi)$ denotes the diagonal matrix with diagonal entries being those of the $n$-vector $\xi$. The proof is therefore complete.

\subsubsection{Proof of Lemma \ref{lem:1.7}} \label{proof:lem:1.7}

We separate the proof of Lemma~\ref{lem:1.7} into Lemmas~\ref{lem:w11}, \ref{lem:w12}, \ref{lem:w13} and \ref{lem:w14} below, each consists of the results on the operator norm bound of the terms in Lemma~\ref{lem:1.7} with the power of $W_1$ equals to 1, 2, 3, and 4, respectively.

\begin{lemma}\label{lem:w11}
Under the conditions of Lemma~\ref{lem:Delta}, it holds that
\begin{align*}
&\Vert H' W_1 H\Vert=O_{\mathbb P}(\|\theta\|_1\sqrt{\log(n)}),~~~~~~~~
\Vert \Pi'\Theta W_1H\Vert=O_{\mathbb P}(\|\theta\|_3^{3/2}\|\theta\|_1^{1/2}\sqrt{\log(n)}),\\
&\Vert\Pi'\Theta W_1 \Theta H \Vert=O_{\mathbb P}(\|\theta\|^2).
\end{align*}
\end{lemma}

\begin{proof}[Proof of Lemma~\ref{lem:w11}]
%\subsubsection*{Proofs for $\Vert H' W_1 H\Vert,\Vert \Pi'\Theta W_1H\Vert,\Vert\Pi'\Theta W_1 \Theta H \Vert$} \label{sec:sub1W}

For $\|H'W_1H\|$, it suffices to bound $\max_{1\leq k,\ell\leq K}|\{H'W_1H\}_{k,\ell}|$. Observe that, for $1\leq k,\ell\leq K$,
\begin{align*}
\{H'W_1H\}_{k,\ell}=\sum_{r,s=1}^nH_{rk}H_{s\ell}(W_{rs}+\Omega_{rr}\delta_{rs}),
\end{align*}
where $\delta_{rs}$ denotes the Kronecker delta, more precisely, $\delta_{rs}=1$ if $r=s$ and $\delta_{rs}=0$ otherwise. We have
\begin{align*}
&|{\rm E}[\{H'W_1H\}_{k,\ell}]|=\bigg|\sum_{r=1}^nH_{rk}H_{r\ell}\Omega_{rr}\bigg|\leq C\|\theta\|^2,\qquad\sum_{r,s=1}^n{\rm var}[H_{rk}H_{r\ell}(W_{rs}+\Omega_{rr}\delta_{rs})]\leq C \|\theta\|_1^2.
\end{align*}
Moreover, we have $|H_{rk}H_{r\ell}(W_{rs}+\Omega_{rr}\delta_{rs})|\leq C \|H\|_{\max}\leq C$. Therefore, by Bernstein's inequality we obtain that, for $c>0$ large enough and for some constant $c'>0$,
\begin{align}\label{temp}
&{\rm P}\Big\{\max_{1\leq k,\ell\leq K}|\{H'W_1H\}_{k,\ell}-{\rm E}[\{H'W_1H\}_{k,\ell}]|> c\|\theta\|_1\sqrt{\log(n)}\Big\}\notag\\
&\leq 2K^2\exp \bigg\{\frac{-c'\|\theta\|_1^2\log(n)/2}{c'\|\theta\|_1^2+c'\|\theta\|_1\sqrt{\log(n)}\|H\|_{\max}^2/3} \bigg\}=o(1).
\end{align}
Therefore, we deduce that
\begin{align}\label{hw1h}
\hspace{-0.5em}\|H'W_1H\|\lesssim\hspace{-0.3em}\max_{1\leq k,\ell\leq K}|\{H'W_1H\}_{k,\ell}|\lesssim \|\theta\|^2+O_{\mathbb P}(\|\theta\|_1\sqrt{\log(n)})=O_{\mathbb P} (\|\theta\|_1\sqrt{\log(n)}).
\end{align}

%Second, for $\|W_1\Theta \Pi\|$, we have $\|W_1\Theta \Pi\|=\|\Pi'\Theta W_1^2\Theta \Pi\|^{1/2}$
%
%, for $1\leq k\leq K$, 
%\begin{align*}
%\{W_1\Theta \Pi\}_{k}=\sum_{r=1}^nW_{ir}\theta_r\Pi_{rk}-\Omega_{ii}\theta_i\Pi_{ik}
%\end{align*}
%Observe that $|W_{ir}\theta_r\Pi_{rk}|\lesssim\theta_r$ and $\sum_{r=1}^n{\rm var}[W_{ir}\theta_r\Pi_{rk}]\lesssim\theta_i\|\theta\|_3^3$. It follows from Bernstein's inequality that $\max_{1\leq k\leq K}|\sum_{r=1}^nW_{ir}\theta_r\Pi_{rk}|\lesssim\theta_i^{1/2}\|\theta\|_3^{3/2}\sqrt{\log(n)}$ with probability at least $1-o(n^{-3})$, which together with the fact that $\max_{1\leq k\leq K}|\Omega_{ii}\theta_i\Pi_{ik}|\lesssim\theta_i^3$ concludes the proof.

For $\Vert \Pi'\Theta W_1H\Vert$, we have, for $1\leq k,\ell\leq K$,
\begin{align*}
\{\Pi'\Theta W_1H\}_{k,\ell}=\sum_{r,s=1}^n\Pi_{rk}\theta_rW_{rs}H_{s\ell}-\sum_{r=1}^n\Pi_{rk}\theta_r\Omega_{rr}H_{r\ell}.
\end{align*}
Observe that $|\Pi_{rk}\theta_rW_{rs}H_{s\ell}|\lesssim\theta_r$ and $\sum_{r,s=1}^n{\rm E}(\Pi_{rk}\theta_rW_{rs}H_{s\ell})^2\lesssim \|\theta\|_3^3\|\theta\|_1$. By Bernstein's inequality and similar derivation in \eqref{temp}, we obtain, with probability at least $1-o(n^{-3})$, $\max_{1\leq k,\ell\leq K}|\sum_{r,s=1}^n\Pi_{rk}\theta_rW_{rs}H_{s\ell}|\lesssim\|\theta\|_3^{3/2}\|\theta\|_1^{1/2}\sqrt{\log(n)}$. In addition, we have $\max_{1\leq k,\ell\leq K}|\sum_{r=1}^n\Pi_{rk}\theta_r\Omega_{rr}H_{r\ell}|\lesssim \|\theta\|_3^3$. By Bernstein's inequality and the assumption that $\theta_{\max}\sqrt{\log(n)}\leq c$, we obtain that
\begin{align*}
\max_{1\leq k,\ell\leq K}|\{\Pi'\Theta W_1H\}_{k,\ell}|\lesssim\|\theta\|_3^3+O_{\mathbb P}( \|\theta\|_3^{3/2}\|\theta\|_1^{1/2}\sqrt{\log(n)})=O_{\mathbb P}( \|\theta\|_3^{3/2}\|\theta\|_1^{1/2}\sqrt{\log(n)}).
\end{align*}

For $\Vert\Pi'\Theta W_1 \Theta H \Vert$, we have, for $1\leq k,\ell\leq K$,
\begin{align*}
\{\Pi'\Theta W_1 \Theta H\}_{k,\ell}=\sum_{r,s=1}^n\Pi_{rk}\theta_r\theta_sW_{rs}H_{s\ell}-\sum_{r=1}^n\Pi_{rk}\theta_r^2\Omega_{rr}H_{r\ell}.
\end{align*}
By Bernstein's inequality and similar derivation as in \eqref{temp}, it follows similar derivation that, with probability at least $1-o(1)$,
\begin{align*}
\max_{1\leq k,\ell\leq K}|\{\Pi'\Theta W_1\Theta H\}_{k,\ell}|\lesssim\|\theta\|_4^4+ O_{\mathbb P}(\|\theta\|_3^3\sqrt{\log(n)})=O_{\mathbb P}(\|\theta\|^2).
\end{align*}
The proof is complete.
\end{proof}

%\subsubsection*{Proofs for $\Vert H' W_1^{2} H\Vert,\| \Pi'\Theta W_1^{2} H\|, \Vert\Pi'\Theta W_1^2 \Theta H \Vert,\Vert \Pi'\Theta {\rm diag}(W_1^2) \Theta \Pi\Vert$}

\begin{lemma}\label{lem:w12}
Under the conditions of Lemma~\ref{lem:Delta}, it holds that
\begin{align*}
&\Vert H' W_1^{2} H\Vert=O_{\mathbb P}(\|\theta\|_1^2),~~~~~~~~~~~~~~~~~~~~~\| \Pi'\Theta W_1^{2} H\|=O_{\mathbb P}(\|\theta\|_1^2),\\
& \Vert\Pi'\Theta W_1^2 \Theta H \Vert=O_{\mathbb P}(\|\theta\|_1\|\theta\|^2),~~~~~~~~~~~~
\Vert \Pi'\Theta {\rm diag}(W_1^2) \Theta \Pi\Vert=O_{\mathbb P}(\|\theta\|_1\|\theta\|^2).
\end{align*}

\end{lemma}

\begin{proof}[Proof of Lemma~\ref{lem:w12}]

For $\Vert H' W_1^{2} H\Vert$, we have
\begin{align*}
&\{H' W_1^{2} H\}_{k,\ell}=\sum_{r,s,t=1}^nH_{rk}H_{t\ell}\{W_{rs}-\Omega_{rr}\delta_{rs}\}\{W_{st}-\Omega_{ss}\delta_{st}\}\\
&=\sum_{r,s,t=1}^nH_{rk}H_{t\ell}W_{rs}W_{st}+\sum_{r=1}^nH_{rk}H_{r\ell}\Omega_{rr}^2-\sum_{r,s=1}^nH_{rk}H_{s\ell}W_{rs}\Omega_{ss}-\sum_{r,t=1}^nH_{rk}H_{t\ell}W_{rt}\Omega_{rr}.
\end{align*}
We have, for $1\leq k,\ell\leq K$, it holds that $|{\rm E}[\{H' W_1^{2} H\}_{k,\ell}]|\lesssim \|\theta\|_1^2+\|\theta\|_4^4\lesssim\|\theta\|_1^2$, and
\begin{align}\label{hw2h}
&{\rm var}[\{H' W_1^{2} H\}_{k,\ell}]\leq 4{\rm var}\bigg(\sum_{r,s=1}^nH_{rk}H_{r\ell}W_{rs}^2\bigg)+4{\rm var}\bigg(\sum_{1\leq r,s,t\leq n,t\neq r}H_{rk}H_{r\ell}W_{rs}W_{st}\bigg)\notag\\
&~~~~~~~~~~~~~~~~~~~~~~~~~~~~~~+4\sum_{r,s=1}^n{\rm var}(H_{rk}H_{s\ell}W_{rs}\Omega_{ss})+4\sum_{r,t=1}^n{\rm var}(H_{rk}H_{t\ell}W_{rt}\Omega_{rr})\notag\\
&\lesssim \sum_{r,s=1}^n{\rm var}(W_{rs}^2)+\sum_{\substack{r,s,t,r',s',t'\\t\neq r,t'\neq r'}}{\rm E}(W_{rs}W_{st}W_{r's'}W_{s't'})+\sum_{r,s=1}^n\Omega_{ss}^2{\rm var}(W_{rs})+\sum_{r,s=1}^n\Omega_{rr}^2{\rm var}(W_{rt})\notag\\
&\lesssim \|\theta\|_1^2+\|\theta\|_1^2\|\theta\|^2+\|\theta\|_5^5\|\theta\|_1\lesssim \|\theta\|_1^2\|\theta\|^2.
\end{align}
By Chebyshev's inequality, we obtain that
\begin{align*}
&\mathbb P\Big(\max_{1\leq k,\ell\leq K}|\{H' W_1^{2} H\}_{k,\ell}-\mathbb E[\{H' W_1^{2} H\}_{k,\ell}]|\geq \|\theta\|_1\|\theta\|\sqrt{\log(n)}\Big)\\
&\leq\sum_{k,\ell=1}^K \|\theta\|_1^{-2}\|\theta\|^{-2}\{\log(n)\}^{-1}{\rm var}(\{H' W_1^{2} H\}_{k\ell})=o(1).
\end{align*}
We therefore obtain from the above equation that
\begin{align*}
\max_{1\leq k,\ell\leq K}|\{H' W_1^{2} H\}_{k,\ell}|\lesssim \|\theta\|_1^2+O_{\mathbb P}(\|\theta\|_1\|\theta\|\sqrt{\log(n)})=O_{\mathbb P}(\|\theta\|_1^2),
\end{align*}
where we have used the assumption that $\theta_{\max}\sqrt{\log(n)}\lesssim 1$ in view of (a) of Condition~\ref{cond:afmSCORE}.

For $\Vert \Pi'\Theta W_1^{2} H\Vert$, we follow the similar derivation. Observe that
\begin{align*}
&\{\Pi'\Theta W_1^{2} H\}_{k,\ell}=\sum_{r,s,t=1}^n\Pi_{rk}H_{t\ell}\theta_r\{W_{rs} -\Omega_{rr}\delta_{rs}\}\{W_{st} - \Omega_{ss}\delta_{st}\}\\
&=\sum_{r,s,t=1}^n\Pi_{rk}H_{t\ell}\theta_rW_{rs}W_{st}
+\sum_{r,s=1}^n\Pi_{rk}H_{s\ell}\theta_rW_{rs}\Omega_{ss}
-\sum_{r,t=1}^n\Pi_{rk}H_{t\ell}\theta_r\Omega_{rr}W_{rt}
-\sum_{r=1}^n\Pi_{rk}H_{r\ell}\theta_r\Omega_{rr}^2.
\end{align*}
We have $|{\rm E}[\{\Pi'\Theta W_1^{2} H\}_{k,\ell}]|\lesssim \|\theta\|_1\|\theta\|^2+\|\theta\|_5^5\lesssim\|\theta\|_1\|\theta\|^2$. In addition, similar to the derivation in \eqref{hw2h}, we obtain
\begin{align*}
&{\rm var}[\{\Pi'\Theta W_1^{2} H\}_{k,\ell}]\lesssim \sum_{r,s=1}^n\theta_r^2{\rm var}(W_{rs}^2)+\sum_{\substack{r,s,t,r',s',t'\\t\neq r,t'\neq r'}}\theta_{r}^2\theta_{r'}^2{\rm E}(W_{rs}W_{st}W_{r's'}W_{s't'})\\
&~~~~~~~~~~~~~~~~~~~~~~~~~~~+\sum_{r,s=1}^n\theta_r^2\Omega_{ss}^2{\rm var}(W_{rs})+\sum_{r,t=1}^n\theta_r^2\Omega_{rr}^2{\rm var}(W_{rt})\\
&\lesssim\|\theta\|_1\|\theta\|_3^3+(\|\theta\|_1\|\theta\|^2\|\theta\|_5^5+\|\theta\|_3^6\|\theta\|^2)+\|\theta\|_3^3\|\theta\|_5^5+\|\theta\|_7^7\|\theta\|_1\lesssim \|\theta\|_1^2\|\theta\|_3^6.
\end{align*}
By Chebyshev's inequality, we obtain that
\begin{align*}
\max_{1\leq k,\ell\leq K}|\{H' \Theta W_1^{2} H\}_{k,\ell}|\lesssim \|\theta\|_1\|\theta\|^2+O_{\mathbb P}(\|\theta\|_1\|\theta\|_3^3\sqrt{\log(n)})=O_{\mathbb P}(\|\theta\|_1^2),
\end{align*}
where we used the assumption that $\theta_{\max}\sqrt{\log(n)}\leq c$ in view of (a) of Condition~\ref{cond:afmSCORE}.

For $\Vert\Pi'\Theta W_1^2 \Theta H \Vert$, we have, for $1\leq k,\ell\leq K$,
\begin{align*}
&\{\Pi'\Theta W_1^{2}\Theta H\}_{k,\ell}
%=\sum_{r,s,t=1}^n\Pi_{rk}H_{t\ell}\theta_r\theta_t\{W_{rs}+\Omega_{rr}\delta_{rs}\}\{W_{st}+\Omega_{ss}\delta_{st}\}\\
%
=\sum_{r,s,t=1}^n\Pi_{rk}H_{t\ell}\theta_r\theta_tW_{rs}W_{st}
-\sum_{r,s=1}^n\Pi_{rk}H_{s\ell}\theta_r\theta_sW_{rs}\Omega_{ss}\\
&\hspace{4cm}-\sum_{r,t=1}^n\Pi_{rk}H_{t\ell}\theta_r\theta_t\Omega_{rr}W_{rt}
+\sum_{r=1}^n\Pi_{rk}H_{r\ell}\theta_r^2\Omega_{rr}^2.
\end{align*}
Observe that $|{\rm E}[\{\Pi'\Theta W_1^{2}\Theta H\}_{k\ell}]|\lesssim \|\theta\|_1\|\theta\|_3^3+\|\theta\|_6^6\lesssim\|\theta\|_1\|\theta\|_3^3$. In addition, following the derivation similar to the one in \eqref{hw2h}, we obtain
\begin{align*}
&{\rm var}[\{\Pi'\Theta W_1^{2}\Theta H\}_{k,\ell}]\lesssim \sum_{r,s=1}^n\theta_r^2\theta_t^2{\rm var}(W_{rs}^2)+\sum_{\substack{r,s,t,r',s',t'\\t\neq r,t'\neq r'}}\theta_{r}^2\theta_{r'}^2\theta_t^2\theta_{t'}^2{\rm E}(W_{rs}W_{st}W_{r's'}W_{s't'})\\
&~~~~~~~~~~~~~~~~~~~~~~~~~~~+\sum_{r,s=1}^n\theta_r^2\theta_s^2\Omega_{ss}^2{\rm var}(W_{rs})+\sum_{r,t=1}^n\theta_r^2\theta_t^2\Omega_{rr}^2{\rm var}(W_{rt})\\
&\lesssim\|\theta\|_1\|\theta\|^2\|\theta\|_3^3+ \|\theta\|^2\|\theta\|_5^{10}+\|\theta\|_7^7\|\theta\|_3^3\lesssim\|\theta\|_1\|\theta\|^2\|\theta\|_3^3.
\end{align*}
By Chebyshev's inequality, we obtain that
\begin{align*}
\max_{1\leq k,\ell\leq K}|\{H' \Theta W_1^{2}\Theta H\}_{k,\ell}|\lesssim \|\theta\|_1\|\theta\|_3^3+O_{\mathbb P}(\|\theta\|_1^{1/2}\|\theta\|\|\theta\|_3^{3/2}\sqrt{\log(n)})=O_{\mathbb P}(\|\theta\|_1\|\theta\|^2),
\end{align*}
where we used the assumption that $\theta_{\max}\sqrt{\log(n)}\leq c$ in view of (a) of Condition~\ref{cond:afmSCORE}.

For $\Vert \Pi'\Theta {\rm diag}(W_1^2) \Theta \Pi\Vert$, we have
\begin{align*}
\{\Pi'\Theta {\rm diag}(W_1^2) \Theta \Pi\}_{k,\ell}&=\sum_{r,s=1}^n\Pi_{rk}\Pi_{r\ell}\theta_r^2\{W_{rs}-\Omega_{rr}\delta_{rs}\}^2
\\
&=\sum_{r,s=1}^n\Pi_{rk}\Pi_{r\ell}\theta_r^2W_{rs}^2+\sum_{r=1}^n\Pi_{rk}\Pi_{r\ell}\theta_r^2\Omega_{rr}^2.
\end{align*}
Observe that, for $1\leq k,\ell\leq K$, it holds that $|{\rm E}\{\Pi'\Theta {\rm diag}(W_1^2) \Theta \Pi\}_{k,\ell}|\lesssim \|\theta\|_3^3\|\theta\|_1+\|\theta\|_6^6\lesssim\|\theta\|_3^3\|\theta\|_1$ and ${\rm var}[\{\Pi'\Theta {\rm diag}(W_1^2) \Theta \Pi'\}_{k,\ell}]\lesssim\|\theta\|_1\|\theta\|_5^5$.
By Chebyshev's inequality, we obtain
\begin{align*}
\max_{1\leq k,\ell\leq K}|\{\Pi'\Theta {\rm diag}(W_1^2) \Theta \Pi\}_{k,\ell}|\lesssim \|\theta\|_1\|\theta\|_3^3+O_{\mathbb P}(\|\theta\|_1^{1/2}\|\theta\|_5^{5/2}\sqrt{\log(n)})=O_{\mathbb P}(\|\theta\|_1\|\theta\|^2).
\end{align*}
The proof is complete.
\end{proof}

%\subsubsection*{Proofs for $\Vert H' W_1^{3} H\Vert,\Vert \Pi'\Theta {\rm diag}(W_1^2) W_1H\Vert,\Vert \Pi'\Theta W_1^{3} H\Vert$}

\begin{lemma}\label{lem:w13}
Under the conditions of Lemma~\ref{lem:Delta}, it holds that
\begin{align*}
&\Vert H' W_1^{3} H\Vert=O_{\mathbb P}(\|\theta\|_1^2),~~~~~~~~~~~~
\Vert \Pi'\Theta {\rm diag}(W_1^2) W_1H\Vert=o_{\mathbb P}(\|\theta\|^3\|\theta\|_1),\\
&\Vert \Pi'\Theta W_1^{3} H\Vert=o_{\mathbb P}(\|\theta\|^3\|\theta\|_1).
\end{align*}

\end{lemma}

\begin{proof}[Proof of Lemma~\ref{lem:w13}]

For $\Vert H' W_1^{3} H\Vert$, we have
\begin{align*}
&\{H' W_1^{3} H\}_{k,\ell}=\sum_{r,s,t,u=1}^nH_{rk}H_{u\ell}\{W_{rs}-\Omega_{rr}\delta_{rs}\}\{W_{st}-\Omega_{ss}\delta_{st}\}\{W_{tu}-\Omega_{tt}\delta_{tu}\}.
%&=\sum_{r,s,t,u=1}^nH_{rk}H_{t\ell}W_{rs}W_{st}W_{tu}+\sum_{r,u=1}^nH_{rk}H_{r\ell}\Omega_{rr}^2W_{ru}-\sum_{r,s,u=1}^nH_{rk}H_{s\ell}W_{rs}\Omega_{ss}W_{su}-\sum_{r,t,u=1}^nH_{rk}H_{t\ell}W_{rt}\Omega_{rr}W_{tu}\\
%&-\sum_{r,s,t=1}^nH_{rk}H_{t\ell}W_{rs}W_{st}\Omega_{tt}-\sum_{r=1}^nH_{rk}H_{r\ell}\Omega_{rr}^3+\sum_{r,s=1}^nH_{rk}H_{s\ell}W_{rs}\Omega_{ss}^2+\sum_{r,t=1}^nH_{rk}H_{t\ell}W_{rt}\Omega_{rr}\Omega_{tt}
\end{align*}
We have $|{\rm E}[\{H' W_1^{3} H\}_{k,\ell}]|\lesssim \|\theta\|_1^2$ and ${\rm var}[\{H' W_1^{3} H\}_{k,\ell}]\lesssim I_1+I_2+I_3+I_4+I_5+I_6$, where
\begin{align*}
&I_1={\rm var}\bigg(\sum_{r,s,t,u=1}^nH_{rk}H_{u\ell}W_{rs}W_{st}W_{tu}\bigg),\qquad I_2={\rm var}\bigg(\sum_{r,u=1}^nH_{rk}H_{u\ell}\Omega_{rr}^2W_{ru}\bigg),\\
&I_3={\rm var}\bigg(\sum_{r,s,u=1}^nH_{rk}H_{u\ell}W_{rs}\Omega_{ss}W_{su}\bigg),\qquad ~I_4={\rm var}\bigg(\sum_{r,t,u=1}^nH_{rk}H_{u\ell}W_{rt}\Omega_{rr}W_{tu}\bigg),\\
%&I_5={\rm var}\bigg(\sum_{r,s,t=1}^nH_{rk}H_{t\ell}W_{rs}W_{st}\Omega_{tt}\bigg)\qquad 
&I_5={\rm var}\bigg(\sum_{r,s=1}^nH_{rk}H_{s\ell}W_{rs}\Omega_{ss}^2\bigg),\qquad ~~~~~~~~~I_6={\rm var}\bigg(\sum_{r,t=1}^nH_{rk}H_{t\ell}W_{rt}\Omega_{rr}\Omega_{tt}\bigg).
\end{align*}
Direct calculations yield
\begin{align*}
&I_1\lesssim\|\theta\|_1^2\|\theta\|^4+\|\theta\|_1^3\|\theta\|_3^3,\qquad I_2\lesssim \|\theta\|_9^9\|\theta\|_1,\qquad I_3\lesssim\|\theta\|_6^6\|\theta\|_1^2,\\
&I_4\lesssim \|\theta\|_1\|\theta\|^2\|\theta\|_5^5,\qquad ~~~~~~~~~~I_5\lesssim \|\theta\|_1\|\theta\|_9^9,\qquad I_6\lesssim \|\theta\|_5^{10}.
\end{align*}
Therefore, by Chebyshev's inequality, we obtain that, with probability at least $1-o(1)$,
\begin{align*}
\max_{1\leq k,\ell\leq K}|\{H' W_1^{3} H\}_{k,\ell}|\lesssim \|\theta\|_1^2+(\|\theta\|_1\|\theta\|^2+\|\theta\|_1^{3/2}\|\theta\|_3^{3/2})\sqrt{\log(n)}\lesssim\|\theta\|_1^2.
\end{align*}

For $\Vert \Pi'\Theta {\rm diag}(W_1^2) W_1H\Vert$, we have, for $1\leq k,\ell\leq K$,
\begin{align*}
&\{\Pi'\Theta {\rm diag}(W_1^2) W_1H\}_{k,\ell}
%=\sum_{r,s,t=1}^n\Pi_{rk}\theta_r\{W_{rs}-\Omega_{rr}\delta_{sr}\}^2(W_{st}-\Omega_{ss}\delta_{st})H_{t\ell}\\
%&=\sum_{r,s,t=1}^n\Pi_{rk}\theta_r H_{t\ell}W_{rs}^2W_{st}-\sum_{r,s=1}^n\Pi_{rk}\theta_r H_{t\ell}W_{rs}^2\Omega_{ss}+\sum_{r,t=1}^n\Pi_{rk}\theta_r H_{t\ell}\Omega_{rr}^2W_{st}-\sum_{r=1}^n\Pi_{rk}\theta_r H_{r\ell}\Omega_{rr}^3\\
=\sum_{r,s=1}^n\Pi_{rk}\theta_r H_{s\ell}W_{rs}^3+\sum_{1\leq r,s,t\leq n,t\neq s}\Pi_{rk}\theta_r H_{t\ell}W_{rs}^2W_{rt}\\
&\qquad~~~~~~~~~~~~~~~~~~~~~~-\sum_{r,s=1}^n\Pi_{rk}\theta_r H_{s\ell}W_{rs}^2\Omega_{ss}+\sum_{r,t=1}^n\Pi_{rk}\theta_r H_{t\ell}\Omega_{rr}^2W_{rt}-\sum_{r=1}^n\Pi_{rk}\theta_r H_{r\ell}\Omega_{rr}^3.
\end{align*}
We obtain that $|{\rm E}[\{\Pi'\Theta {\rm diag}(W_1^2) W_1H\}_{k\ell}]|\lesssim \|\theta\|^2\|\theta\|_1$ and
\begin{align*}
&{\rm var}[\{\Pi'\Theta {\rm diag}(W_1^2) W_1H\}_{k,\ell}]\\
&\lesssim {\rm var}\bigg(\sum_{r,s=1}^n\Pi_{rk}\theta_r H_{r\ell}W_{rs}^3\bigg)+{\rm var}\bigg(\sum_{1\leq r,s,t\leq n,t\neq s}\Pi_{rk}\theta_r H_{t\ell}W_{rs}^2W_{rt}\bigg)\\
&\quad +{\rm var}\bigg(\sum_{r,s=1}^n\Pi_{rk}\theta_r H_{t\ell}W_{rs}^2\Omega_{ss}\bigg)+ {\rm var}  \bigg(\sum_{r,t=1}^n\Pi_{rk}\theta_r H_{t\ell}\Omega_{rr}^2W_{rt}\bigg) \notag\\
&\lesssim \|\theta\|_3^3\|\theta\|_1+(\|\theta\|_1^2\|\theta\|_4^4 + \|\theta\|_3^6\Vert \theta\Vert_1^2+\|\theta\|_5^5\|\theta\|_1^3)+\|\theta\|_3^3\|\theta\|_5^5+ \Vert \theta\Vert_{11}^{11} \Vert \theta\Vert_1\lesssim\|\theta\|_5^5\|\theta\|_1^3.
\end{align*}
Therefore, by Chebyshev's inequality, we obtain that, with probability $1-o(1)$,
\begin{align*}
\max_{1\leq k,\ell\leq K}|\{\Pi'\Theta {\rm diag}(W_1^2) W_1H\}_{k,\ell}|\lesssim \Vert\theta\Vert^2 \Vert\theta\Vert_1+\|\theta\|_5^{5/2}\|\theta\|_1^{3/2}\sqrt{\log(n)}.
\end{align*}
Observing (a) in Condition~\ref{cond:afmSCORE}, we obtain $\|\theta\|_5^{5}\|\theta\|_1^{3}\log(n)/(\|\theta\|^6\|\theta\|_1^2)=o(1)$.
Hence, we conclude that, with probability $1-o(1)$, $\|\Pi'\Theta {\rm diag}(W_1^2) W_1H\|=o_{\mathbb P}(\|\theta\|^3\|\theta\|_1)$.

For $\Vert \Pi'\Theta W_1^{3} H\Vert$, we have, for $1\leq k,\ell\leq K$,
\begin{align*}
\{\Pi'\Theta W_1^{3} H\}_{k,\ell}=\sum_{r,s,t,u=1}^nH_{rk}H_{u\ell}\theta_r\{W_{rs}-\Omega_{rr}\delta_{rs}\}\{W_{st}-\Omega_{ss}\delta_{st}\}\{W_{tu}-\Omega_{tt}\delta_{tu}\}.
\end{align*}
Following similar calculations as $\Vert \Pi' W_1^{3} H\Vert$ and $\Vert \Pi'\Theta {\rm diag}(W_1^2) W_1H\Vert$, we obtain that 
\begin{align*}
|{\rm E}[\{\Pi'\Theta W_1^{3}H\}_{k,\ell}]|\lesssim\|\theta\|^2\|\theta\|_1,\qquad {\rm var}[\{\Pi'\Theta W_1^{3} H\}_{k,\ell}]\lesssim\|\theta\|_5^5\|\theta\|_1^3.
\end{align*}
We therefore deduce from the above equation that, with probability $1-o(1)$,
\begin{align*}
\Vert \Pi'\Theta W_1^{3} H\Vert\lesssim \Vert \theta\Vert^2 \Vert\theta\Vert_1+\|\theta\|_5^{5/2}\|\theta\|_1^{3/2}\sqrt{\log(n)}.
\end{align*}
which implies that $\Vert \Pi'\Theta W_1^{3} H\Vert=o_{\mathbb P}(\|\theta\|^3\|\theta\|_1)$ in view of (a) of Condition~\ref{cond:afmSCORE}.

\end{proof}

%\subsubsection*{Proofs for $\Vert H'W_1^4 H\Vert,\Vert  H'W_1 {\rm diag} (W_1^2) W_1H\Vert$}

\begin{lemma}\label{lem:w14}
Under the conditions of Lemma~\ref{lem:Delta}, it holds that
\begin{align}\label{eq:fourW}
\Vert  H'W_1 {\rm diag} (W_1^2) W_1H\Vert   =O_{\mathbb P}( \Vert \theta\Vert_1^2 \Vert\theta  \Vert^2),\qquad\Vert H'W_1^4 H\Vert =O_{\mathbb P}( \Vert \theta\Vert_1^2 \Vert\theta  \Vert^2). 
\end{align}

\end{lemma}

\begin{proof}[Proof of Lemma~\ref{lem:w14}]

We start with $\Vert  H'W_1 {\rm diag} (W_1^2) W_1H\Vert$. Recall the definition $W_1 = A- \Omega = W - {\rm diag}(\Omega)$. By the fact that ${\rm diag} (W)=0$, we have the following decomposition:
\begin{align*}
H'W_1 {\rm diag} (W_1^2) W_1H  &= H' \big(W - {\rm diag}(\Omega)\big){\rm diag} \big\{\big(W - {\rm diag}(\Omega)\big)^2\big\}\big(W - {\rm diag}(\Omega)\big) H\notag\\
& = I_0 + I_{1} +I_{1} ' + I_{21} + I_{22} + I_3 +I_3'+ I_4 ,
\end{align*} 
where
\begin{align}\label{imany}
& I_0: = H'\big({\rm diag} (\Omega)\big)^4H, \quad~~~~~~~~~~~~~~~~ I_{1}:= - H' W\big({\rm diag} (\Omega)\big)^3 H ,\notag\\
& I_{21}:=H' {\rm diag} (\Omega) {\rm diag}(W^2) {\rm diag} (\Omega) H, \quad~~~~   I_{22}:= H' W \big({\rm diag} (\Omega)\big)^2 WH,\notag\\ 
& I_{3}:= - H' W {\rm diag} (W^2) {\rm diag} (\Omega)H, \quad~~  I_4:= H'W{\rm diag} (W^2) W H\, .
\end{align} 
%\begin{align*}
%& I_{11}:= - H' W\big({\rm diag} (\Omega)\big)^3 H , \quad I_{12}:= - H'{\rm diag} (\Omega) {\rm diag} (W) \big({\rm diag} (\Omega)\big)^2H,\notag\\
%& I_{21}:=H' {\rm diag} (\Omega) W^2 {\rm diag} (\Omega) H, \quad   I_{22}:= H' W \big({\rm diag} (\Omega)\big)^2 WH,\notag\\ 
%&I_{23}:= H'W{\rm diag} (W)\big({\rm diag} (\Omega)\big)^2H, \quad  I_{24}:= H'{\rm diag} (\Omega){\rm diag} (W){\rm diag} (\Omega)WH,\notag\\
%& I_{31}:= - H' W {\rm diag} (W^2) {\rm diag} (\Omega)H, \quad I_{32}:= -H'W{\rm diag} (W){\rm diag} (\Omega) WH, \notag\\
%& I_4:= H'W{\rm diag} (W^2) W H, \quad I_0: = H'\big({\rm diag} (\Omega)\big)^4H\, .
%\end{align*}
We proceed to show that each term in \eqref{imany} are of order $O_{\mathbb P}(\Vert \theta\Vert_1^2 \Vert\theta\Vert^2)$.

In the sequel, $k$ and $\ell$ take values $1,\ldots,K$.
For $I_0$ in \eqref{imany}, direct computation yields $\Vert I_0\Vert \lesssim \max_{k, \ell} |(I_0)_{k,\ell}|\lesssim \sum_{i=1}^n \Omega_{ii}^4\lesssim \Vert \theta\Vert_8^8\ll \Vert \theta\Vert^2 $. For $I_{1}$ in \eqref{imany}, it  only involves one $W$. Following the similar derivation as in the proof of Lemma~\ref{lem:w11}, observing the fact that $({\rm diag}(\Omega))_{ii} \asymp \theta_i^2 $ and the condition $\theta_{\max}\sqrt{\log n} \leq C$ for some constant $C>0$, in view of (a) of Condition~\ref{cond:afmSCORE}, we deduce that $\Vert I_{1}\Vert  =o_{\mathbb P}( \Vert \theta\Vert_1^{{1/2}} \Vert \theta\Vert).$
For the terms $I_{21}$ and $I_{22}$ in \eqref{imany}, following similar derivations in the proofs of $\Vert  \Pi'\Theta {\rm diag}(W_1^2) \Theta H\Vert $ and $\Vert  H'W_1^2 H\Vert $ in Lemma~\ref{lem:w12}, we obtain $\Vert I_{21}\Vert =o_{\mathbb P}( \Vert \theta\Vert_1\Vert\theta\Vert^2)$ and $\Vert I_{22}\Vert =o_{\mathbb P}( \Vert \theta\Vert_1\Vert\theta\Vert^2)$. Next for $I_3$ in \eqref{imany}, following similar computation in the proof for $\|H' W{\rm diag} (W^2)\Theta \Pi\|$ in Lemma~\ref{lem:w13}, we obtain that $\Vert I_{3}\Vert =o_{\mathbb P}( \Vert \theta\Vert_1 \Vert\theta\Vert^3)$. Next, we investigate the term $I_4$ in \eqref{imany}. Observe that
\begin{align*}
\big|(I_{4})_{k, \ell}\big| &= \bigg|\sum_{\substack{1\leq i,j,s,t\leq n\\ j\neq i, j\neq s, j\neq t}} H_{ik} W_{ij} W_{js}^2 W_{jt}H_{t\ell}\bigg|\leq  |\mathcal{I}_1| + |\mathcal{I}_2| + |\mathcal{I}_3| + | \mathcal{I}_4| + |\mathcal{I}_5|,
\end{align*}
where
\begin{align}\label{def:I1_5}
& \mathcal{I}_1= \sum_{i\neq j} H_{ik}H_{i\ell} W_{ij}^4, \qquad~~~~~~~\mathcal{I}_2=  \sum_{i\neq j\neq t}H_{ik}H_{t\ell}W_{ij}^3W_{jt},\qquad~~~~~~~ \mathcal{I}_3= \sum_{i\neq j\neq t}H_{ik}H_{t\ell}W_{ij}W_{jt}^3,\notag\\
&\mathcal{I}_4= \sum_{i\neq j\neq s} H_{ik} H_{i\ell}W_{ij}^2 W_{js}^2 ,\qquad \mathcal{I}_5= \sum_{i\neq j\neq s\neq t}^n H_{ik} W_{ij} W_{js}^2 W_{jt}H_{t\ell}.
\end{align}

We proceed to investigate the five terms $\mathcal{I}_1, \cdots, \mathcal{I}_5$ separately. For the first term $\mathcal{I}_1$ in \eqref{def:I1_5}, it follows that
\begin{align*}
\big| \mathbb{E} \mathcal{I}_1\big|  \lesssim \sum_{i,j} \theta_i \theta_j =\Vert \theta\Vert_1^2, \qquad {\rm var} (\mathcal{I}_1) \lesssim \sum_{i<j}{\rm var} (W_{ij}^4) \lesssim \sum_{i,j} \theta_i\theta_j = \Vert\theta\Vert_1^2
\end{align*} 
Then by Chebyshev's inequality, $|\mathcal{I}_1|\lesssim \Vert \theta\Vert_1^2 +O_{\mathbb P} (\Vert \theta\Vert_1 \sqrt{\log n} )=O_{\mathbb P} (\Vert \theta\Vert_1^2 )$. The second and third terms  $\mathcal I_2$ and $\mathcal I_3$ in \eqref{def:I1_5} are similar with the only difference in the coefficient which will not affect the estimation of the order of mean and variance. For simplicity, we only show the details for bounding $\mathcal I_2$. Note that $\mathbb{E} \mathcal{I}_2 = 0 $ and its variance be derived as 
\begin{align*}
{\rm var} (\mathcal{I}_2) = \mathbb{E} \Big( \sum_{i\neq j\neq t}H_{ik}H_{t\ell}W_{ij}^3W_{jt}\Big)^2  = \sum_{\substack{i\neq j\neq t\\ i'\neq j'\neq t'}} H_{ik}H_{t\ell} H_{i'k}H_{t'\ell} \, \mathbb{E}(W_{ij}^3 W_{i'j'}^3 W_{jt}  W_{j't'} ).
\end{align*}
The summand on the RHS above is nonzero if it takes one of the following forms: $\mathbb{E} (W_{ij}^4W_{i'j'}^4)$, $\mathbb{E} (W_{ij}^3W_{i'j}^3W_{jt}^2)$, $\mathbb{E} (W_{ij}^3W_{jt}^2W_{ti'}^3)$, or $\mathbb{E} (W_{ij}^6W_{jt}^2)$, without consideration for the coefficient in front. This results in 
\begin{align*}
{\rm var} (\mathcal{I}_2) &\lesssim \sum_{  i\neq j\neq t } \Big(\mathbb{E} (W_{ij}^4W_{jt}^4) + \mathbb{E} (W_{ij}^6W_{jt}^2 ) \Big) +  \sum_{i,j, i', t \, (dist)} \Big(\mathbb{E}( W_{ij}^3W_{i'j}^3W_{jt}^2) + \mathbb{E} (W_{ij}^3W_{jt}^2W_{ti'}^3)\Big) \notag\\
&\lesssim \Vert \theta\Vert^4 \Vert \theta\Vert_1^2 + \Vert \theta\Vert_3^3\Vert\theta\Vert_1^3
\ll \Vert \theta\Vert_1^4.
\end{align*}
Therefore, we deduce that $|\mathcal{I}_2 | =O_{\mathbb P}( \Vert\theta\Vert_1^2 \sqrt{\log n})$. For the fourth term $\mathcal{I}_4$ in \eqref{def:I1_5}, we compute $\mathbb{E} (\mathcal{I}_4)  \lesssim \sum_{i\neq j \neq s} \theta_i \theta_j^2 \theta_s \leq  \Vert \theta\Vert_1^2 \Vert \theta\Vert^2$ and
\begin{align*}
 {\rm var} (\mathcal{I}_4)  &= \sum_{\substack{i\neq j \neq s\\ i'\neq j' \neq s'}}H_{ik}H_{i'k} H_{i\ell}H_{i'\ell} \Big(\mathbb{E} \big(W_{ij}^2W_{js}^2 W_{i'j'}^2W_{j's'}^2\big) - \mathbb{E}W_{ij}^2W_{js}^2 \cdot \mathbb{E}W_{i'j'}^2W_{j's'}^2\Big)\notag\\
 & \lesssim \sum_{i,j, s,s' \,  (dist)} \Big( \mathbb{E}(W_{ij}^4W_{js}^2 W_{js'}^2) + \mathbb{E}(W_{ij}^2W_{js}^4 W_{ss'}^2)  \Big) + \sum_{i\neq j\neq s}\mathbb{E} (W_{ij}^4W_{js}^4)   \notag\\
 & \lesssim  \sum_{i,j, s,s' \,  (dist)} (\theta_i\theta_j^3\theta_s\theta_{s'}+ \theta_i \theta_j^2\theta_s^2\theta_{s'}) + \sum_{i\neq j\neq s} \theta_i\theta_j^2\theta_s\lesssim \Vert \theta\Vert_1^3 \Vert \theta\Vert^2 \theta_{\max} ,
\end{align*}
where in the second step we used the fact that each  summand in the first line is nonzero if and only if the two paths, $i\to j\to s$ and $i'\to j'\to s'$, have an overlap in at least one edge. As a result, we obtain $|\mathcal{I}_4|=O_{\mathbb P}( \Vert \theta\Vert_1^2 \Vert \theta\Vert^2)$. Last, for $\mathcal{I}_5$,  its mean is zero and its variance can be computed as below.
\begin{align*}
{\rm var}(\mathcal{I}_5) &=  \sum_{\substack{i, j, s, t \, (dist)\\i', j', s', t' \, (dist)}}^n H_{ik} H_{t\ell}H_{i'k'} H_{t'\ell'}\mathbb{E} (W_{ij} W_{js}^2 W_{jt}W_{i'j'} W_{j's'}^2 W_{j't'})\notag\\
&\lesssim \sum_{i,j,s,t, s'\, (dist)}\mathbb{E}( W^2_{ij} W_{js}^2 W^2_{jt}W_{js'}^2 )+ \sum_{i,j,s,t \, (dist)}\mathbb{E} (W^3_{ij} W_{js}^3 W^2_{jt} )+\mathbb{E} (W^2_{ij} W_{js}^4 W^2_{jt})\notag\\
& \lesssim\sum_{i,j,s,t,s' \, (dist)} \theta_i\theta_s\theta_t\theta_{s'} \theta_j^4  + \sum_{i,j,s,t \, (dist)} \theta_i\theta_s\theta_t \theta_j^3 \lesssim \Vert\theta\Vert_1^4 \Vert \theta\Vert^2 \theta_{\max}^2.
\end{align*}
Hence, $|\mathcal{I}_5|=O_{\mathbb P}( \Vert \theta\Vert_1^2 \Vert \theta\Vert \theta_{\max} \sqrt{\log n} )=o_{\mathbb P}(\Vert \theta\Vert_1^2 \Vert \theta\Vert^2)$, where we used the condition that $\theta_{\max}\sqrt{\log n} = o(1) $.  Combining the estimates of $\mathcal{I}_1,\ldots,\mathcal I_5$ in \eqref{def:I1_5} yields that $\Vert I_4\Vert =O_{\mathbb P} (\Vert\theta\Vert_1^2\Vert \theta\Vert^2)$. The proof for $\Vert  H'W_1 {\rm diag} (W_1^2) W_1H\Vert$ is therefore complete by combining the bounds the terms in \eqref{imany}.

Next, we prove the bound for  $\Vert H'W_1^4 H\Vert $ in (\ref{eq:fourW}). Observe that
\begin{align*}
H'W_1^4H = H'\big(W- {\rm diag} (\Omega)\big)^4H  =\sum_{\alpha_1, \cdots, \alpha_4\in \{0,1\}} H'\prod_{u=1}^4 W^{\alpha_u} \big(- {\rm diag} (\Omega)\big)^{1-\alpha_u} H\,.
\end{align*}
If $\sum_{u=1}^4 \alpha_u  = 0$, the summand on the RHS becomes $H' [{\rm diag} (\Omega)]^4 H$ and it can be trivially  bounded by $\sum_{i=1}^n \theta_i^8 = o(\Vert \theta\Vert_1)$. 
For those summands such that $\sum_{u=1}^4 \alpha_u =1, 2, 3$, their bounds are similar to that of $\Vert H'W^aH\Vert$, for $a=1, 2,3$, since ${\rm diag} (\Omega)$ is deterministic and its diagonal entries are all bounded by $\theta_{\max}^2$. We therefore obtain that
\begin{align*}
&\Big\Vert \sum_{\alpha_1+ \cdots + \alpha_4=1} H'\prod_{u=1}^4 W^{\alpha_u} \big(- {\rm diag} (\Omega)\big)^{1-\alpha_u} H \Big\Vert =o_{\mathbb P} (\Vert \theta\Vert_1)=o_{\mathbb P}( \Vert \theta\Vert_1^2 \Vert \theta\Vert^2)\,, \notag\\
&\Big\Vert \sum_{\alpha_1+ \cdots + \alpha_4=2} H'\prod_{u=1}^4 W^{\alpha_u} \big(- {\rm diag} (\Omega)\big)^{1-\alpha_u} H \Big\Vert = o_{\mathbb P}(\Vert \theta\Vert_1^2)=o_{\mathbb P}( \Vert \theta\Vert_1^2  \Vert \theta\Vert^2)\,, \notag\\
&\Big\Vert \sum_{\alpha_1+ \cdots + \alpha_4=3} H'\prod_{u=1}^4 W^{\alpha_u} \big(- {\rm diag} (\Omega)\big)^{1-\alpha_u} H \Big\Vert=o_{\mathbb P}( \Vert \theta\Vert_1^2) =o_{\mathbb P}( \Vert \theta\Vert_1^2  \Vert \theta\Vert^2)\,,
\end{align*}
where we used the condition that $\theta_{\max} = o(1/\sqrt{\log n}\, ) $.
Next, we bound $\Vert H' W^4 H \Vert $, which corresponds to the case $\alpha_1+ \cdots + \alpha_4=4$. It suffices to claim that for a fixed $(k, \ell)$, $\big| (H' W^4H)_{k, \ell}\big|=O_{\mathbb P} (\Vert \theta\Vert_1^2 \Vert \theta\Vert^2)$.
  To see this, we first write 
  \begin{align*}
  (H' W^4H)_{k, \ell} &= \sum_{i,j,s,t,q=1}^nH_{ik} H_{q\ell} W_{ij}W_{js}W_{st}W_{tq}.
  %\notag\\
   %& \leq \Big| \sum_{i,j,s,t,q=1}^nH_{ik} H_{q\ell} W_{ij}W_{js}W_{st}W_{tq}\Big|
   \end{align*}
Direct calculations yield
  \begin{align*}
  \mathbb{E} \{(H' W^4H)_{k, \ell}\}  & = \sum_{i,j,s,t,q=1}^nH_{ik} H_{q\ell} \mathbb{E} W_{ij}W_{js}W_{st}W_{tq} = \sum_{i,j,s=1}^nH_{ik} H_{i\ell} \big( \mathbb{E} W_{ij}^4 + W_{ij}^2W_{it}^2+  W_{ij}^2W_{js}^2\big)
 \notag\\
  & \lesssim \sum_{i,j}\theta_i\theta_j + \sum_{i,j,s} \theta_i\theta_j^2 \theta_s \lesssim \Vert \theta\Vert_1^2 \Vert \theta\Vert^2
  \end{align*}
In addition, observing the definitions of $\mathcal I_1,\ldots,\mathcal I_5$ in \eqref{def:I1_5}, we have
\begin{align*}
{\rm var} \{(H' W^4H)_{k, \ell}\} &\lesssim \sum_{a=1}^5 {\rm var} (\mathcal{I}_a)+ {\rm var} \Big(\sum_{i,j,s,t\, (dist) }^nH_{ik} H_{i\ell}  W_{ij}^2W_{it}^2 \Big)
\\
&\quad+{\rm var} \Big(\sum_{i,j,s,t\, (dist) }^nH_{ik} H_{s\ell}  W_{ij}^2W_{it}W_{ts} \Big)\\
% &= {\rm var}\Big(\sum_{i\neq j } H_{ik}H_{i\ell} W_{ij}^4\Big)+ {\rm var} \Big(\sum_{i\neq j\neq s}^nH_{ik} H_{i\ell}  W_{ij}^2 W_{js}^2 \Big) + {\rm var} \Big(\sum_{i\neq j\neq s }^nH_{ik} H_{s\ell}  W_{ij}^3W_{js} \Big)\notag\\
%& \quad  + {\rm var} \Big(\sum_{i\neq j\neq s }^nH_{ik} H_{s\ell}  W_{ij}W_{js}^3 \Big) +   {\rm var} \Big(\sum_{i,j,s,t\, (dist) }^nH_{ik} H_{s\ell}W_{ij}  W_{js}^2W_{jt}\Big) \notag\\
& \quad  +   {\rm var} \Big(\sum_{i,j,s,t, q\, (dist) }^nH_{ik} H_{q\ell} \mathbb{E} W_{ij}W_{js}W_{st}W_{tq}\Big) .
%\notag\\
%&\leq \sum_{a=1}^5 {\rm var} (\mathcal{I}_a) + {\rm var} \Big(\sum_{i,j,s,t\, (dist) }^nH_{ik} H_{s\ell}  W_{ij}^2W_{it}W_{ts} \Big) +   {\rm var} \Big(\sum_{i,j,s,t, q\, (dist) }^nH_{ik} H_{q\ell} \mathbb{E} W_{ij}W_{js}W_{st}W_{tq}\Big)
\end{align*}
It is worthy noting that the calculations of $ {\rm var} \Big(\sum_{i,j,s,t\, (dist) }^nH_{ik} H_{i\ell}  W_{ij}^2W_{it}^2 \Big)$ are essentially similar to ${\rm var}(\mathcal{I}_4)$. Therefore, we can obtain that $ {\rm var} \Big(\sum_{i,j,s,t\, (dist) }^nH_{ik} H_{i\ell}  W_{ij}^2W_{it}^2 \Big) = o(\Vert \theta\Vert_1^3 \Vert \theta\Vert^2)$.
We also observe that
\begin{align*}{\rm var} 
&\Big(\sum_{i,j,s,t\, (dist) }^nH_{ik} H_{s\ell}  W_{ij}^2W_{it}W_{ts} \Big) \lesssim \sum_{\substack {i,j,s,t\, (dist) \\i',j',s',t'\, (dist) }} \mathbb{E}(W_{ij}^2W_{it}W_{ts} W_{i'j'}^2W_{i't'}W_{t's'}) \notag\\
& \lesssim \sum_{i,j,t,s \, (dist)}\theta_i^2\theta_j\theta_t^2 \theta_s + \sum_{i,j,s,t,s', (dist) } (\theta_i\theta_j^2 \theta_s^2 \theta_t^2 \theta_{s'} + \theta_i^3\theta_j^2 \theta_s\theta_t\theta_{s'})\lesssim \Vert \theta\Vert_1^4\Vert \theta\Vert^{2}\theta_{\max}^2 .
\end{align*}
Here we get the second step since each summand is nonzero if and only if the two paths, $ j\to i \to t\to s$ and $ j'\to i' \to t'\to s'$, merge into one of the following cases:
(a) a directed tree: $j\to i\to t\to s$; (b) a directed tree: $j\to i\to t\to s\to s'$; (c) a tree structure rooted at $i$ with three leaves $i\to t\to s$, $i\to j$ and $i\to j'$. Each edge mentioned  must be a multiple edge. And each multiple edge, denoted by $a\to b$,  contributes to $\theta_a\theta_b$. Thereby, the three cases correspond to the three summations on the second line of the above inequality.
In addition,
\begin{align*}
&{\rm var} \bigg\{\sum_{i,j,s,t, q\, (dist) }^n\hspace{-0.5em} H_{ik} H_{q\ell} \mathbb{E} (W_{ij}W_{js}W_{st}W_{tq})\bigg\} \lesssim \hspace{-0.5em} \sum_{\substack{i,j,s,t, q\, (dist)\\ i',j',s',t', q'\,  (dist)} }\hspace{-1em}\mathbb{E} (W_{ij}W_{js}W_{st}W_{tq}W_{i'j'}W_{j's'}W_{s't'}W_{t'q'})\notag\\
& \lesssim \sum_{i,j,s,t,q\, (dist)}\mathbb{E} (W_{ij}^2W_{js}^2W_{st}^2W_{tq}^2)\lesssim \sum_{i,j,s,t,q} \theta_i \theta_j^2 \theta_s^2 \theta_t^2 \theta_q \ll \Vert \theta\Vert_1^2 \Vert \theta\Vert^6\, .
\end{align*}
In conclusion, we obtain that ${\rm var}\{ (H' W^4H)_{k, \ell}\}\ll \Vert \theta\Vert_1^4 \Vert \theta\Vert^2 $. The proof is therefore complete by applying Chebyshev's inequality.

\end{proof}

\section{Analysis of MSCORE and the $H$ from net-rounding} \label{sec:MSCORE-supp}
In this section, we aim to present the proof of Theorem~\ref{thm:MSCORE}. We first improve the results in \cite{MSCORE} and provide sharp entry-wise large deviation bounds for  eigenvectors in Section~\ref{subsec:entrywise-eigen}. The complete proof of 
Theorem~\ref{thm:MSCORE} is given in Section~\ref{susec:pf-MSCORE}.  
%Throughout this section, for any two sequences $a_n$ and $b_n$, $ a_n \asymp b_n$ means  $a_n \leq C b_n $ and $b_n\leq C' a_n$ for some constants $C, C' >0$;  $a_n \lesssim b_n$ means $a_n\leq C b_n$ for some constant $C>0$. We use $C, C', c, c'$ to represent some generic positive constants independent of dimension $n$, which may vary from line to line.  

\subsection{Entry-wise eigenvector analysis of $A$} \label{subsec:entrywise-eigen}
The node-wise error bounds for Mixed-SCORE heavily rely on entry-wise large deviation bounds for the eigenvectors of $A$. We collect the results in the following lemma. 
\begin{lemma}\label{lem:entrywise-eigen}
Suppose the assumptions in Theorem~\ref{thm:MSCORE} hold. Let $(\lambda_k, \xi_k)$ and  $(\hat{\lambda}_k, \hat{\xi}_k)$ be the $k$-th eigen-pairs of $\Omega$ and $A$, respectively. Write $\Xi_1:= (\xi_2, \cdots, \xi_K)$ and similarly for $\widehat{\Xi} _1$. With probability $1- o(n^{-3})$, there exists $\omega \in \{1, -1\}$ and an orthogonal matrix $O_1\in \mathbb{R}^{(K-1)\times (K-1)}$ such that simultaneously for all $1\leq i \leq n$, 
\begin{align*}
\big| \hat{\xi}_1(i) - \omega\xi_1(i) \big| & \lesssim \frac{\sqrt{n\overline{\theta} \theta_{\max}} \, \theta_i}{ \Vert \theta\Vert^{3}} + \frac{\Vert\theta\Vert_3^{3/2}\sqrt{\theta_i \log n}  }{\Vert \theta\Vert^3 } + \frac{\theta_{\max} \log n}{\Vert \theta\Vert^3}  + \frac{\theta_{\max} \sqrt{n\overline{\theta} \theta_i \log n} }{\Vert \theta\Vert^4} , \notag\\
\Vert \widehat{\Xi}_1(i) - \Xi_1(i) O_1\Vert &\lesssim \frac{\sqrt{n\overline{\theta} \theta_{\max}} \, \theta_i}{ \beta_n  \Vert \theta\Vert^{3}} + \frac{\Vert\theta\Vert_3^{3/2}\sqrt{\theta_i \log n}  }{ \beta_n \Vert \theta\Vert^3 } + \frac{\theta_{\max} \log n}{\beta_n\Vert \theta\Vert^3}  + \frac{\theta_{\max} \sqrt{n\overline{\theta} \theta_i \log n} }{\beta_n^2\Vert \theta\Vert^4} \,.  
\end{align*}
where $\widehat{\Xi}_1(i)$ represents $i$th row of $\widehat{\Xi}$, and similarly for $\Xi(i)$. 
\end{lemma}

\begin{remark} Recall $\delta_n = {\theta_{\min} \Vert \theta\Vert }/({ \theta_{\max}^{3/2} \Vert \theta\Vert_1^{1/2})}$. Under the condition $\delta_n \cdot \beta_n \Vert \theta\Vert /\sqrt{\log (n)} \to \infty$, it is easy to obtain that $\Vert \theta\Vert^2 \geq C \sqrt{n\overline{\theta}\theta_{\max} \log (n)}$. Based on this, some elementary computations give that the upper bounds presented in Lemma~\ref{lem:entrywise-eigen} are dominated by $(\delta_n \beta_n \Vert \theta\Vert)^{-1} \sqrt{\log (n)} \cdot \theta_i/\Vert \theta\Vert$. We remark that compared to Lemma D.2 of \cite{MSCORE}, the results in Lemma~\ref{lem:entrywise-eigen} are much stronger in the sense that it gives precisely  the sharp entry-wise bounds instead of `` $2$-to-infinity'' bounds. We also remark  that the proof of Lemma~\ref{lem:entrywise-eigen} is similar to part of the entry-wise eigenvector analysis in  \cite{ke2022optimal}, more specifically,  the entry-wise large deviation bounds between $(\hat{\xi}, \widehat{\Xi}_1)$ and  $( \widetilde{\xi}_1, \widetilde{\Xi}_1)$. Here  $\hat{\xi}_1, \widehat{\Xi}_1$ are generated from Laplacian matrix, while $ \widetilde{\xi}_1, \widetilde{\Xi}_1$ are  obtained from a partial leave-one-out proxy. More formal definitions can be found in \cite{ke2022optimal}. 
\end{remark}

\begin{proof}[Proof of Lemma~\ref{lem:entrywise-eigen}]
 We recall the spectral properties of the population matrix $\Omega$ in \cite{MSCORE}. For the eigenvalues,  by Lemma C.2 of \cite{MSCORE}
\begin{align*}
\lambda_1 \asymp \Vert \theta\Vert^2, \quad |\lambda_2(\Omega)| \leq  (1-c) \lambda_1(\Omega), \quad |\lambda_K(\Omega) |\asymp \beta_n\Vert \theta\Vert^2
\end{align*}
for some fixed constant $c\in (0,1)$. For the associated eigenvectors,   Lemma C.3 of \cite{MSCORE} gives that  for all $1\leq i \leq n$, 
\begin{align}\label{2023061601}
\xi_1(i) \asymp \theta_i \Vert \theta\Vert^{-1}, \qquad  \Vert \Xi_1 (i) \Vert \lesssim  \theta_i \Vert \theta\Vert^{-1}\, . 
\end{align}

In the sequel, we focus on the details of the proof of $\hat{\xi}_1(i) - \xi_1(i)$. And the analysis of the row-wise error rate of $\widehat{\Xi}_1$ can be derived by the same technique, along with a similar technical trick used in \cite{ke2022optimal} addressing the issue arising from the non-commutativity of matrix product. More precisely, unlike $\xi_1'\hat{\xi}_1 \lambda_1 - \lambda_1\xi_1'\hat{\xi}_1=0$, $\Lambda_1 \Xi'_1 \widehat{\Xi}_1 - \Xi'_1 \widehat{\Xi}_1 \Lambda_1\neq 0$ where we write $\Lambda_1 = {\rm diag} (\lambda_2, \cdots, \lambda_K)$. To resolve this, we use the decomposition $\Lambda_1 \Xi'_1 = \Xi_1' \Omega = \Xi_1'A  + \Xi_1' (\Omega - A)  $. The first term further contributes to $\Xi_1' A\widehat{\Xi}_1 = \Xi_1' \widehat{\Xi}_1 \widehat{\Lambda}_1 $ which gives the desired form and the perturbation from $\Lambda_1$ to $\widehat{\Lambda}_1$ is easy to cope with. To control the other term stemming from the second term in the aforementioned decomposition, we simply use the operator norms of both $\Omega - A $ and $\Lambda_1$.  Another minor complexity in the analysis of $\widehat{\Xi}_1$ comes from the presence of the orthogonal matrix $O_1$. We occasionally need to write $O_1$ as $\Xi_1'\widehat{\Xi}_1  +( O_1 - \Xi_1'\widehat{\Xi}_1)$ to overcome the non-commutativity of matrix product and facilitate the analysis. We refer the readers to Section~C.4-C.7 in \cite{ke2022optimal} for more details and omit the  proof for  $\Vert \widehat{\Xi}_1(i) - \Xi_1(i) O_1\Vert $ for the sake of simplicity.

 Without loss of generality, we assume $\omega=1$ since we can always choose $\xi_1$ and $\hat{\xi}_1$ with their first components non-negative.   By the fact that $\hat{\xi}_1 = \hat{\lambda}_1^{-1} A\hat{\xi}_1$, we can derive 
\begin{align*}
\hat{\xi}_1(i) - \xi_1(i)  & = \hat{\lambda}_1^{-1}  e_i' \Omega \hat{\xi}_1 - \xi_1(i) +   \hat{\lambda}_1^{-1}  e_i' (A-\Omega) \hat{\xi}_1 \notag\\
&= \big(\lambda_1 \hat{\lambda}_1^{-1} \xi_1' \hat{\xi}_1 - 1\big)\xi_1(i) + \sum_{k=2}^K \lambda_k \hat{\lambda}_1^{-1}  \xi_k' \hat{\xi}_1 \xi_k(i)  + \widehat{\lambda}_1^{-1} e_i' \big(W- {\rm diag} (\Omega)\big) \hat{\xi}_1
\end{align*}
Employing Weyl's inequality and sine-theta theorem, it is not hard to bound 
\begin{align*}
\big| \big(\lambda_1 \hat{\lambda}_1^{-1} \xi_1' \hat{\xi}_1 - 1\big)\xi_1(i)\big| & \lesssim \big(|\lambda_1 \hat{\lambda}_1^{-1} - 1| + |\xi_1' \hat{\xi}_1 - 1|\big)|\xi_1(i)| \notag\\
& \lesssim  \frac{\Vert W - {\rm diag}(\Omega)\Vert }{\Vert \theta\Vert^2}  \cdot |\xi_1(i)|\lesssim \sqrt{n\overline{\theta} \theta_{\max}} \, \theta_i \Vert \theta\Vert^{-3}
\end{align*}
and 
\begin{align*}
\Big| \sum_{k=1}^K \lambda_k \hat{\lambda}_1^{-1}  \xi_k' \hat{\xi}_1 \xi_k(i)  \Big| \leq \Vert \Xi_1'\hat{\xi}_1\Vert\cdot \Vert \Xi_1(i)\Vert \lesssim \frac{\Vert W - {\rm diag}(\Omega)\Vert }{\Vert \theta\Vert^2}  \cdot  \Vert \Xi_1(i)\Vert \lesssim  \sqrt{n\overline{\theta} \theta_{\max}} \, \theta_i \Vert \theta\Vert^{-3}
\end{align*}
with probability at least $1- o(n^{-4})$,  where we used the result that $\Vert W - {\rm diag}(\Omega)\Vert \leq C \sqrt{n\overline{\theta}\theta_{\max} }$ in the proof of Lemma D.1 in \cite{MSCORE} and  applied (\ref{2023061601}) in the last steps. We further bound 
\begin{align*}
|\hat{\lambda}_1^{-1}e_i'{\rm diag} (\Omega) \hat{\xi}_1| \lesssim \theta_i^2\Vert \theta\Vert^{-2} |\hat{\xi}_1(i)| \leq \theta_i^2\Vert \theta\Vert^{-2} |{\xi}_1(i)| + \theta_i^2\Vert \theta\Vert^{-2} |\hat{\xi}_1(i) - \xi(i)|\, . 
\end{align*}
Thereby, we arrive at 
\begin{align*}
\big| \hat{\xi}_1(i) - \xi_1(i) \big| \lesssim \sqrt{n\overline{\theta} \theta_{\max}} \, \theta_i \Vert \theta\Vert^{-3} + \Vert \theta\Vert^{-2}\big|e_i' W\hat{\xi}_1 \big|+ \theta_i^2\Vert \theta\Vert^{-2} |\hat{\xi}_1(i) - \xi(i)|\,. 
\end{align*}
Rearranging both sides gives
\begin{align}\label{2023061500}
\big| \hat{\xi}_1(i) - \xi_1(i) \big| \lesssim \sqrt{n\overline{\theta} \theta_{\max}} \, \theta_i \Vert \theta\Vert^{-3} + \Vert \theta\Vert^{-2}\big|e_i' W\hat{\xi}_1 \big|\, . 
\end{align}

Next,  we bound $\big|e_i' W\hat{\xi}_1 \big|$. We decompose it as 
\begin{align}\label{2023061501}
\big|e_i' W\hat{\xi}_1 \big| \leq | e_i' W {\xi}_1|  + | e_i' W (\widetilde{\xi}_1 - \xi_1) | +  |e_i' W (\hat{\xi}_1 - \widetilde{\xi}_1) |
\end{align}
where $(\widetilde{\xi}, \widetilde{\Xi}_1)$ is the eigenvectors associated with the top $K$ eigenvalues of $A^{(i)} = \Omega - {\rm diag} (\Omega) + W^{(i)}$. Here $W^{(i)}$ is the random matrix obtained by zeroing out the $i$-th row and column of $W$. In addition, we choose $\widetilde{\xi}_1$ whose first component is non-negative. By this choice, we have both ${\rm sgn}(\hat{\xi}_1'\widetilde{\xi}_1) =1$ and ${\rm sgn}(\widetilde{\xi}_1'\xi_1) =1$.  Note that $\widetilde{\xi}_1 - \xi_1$, as well as ${\xi}_1$,  is independent of  $W_i$, $i$-th row or column of $W$.  We can then restrict to the randomness of $W_i$ and apply Bernstein inequality to the first two terms on the RHS of (\ref{2023061501}). It yields that with probability at least $1- o(n^{-4})$, 
\begin{align}
& | e_i' W {\xi}_1|  = \big| \sum_{j\neq i} W_{ij} \xi_1(j)\big| \lesssim \frac{\Vert\theta\Vert_3^{3/2}\sqrt{\theta_i \log n}  }{\Vert \theta\Vert } + \frac{\theta_{\max} \log n}{\Vert \theta\Vert} \label{2023061502}\\
&  | e_i' W (\widetilde{\xi}_1 - \xi_1) |  = \Big|  \sum_{j\neq i} W_{ij} \big(\widetilde{\xi}_1(i)- \xi_1(i) \big)\Big| \lesssim \sqrt{\theta_i \theta_{\max} \Vert \widetilde{\xi}_1 - \xi_1\Vert^2 \log n} + \Vert \widetilde{\xi}_1   - \xi_1 \Vert_{\infty} \log n \notag
\end{align}
Further applying sine-theta theorem to $\Vert \widetilde{\xi}_1 - \xi_1\Vert^2$ and decomposing $\Vert \widetilde{\xi}_1   - \xi_1 \Vert_{\infty}$, we obtain
\begin{align}\label{2023061503}
| e_i' W (\widetilde{\xi}_1 - \xi_1) |\lesssim \frac{\theta_{\max} \sqrt{n\overline{\theta} \theta_i \log n} }{\Vert \theta\Vert^2}   +  \log n\cdot \big( \Vert \hat{\xi}_1 - \widetilde{\xi}_1\Vert + \Vert \hat{\xi}_1 - \xi_1\Vert_\infty \big)\,.
\end{align}
For the last term on the RHS of (\ref{2023061501}), we crudely bound 
\begin{align}\label{2023061504}
|e_i' W (\hat{\xi}_1 - \widetilde{\xi}_1) | \lesssim \sqrt{n\overline{\theta} \theta_i } \, \Vert \hat{\xi}_1 - \widetilde{\xi}_1 \Vert \, .
\end{align}
%Combining (\ref{2023061501}) - (\ref{2023061504}) into (\ref{2023061500}), it  gives  that 
We bound $\Vert \hat{\xi}_1 - \widetilde{\xi}_1 \Vert $ by sine-theta theorem as below. 
\begin{align*}
\Vert \hat{\xi}_1 - \widetilde{\xi}_1 \Vert  \lesssim \frac{\Vert (A^{(i)}  - A)\hat{\xi} \Vert }{\Vert \theta\Vert^2} &= \frac{\Vert e_i W_i'\hat{\xi}_1 + W_ie_i' \hat{\xi}_1\Vert }{\Vert \theta\Vert^2} \notag\\
&\lesssim \Vert \theta\Vert^{-2} |e_i'W\hat{\xi}_1| + \sqrt{n\overline{\theta}\theta_i} \, \Vert \theta\Vert^{-2} \big(|\xi_1(i)| + |\hat{\xi}_1(i) - {\xi}_1(i)|\big) \notag\\
&\lesssim \Vert \theta\Vert^{-2} |e_i'W\hat{\xi}_1| + \sqrt{n\overline{\theta}\theta_i} \, \theta_i\Vert \theta\Vert^{-3}  +  \sqrt{n\overline{\theta}\theta_i} \, \Vert \theta\Vert^{-2} |\hat{\xi}_1(i) - {\xi}_1(i)|
\end{align*}
where we denote the column vector $W_i$ the $i$-th row of $W$. Combing this with (\ref{2023061502}) - (\ref{2023061504}) into (\ref{2023061501}), we get 
\begin{align*}
\big|e_i' W\hat{\xi}_1 \big| \lesssim & \frac{\Vert\theta\Vert_3^{3/2}\sqrt{\theta_i \log n}  }{\Vert \theta\Vert } + \frac{\theta_{\max} \log n}{\Vert \theta\Vert}  + \frac{\theta_{\max} \sqrt{n\overline{\theta} \theta_i \log n} }{\Vert \theta\Vert^2}   +  \log n\cdot  \Vert \hat{\xi}_1 - \xi_1\Vert_\infty  \notag\\
& + \frac{n\overline{\theta}\theta_i^2 }{\Vert\theta\Vert^3} +\big(\sqrt{n\overline{\theta}\theta_i} + \log n\big)  \sqrt{n\overline{\theta}\theta_i} \, \Vert \theta\Vert^{-2} |\hat{\xi}_1(i) - {\xi}_1(i)| + \frac{\sqrt{n\overline{\theta}\theta_i} + \log n}{\Vert \theta\Vert^2 } \big|e_i' W\hat{\xi}_1 \big|\, .
\end{align*}
Rearranging the terms and using the bounds $\Vert \theta\Vert^2 \geq C \sqrt{n\overline{\theta}\theta_{\max} \log (n)}$ mentioned in the remark of Lemma~\ref{lem:entrywise-eigen} and $\beta_n \Vert \theta\Vert /\sqrt{\log (n)} \to \infty$ in (b) of Condition~\ref{cond:afmSCORE}, we arrive at 
\begin{align*}
\big|e_i' W\hat{\xi}_1 \big| \lesssim & \frac{\Vert\theta\Vert_3^{3/2}\sqrt{\theta_i \log n}  }{\Vert \theta\Vert } + \frac{\theta_{\max} \log n}{\Vert \theta\Vert}  + \frac{\theta_{\max} \sqrt{n\overline{\theta} \theta_i \log n} }{\Vert \theta\Vert^2}   +  \log n\cdot  \Vert \hat{\xi}_1 - \xi_1\Vert_\infty  \notag\\
& + \frac{n\overline{\theta}\theta_i^2 }{\Vert\theta\Vert^3} +\big(\sqrt{n\overline{\theta}\theta_i} + \log n\big)  \sqrt{n\overline{\theta}\theta_i} \, \Vert \theta\Vert^{-2} |\hat{\xi}_1(i) - {\xi}_1(i)|\, . 
\end{align*}
We now substitute the above inequality to (\ref{2023061500}) and rearrange the terms. It follows that 
\begin{align*}
\big| \hat{\xi}_1(i) - \xi_1(i) \big| &\lesssim \frac{\sqrt{n\overline{\theta} \theta_{\max}} \, \theta_i}{ \Vert \theta\Vert^{3}} + \frac{\Vert\theta\Vert_3^{3/2}\sqrt{\theta_i \log n}  }{\Vert \theta\Vert^3 } + \frac{\theta_{\max} \log n}{\Vert \theta\Vert^3} 
\\
& \qquad+ \frac{\theta_{\max} \sqrt{n\overline{\theta} \theta_i \log n} }{\Vert \theta\Vert^4}   +  \frac{\log n}{\Vert \theta\Vert^2}\cdot  \Vert \hat{\xi}_1 - \xi_1\Vert_\infty\, . 
\end{align*}
We then take maximum over $i$ for both sides and get the bound 
\begin{align*}
\Vert \hat{\xi}_1 - \xi_1\Vert_\infty \lesssim \frac{\sqrt{n\overline{\theta} \theta_{\max}} \, \theta_{\max}}{ \Vert \theta\Vert^{3}} + \frac{\Vert\theta\Vert_3^{3/2}\sqrt{\theta_{\max} \log n}  }{\Vert \theta\Vert^3 } + \frac{\theta_{\max} \log n}{\Vert \theta\Vert^3}  + \frac{\theta_{\max} \sqrt{n\overline{\theta}\theta_{\max} \log n} }{\Vert \theta\Vert^4}  \, .  
\end{align*}
Combining the above two inequalities, by $\Vert \theta\Vert^2 \geq C \sqrt{n\overline{\theta}\theta_{\max} \log (n)}$ and $\beta_n \Vert \theta\Vert /\sqrt{\log (n)} \to \infty$ again,  we finally conclude that with probability at least $1- o(n^{-4})$, 
\begin{align*}
\big| \hat{\xi}_1(i) - \xi_1(i) \big| \lesssim \frac{\sqrt{n\overline{\theta} \theta_{\max}} \, \theta_i}{ \Vert \theta\Vert^{3}} + \frac{\Vert\theta\Vert_3^{3/2}\sqrt{\theta_i \log n}  }{\Vert \theta\Vert^3 } + \frac{\theta_{\max} \log n}{\Vert \theta\Vert^3}  + \frac{\theta_{\max} \sqrt{n\overline{\theta} \theta_i \log n} }{\Vert \theta\Vert^4}  \, . 
\end{align*}
Then we conclude the proof by combining all $1\leq i \leq n$ together.

\end{proof}

\subsection{Proof Theorem~\ref{thm:MSCORE}} \label{susec:pf-MSCORE}
We finish the proof of Theorem~\ref{thm:MSCORE} in this section. First, we claim the node-wise error bounds for MSCORE. Write $\hat{R}: = {\rm diag}(\hat{\xi_1})^{-1} \widehat{\Xi}_1$ and $R =  {\rm diag}({\xi_1})^{-1} {\Xi}_1$. Employing Lemma~\ref{lem:entrywise-eigen} and  \eqref{2023061601},
 it is not hard to obtain that 
\begin{align*}
\Vert \widehat{R}(i) - R(i)O\Vert &\lesssim   \frac{\sqrt{n\overline{\theta} \theta_{\max}} \, }{ \beta_n  \Vert \theta\Vert^{2}} + \frac{\Vert\theta\Vert_3^{3/2}\sqrt{ \log n}  }{ \beta_n \Vert \theta\Vert^2\sqrt{\theta_i}} + \frac{\theta_{\max} \log n}{\beta_n\Vert \theta\Vert^2\theta_i }  + \frac{\theta_{\max} \sqrt{n\overline{\theta}  \log n} }{\beta_n^2\Vert \theta\Vert^3\sqrt{\theta_i}}  \notag\\
&\lesssim (\delta_n \beta_n \Vert \theta\Vert)^{-1} \sqrt{\log (n)} \, . 
\end{align*}
Based on the above error rate,  further with effective vertex hunting algorithm, we can use the same arguments in the proof of the rate of convergence of Mixed-SCORE in \cite{MSCORE} and finally get the node-wise error rate of $\widehat{\Pi}^{\rm MS}$, which is $(\delta_n \beta_n \Vert \theta\Vert)^{-1} \sqrt{\log (n)}$. Since the remaining proof is simply a copy of the proof in Section F.1 of \cite{MSCORE}, we skip the details and conclude the proof. Moreover, under (e) of Condition~\ref{cond:MSCORE}, under the high probability event that the rate of Mixed-SCORE holds,   we can claim that $\hat{\pi}_{0i} = \pi_{0i} $ simultaneously for all $1\leq i \leq n$. To see this, let us suppose that $\pi_{0i^*} = e_{k}'$ and $\hat{\pi}_{0i^*} = e_{k^*}'$ for some $1\leq i^* \leq n$  and $1\leq k\neq k^* \leq n$. Under Condition~\ref{cond:MSCORE}, it is seen that $\pi_{i^*}(k) - \max_{i\neq k} \pi_{i^*}(i) \geq (\delta_n \beta_n \Vert \theta\Vert)^{-1} \log (n)$. If follows from the rate of Mixed-SCORE that 
\begin{align*}
\hat{\pi}_{i^*}(k) &\geq  \pi_{i^*}(k) -C(\delta_n \beta_n \Vert \theta\Vert)^{-1} \sqrt{\log (n)}\notag\\
& \geq  \max_{i\neq k} \pi_{i^*}(i) + C (\delta_n \beta_n \Vert \theta\Vert)^{-1} \log (n) \notag\\
& \geq \max_{i\neq k} \hat{\pi}_{i^*}(i)  - \Vert \hat{\pi}_{i^*}- \pi_{i^*}\Vert + C (\delta_n \beta_n \Vert \theta\Vert)^{-1} \log (n) \notag\\
& >\max_{i\neq k} \hat{\pi}_{i^*}(i),
\end{align*} 
then $\hat{\pi}_{0i^*} = e_{k}'$ which is a contradiction. Consequently, $\widehat{\Pi}_0  = \Pi_0$ with probability $1- o(n^{-3})$. This finishes the proof of Theorem~\ref{thm:MSCORE}. 
%
%Once we get the sharp entry-wise error rate of SCORE vector $\hat{R}$, the remaining proof for node-wise error is the same as that in Mixed-SCORE or Mixed-SCORE Laplacian. Therefore, below we only prove 
%\begin{align*}
%\Vert \widehat{R}(i)  -  \widehat{R}(i) O\Vert \lesssim \delta_n 
%\end{align*}
%which heavily relies on the entry-wise large deviation bounds of eigenvector. 
\subsection{A corollary for MMSBM \& SBM \& DCBM}
For convenience, we offer the entry-wise eigenvector large deviation bounds for the three block models: MMSBM, SBM and DCBM. These bounds are derived through the adaptation of Lemma~\ref{lem:entrywise-eigen} with its proof. It is worthy noting that DCBM varies from DCMM in the structure of membership vectors, this distinction doesn't impact the eigenvector analysis. Hence,  the eigenvector results for DCBM are identical to those of DCMM, encompassing both the regularity conditions and the error rates of $\hat{\xi}_1$ and $\widehat{\Xi}_1$. 

In contrast, MMSBM and SBM have no degree parameter, resulting in the use of SCORE algorithm unnecessary. Due to this fact, there is no need to seperate the first eigenvector $\hat{\xi}_1$ from the other eigenvectors $\widehat{\Xi}_1$. Consequently, fewer regularity conditions are needed for MMSBM and SBM,  and the error rate expression will be much simplified. Let $\widehat{\Xi} = (\hat{\xi}_1, \cdots, \hat{\xi}_K)  $ be the eigenvectors of $A$ associated with the first $K$ largest eigenvalues (in magnitude). Similarly, $\Xi$ are the top $K$ eigenvectors of $\Omega$. We denote by $\widehat{\Xi}(i), \Xi(i)$, the $i$-th row of $\widehat{\Xi}$ and $\Xi$, respectively. Based on these notations, we have the following corollary for MMSBM and SBM.
\begin{cor}\label{cor:eigenMMSBM}
Fixed $K\geq 1$. Consider MMSBM  \eqref{MMSBM} or SBM (each $\pi_i$ in \eqref{MMSBM} ranges in $\{e_1,\ldots,e_K\}$).  Let $\lambda_K$ be the $K$-th  largest  (in magnitude) right eigenvalue of $n^{-1}P\Pi'\Pi $. Suppose that $\Vert P\Vert_{\max} \leq c_1$, $n(\Pi'\Pi)^{-1}\leq c_1$ for some constant $c_1>0$ and $|\lambda_K|\sqrt{n\alpha_n/\log (n)}\to \infty$. Then, with probability $1- o(n^{-3})$, there exists an orthogonal matrix $O\in \mathbb{R}^{K \times K}$ such that simultaneously for all $1\leq i \leq n$, 
\begin{align*}
\Vert \widehat{\Xi}(i) - \Xi(i) O\Vert &\leq  \frac{C\sqrt{\log (n)}}{ |\lambda_K| n \sqrt{\alpha_n}  } 
\end{align*}
for some fixed constant $C>0$. 
\end{cor}
The proof of this Corollary closely resembles  the proof of  $\Vert \widehat{\Xi}_1(i) - \Xi_1(i) O_1\Vert$ in Lemma~\ref{lem:entrywise-eigen} by taking $\theta_i = \bar{\theta} = \sqrt{\alpha_n}$  for all $1\leq i \leq n $ and employing the condition $|\lambda_K|\sqrt{n\alpha_n/\log (n)}\to \infty$. Hence, we opt to skip the redundant details.

\section{GoF for SBM, MMSBM, and DCBM}\label{supp:other3}

In this section, we show the proof of Theorems~\ref{thm:MMSBM}, \ref{thm:DCBM}, and \ref{thm:SBM} which follow the same outline as Theorem~\ref{thm:main}.
\subsection{Proof of Theorem~\ref{thm:MMSBM} (MMSBM)} \label{supp:pfMMSBM}
%{\color{blue}This proof is essentially the same as H approach proof for DCMM. We can simply omit the details and say something like: it is a special case of DCMM where $\Theta=I_n$, we can adapt the proof by letting $\Theta=I_n$ and conclude the proof. No need to include the lemma and additional arguments. }

%For MMSBM, we first state several population quantities. 

Recall that under MMSBM, $\Omega^{\rm MMSBM}=\alpha_n\Pi P\Pi'$. In the sequel, for simplicity, we drop the superscript ``MMSBM''. 
%In addition, $V_H=\alpha_nP\Pi'H$, and $R_H=\Omega H=\Pi V_H$. 
%This implies that, for each $1\leq i\leq n$, the $i$-th row $r_i'$ of $R_H$ satisfies $r_i'=\sum_{k=1}^K\pi_i(k)v_k'$, where $v_k'$ is the $k$-th row of $V_H$. Since $[\pi_i(1),\ldots,\pi_i(K)]'$ is a probability weight vector under MMSBM, this observation further implies that $r_i$ is a weighted average of $v_k$, for each $1\leq i\leq n$. Therefore, the vertices $v_1,\ldots, v_K$ can be estimated by $\widehat v_1,\ldots,\widehat v_K$ via the Vertex Hunting algorithm, applied to the rows of the $R_H$ matrix.
%Similar to the analysis in the proof of Theorem~\ref{thm:main}, for convenience of analysis, we skip the operation of removing negative entries in $\widehat w_i$, that is, we assume $\widehat w_i=\widetilde w_i$ in Section~\ref{subsec:MMSBM}. Note that in this case, by the fact that $H\mathbf 1_K=\mathbf 1_n$, we have $\|\widehat w_i\|_1=1$ is automatically satisfied. This fact follows from the same arguments as in \eqref{TtoH-8}. Hence, we deduce that $\widehat\pi_i=\widehat w_i=(\widehat V_H^{-1})'\widehat r_i$, for each $1\leq i\leq n$. Now, we have $\widehat\Pi=AH\widehat V_H^{-1}$. 
%
%
%
%
%
Based on the  Algorithm in Section~\ref{subsec:MMSBM}, we conclude from $\hat\pi_i=\hat w_i=(\widehat V_H^{-1})'\hat r_i$ that $\widehat \Pi=A\widehat H\widehat  V_H^{-1}$. Further, $\widehat \Pi'\widehat \Pi=(\widehat V_H')^{-1} \widehat H'A^2 \widehat H\widehat V_H^{-1}$ and $\widehat\Pi' A\widehat \Pi=(\widehat V_H')^{-1}(\widehat H'A^3 \widehat H)\widehat V_H^{-1}$.
%
%Similar to the analysis in the proof of Theorem~\ref{thm:main}, for convenience of analysis, we skip the operation of removing negative entries in $\hat w_i$, that is, we assume $\hat w_i=\widetilde w_i$ in Section~\ref{subsec:MMSBM}. Note that in this case, by the fact that $H\mathbf 1_K=\mathbf 1_n$, we have $\|\hat w_i\|_1=1$ is automatically satisfied. This fact follows from the same arguments as in \eqref{TtoH-8}. Hence, we deduce that $\hat\pi_i=\hat w_i=(\widehat V_H^{-1})'\hat r_i$, for each $1\leq i\leq n$. Now, we have $\widehat \Pi=AH\widehat  V_H^{-1}$, so that $\widehat \Pi'\widehat \Pi=(\widehat V_H')^{-1}H'A^2H\widehat V_H^{-1}$ and $\widehat\Pi' A\widehat \Pi=(\widehat V_H')^{-1}(H'A^3H)\widehat V_H^{-1}$.
 We therefore deduce that
\begin{align}\label{omegahat}
\widehat \Omega = \widehat\Omega^{\rm MMSBM}&=\widehat\Pi(\widehat\Pi'\widehat\Pi)^{-1}\widehat\Pi'A\widehat\Pi(\widehat\Pi'\widehat\Pi)^{-1}\widehat\Pi'\notag
\\
&=A\widehat H( \widehat H'A^2 \widehat H)^{-1}(\widehat H'A^3 \widehat H)(\widehat H'A^2 \widehat H)^{-1} \widehat H'A.
\end{align}
Observing that $\Omega=\Pi P\Pi'$, we obtain by direct calculations that
\begin{align}\label{omega}
\Omega=\Omega \widehat H( \widehat H'\Omega^2 \widehat H)^{-1}( \widehat H'\Omega^3 \widehat H)( \widehat H'\Omega^2  \widehat H)^{-1} \widehat H'\Omega.
\end{align}

We first claim $\widehat H=\Pi_0$ with probability $1- o(1)$. 
Recall the procedures by which we get $\widehat H$ in Section~\ref{subsec:MMSBM}. Using Corollary~\ref{cor:eigenMMSBM}, we get that 
\begin{align*}
\max_{1\leq i\leq n } \Vert \widehat{\Xi}(i) - \Xi(i)O\Vert \lesssim \frac{\sqrt{\log n}}{\beta_n\sqrt{n\alpha_n} } \cdot \frac{1}{\sqrt{n}}\, . 
\end{align*} 
By our vertex hunting algorithm,  $\max_{1\leq k\leq K} \Vert O \hat{v}^*_k - {v}^*_k \Vert \lesssim {\sqrt{\log n}}/{(\beta_nn \sqrt{\alpha_n})  }$ (for simplicity, we assume a perfect alignment between $\{\hat{v}^*_k\}$ and $\{{v}^*_k\}$, thereby excluding the consideration of potential permutations of $\widehat{\Pi}^{\rm MS_0}$ for approximating $\Pi$).  Note that $\tilde{\pi}_i = (\widehat{V}^*)^{-1}\widehat{\Xi}'(i) $. It follows that with probability $1- o(n^{-3})$, simultaneously for all $1\leq i \leq n$, 
\begin{align} \label{2023081901}
 \Vert \tilde{\pi}_i -  \pi_i \Vert \lesssim \Vert (V^*)^{-1} \Vert \Vert  \widehat{\Xi}(i) - \Xi_1(i)O\Vert + \Vert (V^*)^{-1} - (O\widehat{V}^*)^{-1} \Vert \Vert \Xi(i)\Vert \lesssim \frac{\sqrt{\log n}}{\beta_n\sqrt{n\alpha_n} },
\end{align}
where we employed the facts that $V^*(V^*)' = G^{-1}$ so that $ \Vert (V^*)^{-1} \Vert  \asymp n^{1/2}$ and $\Vert \Xi(i)\Vert \lesssim 1/\sqrt{n}$. Next, we consider truncation and renormalization on $\tilde{\pi}_i$'s. Let $\tilde{\pi}_i^*$ be the truncated version of $\tilde{\pi}_i$. It is worthy noting that if $\tilde{\pi}_i (k)<0$ for some $1\leq k \leq K$, by (\ref{2023081901}) and the fact that ${\pi}_i (k)\geq 0$, it holds that $|\tilde{\pi}_i^*(k) - {\pi}_i (k) |  =  {\pi}_i (k) \leq |\tilde{\pi}_i^*(k) - {\pi}_i^*(k)|$. Consequently, 
$$\Vert \tilde{\pi}_i^*- \pi_i \Vert_1\leq\Vert \tilde{\pi}_i -  \pi_i \Vert_1 \leq \sqrt{K} \Vert \tilde{\pi}_i -  \pi_i \Vert \lesssim \frac{\sqrt{\log n}}{\beta_n\sqrt{n\alpha_n} }. $$
By definition, $\hat{\pi}_i^{{\rm MS}_0} = \tilde{\pi}_i^*/\Vert \tilde{\pi}_i^*\Vert_1$ for all $1\leq i \leq n $.  It is straightforward to derive 
\begin{align*}
\Vert \hat{\pi}_i^{{\rm MS}_0}  - \pi_i\Vert_1 = \sum_{k=1}^K \big|  \hat{\pi}_i^{{\rm MS}_0}(k)  - \pi_i(k)\big| &\leq \sum_{k=1}^K  \frac{1}{\Vert \tilde{\pi}_i^*\Vert_1 } \big|  \tilde{\pi}_i^*(k) - \pi_i(k)\big| + \pi_i(k) \Big|\frac{1}{\Vert \tilde{\pi}_i^*\Vert_1} - \frac{1}{\Vert {\pi}_i\Vert_1}\Big| \notag\\
& \leq \frac{\Vert \tilde{\pi}_i^*- \pi_i \Vert_1 + \big|\Vert \tilde{\pi}_i^*\Vert_1 - \Vert {\pi}_i\Vert_1 \big|}{\Vert \tilde{\pi}_i^*\Vert_1} \notag\\
& \leq \frac{2\Vert \tilde{\pi}_i^*- \pi_i \Vert_1}{\Vert \tilde{\pi}_i^*\Vert_1} \lesssim   \frac{\sqrt{\log n}}{\beta_n\sqrt{n\alpha_n} },
\end{align*}
where we use the fact that $\Vert \tilde{\pi}_i^*\Vert_1 \geq \Vert {\pi}_i\Vert_1- \Vert \tilde{\pi}_i^*- \pi_i \Vert_1  = 1- o_{\mathbb{P}}(1)  $. 
 Furthermore, using (e) of Condition~\ref{cond:MSCORE}, and employing the identical reasoning for $\widehat{\Pi}_0 = \Pi_0$ in  the proof Theorem~\ref{thm:MSCORE} in Section~\ref{susec:pf-MSCORE}, we are able to show that obtaining  $\widehat H$ by applying net-rounding to $ \widehat \Pi^{\rm MS_0}$, it holds that $\mathbb P(\widehat H=\Pi_0)\to1$, as $n\to\infty$. Here $\Pi_0$ is defined in a similar way as in Condition~\ref{cond:MSCORE} for DCMM where $\Theta$ is taken as $\sqrt{\alpha_n}I_n$.

We proceed to show the asymptotic normality $T_n(\widehat \Omega)\to N(0,1)$ using arguments similar to those used in the proof of Theorem~\ref{thm:main}.
Since $\widehat H=\Pi_0$ with probability $1- o(1)$, we can replace $\widehat H$ by a deterministic $\Pi_0$ in both ${\Omega} $ and $\widehat{\Omega}^{\rm MMSBM}$ and study the normality. For notation simplicity, we use $H$ in the following analysis. However, one should keep in mind that it actually represents a deterministic $n\times K$ matrix each row of which is a K-dimensional weight vector.

We will follow the proof of Theorem~\ref{thm:main} by considering $\Delta:= \widehat{\Omega }^{\rm MMSBM} - \Omega$ and proving each term in Lemma~\ref{lem:diffU} of order $o_{\mathbb{P}}({\rm tr} (\Omega^3)) = o_{\mathbb{P}}(n^3\alpha_n^3) $ by viewing $\Theta = \sqrt{\alpha_n}\, I_n  $ so that $\Vert \theta\Vert \asymp \sqrt{n\alpha_n}$. The reason behind this is that MMSBM, as a special case of DCMM, satisfies Corollary~\ref{cor:SCC}, leading to $T_n(\Omega) \to N(0,1)$. All remains to check $U_{n, 3}(\widehat{\Omega}^{\rm MMSBM}) - U_{n,3} (\Omega) \ll \sqrt{{\rm tr} (\Omega^3)} $.  To do so, it suffices to prove a key lemma analogue to Lemma~\ref{lem:Delta} by taking $\Vert \theta \Vert \asymp \sqrt {n\alpha_n} $. More specifically, we shall show that 
\begin{itemize}
\item[(a1)] $\Vert \Delta \Vert =o_{\mathbb{P}}( n^{1/2}\alpha_n^{1/2})$;
\item[(b1)] $ \Vert W_1\Delta\Vert =O_{\mathbb{P}}(n\alpha_n)$;
\item[(c1)] $\Vert W_1^2 \Delta \Vert  = o_{\mathbb{P}}(n^{3/2}\alpha_n^{3/2})$;
\item[(d1)] $\Vert {\rm diag} (W_1^2) \Delta \Vert = o_{\mathbb{P}}(n^{3/2}\alpha_n^{3/2})$.
\end{itemize}
Based on these estimates, similar to the proof in Section~\ref{supp:main}, we can complete the proof of normality by 
\begin{align*}
\big| U_{n, 3}(\widehat{\Omega}) - U_{n,3} (\Omega)  \big|& \lesssim \Vert \Delta\Vert^3 +\Vert W_1\Delta\Vert \Vert \Delta \Vert + \Vert W_1^2 \Delta\Vert +\Vert {\rm diag} (\Omega)\Vert   \big(\Vert \Delta\Vert^2 + \Vert W_1\Delta\Vert \big)  \notag\\
&\quad + \Vert {\rm diag} (W_1^2) \Delta\Vert  + \Vert {\rm diag} (\Omega)\Vert^2 \Vert \Delta\Vert \notag\\
& = o_{\mathbb{P}}(n^3\alpha_n^3) \,.
\end{align*}

The proofs of the estimates (a1)-(d1) are similar to the proofs of Lemma~\ref{lem:Delta} by leveraging Lemma~\ref{lem:1.7} with $\Vert \theta\Vert_1 \asymp n\sqrt{\alpha_n}$, $\Vert \theta\Vert \asymp \sqrt{n\alpha_n} $ and $ \Vert \theta\Vert_3\asymp n^{1/3}\sqrt{\alpha_n}$. However, in contrast to $A- \Omega = W_1$, the difference between square or cube terms involves more sub-terms. This renders the overall analysis more tedious. For convenience, we list out some significant estimates that will be frequently used in the later analysis. These estimates can be deduced from Lemma~\ref{lem:1.7} by letting $\Vert \theta\Vert_1 \asymp n\sqrt{\alpha_n}$, $\Vert \theta\Vert \asymp \sqrt{n\alpha_n} $ and $ \Vert \theta\Vert_3\asymp n^{1/3}\sqrt{\alpha_n}$. Therefore, we skip their proofs.  
\begin{align} \label{2023080201}
& \Vert \Pi' W_1\Pi \Vert  = O_{\mathbb{P}}\big( n\alpha_n^{1/2}\sqrt{ \log (n)}\big),\qquad \Vert \Pi' W_1 H \Vert = O_{\mathbb{P}}\big( n\alpha_n^{1/2}\sqrt{\log (n)} \big),\notag\\
&\Vert H' W_1^2 H\Vert  = O_{\mathbb{P}}( n^2\alpha_n),\qquad~~~~~~~~~\Vert \Pi' W_1^2 H\Vert  = O_{\mathbb{P}}( n^2 \alpha_n) , \notag\\
&\Vert \Pi' W_1^2\Pi \Vert  = O_{\mathbb{P}}( n^2\alpha_n), \qquad ~~~~~~~~~~\Vert \Pi' {\rm diag} (W_1^2) \Pi \Vert  = O_{\mathbb{P}}( n^2\alpha_n) , \notag \\
&\Vert H' W_1^3 H\Vert  = O_{\mathbb{P}}( n^2\alpha_n ),\quad\Vert \Pi' W_1^3 H\Vert  = o_{\mathbb{P}}( n^{5/2}\alpha_n^{3/2}),\quad\Vert \Pi'{\rm diag} (W_1^2)W_1 H\Vert  = o_{\mathbb{P}}( n^{5/2} \alpha_n^{3/2}),\notag\\
& \Vert H' W_1^4 H\Vert  = O_{\mathbb{P}}( n^3\alpha_n^{2}), \quad \Vert H' W_1 {\rm diag} (W_1^2)W_1 H\Vert  = O_{\mathbb{P}}( n^3\alpha_n^{2}).
\end{align}
Under the assumptions in Theorem~\ref{thm:MMSBM}, we also have 
\begin{align}\label{2023080301}
\Vert H' \Omega^2 H\Vert \lesssim n^3\alpha_n^2 , \qquad \Vert (H' \Omega^2 H)^{-1}\Vert \lesssim n^{-3}\alpha_n^{-2}\beta_n^{-2}\, .
\end{align}
Hereafter, our analysis is based on (\ref{2023080201}). All the bounds there hold under an event whose probability is $1- o(1) $, therefore, all of our derivations below hold under this event. For the sake of convenience, we will often derive the upper bounds without specifying the high probability statement. 

Define $\mathcal{T}_2 = \sum_{j=1}^{\infty} \big(\big[ H'(\Omega^2 - A^2 )H \big] (H'\Omega^2H)^{-1}\big)^j  $. Let us start with some useful decompositions:
\begin{align} 
&(H'A^2H)^{-1} = (H'\Omega^2H)^{-1}  + (H'\Omega^2H)^{-1} (I_K + \mathcal{T}_2) \big[ H'(\Omega^2 - A^2 )H \big] (H'\Omega^2H)^{-1}\, , \label{eq:deHA2H}\\
 &H'(A^2 - \Omega^2 )H  = H' W_1^2 H + H' \Omega W_1 H + H' W_1 \Omega H \, , \label{eq:deHA2-O2H}
\end{align}
and 
\begin{align}\label{eq:deHA3H}
H'A^3H = H' \Omega^3 H + H'W_1^3 H + H' W_1 \Omega W_1 H  + \sum_{\substack{a_1+a_2+a_3 = 3;\\ a_2 =1, 2\, a_1, a_3\geq 0}}H'\Omega^{a_1}  W_1^{a_2} \Omega^{a_3} H \,. 
\end{align}
With a little ambiguity of the notations, throughout this  subsection, we define 
\begin{align*}
G_H: = \Pi' H, \qquad G = \Pi'\Pi \, . 
\end{align*}
Under the assumptions of Theorem~\ref{thm:MMSBM}, we have the estimates $\Vert G_H\Vert \asymp n$,  $\Vert G_H^{-1} \Vert \asymp 1/n$, $\Vert G\Vert \asymp n$, and $\Vert G^{-1} \Vert \asymp 1/n$.  Based on these shorthand notations, we can write 
\begin{align}\label{2023080203}
H' \Omega^{a} H = \alpha_n^{a}G_H'P (G P)^{a-1} G_H, \qquad \text{for } a=1, 2, 3\, . 
\end{align}
%
%\begin{align} \label{2023080201}
%& \Vert H' \Omega^2 H\Vert \lesssim n^3 , \quad \Vert (H' \Omega^2 H)^{-1}\Vert \lesssim n^{-3}\beta_n^{-2}, \quad   \Vert H' W_1^2 H\Vert  \lesssim n^2  \quad \Vert H' \Omega W_1 H\Vert  \lesssim n^2 \sqrt{\log (n)}, \notag\\
%& \Vert H' \Omega^3 H\Vert \lesssim n^4, \quad   \Vert H'W_1^3 H  \Vert \ll n^{5/2}, \quad \Vert H' W_1 \Omega W_1 H\Vert \lesssim n^2 \log n, \notag\\
%& \sum_{\substack{a_1+a_2+a_3 = 3;\\ a_2 =1, 2, \, a_1, a_3\geq 0}} \Big\Vert H'\Omega^{a_1}  W_1^{a_2} \Omega^{a_3} H \Big \Vert \lesssim n^3 \sqrt{\log n }, \quad \sum_{\substack{a_1+a_2\leq 3\\ a_2=1, 2, a_1\geq 0}}\Vert H'\Omega^{a_1}  W_1^{a_2} \Pi\Vert \lesssim n^{a_1+ a_2} \sqrt{\log n} 
%\end{align}
 It follows from (\ref{eq:deHA2-O2H}), (\ref{2023080301}) and the first estimates in the first and second line of  (\ref{2023080201}) that 
\begin{align*}
\Vert [ H'(\Omega^2 - A^2 )H ] (H'\Omega^2H)^{-1} \Vert &\leq \big(\Vert H' W_1^2 H \Vert + \alpha_n\Vert \Pi' W_1 H\Vert \Vert G_H' P\Vert \big) \Vert (H'\Omega^2H)^{-1} \Vert
\\
& \lesssim \frac{\sqrt{\log (n)}}{n \alpha_n \beta_n^2}.
\end{align*}
%which is $o(1)$.
As a result, $\Vert \mathcal{T}_2 \Vert  = o_{\mathbb{P}}(1) $. 
%
%Moreover, for $a_3>0$, $a_2= 1, 2$, $a_1\geq 0$ such that  $a_1+ a_2+ a_3= 3$,  we bound 
%\begin{align}\label{2023080202}
%\Vert (H'\Omega^{a_1}  W_1^{a_2} \Omega^{a_3} H)  (H'A^2H)^{-1} \Vert &\lesssim \Vert (H'\Omega^{a_1}  W_1^{a_2} \Omega^{a_3} H)  (H'\Omega^2H)^{-1} \Vert \notag\\
%& = \Vert  H' \Omega^{a_1} W_1^{a_2} \Pi (P G)^{a_3-1}G^{-1} P^{-1} (G_H')^{-1} \Vert \notag\\
%& \lesssim n^{a_1 + a_2} \sqrt{\log (n)}\, n^{a_3-1} \cdot \frac{1}{n^2 \beta_n } \notag\\
%& \lesssim \frac{\sqrt{\log (n)} }{\beta_n}
%\end{align}
%If $a_3=0$, then it must hold that $a_1>0$, we can consider $ (H'A^2H)^{-1} (H'\Omega^{a_1}  W_1^{a_2} \Omega^{a_3} H) $ instead and get the same result as above. 

In view of (\ref{eq:deHA2H}) and (\ref{eq:deHA3H}), we obtain that $\Delta = \Delta_1 + \Delta_2 + \Delta_3$, where 
\begin{align*}
&\Delta_1:= AH(H'A^2H)^{-1}(H'A^3H - H' \Omega^3 H)(H'A^2H)^{-1}H'A \, , \notag\\
& \Delta_{2}: = (A-\Omega ) H(H'A^2H)^{-1}H' \Omega^3 H (H'A^2H)^{-1}H'A  \notag\\
&\qquad+ \Omega H(H'A^2H)^{-1}H' \Omega^3 H (H'A^2H)^{-1}H'(A- \Omega) \, ,  \notag\\
 & \Delta_{3} : = \Omega H (H'\Omega^2H)^{-1} (I_K + \mathcal{T}_2) \big[ H'(\Omega^2 - A^2 )H \big] (H'\Omega^2H)^{-1}H' \Omega^3 H (H'A^2H)^{-1} H' \Omega \notag\\
 &\quad \quad +  \Omega H (H'\Omega^2H)^{-1}H' \Omega^3 H (H'\Omega^2H)^{-1} (I_K + \mathcal{T}_2) \big[ H'(\Omega^2 - A^2 )H \big] (H'\Omega^2H)^{-1} H' \Omega .
\end{align*}
We will claim below that $\Delta_a$ satisfies (a1)--(d1), where $\Delta$ is substituted by $\Delta_a$, for each $a=1,2,3$. Then, by triangle inequality, (a1)--(d1) is proved. 

To study $\Delta_1$,  we need to bound the generic form $\Vert \Gamma \Delta_1 \Vert  $ for $\Gamma  = I_n, W_1, W_1^2, {\rm diag}(W_1^2)$. We first use the fact, for arbitrary matrix $B\in \mathbb{R}^{K\times K} $,
\begin{align*}
&\Vert (H'A^2H)^{-1}B \Vert = \Vert B(H'A^2H)^{-1}\Vert\\
&\leq  \Vert B(H'\Omega^2H)^{-1}\Vert+\|B(H'\Omega^2H)^{-1} (I_K + \mathcal{T}_2) [ H'(\Omega^2 - A^2 )H ] (H'\Omega^2H)^{-1}\|\\
&\leq \Vert B(H'\Omega^2H)^{-1}\Vert\times\big(1+\|(I_K + \mathcal{T}_2) [ H'(\Omega^2 - A^2 )H ] (H'\Omega^2H)^{-1}\|\big) \lesssim \Vert (H'\Omega^2H)^{-1}B \Vert,
\end{align*}
where we used (\ref{eq:deHA2H}) and the estimate that $\Vert \mathcal{T}_2 \Vert  = o_{\mathbb{P}}(1)$ and $\|[ H'(\Omega^2 - A^2 )H ] (H'\Omega^2H)^{-1}\|=o_{\mathbb P}(1)$. As a result,
\begin{align*}
\Vert \Gamma \Delta_1 \Vert & = \Vert H'A \Gamma AH(H'A^2H)^{-1}(H'A^3H - H' \Omega^3 H)(H'A^2H)^{-1} \Vert  \notag\\
& \lesssim  \Vert H'A \Gamma AH(H'A^2H)^{-1}(H'A^3H - H' \Omega^3 H)(H' \Omega^2H)^{-1} \Vert \notag\\
&= \Vert (H'A^2H)^{-1}(H'A^3H - H' \Omega^3 H)(H' \Omega^2H)^{-1} H'A \Gamma AH\Vert \notag\\
& \lesssim \Vert (H'\Omega^2H)^{-1}(H'A^3H - H' \Omega^3 H)(H' \Omega^2H)^{-1} H'A \Gamma AH\Vert \, . 
\end{align*}
We further plug in the representation (\ref{2023080203}) and get 
\begin{align}\label{eq:GammaDelta1}
&\Vert \Gamma \Delta_1 \Vert \lesssim \alpha_n^{-4}\Vert(PG_H)^{-1} G^{-1}  (G_H'P)^{-1} (H'A^3H - H' \Omega^3 H)(PG_H)^{-1} G^{-1}  (G_H'P)^{-1} H'A \Gamma AH\Vert \notag\\
& = \alpha_n^{-4}\Vert G^{-1}  (G_H'P)^{-1} (H'A^3H - H' \Omega^3 H)(PG_H)^{-1} G^{-1}  (G_H'P)^{-1} H'A \Gamma AH (PG_H)^{-1} \Vert \notag\\
& \lesssim  \alpha_n^{-4}\Vert G^{-1} \Vert^2  \Vert (G_H'P)^{-1} (H'A^3H - H' \Omega^3 H)(PG_H)^{-1}\Vert \Vert (G_H'P)^{-1} H'A \Gamma AH (PG_H)^{-1} \Vert  \notag\\
&  \lesssim n^{-2}\alpha_n^{-4} \Vert (G_H'P)^{-1} (H'A^3H - H' \Omega^3 H)(PG_H)^{-1}\Vert \Vert (G_H'P)^{-1} H'A \Gamma AH (PG_H)^{-1} \Vert \, . 
\end{align}
We write for short
\begin{align*}
 \mathbb{T}_1: = \Vert (G_H'P)^{-1} (H'A^3H - H' \Omega^3 H)(PG_H)^{-1}\Vert, \quad \mathbb{T}_2(\Gamma) : = \Vert (G_H'P)^{-1} H'A \Gamma AH (PG_H)^{-1} \Vert \, . 
 \end{align*}
 In the sequel, we show the upper bounds of $\mathbb{T}_1$ and $\mathbb{T}_2(\Gamma) $ for $\Gamma  = I_n, W_1, W_1^2, {\rm diag}(W_1^2)$ by leveraging (\ref{eq:deHA3H}) and (\ref{2023080201}). 

By (\ref{eq:deHA3H}) and triangle inequality, we first bound 
\begin{align*}
\mathbb{T}_1&\leq \Vert (G_H'P)^{-1} H'W_1^3 H(PG_H)^{-1}\Vert + \Vert (G_H'P)^{-1} H'W_1\Omega W_1H (PG_H)^{-1}\Vert  \notag\\
& \quad + 2 \Vert (G_H'P)^{-1} H'W_1\Omega ^2 H (PG_H)^{-1}\Vert  + \Vert (G_H'P)^{-1} H'\Omega W_1\Omega H (PG_H)^{-1}\Vert \notag\\
& =: \mathbb{T}_{11} + \mathbb{T}_{12} + 2\mathbb{T}_{13} + \mathbb{T}_{14}
\end{align*}
Using the definition $\Omega = \Pi P \Pi'$ and (\ref{2023080201}), we can bound $\mathbb{T}_{1a}$ for $a= 1,2 ,3, 4$ as follows. 
\begin{align*}
& \mathbb{T}_{11} \leq \Vert (PG_H)^{-1} \Vert^2 \Vert H'W_1^3 H\Vert \lesssim \frac{\alpha_n}{ \beta_n^2}, \notag\\
& \mathbb{T}_{12} \leq \alpha_n \Vert (PG_H)^{-1} \Vert^2 \Vert H'W_1\Pi \Vert^2 \Vert P\Vert \lesssim \frac{\alpha_n^2\log (n) }{ \beta_n^2} ,\notag\\
& \mathbb{T}_{13} \leq \alpha_n^2\Vert (PG_H)^{-1} \Vert \Vert H'W_1\Pi \Vert \Vert PG\Vert \lesssim \frac{n\alpha_n^{5/2}\sqrt{\log (n)} }{ \beta_n},\notag\\
& \mathbb{T}_{14} \leq  \alpha_n^2\Vert \Pi'W_1\Pi \Vert  \lesssim n \alpha_n^{5/2}\sqrt{\log(n)} .
\end{align*}
Here to bound $\mathbb{T}_{13}$ and $\mathbb{T}_{14}$, we used the identity $\Omega H (PG_H)^{-1} = \Pi P G_H  (PG_H)^{-1} = \Pi$. From the above inequalities and the condition $\beta_n \sqrt{n\alpha_n} \gg\sqrt{\log (n)} $ , we conclude that with probability $1- o(1) $, 
\begin{align}\label{est:T1}
\mathbb{T}_1 \lesssim \frac{n\alpha_n^{5/2} \sqrt{\log (n)} }{ \beta_n}\, . 
\end{align}

Next, for $\mathbb{T}_2(\Gamma) $, we have
\begin{align*}
\mathbb{T}_2(\Gamma) &\leq \Vert (G_H'P)^{-1} H'W_1\Gamma W_1H (PG_H)^{-1} \Vert +2 \Vert (G_H'P)^{-1} H'\Omega\Gamma W_1H (PG_H)^{-1} \Vert \notag\\
& \quad + \Vert (G_H'P)^{-1} H'\Omega\Gamma \Omega H (PG_H)^{-1} \Vert \notag\\
&=:  \mathbb{T}_{21}(\Gamma) + 2\mathbb{T}_{22}(\Gamma) + \mathbb{T}_{23} (\Gamma)\, . 
\end{align*}
Direct calculations yield
\begin{align}\label{t23}
&\mathbb{T}_{21}(\Gamma)  \lesssim \Vert (PG_H)^{-1} \Vert^2 \Vert H'W_1\Gamma W_1H\Vert \lesssim \frac{1}{n^2 \beta_n^2 } \cdot \left\{
\begin{array}{ll}
n^2\alpha_n, &\Gamma =I_n, W_1\\
n^3\alpha_n^2,&\Gamma= W_1^2, {\rm diag}(W_1^2)
\end{array}
\right. \notag\\
&\mathbb{T}_{22}(\Gamma)  \lesssim \alpha_n\Vert (PG_H)^{-1} \Vert \Vert \Pi'\Gamma W_1H\Vert \lesssim \frac{\alpha_n}{n \beta_n } \cdot \left\{
\begin{array}{ll}
n\alpha_n^{1/2}\sqrt{\log n} , &\Gamma =I_n\\
n^2\alpha_n, &\Gamma =W_1\\
o_{\mathbb{P}} (n^{5/2}\alpha_n^{3/2}),&\Gamma=W_1^2, {\rm diag}(W_1^2)
\end{array}
\right. \notag\\
&\mathbb{T}_{23}(\Gamma) \leq \alpha_n^2\Vert \Pi' \Gamma\Pi\Vert \lesssim  \alpha_n^2 \cdot \left\{
\begin{array}{ll}
n , &\Gamma =I_n\\
n\alpha_n^{1/2}\sqrt{\log n}\, , &\Gamma =W_1\\
n^2\alpha_n,&\Gamma=W_1^2, {\rm diag}(W_1^2)
\end{array}
\right. 
\end{align}
which follows from (\ref{2023080201}). 
Combining these together with the condition $\beta_n \sqrt{n\alpha_n} \gg\sqrt{\log (n)} $,  we arrive at 
\begin{align*}
&\mathbb{T}_{2}(I_n) \lesssim n\alpha_n^2, \qquad ~~~~ \mathbb{T}_{2}(W_1) \lesssim n\alpha_n^2\beta_n^{-1} + n\alpha_n^{5/2}\sqrt{\log (n)},
\\
&\mathbb{T}_{2}(W_1^2) \lesssim n^2\alpha_n^3, \quad~~~~ \mathbb{T}_{2}({\rm diag}(W_1^2)) \lesssim n^2\alpha_n^3 \, . 
\end{align*}
We further substitute the above inequalities and (\ref{est:T1}) into (\ref{eq:GammaDelta1}) and conclude the proof of (a1)-(d1) for $\Delta_1$. 

We proceed to prove  (a1)-(d1) for $\Delta_2$ by bounding the generic form $\Vert \Gamma \Delta_2\Vert$ for $\Gamma  = I_n, W_1, W_1^2, {\rm diag}(W_1^2)$. Similarly to how we deal with $\Vert \Gamma \Delta_1\Vert$, we first derive 
\begin{align*}
&\Vert \Gamma \Delta_2\Vert \leq \Vert H'A\Gamma W_1H (H'A^2H)^{-1} H' \Omega^3 H (H'A^2 H)^{-1} \Vert \\
&\qquad\qquad+  \Vert H'W_1\Gamma \Omega H (H'A^2H)^{-1} H' \Omega^3 H (H'A^2 H)^{-1} \Vert \notag\\
& \lesssim \Vert H'A\Gamma W_1H (H'\Omega^2H)^{-1} H' \Omega^3 H (H'\Omega^2 H)^{-1} \Vert \hspace{-0.3em}+ \hspace{-0.3em} \Vert H'W_1\Gamma \Omega H (H'\Omega^2H)^{-1} H' \Omega^3 H (H'\Omega^2 H)^{-1} \Vert \notag\\
& \lesssim \alpha_n^{-1} \Vert H'A\Gamma W_1H (G_H)^{-1} (G_H'P)^{-1} \Vert +  \Vert H'W_1\Gamma \Pi (G_H')^{-1}  \Vert ,
\end{align*}
where we frequently used (\ref{2023080203}). Recalling that $A=\Omega+W_1$, we further have 
\begin{align*}
\Vert \Gamma \Delta_2\Vert&  \lesssim 
% \Vert  (G_H'P)^{-1}H'A\Gamma W_1H (G_H)^{-1}  \Vert +  \Vert H'W_1\Gamma \Pi (G_H')^{-1}  \Vert  \notag\\
%& \leq
 \alpha_n^{-1} \Vert  (G_H'P)^{-1}H'\Omega \Gamma W_1H (G_H)^{-1}  \Vert
\\
&\qquad+  \alpha_n^{-1}\Vert  (G_H'P)^{-1}H'W_1\Gamma W_1H (G_H)^{-1}  \Vert+ \Vert H'W_1\Gamma \Pi (G_H')^{-1}  \Vert \\
& \leq  \Vert  \Pi' \Gamma W_1H\Vert \Vert  (G_H)^{-1}  \Vert +  \alpha_n^{-1}\Vert  (G_H'P)^{-1}\Vert \Vert H'W_1\Gamma W_1H \Vert \Vert  (G_H)^{-1}  \Vert\notag\\
& \lesssim \left\{
\begin{array}{ll}
\sqrt{\alpha_n \log n} + \beta_n^{-1} & \Gamma = I_n \\
n\alpha_n  & \Gamma = W_1\\
o_{\mathbb{P}}(n^{3/2}\alpha_n^{3/2})& \Gamma = W_1^2, {\rm diag} (W_1^2)
\end{array}
\right. 
\end{align*}
where the last step follows from (\ref{2023080201}). Therefore, with the condition that $\beta_n\sqrt{n\alpha_n} \gg \sqrt{\log n} $,  we finish the proof for $\Delta_2$. 

Lastly, we consider $\Delta_3$. Analogously to $\Vert \Gamma \Delta_1 \Vert $ and $\Vert \Gamma \Delta_2 \Vert $, we derive 
\begin{align*}
\Vert \Gamma \Delta_3\Vert &\lesssim \Vert H'\Omega \Gamma \Omega H (H'\Omega^2H)^{-1} (I_K + \mathcal{T}_2) \big[ H'(\Omega^2 - A^2 )H \big] (H'\Omega^2H)^{-1} H' \Omega^3 H (H'\Omega^2 H)^{-1} \Vert \notag\\
& \quad +  \Vert H'\Omega \Gamma \Omega H (H'\Omega^2H)^{-1}H' \Omega^3 H (H'\Omega^2H)^{-1} (I_K + \mathcal{T}_2) \big[ H'(\Omega^2 - A^2 )H \big] (H'\Omega^2H)^{-1}  \Vert \notag\\
& \lesssim \alpha_n^{-1} \Vert \Pi' \Gamma \Pi G^{-1} (G_H'P)^{-1} (I_K + \mathcal{T}_2)[ H'(\Omega^2 - A^2 )H ] G_H^{-1} \Vert \notag\\
& \quad +  \alpha_n^{-1}\Vert \Pi' \Gamma \Pi (G_H')^{-1} (I_K + \mathcal{T}_2)[ H'(\Omega^2 - A^2 )H ] (PG_H)^{-1}G^{-1} \Vert.
\end{align*}
Notice that 
\begin{align*}
& (G_H'P)^{-1}\mathcal{T}_2 =   \sum_{j=1}^{\infty} \Big( (G_H'P)^{-1}\big[ H'(\Omega^2 - A^2 )H \big] (PG_H)^{-1}G^{-1}\Big)^j  (G_H'P)^{-1} =: \widetilde{\mathcal{T}}_2 (G_H'P)^{-1}\notag\\
&  (G_H')^{-1}\mathcal{T}_2 =   \sum_{j=1}^{\infty} \Big( (G_H')^{-1}\big[ H'(\Omega^2 - A^2 )H \big] (PG_H)^{-1}G^{-1}P^{-1} \Big)^j  (G_H')^{-1} =: \overline{\mathcal{T}}_2 (G_H')^{-1}
\end{align*}
 Similar arguments for proving $\Vert \mathcal{T}_2\Vert  = o_{\mathbb{P}}(1)$ also hold for $\Vert \widetilde{\mathcal{T}}_2\Vert = o_{\mathbb{P}}(1)$ and $\Vert \overline{\mathcal{T}}_2\Vert = o_{\mathbb{P}}(1)$. We refrain ourselves from repeated details. Consequently, we further bound 
\begin{align*}
\Vert \Gamma \Delta_3\Vert & \lesssim  \alpha_n^{-1} \Vert \Pi' \Gamma \Pi \Vert \Vert G^{-1} \Vert \Vert \big(\Vert I_K + \overline{\mathcal{T}}_2 \Vert + \Vert I_K + \widetilde{ \mathcal{T}} _2 \Vert \big) \Vert (G_H')^{-1}\big[ H'(\Omega^2 - A^2 )H \big] (PG_H)^{-1}\Vert  \notag\\
& \lesssim n^{-1} \alpha_n^{-1}  \Vert \Pi' \Gamma \Pi\Vert  \Vert (G_H')^{-1} \big[ H'(\Omega^2 - A^2 )H \big] (PG_H)^{-1}\Vert .
\end{align*}
We then employ (\ref{2023080201}) and (\ref{eq:deHA2-O2H}), which yields
\begin{align*}
&\Vert (G_H')^{-1} \big[ H'(\Omega^2 - A^2 )H \big] (PG_H)^{-1}\Vert  \\
& \leq \Vert (G_H')^{-1} H'W_1^2 H  (PG_H)^{-1} \Vert  + 2\Vert (G_H')^{-1} H'W_1\Omega H (PG_H)^{-1}\Vert \notag\\
& \lesssim  n^{-2} \beta_n^{-1}\Vert H' W_1^2 H\Vert  + 2n^{-1} \alpha_n  \Vert H'W_1\Pi\Vert \lesssim \alpha_n \beta_n^{-1} + \alpha_n^{3/2} \sqrt{\log (n)}\, . 
\end{align*}
This, together with \eqref{t23} leads to (a1)--(d1) for $\Delta_3$.  We therefore complete the proof.

%
% {\color{red} {\bf Revise and Claim later} First, by treating MMSBM as a special case of DCMM by taking $\Theta=\sqrt{\alpha_n}I_n$, we obtain from Theorem~\ref{thm:MSCORE} that, by taking $H=\widehat \Pi_0^{\rm MMSBM}$, it holds that $\mathbb P(H=\Pi_0^{\rm MMSBM})\to1$, as $n\to\infty$, where $\widehat\Pi_0^{\rm MMSBM}$ is defined in Section~\ref{subsec:MMSBM} and $\Pi_0^{\rm MMSBM}$ is defined in a similar way as in Condition~\ref{cond:MSCORE} for DCMM where $\Theta$ is taken as $\sqrt{\alpha_n}I_n$. Therefore,  it suffices to show the asymptotic normality on the event that $H=\Pi_0^{\rm MMSBM}$.  }

\subsection{Proof of Theorem~\ref{thm:DCBM} (DCBM)} \label{supp:subsec_DCBM}
In this subsection, we provide the proof of Theorem~\ref{thm:DCBM} (DCBM) whose proof outline is the same as Theorem~\ref{thm:main}. Moreover, since DCBM is a special case of DCMM, analogous to the proofs in Sections~\ref{supp:pfMMSBM}  and \ref{supp:pfSBM}, we only need to verify Lemma~\ref{lem:Delta} under the assumptions in Theorem~\ref{thm:DCBM}. With a little ambiguity of notations, throughout this subsection, we still use $\Delta$ and $\Delta_a$'s with the clarification that they all relate to the setting of DCBM.

 We start with the derivation for the matrix form of $ \widehat{\Omega}^{\text{DCBM}}$.  To distinguish from $\{e_i\}_{i=1}^n$, the standard basis of $\mathbb{R}^n$, we use $\{\tilde{e}_k\}_{k=1}^K$ to represent the standard basis of $\mathbb{R}^K$. Note that $\widehat{\Pi}{\bf 1}_K = \mathbf{1}_n$.  Based on the proposed algorithm in Section~\ref{subsec:DCBM}, suppose that $\hat \pi_i = \tilde{e}_k$ and $\hat \pi_j = \tilde{e}_t $, we see that 
\begin{align*}
 \widehat{\Omega}^{\text{DCBM}} (i,j)
 % = \widehat \theta_i  \widehat\theta_j\cdot  \hat \pi_i' \widehat P \hat \pi_j 
  & = \frac{e_i' A \mathbf{1}_n \cdot e_j' A \mathbf{1}_n }{\hat \pi_i' \widehat \Pi' A\mathbf 1_n \cdot \hat \pi_j' \widehat \Pi' A\mathbf 1_n} \sqrt{\hat \pi _i' M \hat \pi_i } \, \hat \pi_i'  \big[{\rm diag}( M )\big]^{-{1/2}} M \big[{\rm diag}( M )\big]^{-{1/2}} \hat \pi_j  \sqrt{\hat \pi _j' M \hat \pi_j } & \notag\\
  & = \frac{e_i' A \mathbf{1}_n \cdot e_j' A \mathbf{1}_n }{\hat \pi_i' \widehat \Pi' A\mathbf 1_n \cdot \hat \pi_j' \widehat \Pi' A\mathbf 1_n} \hat \pi_i ' M \hat \pi_j \notag\\
  & = e_i' \, {\rm diag} (A \mathbf{1}_n) \widehat \Pi \, \big[{\rm diag} \big(\widehat \Pi' A\mathbf 1_n\big)\big]^{-1} M\,  \big[ {\rm diag} \big(\widehat \Pi' A\mathbf 1_n\big) \big]^{-1}  \widehat \Pi' \, {\rm diag} (A \mathbf{1}_n) \, e_j
\end{align*}
where we used the identity $\sqrt{\hat \pi _i' M \hat \pi_i } \, \hat \pi_i'  \big[{\rm diag}( M )\big]^{-{1/2}} = M_{kk}^{1/2} \tilde{e}_k'  \big[{\rm diag}( M )\big]^{-{1/2}}  =\tilde{ e}_k' = \hat \pi _i'$. This further gives rise to 
\begin{align} \label{eq:hOmegaDCBM}
\widehat{\Omega}= \widehat{\Omega}^{\text{DCBM}}  & = {\rm diag} (A \mathbf{1}_n) \widehat \Pi \, \big[{\rm diag} \big(\widehat \Pi' A\mathbf 1_n\big) \big]^{-1}  \widehat\Pi' A \widehat \Pi\, \big[  {\rm diag} \big(\widehat \Pi' A\mathbf 1_n\big) \big]^{-1}  \widehat \Pi' \, {\rm diag} (A \mathbf{1}_n)\notag\\
& ={\rm diag} (A \mathbf{1}_n)  \Pi \, \big[ {\rm diag} \big( \Pi' A\mathbf 1_n\big) \big]^{-1} \Pi' A  \Pi\,  \big[ {\rm diag} \big( \Pi' A\mathbf 1_n\big) \big]^{-1}  \Pi' \, {\rm diag} (A \mathbf{1}_n)  \, ,
\end{align}
with probability $1- o(1)$,
since the clustering step in the algorithm achieves exact recovery $\mathbb{P} (\widehat \Pi = \Pi ) = 1- o(1) $. The proof of exact recovery follows from the arguments for DCMM and SBM  in Sections~\ref{susec:pf-MSCORE} and \ref{supp:pfSBM}. We need first have $\Vert\hat{r}_i - r_i\Vert \lesssim (\delta_n \beta_n \Vert \theta\Vert)^{-1} \sqrt{\log (n)}$  from Section~\ref{susec:pf-MSCORE}. Second,  $R = \Pi V$ with  $V= (v_1', \cdots , v_K')'$, we have the inequality $\max_{k\neq \ell} \Vert v_k - v_\ell\Vert \gg  (\delta_n \beta_n \Vert \theta\Vert)^{-1} \sqrt{\log (n)}$ following from Lemma C.4 of \cite{MSCORE}. Based on these two inequalities,
further with the fact that $\pi \in \{e_1,\cdots, e_K\}$, the standard basis of $\mathbb{R}^{K}$, we therefore achieve exact accuracy with high probability by k-means algorithm. To see this, assume $\pi_i = e_{k^*} $, it follows that $\Vert \hat r_i - v_{k^*}\Vert  = \Vert \hat r_i  - r_i \Vert \ll \max_{k\neq k^*} \| v_k - v_{k^*}\|$. As a result, for $1\leq k \neq k^*\leq K$, $\Vert \hat r_i  - v_{k}\Vert \geq     \| v_k - v_{k^*}\|  - \Vert \hat r_i - v_{k^*}\Vert \gg \Vert\hat r_i  - v_{k^*}\Vert$. Therefore, $\hat{\pi}_i = e_{k^*} $.

By direct computations, under current DCBM setting, we also have 
\begin{align} \label{eq:OmegaDCBM}
\Omega  = {\rm diag} (\Omega  \mathbf{1}_n)  \Pi \, \big[ {\rm diag} \big( \Pi' \Omega\mathbf 1_n\big) \big]^{-1}  \Pi' \Omega  \Pi\,  \big[ {\rm diag} \big(\Pi' \Omega\mathbf 1_n\big)\big]^{-1}    \Pi' \, {\rm diag} (\Omega \mathbf{1}_n) \,. 
\end{align}
%
%In the sequel, we validate Lemma~\ref{lem:Delta} leveraging the representations (\ref{eq:OmegaDCBM}) and (\ref{eq:OmegaDCBM}). 
Unlike the DCMM and MMSBM model, where the forms $H'A^aH$ and $AH$ play a crucial role by exploiting the estimates in Lemma~\ref{lem:1.7}, (\ref{eq:hOmegaDCBM}) and (\ref{eq:OmegaDCBM}) involves  certain diagonal matrices. To facilitate the proofs, we require an additional auxiliary lemma, analogous to  Lemma~\ref{lem:1.7}.  

Before presenting the auxiliary lemma, we introduce some notations for convenience. We define a sequence of column vectors $\widetilde{\bf 1}_k \in \mathbb{R}^n$ for $1\leq k \leq K$ such that $\widetilde{\bf 1}_k(i ) = 1$ if $\pi_i =\tilde{ e}_k$ and  $0$ otherwise, where $1\leq i\leq n $. We observe  that $\widetilde{\bf 1}_k = \Pi \tilde{e}_k$. Additionally, we denote $\mathcal{C}_k$ the index set of nodes in the $k$-th community for $1\leq k \leq K$.  Furthermore, we introduce a sequence of diagonal matrices $\mathbb{I}_{n, k}  :=  \sum_{i\in \mathcal{C}_k}e_i e_i'$ for $1\leq k \leq K$. 

The following lemma is essential in our analysis. Since its proof is similar to that of Lemma~\ref{lem:1.7}, we will briefly discuss the proof at the end of this section.
\begin{lemma}\label{lem:1.7DCBM}
Under the conditions of Theorem~\ref{thm:DCBM}, the following estimates hold for any fixed $1\leq k, \ell\leq K$ and $\beta=1, 2$, it holds that
\begin{align}
&|\widetilde{\bf 1}_k' W_1 {\bf 1}_n|  = O_{\mathbb{P}} (\Vert \theta\Vert_1\sqrt{\log (n)}). \label{est:DCBM1}
\end{align}
In addition, we have
\begin{align}
&\Vert \Pi' \Theta \mathbb{I}_{n, k}W_1^{\beta} \mathbb{I}_{n, \ell} \Theta\Pi \Vert =  O_{\mathbb{P}} (\Vert\theta\Vert^2\Vert \theta\Vert_1^{\beta - 1}), \quad ~~~~~\Vert \Pi' \Theta \mathbb{I}_{n, k}{\rm diag}(W_1^2)\mathbb{I}_{n, \ell} \Theta\Pi \Vert =  O_{\mathbb{P}} (\Vert\theta\Vert^2 \Vert \theta\Vert_1); \notag\\
& \Vert  \mathbf{1}_n' W_1 \mathbb{I}_{n, k} \Theta\Pi\Vert = O_{\mathbb{P}} (\Vert\theta\Vert_3^{3/2} \Vert \theta\Vert_1^{1/2}\sqrt{\log(n)}\, ) ,
~~ 
 \Vert  \mathbf{1}_n' W_1 \mathbb{I}_{n, k}{\rm diag}(W_1^2) \mathbb{I}_{n, \ell}\Theta\Pi\Vert=o_{\mathbb{P}} (\Vert\theta\Vert^3 \Vert \theta\Vert_1),\notag\\
 & \Vert  \mathbf{1}_n' W_1 \mathbb{I}_{n, k}W_1^2 \mathbb{I}_{n, \ell}\Theta\Pi\Vert =o_{\mathbb{P}} (\Vert\theta\Vert^3 \Vert \theta\Vert_1), \qquad\Vert  \mathbf{1}_n' W_1 \mathbb{I}_{n, k}W_1 \mathbb{I}_{n, \ell}\Theta\Pi\Vert =O_{\mathbb{P}} (\Vert\theta\Vert^2 \Vert \theta\Vert_1) 
  ; \label{est:DCBM3}
\end{align}
Furthermore, we have
\begin{align}
 & |\mathbf{1}_n' W_1 \mathbb{I}_{n, k}W_1^2 \mathbb{I}_{n, \ell}W_1{\bf 1}_n | = O_{\mathbb{P}} ( \Vert \theta\Vert^2\Vert \theta\Vert_1^2), 
 \quad ~|\mathbf{1}_n' W_1 \mathbb{I}_{n, k}{\rm diag}(W_1^2) \mathbb{I}_{n, \ell}W_1{\bf 1}_n | = O_{\mathbb{P}} ( \Vert \theta\Vert^2 \Vert \theta\Vert_1^2),\notag\\
 & |\mathbf{1}_n' W_1 \mathbb{I}_{n, k}W_1{\bf 1}_n | = O_{\mathbb{P}} ( \Vert \theta\Vert_1^2),
 \quad~~~~~~~~~~~~~~ |\mathbf{1}_n' W_1 \mathbb{I}_{n, k}W_1\mathbb{I}_{n, \ell}W_1{\bf 1}_n |= O_{\mathbb{P}} ( \Vert \theta\Vert_1^2). \label{est:DCBM4}
\end{align}
\end{lemma}

We will also use the following facts: for $1\leq k\leq K$,
\begin{align}\label{2023080401}
\widetilde{\bf 1}_k' \Omega\mathbf 1_n  \asymp \Vert \theta\Vert _1^2,\qquad \big\Vert \big[{\rm diag} \big( \Pi' \Omega\mathbf 1_n\big) \big]^{-1} \big\Vert  \asymp \Vert \theta\Vert_1^{-2}\, . 
\end{align}
To prove \eqref{2023080401}, first we notice that  $\lambda_{\min}(\Pi'\Theta \Pi)/\lambda_{\max}(\Pi'\Theta \Pi) \geq c_3$ and $\tilde{e}_k'  \Pi'\Theta \mathbf 1_n = \widetilde{\bf 1}_k' \Theta {\bf 1}_n =  \widetilde{\bf 1}_k'\Pi'\Theta \Pi  {\bf 1}_k$. Therefore, $\Pi' \Theta \Pi$ is a $K\times K$ diagonal matrix with diagonal entries  of order $\Vert \theta\Vert_1 $ and the components of $ \Pi'\Theta \mathbf 1_n$ are of order $\Vert \theta\Vert_1 $. Hence, \eqref{2023080401} is proved since
\begin{align*}
\widetilde{\bf 1}_k' \Omega\mathbf 1_n  = \tilde{e}_k' (\Pi' \Theta \Pi) P \Pi'\Theta \mathbf 1_n \asymp \Vert \theta\Vert_1  \tilde{e}_k' P  \Pi'\Theta \mathbf 1_n \asymp \Vert \theta\Vert_1^2 \tilde{e}_k' P\mathbf{1}_K  \asymp \Vert \theta\Vert_1^2.
\end{align*}

In the sequel, we will validate Lemma~\ref{lem:Delta} leveraging the representations (\ref{eq:hOmegaDCBM}) and (\ref{eq:OmegaDCBM}). Our analysis will be conducted under the intersection of all the good events where the estimates in Lemma~\ref{lem:1.7DCBM} hold. It can be observed that the probability of this intersected event is $1- o(1)$. To maintain conciseness,  we will avoid explicitly stating ``with high probability'' arguments.
 
Similarly to our analysis on the terms $(H'AH)^{-1} $ and $(H'A^2 H)^{-1}$ in the proofs of Theorems~\ref{thm:main} and \ref{thm:MMSBM}, we have $[{\rm diag} ( \Pi' A\mathbf 1_n) ]^{-1}=\mathcal D_0-\mathcal D_1$, where
\begin{align*}
&\mathcal{D}_0=[{\rm diag} ( \Pi' \Omega\mathbf 1_n) ]^{-1},\qquad\mathcal D_1=[{\rm diag} ( \Pi' \Omega\mathbf 1_n) ]^{-1} (I_K + \mathcal{T}_3) {\rm diag} ( \Pi' W_1\mathbf 1_n) [{\rm diag} ( \Pi' \Omega\mathbf 1_n) ]^{-1},\\
&\mathcal{T}_3=\sum_{j=1}^{\infty} \big({\rm diag} ( \Pi' (-W_1)\mathbf 1_n) [{\rm diag} ( \Pi' \Omega\mathbf 1_n) ]^{-1} \big)^j.
\end{align*}
Applying (\ref{est:DCBM1}), we get $\Vert {\rm diag} ( \Pi' W_1\mathbf 1_n)  \Vert = O_{\mathbb{P}}(\Vert \theta \Vert_1 \sqrt{\log (n)})$, which, together with (\ref{2023080401}), leads to $\Vert {\rm diag} ( \Pi' (-W_1)\mathbf 1_n) [{\rm diag} ( \Pi' \Omega\mathbf 1_n) ]^{-1} \Vert = o_{\mathbb{P}}(1)$, since $\Vert \theta\Vert_1\gg \sqrt{\log (n)}$. Consequently, $\Vert \mathcal{T}_3\Vert = o_{\mathbb{P}}(1) $ and 
\begin{align}\label{2023080305}
\Vert \mathcal{D}_1\Vert \lesssim \Vert \theta\Vert_1^{-3}\sqrt{\log (n)}\,,  \quad 
\Vert [{\rm diag} ( \Pi' A\mathbf 1_n) ]^{-1}  \Vert  \lesssim \Vert \mathcal{D}_0\Vert + \Vert \mathcal{D}_0\Vert ^2 \lesssim \Vert \theta\Vert_1^{-2}\,. 
\end{align}
We now decompose 
$
\Delta = \Delta_1 + \Delta_2 + \Delta_3
$
where 
\begin{align*}
&\Delta_1 : = {\rm diag} (A \mathbf{1}_n)  \Pi \, [{\rm diag} ( \Pi' A\mathbf 1_n) ]^{-1} \Pi' W_1   \Pi\,  [ {\rm diag} ( \Pi' A\mathbf 1_n)  ]^{-1}\Pi' \, {\rm diag} (A \mathbf{1}_n) \notag\\
&\Delta_2 : =  {\rm diag} (W_1 \mathbf{1}_n)  \Pi \, [ {\rm diag} ( \Pi' A\mathbf 1_n)  ]^{-1} \Pi' \Omega   \Pi\, [ {\rm diag} ( \Pi' A\mathbf 1_n) ]^{-1} \Pi' \, {\rm diag} (A \mathbf{1}_n)  \notag\\
& \qquad +  {\rm diag} (\Omega  \mathbf{1}_n)  \Pi \,[ {\rm diag} ( \Pi' A\mathbf 1_n)  ]^{-1} \Pi' \Omega   \Pi\, [ {\rm diag} ( \Pi' A\mathbf 1_n)  ]^{-1} \Pi' \, {\rm diag} (W_1 \mathbf{1}_n) 
 \notag\\
&\Delta_3: =  {\rm diag} (\Omega  \mathbf{1}_n)  \Pi \, \mathcal{D}_1\Pi' \Omega   \Pi\,   [ {\rm diag} ( \Pi' A\mathbf 1_n) ]^{-1} \Pi' \, {\rm diag} (\Omega \mathbf{1}_n)  \notag\\
& \qquad +  {\rm diag} (\Omega  \mathbf{1}_n)  \Pi \,\mathcal{D}_0 \Pi' \Omega   \Pi\,  \mathcal{D}_1\Pi' \, {\rm diag} (\Omega  \mathbf{1}_n) 
\end{align*}
It suffices to claim that (a)-(d) in Lemma~\ref{lem:Delta}  hold for $\Delta_1, \Delta_2, \Delta_3$. In the same manner to Section~\ref{supp:pfMMSBM}, we will study the generic form $\Vert \Gamma \Delta_a \Vert $ for $\Gamma = I_n, W_1, W_1^2, {\rm diag} (W_1^2)$ and $a=1,2, 3$. For simplicity, we denote 
\begin{align*}
&\mathbb{J}_1(\Gamma):= \Vert \Pi' {\rm diag}(\Omega \mathbf{1}_n) \Gamma {\rm diag}(\Omega  \mathbf{1}_n) \Pi\Vert,   \notag\\
& \mathbb{J}_2(\Gamma): =\Vert \Pi' {\rm diag}(W_1 \mathbf{1}_n) \Gamma {\rm diag}(\Omega  \mathbf{1}_n) \Pi\Vert , \notag\\
&\mathbb{J}_3(\Gamma):= \Vert \Pi' {\rm diag}(W_1 \mathbf{1}_n) \Gamma {\rm diag}(W_1 \mathbf{1}_n) \Pi\Vert. 
\end{align*}
First, by applying (\ref{2023080305}) and the first estimate in Lemma~\ref{lem:1.7}, we deduce that
\begin{align*}
\Vert \Gamma \Delta_1 \Vert  &\leq \Vert \Pi' {\rm diag}(A \mathbf{1}_n) \Gamma {\rm diag}(A \mathbf{1}_n) \Pi\Vert \times \Vert [{\rm diag} ( \Pi' A\mathbf 1_n) ]^{-1} \Vert ^2\times \Vert \Pi' W_1   \Pi\Vert \notag\\
&\lesssim \Vert \theta\Vert_1^{-3} \sqrt{\log (n)}\times \big(\mathbb{J}_1( \Gamma ) + \mathbb{J}_2( \Gamma ) + \mathbb{J}_3( \Gamma )\big) , \notag\\
\Vert \Gamma \Delta_2 \Vert  &\leq \big(\Vert \Pi' {\rm diag}(A \mathbf{1}_n) \Gamma {\rm diag}(W_1\mathbf{1}_n) \Pi\Vert + \Vert \Pi' {\rm diag}(\Omega\mathbf{1}_n) \Gamma {\rm diag}(W_1\mathbf{1}_n) \Pi\Vert\big)  \\
&\quad\times\Vert [{\rm diag} ( \Pi' A\mathbf 1_n) ]^{-1} \Vert ^2 \times\Vert \Pi' \Omega   \Pi\Vert\lesssim \Vert \theta\Vert_1^{-2} \times[\mathbb{J}_2(\Gamma )+ \mathbb{J}_3( \Gamma ) ],   \notag\\
\Vert \Gamma \Delta_3 \Vert  &\leq \Vert \Pi' {\rm diag}(\Omega \mathbf{1}_n) \Gamma {\rm diag}(\Omega\mathbf{1}_n) \Pi\Vert \times \Vert [{\rm diag} ( \Pi' \Omega\mathbf 1_n) ]^{-1} \Vert \times\Vert \mathcal{D}_1\Vert  \times\Vert \Pi' \Omega   \Pi\,\Vert \notag\\
&\ll \Vert \theta\Vert_1^{-3} \sqrt{\log (n)} \times \mathbb{J}_1( \Gamma )  .
\end{align*}
We therefore deduce from the above equation that
\begin{align}\label{form:GammaDelta}
\Vert \Gamma \Delta \Vert \lesssim \Vert \theta\Vert_1^{-2} [\mathbb{J}_2(\Gamma )+ \mathbb{J}_3( \Gamma ) ] + \Vert \theta\Vert_1^{-3} \sqrt{\log (n)}  \, \cdot  \mathbb{J}_1( \Gamma ) \, . 
\end{align}

We proceed to investigate $\mathbb{J}_1(\Gamma) , \mathbb{J}_2(\Gamma), \mathbb{J}_3(\Gamma)$, 
for $\Gamma = I_n, W_1, W_1^2, {\rm diag} (W_1^2)$\,. Since the matrices $\mathbb{J}_1(\Gamma) , \mathbb{J}_2(\Gamma), \mathbb{J}_3(\Gamma)$ are all $K$-by-$K$, by observing the fact that $\Vert \mathbf{1}_n'  \Theta \Pi P \Vert \asymp  \Vert \theta \Vert_1 $, we obtain
\begin{align*}
\mathbb{J}_1(\Gamma) &\lesssim   \max_{1\leq k,\ell \leq K} |\widetilde{\bf 1}_k ' {\rm diag}(\Omega \mathbf{1}_n)  \Gamma {\rm diag}(\Omega {\bf 1}_n ) \widetilde{\bf 1}_{\ell} |=\max_{1\leq k,\ell \leq K}|\mathbf{1}_n' \Omega \mathbb{I}_{n, k}\Gamma \mathbb{I}_{n, \ell}\Omega {\bf 1}_n |\\
& \lesssim \Vert \theta\Vert_1^2 \cdot  \max_{1\leq k,\ell \leq K}\Vert \Pi' \Theta\mathbb{I}_{n, k}\Gamma \mathbb{I}_{n, \ell}\Theta\Pi \Vert.
\end{align*}
%Then we can employ Lemma~\ref{lem:1.7DCBM}. 
Following similar arguments, we obtain that
\begin{align*}
&\mathbb{J}_2(\Gamma)\lesssim \Vert \theta\Vert_1\cdot \max_{1\leq k,\ell \leq K}  \Vert  \mathbf{1}_n' W_1 \mathbb{I}_{n, k}\Gamma \mathbb{I}_{n, \ell}\Theta\Pi\Vert, \qquad
 \mathbb{J}_3(\Gamma)\lesssim \max_{1\leq k,\ell \leq K} |\mathbf{1}_n' W_1 \mathbb{I}_{n, k}\Gamma \mathbb{I}_{n, \ell}W_1{\bf 1}_n |\, . 
\end{align*}
%\begin{align*}
%&\mathbb{J}_1(\Gamma)\asymp \max_{1\leq k,\ell \leq K} \big|\mathbf{1}_n' \Omega \mathbb{I}_{n, k}\Gamma \mathbb{I}_{n, \ell}\Omega {\bf 1}_n \big|, 
%\quad
% \mathbb{J}_2(\Gamma)\asymp \max_{1\leq k,\ell \leq K} \big|\mathbf{1}_n' W_1 \mathbb{I}_{n, k}\Gamma \mathbb{I}_{n, \ell}\Omega {\bf 1}_n \big|, 
%\notag\\
% & \mathbb{J}_3(\Gamma)\asymp \max_{1\leq k,\ell \leq K} \big|\mathbf{1}_n' W_1 \mathbb{I}_{n, k}\Gamma \mathbb{I}_{n, \ell}W_1{\bf 1}_n \big|\, . 
%\end{align*}
Next, we employ Lemma~\ref{lem:1.7DCBM}. In particular, when $\Gamma = I_n$, 
\begin{align*}
\max_{1\leq k,\ell \leq K}  \big\Vert \Pi' \Theta \mathbb{I}_{n, k} \mathbb{I}_{n, \ell} \Theta\Pi \big\Vert  = \max_{1\leq k \leq K}  \big\Vert \Pi' \Theta \mathbb{I}_{n, k} \Theta\Pi \big\Vert = \max_{1\leq k \leq K} \sum_{i\in \mathcal{C}_k} \theta_i^2 \asymp \Vert \theta\Vert^2 
\end{align*}
It is straightforward to conclude from (\ref{est:DCBM3}) and  (\ref{est:DCBM4}) that
\begin{align*}
\mathbb{J}_1(\Gamma) &=  \left\{
\begin{array}{ll}
O_{\mathbb{P}} (\Vert \theta\Vert_1^2   \Vert \theta\Vert^2),  &\qquad\Gamma = I_n, W_1;\\
O_{\mathbb{P}} ( \Vert \theta\Vert_1^3   \Vert \theta\Vert^2), &\qquad\Gamma =W_1^2, {\rm diag} (W_1^2). 
\end{array}\right.
\\
\mathbb{J}_2(\Gamma)& = \left\{
\begin{array}{ll}
O_{\mathbb{P}} (\Vert\theta\Vert_3^{3/2} \Vert \theta\Vert_1^{1/2}\sqrt{\log(n)}),  &\qquad\Gamma = I_n;\\
O_{\mathbb{P}} ((\Vert \theta\Vert_1^2   \Vert \theta\Vert^2), &\qquad\Gamma =W_1; \\
o_{\mathbb{P}} (\Vert \theta\Vert_1^2   \Vert \theta\Vert^3), &\qquad\Gamma =W_1^2, {\rm diag} (W_1^2) .
\end{array}\right.
\\
\mathbb{J}_3(\Gamma) &=  \left\{
\begin{array}{ll}
O_{\mathbb{P}} (\Vert \theta\Vert_1^2 )  ,  &\qquad\Gamma = I_n, W_1; \\
O_{\mathbb{P}} (\Vert \theta\Vert_1^2   \Vert \theta\Vert^2), &\qquad\Gamma =W_1^2, {\rm diag} (W_1^2) .
\end{array}\right.
\end{align*}
To proceed, we plug in the above estimates back into (\ref{form:GammaDelta}). Simple computations and the condition $\Vert \theta\Vert \gg \sqrt{\log n}\, $ yield (a)-(d) in Lemma~\ref{lem:Delta}. We thus finish the proof.

It then remains to prove Lemma~\ref{lem:1.7DCBM}. The proof is similar to that of Lemma~\ref{lem:1.7} in Section~\ref{proof:lem:1.7}, which for simplicity is briefly stated as follows.
\begin{proof}[Proof of Lemma~\ref{lem:1.7DCBM}]
%We recall the proof of  Lemma~\ref{lem:1.7} . The basic idea is that we examine a fixed entry in the $K \times K$ matrix and express it as a summation with respect to entries in $W_1$. The coefficients are determined by certain products of entries in $H, \Pi, \Theta$. A high probability upper bound of this summation can be obtained by investigating its  mean and variance with respect to the randomness in $W_1$.  To establish the final order of the mean and variance, we conservatively bound the entries in $H$ and $\Pi$ by $1$. 

The proofs of the estimates in Lemma~\ref{lem:1.7DCBM} follow the same principle as in Section~\ref{proof:lem:1.7}, which boils down to bounding the entries in a $K\times K$ matrix which is a functional of $W_1=A-\Omega$.
%. As long as we can express the quantity of interest as a summation of certain products of entries in $W_1$, with coefficients determined by entries from $\theta$ and other factors which are bounded by a constant, the analysis remains essentially equivalent to that in Section~\ref{proof:lem:1.7}. 
For example, $|\widetilde{\bf 1}_j'W_1{\bf 1}_n|$ shares the same upper bound with $\Vert H'W_1H\Vert $; since $\Theta \mathbb{I}_{n, k} = \mathbb{I}_{n, k} \Theta  $,  we have $\Vert \Pi \Theta \mathbb{I}_{n, k}\Gamma \mathbb{I} _{n,\ell} \Theta\Pi \Vert  =\Vert \Pi \mathbb{I}_{n, k} \Theta \Gamma \Theta \mathbb{I} _{n,\ell} \Pi \Vert $ for $\Gamma = W_1, W_1^2, {\rm diag} (W_1^2)$ and $\Vert  \mathbf{1}_n' W_1 \mathbb{I}_{n, k} \Theta\Pi\Vert= \Vert  \mathbf{1}_n' W_1 \Theta\mathbb{I}_{n, k} \Pi\Vert$. Therefore, $\Vert \Pi \mathbb{I}_{n, k} \Theta \Gamma \Theta \mathbb{I} _{n,\ell} \Pi \Vert$ and $\Vert  \mathbf{1}_n' W_1 \mathbb{I}_{n, k} \Theta\Pi\Vert$  shares the same upper bound as $\|\Pi' \Theta \Gamma \Theta H\|$ and $\Vert H' W_1\Theta\Pi\Vert$, respectively. 

The main challenge in handling the remaining quantities compared to those in Lemma~\ref{lem:1.7} lies in the presence of diagonal matrices $\mathbb{I}_{n, k}$ sandwiched between $W_1$ and $\Gamma$, for $\Gamma = W_1, W_1^2, {\rm diag}(W_1)$. However,  this  difference does not affect the analysis significantly thanks to the special form of $\mathbb{I}_{n, k} $. Consider $ {\bf 1}_n' W_1 \mathbb{I}_{n,k} W_1 {\bf 1}_n $ for an example, and the other quantities in (\ref{est:DCBM3}) of Lemma~\ref{lem:1.7DCBM} can be treated similarly.  We write
\begin{align*}
 {\bf 1}_n' W_1 \mathbb{I}_{n,k} W_1 {\bf 1}_n  =  \sum_{i,r, j=1}^n \delta_{k}(r,r) \cdot W_1(i,r)W_1(r,j) 
\end{align*}
where $\delta_{k}(r,r) = 1$ if $r\in \mathcal{C}_k$ and $0$ otherwise. We aim to analyze the order of mean and variance of the sum on the RHS concerning the randomness in $W_1$. To achieve this, we keep the coefficients $\delta_{k}(r,r)$'s untouched until we eliminate all  randomness in the sum by taking expectation. Only in the step of upper bounding the absolute value of mean and variance, we crudely bound $\delta_{k}(r,r) = 1$ for all $r$. Fortunately, the results will be exactly the same as if we had pretended $\mathbb{I}_{n, k} = I_n$ from the beginning. This means that bounding $| {\bf 1}_n' W_1 \mathbb{I}_{n,k} W_1 {\bf 1}_n | $ is essentially the same as bounding $| {\bf 1}_n' W_1^2 {\bf 1}_n |$  which is equivalent  to bounding $\Vert H' W_1^2 H \Vert$.

%To avoid repeatedly presenting cumbersome details, we omit the full computations and refer readers to Section~\ref{proof:lem:1.7} for more technical elaborations.

\end{proof}

\subsection{Proof of Theorem~\ref{thm:SBM} (SBM)} \label{supp:pfSBM} 

The SBM model, which is a special case of MMSBM, automatically satisfies $\Vert G^{-1}\Vert \leq c$ for some constant $c>0$, where $G =n^{-1} \Pi'\Pi$, given the condition $\max_k\{n_k\} \leq C \min_k \{n_k\}$. Let $\beta_n$ denote the order of $|\lambda_k|$. Recall Corollary~\ref{cor:eigenMMSBM}. It gives that  $\Vert \widehat{\Xi}(i) - \Xi(i)O\Vert \lesssim {\sqrt{\log n}}/{(\beta_n n \sqrt{\alpha_n} ) } \ll n^{-1/2} $ under the condition that $\beta_n \sqrt{n\alpha_n}/\sqrt{\log (n)}\to \infty$. 
Since $\Omega=\alpha_n\Pi P\Pi'=\Xi\Lambda\Xi'$, it holds that $\Xi=\Pi B$ for some full-rank matrix $B=[b_1,b_2,\ldots,b_K]'\in\mathbb R^{K,K}$. Then, the rows of $\Xi$ only consist of $K$ distinct rows $b_1',\ldots,b_K'$. In addition, it holds that $BB'=(\Pi'\Pi)^{-1} = {\rm diag}(1/n_1, \cdots, 1/n_K)$, so  it is true that $\min_{1\leq k\neq \ell\leq K}\|b_k-b_\ell\|\gtrsim n^{-1/2}$. By these arguments, 
further with the fact that $\pi \in \{e_1,\cdots, e_K\}$, the standard basis of $\mathbb{R}^{K}$, we therefore achieve exact accuracy with high probability by k-means algorithm. To see this, assume $\pi_i = e_{k^*} $, it follows that $\Vert \widehat{\Xi}(i) - b_{k^*}\Vert  = \Vert \widehat{\Xi}(i) - \Xi(i) \Vert \ll \max_{k\neq k^*} \| b_k - b_{k^*}\|$. As a result, for $1\leq k \neq k^*\leq K$, $\Vert \widehat{\Xi}(i) - b_{k}\Vert \geq     \| b_k - b_{k^*}\|  - \Vert \widehat{\Xi}(i) - b_{k^*}\Vert \gg \Vert \widehat{\Xi}(i) - b_{k^*}\Vert$. Therefore, $\hat{\pi}_i = e_{k^*} $. 
%
%
%suppose that $\hat{\pi}_i = e_{k_0} \neq e_{k^*} = \pi_i  $. It implies that $\Vert \widehat{\Xi}(i) - b_{k_0}\Vert \leq  \Vert \widehat{\Xi}(i) - b_{k^*}\Vert$. However, $\Vert \widehat{\Xi}(i) - b_{k^*}\Vert = \Vert \widehat{\Xi}(i) - \Xi(i) \Vert \ll n^{-1/2}\lesssim \Vert b_{k_0} -  b_{k^*}\Vert   $.  It follows that $\Vert \widehat{\Xi}(i) - b_{k_0}\Vert \geq \Vert b_{k_0} -  b_{k^*}\Vert - \Vert \widehat{\Xi}(i) - b_{k^*}\Vert \gg \Vert \widehat{\Xi}(i) - b_{k^*}\Vert  $, which is contradiction. 
We conclude that  $\mathbb{P} (\widehat{\Pi} = \Pi) = 1- o(1)$. Then, it suffices to restrict the proof on the event $\{\widehat{\Pi} = \Pi\}$.
%
%Since $\Omega=\alpha_n\Pi P\Pi'=\Xi\Lambda\Xi'$, it holds that $\Xi=\Pi B$ for some full-rank matrix $B=[b_1,b_2,\ldots,b_K]'\in\mathbb R^{K,K}$. Then, the rows of $\Xi$ only consist of $K$ distinct rows $b_1',\ldots,b_K'$. In addition, it holds that $B'B=(\Pi'\Pi)^{-1}$, so that by condition \eqref{Main-Cond-1} (setting $\Theta=\alpha_nI_n$), it is true that $\min_{1\leq k\neq \ell\leq K}\|b_k-b_\ell\|\gtrsim n^{-1/2}$. Furthermore, it follows from Lemma D.2 in \cite{MSCORE} that, with probability $1-o(1)$, the $\ell^2$-row-wise bound of the error of $\widehat\Xi$ is given by $|\lambda_{\min}(\alpha_nP)|^{-1} n^{-1}\sqrt{\log n}\ll n^{-1/2}$.
%%there exists $\omega\in{1,-1}$ and $X\in\mathbb R^{K-1,K-1}$ such that
%%\begin{align*}
%%&\|\omega_1\widehat\xi_1-\xi_1\|_\infty\lesssim |\lambda_{\min}(P)|^{-1} n^{-1}\sqrt{\log n}=o(n^{-1/2}),\\
%%&\max_{1\leq i\leq n}\|X'\widehat \Xi_0'e_i-\Xi_0'e_i\|\lesssim 
%%\end{align*}
%%since $|\lambda_{\min}(P)|\gg n^{-1/2}\{\log(n)\}^{1/2}$.
%%\begin{align*}
%%\max_{1\leq k\leq K,1\leq i\leq n}|\widehat\xi_k(i)-\xi_k(i)|\lesssim \{\lambda_{\min}(P)\}^{-1} n^{-1}\sqrt{\log n}\,.
%%\end{align*}
%Therefore, if we define the event $\mathcal E_n=\{\widehat\Pi=\Pi\}$, in view of Theorem~\ref{thm:MSCORE} (setting $\Theta=\sqrt{\alpha_n}I_n$), it holds that ${\mathbb P}(\mathcal E_n)=1-O(n^{-3})$. Therefore, it suffices to restrict the proof on the event $\mathcal E_n$.
It holds that $\widehat\Omega^{\rm SBM}=\Pi(\Pi'\Pi)^{-1}\Pi'A\Pi(\Pi'\Pi)^{-1}\Pi'$. In addition, observe that under SBM, ${\Omega} = \Pi(\Pi'\Pi)^{-1}\Pi\Omega\Pi(\Pi'\Pi)^{-1}\Pi'$. It follows that
\begin{align}\label{deltasbm}
\Delta^{\rm SBM}=\widehat\Omega^{\rm SBM}-\Omega=\Pi(\Pi'\Pi)^{-1}\Pi'W_1\Pi(\Pi'\Pi)^{-1}\Pi'\, . 
\end{align}
In the sequel, we use the notation $\Delta$ to substitute $\Delta^{\rm SBM}$ for simplicity. By the assumption that $\sum_{k=1}^K \lambda_k^3 \geq C^{-1}\lambda_1^3$, it follows that ${\rm tr} (\Omega^3)={\rm tr}\{(n\alpha_nPG )^3\}\asymp  n^3 \alpha_n^3$, since $c_1<\lambda_1(PG)<c_2$ for some constants $c_1,c_2>0$. The estimates in (\ref{2023080201}) also hold for SBM. In particular, 
\begin{align*}
&\|\Pi'W_1\Pi\|=O_{\mathbb P}( n\alpha_n^{1/2} \sqrt{\log(n)}\, ),\qquad \quad \|\Pi'W_1^2\Pi\|=O_{\mathbb P}( n^2\alpha_n ),\\
&\|\Pi'W_1^4\Pi\|=O_{\mathbb P}(n^3\alpha_n^2),\quad~~~~~~~~~~~~~~~~ \|\Pi'\diag(W_1^2)\Pi\|=O_{\mathbb P}(n^2\alpha_n).
\end{align*}
We therefore deduce from the above equation that
\begin{align*}
&\|\Delta\|\lesssim n^{-1}\|\Pi'W_1\Pi\|=O_{\mathbb P}(\sqrt{\alpha_n \log(n)}\, ),\notag\\
&\|W_1\Delta\|
%=\|W_1\Pi\|\times\|\Pi'W_1\Pi\|\times\|\Pi'\Pi\|^{-3/2}
\lesssim n^{-2}\|\Pi'W_1\Pi\|^2=O_{\mathbb P}( \alpha_n\sqrt{\log(n)}\,),\notag\\
&\|W_1^2\Delta\|\lesssim n^{-2}\|\Pi'W_1^2\Pi\|\|\Pi'W_1\Pi\|=O_{\mathbb P}( n\alpha_n^{3/2}\sqrt{\log(n)}\,),\notag\\
&\|{\rm diag} (W_1^2)\Delta\|\lesssim n^{-2}\|\Pi'\diag(W_1^2)\Pi\|\|\Pi'W_1\Pi\|=O_{\mathbb P}( n\alpha_n^{3/2}\sqrt{\log(n)}\,).
\end{align*}
%We therefore apply condition XX to the above equation and obtain
%\begin{align}\label{terms}
%&\Vert \Delta \Vert \ll n^{1/2}, \qquad\Vert W_1\Delta\Vert \lesssim  n,\qquad\Vert W_1^2 \Delta \Vert \ll n^{3/2}, \qquad \Vert {\rm diag} (W_1^2) \Delta \Vert\ll n^{3/2}\, .
%\end{align}
which satisfy (a1)-(d1) in Section~\ref{supp:pfMMSBM} by the condition $\beta_n \sqrt{n\alpha_n}\gg \sqrt{\log (n)}$. This implies 
\begin{align*}
\big| U_{n, 3}(\widehat{\Omega}) - U_{n,3} (\Omega)  \big|& \lesssim \Vert \Delta\Vert^3 +\Vert W_1\Delta\Vert \Vert \Delta \Vert + \Vert W_1^2 \Delta\Vert +\Vert {\rm diag} (\Omega)\Vert   \big(\Vert \Delta\Vert^2 + \Vert W_1\Delta\Vert \big)  \notag\\
&\quad + \Vert {\rm diag} (W_1^2) \Delta\Vert  + \Vert {\rm diag} (\Omega)\Vert^2 \Vert \Delta\Vert \notag\\
& = o_{\mathbb{P}}(n^3\alpha_n^3) \,.
\end{align*}
As a consequence, $T_n(\widehat{\Omega}^{\rm SBM})  = T_n(\Omega) + O_{\mathbb{P}}(\big| U_{n, 3}(\widehat{\Omega}) - U_{n,3} (\Omega)  \big|/\sqrt{n^3\alpha_n^3}\, )$. This eventually gives $T_n(\widehat{\Omega}^{\rm SBM}) \to N(0,1)$.  

\section{Power analysis}\label{sec:power}

This section on power analysis is organized as follows: We provide the proofs of Theorem~\ref{thm:SCC-power-general}, Lemmas~\ref{lem:SNR-higher-rank} - \ref{lem:DCBMvsDCMM}   in Sections~\ref{subsec:thm3.6}-\ref{app:lem:DCBMvsDCMM}, respectively. The theoretical details for Theorem~\ref{thm:power1-mainpaper} of the main paper are presented  in Section~\ref{app:thm:power1-mainpaper} where an auxiliary lemma (Lemma~\ref{lem:power-DCBM}) is also introduced. Finally, in Section~\ref{app:power:2}, we revisit the setting DCBM versus DCMM, as discussed in Lemma~\ref{lem:DCBMvsDCMM}, and extend the case from $K=2$ to the general case of $K\geq 2$.

\subsection{Proof of Theorem~\ref{thm:SCC-power-general}} \label{subsec:thm3.6}

Recall our assumptions that 
\beq \label{thm3.6-proof-conditions}
|\Omega_{ij}-\widetilde{\Omega}_{ij}|\leq C\alpha_nu_iu_j/(n\bar{u}), \qquad  |\tr((\Omega-\widetilde\Omega)^m)|\asymp  C\alpha^m_n \|u\|^{2m}/(n\bar{u})^m. 
\eeq
We first use the above conditions to connect $\alpha_n$ with SNR. By definition, 
\beq \label{thm3.6-proof-order-SNR}
\mathrm{SNR}_{n,m}(\Omega)\asymp \frac{|\tr((\Omega-\widetilde\Omega)^m)|}{\sqrt{\tr(\Omega^m)}} \asymp \frac{\alpha_n^m\|u\|^{2m}/(n\bar{u})^m}{\sqrt{\|u\|^{2m}/(n\bar{u})^m}}\asymp \frac{\alpha_n^{m}\|u\|^m}{(n\bar{u})^{m/2}}. 
\eeq
As a result, $\mathrm{SNR}_{n,m}(\Omega) \to\infty$ if and only if 
\beq \label{thm3.6-proof-order-alpha}
\alpha_n\|u\|/(n\bar{u})^{1/2}\to\infty. 
\eeq

Write $\widetilde \Delta = \Omega - \widetilde \Omega $. We now study $\psi_{n,m}(\widetilde \Omega)$ by decomposing
\begin{align*}
\psi_{n,m}(\widetilde \Omega)
% &= \frac{\sum_{i_1, i_2, \cdots, i_m (dist)} (A - \widetilde{\Omega}) _{i_1i_2}(A - \widetilde{\Omega}) _{i_2i_3} \cdots (A - \widetilde{\Omega}) _{i_mi_1}}{\sqrt{2m C_{n,m}}} \notag\\
& = \frac{\sum_{i_1, i_2, \cdots, i_m (dist)} (W - \widetilde \Delta) _{i_1i_2} (W - \widetilde \Delta) _{i_2i_3} \cdots  (W - \widetilde \Delta) _{i_mi_1}}{\sqrt{2m C_{n,m}}} \notag\\
& = \frac{\sum_{i_1, i_2, \cdots, i_m (dist)} (W - \widetilde \Delta) _{i_1i_2} (W - \widetilde \Delta) _{i_2i_3} \cdots  (W - \widetilde \Delta) _{i_mi_1}}{\sqrt{2m {\rm tr}(\Omega^m) }} \cdot \sqrt{\frac{{\rm tr}(\Omega^m) }{C_{n,m}} }
\notag\\
& = (\mathcal M +  \mathcal R) \cdot \sqrt{\frac{{\rm tr}(\Omega^m) }{C_{n,m}} }
%: (-1)^m  \frac{\sum_{i_1, i_2, \cdots, i_m (dist)}  \widetilde \Delta _{i_1i_2}  \widetilde \Delta _{i_2i_3} \cdots   \widetilde \Delta _{i_mi_1}}{\sqrt{2m C_{n,m}}} + \mathcal R
\end{align*}
where $\mathcal M$ and  $\mathcal R$ represent the alternative mean and variance terms, respective. In particular, 
\begin{align*}
\mathcal M  & = (-1)^m  \frac{\sum_{i_1, i_2, \cdots, i_m (dist)}  \widetilde \Delta _{i_1i_2}  \widetilde \Delta _{i_2i_3} \cdots   \widetilde \Delta _{i_mi_1}}{\sqrt{2m {\rm tr}(\Omega^m)}} \notag\\
& =  (-1)^m  \frac{{\rm tr}(\widetilde \Delta^m ) + ({\rm Bias})}{\sqrt{2m {\rm tr}(\Omega^m)}} 
\end{align*}
The $({\rm Bias})$ term corresponds the sum such that two indices from $\{i_1, i_2, \cdots, i_m\}$ are identical. Therefore, we can bound 
\[
| ({\rm Bias})|\leq C \sum_{k=1}^{m-1}{\rm tr} (\widetilde\Delta^k \circ \widetilde \Delta^{m-k} )
\]
Our assumption says that $|\widetilde{\Delta}_{ij}|\leq C\alpha_nu_iu_j/(n\bar{u})$. It follows that
\begin{align*}
\frac{|(\mathrm{Bias})|}{{\rm tr}(\Omega^m) } & \leq \frac{C\sum_{i_1=i_2,i_3,\ldots,i_{m}}[\alpha_n^mu_{i_1}^4u^2_{i_3}\ldots u^2_{i_m}/(n\bar{u})^{2m}]}{\sqrt{\|u\|^{2m}/(n\bar{u})^m}}\cr
&\leq \frac{C\alpha_n^m\|u\|^{2m-4}\|u\|^4_4/(n\bar{u})^{2m}}{\sqrt{\|u\|^{2m}/(n\bar{u})^m}}\cr
&\leq \frac{C\alpha_n^m\|u\|^{2m-2}u_{\max}^2/(n\bar{u})^{2m}}{\sqrt{\|u\|^{2m}/(n\bar{u})^m}}\cr
&\leq C(u_{\max}^2/\|u\|^2)\cdot \frac{\alpha_n^m\|u\|^m}{(n\bar{u})^{m/2}} \quad =\quad  o(1)\cdot \mathrm{SNR}_{n,m}, 
\end{align*}
where the second line is due to $\|u\|_4^4\leq u_{\max}^2\|u\|^2$, and the last line follows from \eqref{thm3.6-proof-order-SNR} and the assumptions about $u$ in Theorem~\ref{thm:SCC}. Consequently, 
\[
\mathcal M = (-1)^m {\rm SNR}_{n,m}(\Omega) \big(1+ o(1)\big)
\]

To proceed, we analyze the remaining variance term $\mathcal R$. 
Consider the case that $m=3$ first. We have 
\begin{align*}
\big| \mathcal R \big| & \lesssim \frac{\Big| \sum_{i,j,k (dist)} W_{ij} \widetilde \Delta_{jk} \widetilde \Delta_{ki}\Big|   }{\sqrt{6 {\rm tr }(\Omega^3)}} + \frac{\Big| \sum_{i,j,k (dist)} W_{ij} W_{jk} \widetilde \Delta_{ki}\Big|   }{\sqrt{6 {\rm tr }(\Omega^3)}}+ \frac{\Big| \sum_{i,j,k (dist)} W_{ij} W_{jk} W_{ki}\Big|   }{\sqrt{6 {\rm tr }(\Omega^3)}}\notag\\
& =:(| \mathcal R_1| + | \mathcal R_2| +  | \mathcal R_3| ) 
\end{align*}
The means of $\mathcal  R_1$,  $\mathcal  R_2$ and $\mathcal  R_3$ are all $0$ and their variances can be computed as follows.
\begin{align*}
{\rm var} (\mathcal R_1) & = \frac{ \mathbb E\big( \sum_{i,j,k (dist)} W_{ij} \widetilde \Delta^*_{jk} \widetilde \Delta^*_{ki}\big)^2 }{6 {\rm tr }(\Omega^3)} = \frac{\sum_{i\neq j} \Omega_{ij}\big(\sum_{k\neq i,j} \widetilde \Delta^*_{jk} \widetilde \Delta^*_{ki}\big)^2 }{6 {\rm tr }(\Omega^3)}\notag\\
& \lesssim \frac{\alpha_n^4\| u\|^4 \|u\|_3^6 /(n\bar u)^{5}}{6 {\rm tr }(\Omega^3)}
%\frac{\| u\|^4 \big( \sum_i u_i^3\big)^2}{(n\bar u)^5} \Big/ \Big( \frac{\| u\|^2}{n\bar u}\Big)^3 \leq \frac{\| u\|^2}{n\bar u}\cdot  \frac{u_{\max}^2}{n\bar u} = o(1) 
\end{align*}
and 
\begin{align*}
{\rm var} (\mathcal R_2) & = \frac{ \mathbb E\big( \sum_{i,j,k (dist)} W_{ij} W_{jk} \widetilde \Delta^*_{ki}\big)^2 }{6 {\rm tr }(\Omega^3)} = \frac{\sum_{i, j,k (dist)} \Omega_{ij}\Omega_{jk} (\widetilde \Delta^*_{ki})^2  }{6 {\rm tr }(\Omega^3)}\notag\\
& \lesssim \frac{\alpha_n^2\|u \|^2 \|u\|_3^6/(n\bar u)^4}{6 {\rm tr }(\Omega^3)}
\end{align*}
and 
\begin{align*}
{\rm var} (\mathcal R_3) & = \frac{ \mathbb E\big( \sum_{i,j,k (dist)} W_{ij} W_{jk} W_{ki}\big)^2 }{6 {\rm tr }(\Omega^3)} = \frac{\sum_{i, j,k (dist)} \Omega_{ij}\Omega_{jk}\Omega_{ki}  }{6 {\rm tr }(\Omega^3)}\notag\\
& \lesssim \frac{\|u \|^6 /(n\bar u)^3}{6 {\rm tr }(\Omega^3)}
\end{align*}
It follows from $\alpha_n^2\| u\|^2/n\bar u \gg 1$ (see (\ref{thm3.6-proof-order-alpha})) that ${\rm var} (\mathcal R_2)= o({\rm var} (\mathcal R_1))$. Using ${\rm tr }(\Omega^3)\geq \|u\|^6/(n\bar u)^3$, it yields that 
\[
{\rm var}(\mathcal R) \lesssim {\rm var}(\mathcal R_1) + {\rm var} (\mathcal R_3)\lesssim 1+ \frac{\alpha_n^4\|u\|_3^6}{(n\bar u)^2 \|u\|^2}
\]
We now  compare the orders of mean and standard deviation. It gives that 
\begin{align*}
\frac{{\rm SNR}_{n,3}(\Omega) }{\sqrt{{\rm var}(\mathcal R)  }}\asymp \frac{\alpha_n^3 \|u\|^3}{(n\bar u)^{3/2}} + \frac{\alpha_n \|u\|^4}{(n\bar u)^{1/2}\|u\|_3^3} 
%\geq \frac{\alpha_n^3 \|u\|^3}{(n\bar u)^{3/2}} + \frac{\alpha_n \|u\|^2}{(n\bar u)^{1/2}u_{\max}} 
\to \infty
\end{align*}
Here to obtain the RHS, we used (\ref{thm3.6-proof-order-alpha}) and the derivation 
\[
\frac{\|u\|^3 }{\|u\|_3^3} \geq  \frac{\|u\|^3 }{\|u\|^2 u_{\max}}\gg \frac{\sqrt{n\bar u} }{u_{\max}}\gg 1
\]
following from the conditions $\|u \|^2/ n\bar u\to \infty$ and $u_{\max}^2/n\bar u = o(1)$.
In addition, by Lemma~\ref{lem:SCC} and (2) of (\ref{SCC:2goal}), 
\[
\sqrt{\frac{{\rm tr}(\Omega^m) }{C_{n,m}} } = 1+o_p(1)
\]
This further implies 
\begin{align*}
\psi_{n,3}(\widetilde \Omega) =\big[ {\rm SNR}_{n,3}(\Omega) ( - 1+ o(1)) +  \mathcal R(\Omega )\big] (1+o_p(1)) 
\end{align*}
where the ${\rm var }(  \mathcal R/ {\rm SNR}_{n,3}(\Omega)) \to 0$. 
Therefore, we conclude if ${\rm SNR}_{n,3}(\Omega)\to \infty$, $|\psi_{n,3}(\widetilde \Omega)| \to \infty$.

In the sequel, we consider the case $m=4$. We first bound
\begin{align*}
\big| \mathcal R \big| & \lesssim \frac{\Big| \sum_{i,j,k,\ell (dist)} W_{ij} \widetilde \Delta^*_{jk} \widetilde \Delta^*_{k\ell} \widetilde \Delta^*_{\ell i}\Big|   }{\sqrt{8 {\rm tr }(\Omega^4)}} + \frac{\Big| \sum_{i,j,k, \ell (dist)} W_{ij} W_{jk} \widetilde \Delta^*_{k\ell } \widetilde \Delta^*_{\ell i}\Big|   }{\sqrt{8 {\rm tr }(\Omega^4)}} +  \frac{\Big| \sum_{i,j,k,\ell (dist)} W_{ij}  \widetilde \Delta^*_{jk} W_{k\ell } \widetilde \Delta^*_{\ell i}\Big|   }{\sqrt{8 {\rm tr }(\Omega^4)}} \notag\\
& \qquad 
+  \frac{\Big| \sum_{i,j,k,\ell (dist)} W_{ij}  W_{jk} W_{k\ell } \widetilde \Delta^*_{\ell i}\Big|   }{\sqrt{8 {\rm tr }(\Omega^4)}}
+  \frac{\Big| \sum_{i,j,k,\ell (dist)} W_{ij}  W_{jk} W_{k\ell } W_{\ell i}\Big|   }{\sqrt{8 {\rm tr }(\Omega^4)}}
 \notag\\
& =: \sum_{i=1}^{5} | R_i|
\end{align*}
with high probability. Similarly to the case of $m=3$. The means of $R_i$'s are zero. We compute their variance below. 
\begin{align*}
{\rm var} (R_1) &\lesssim \frac{\sum_{i\neq j} \Omega_{ij}\big(\sum_{k, \ell (dist) \neq i,j} \widetilde \Delta^*_{jk} \widetilde \Delta^*_{k\ell}\widetilde \Delta^*_{\ell i}\big)^2 }{8 {\rm tr }(\Omega^4)}\lesssim \frac{\alpha_n^6\| u\|^{8} \|u\|_3^6/(n\bar u)^{7}}{8 {\rm tr }(\Omega^4)} \notag\\
{\rm var} (R_2)& \lesssim \frac{\sum_{i, j, k (dist)} \Omega_{ij} \Omega_{jk} \big(\sum_{\ell \neq i,j,k} \widetilde \Delta^*_{k\ell}\widetilde \Delta^*_{\ell i}\big)^2 }{8 {\rm tr }(\Omega^4)}\lesssim \frac{\alpha_n^4\| u\|^{6} \|u\|_3^6/(n\bar u)^{6}}{8 {\rm tr }(\Omega^4)} = o({\rm var} (R_1) )\notag\\
{\rm var} (R_3)& \lesssim \frac{\sum_{i, j, k , \ell(dist)} \Omega_{ij} \Omega_{k\ell} \big[\big(\widetilde \Delta^*_{jk}\widetilde \Delta^*_{\ell i}\big)^2 + \widetilde \Delta^*_{jk}\widetilde \Delta^*_{\ell i}\widetilde \Delta^*_{j\ell}\widetilde \Delta^*_{k i}\big] }{8 {\rm tr }(\Omega^4)}\lesssim \frac{\alpha_n^4\|u\|_3^{12}/(n\bar u)^{6}}{8 {\rm tr }(\Omega^4)} = o(1)\notag\\
{\rm var} (R_4)& \lesssim \frac{\sum_{i, j, k, \ell (dist)} \Omega_{ij} \Omega_{jk}\Omega_{k\ell} \big( \widetilde \Delta^*_{\ell i}\big)^2 }{8 {\rm tr }(\Omega^4)}\lesssim \frac{\alpha_n^2\| u\|^{4} \|u\|_3^6/(n\bar u)^{5}}{8 {\rm tr }(\Omega^4)} = o(1) \notag\\
{\rm var} (R_5)& \lesssim \frac{\sum_{i, j, k,\ell (dist)} \Omega_{ij} \Omega_{jk}  \Omega_{k\ell}  \Omega_{\ell i} }{8 {\rm tr }(\Omega^4)}\lesssim \frac{\| u\|^{8} / (n\bar u)^{4}}{8 {\rm tr } (\Omega^4)} = O(1) 
\end{align*}
Here we used  (\ref{thm3.6-proof-order-alpha}) and conditions $\|u\|^2/n\bar u \to \infty $and $u_{\max}^2/n\bar u \to 0$.
Consequently, 
\[
{\rm var }(\mathcal R) \lesssim {\rm var} (R_1) + {\rm var} (R_5) \lesssim 1+ \frac{\alpha_n^6\|u\|_3^6}{(n\bar u)^3}
\]
In the same way, we can derive 
\begin{align*}
\frac{{\rm SNR}_{n,4}(\Omega) }{\sqrt{{\rm var}(\mathcal R)  }}\asymp \frac{\alpha_n^4 \|u\|^4}{(n\bar u)^2} + \frac{\alpha_n \|u\|^4}{(n\bar u)^{1/2}\|u\|_3^3} 
%\geq \frac{\alpha_n^3 \|u\|^3}{(n\bar u)^{3/2}} + \frac{\alpha_n \|u\|^2}{(n\bar u)^{1/2}u_{\max}} 
\to \infty
\end{align*}
Therefore, 
\begin{align*}
\psi_{n,4}(\widetilde \Omega) = \big[{\rm SNR}_{n,4}(\Omega) ( 1 + o(1))+  \mathcal R/ {\rm SNR}_{n,4}(\Omega) \big] (1+o_p(1)) 
\end{align*}
and we conclude similarly that $\psi_{n,4}(\widetilde \Omega) \to \infty$ if ${\rm SNR}_{n,4}(\Omega) \to \infty$. 
%For general $m\geq 5$, we can analyze it in the same manner and draw the conclusion. We hence omit the repeated details and conclude the proof. 

%and using Lemma~\ref{lem:diffU}, 
%\begin{align*}
%\psi_{n,m}(\widetilde \Omega)  -  \psi_{n,m}( \Omega) & =-{\rm tr}(\Delta^3) + 3{\rm tr}(W_1\Delta^2)-3{\rm tr}(W_1^2\Delta)+6\tr(W_1\circ W_1\Delta)\notag\\
%& \qquad -3\tr(W_1\circ\Delta^2)+3\tr(\Delta\circ W_1^2)-6\tr(\Delta\circ W_1\Delta)+3\tr(\Delta\circ\Delta^2)+6\tr(W_1\circ\Delta\circ\Delta)-6\tr(W_1\circ W_1\circ\Delta)-2\tr(\Delta\circ\Delta\circ\Delta)
%\end{align*}
\subsection{Proof of Lemma~\ref{lem:SNR-higher-rank}} \label{subsec:lem3.3}
For any symmetric $n\times n$ matrix $M$, let $\lambda^*_j(M)$ denote the $j$th largest eigenvalue of $M$, for $1\leq j\leq n$. The meaning of $\lambda^*_j(\cdot)$ is different from $\lambda_j(\cdot)$ in our main paper, where the latter represents the $j$th largest eigenvalue {\it in magnitude}. The following lemma is the well-known Weyl's inequality (e.g., see \cite[Theorem 4.3.1]{HornJohnson}):
%%%%%%%%%%%%%%%%%%%%%%%%%%%%%%%%%
\begin{lemma}[Weyl's inequality]\label{lem:Weyl's}
Let $M, N\in\mathbb{R}^{n\times n}$ be symmetric matrices. For $1\leq i\leq n$, the following statements hold: 
\[
\lambda_{i}(M+N)\leq \lambda_{i-j}(M)+\lambda_{j+1}(N), \qquad 0\leq j\leq i-1,
\]
and
\[
\lambda_{i}(M+N)\geq \lambda_{i+j-1}(M)+\lambda_{n-j+1}(N), \qquad 1\leq j\leq n+1-i. 
\]
\end{lemma}

For every $1\leq s\leq n-K$, we apply the first inequality of Lemma~\ref{lem:Weyl's} to the matrices $M=\widetilde{\Omega}$ and $N=\Omega-\widetilde{\Omega}$, with $i=s+K$ and $j=s-1$. It follows that
\[
\lambda_{K+s}(\Omega)\leq \lambda_{K+1}(\widetilde{\Omega}) + \lambda_{s}(\Omega-\widetilde{\Omega}). 
\]
Since $\widetilde{\Omega}$ has a rank at most $K$, we have $\lambda_{K+1}(\widetilde{\Omega})=0$. It follows that
\beq \label{lem-highrank-1}
\lambda_{s}(\Omega-\widetilde{\Omega})\geq \lambda_{K+s}(\Omega), \qquad\mbox{for all}\quad 1\leq s\leq n-K. 
\eeq
By our assumption, $\Omega$ is positive semi-definite. \eqref{lem-highrank-1} implies that the first $(n-K)$ eigenvalues of $\Omega-\widetilde{\Omega}$ are nonnegative (i.e., $\Omega-\widetilde{\Omega}$ has at most $K$ negative eigenvalues).  Additionally, by our assumption, for a constant $c'>0$, 
\[
\biggl|\sum_{s: \lambda_s(\Omega-\widetilde{\Omega})<0}\lambda_s^3(\Omega-\widetilde{\Omega})\biggr| \leq (1-c')\cdot |\tr((\Omega-\widetilde{\Omega})^3)|. 
\]
It follows that
\beq  \label{lem-highrank-2}
\tr((\Omega-\widetilde{\Omega})^3)\geq c'\sum_{s: \lambda_s(\Omega-\widetilde{\Omega})\geq 0}\lambda_s^3(\Omega-\widetilde{\Omega}) =\sum_{s=1}^{n-K}\lambda_s^3(\Omega-\widetilde{\Omega}). 
\eeq
We combine \eqref{lem-highrank-1}-\eqref{lem-highrank-2} to obtain: 
\begin{align}
\mathrm{SNR}_{n,3}&=\frac{\tr[(\Omega-\widetilde{\Omega})^3]}{\sqrt{6\tr(\Omega^3)}} \geq \frac{\sum_{s=1}^{n-K}\lambda_s^3(\Omega-\widetilde{\Omega})}{\sqrt{6\sum_{j=1}^{K_0}\lambda^3_j(\Omega)}}\geq \frac{\sum_{s=1}^{K_0-K}\lambda_{K+s}^3(\Omega)}{\sqrt{6\sum_{j=1}^{K_0}\lambda^3_j(\Omega)}}\label{lem-highrank-3}\\
&\geq \frac{\lambda^3_{K+1}(\Omega)}{\sqrt{6K_0\lambda^3_1(\Omega)}} \geq \frac{1}{\sqrt{6}}\biggl(\frac{\lambda_{K+1}(\Omega)}{\sqrt{\lambda_1(\Omega)}}\biggr)^3. \label{lem-highrank-4}
\end{align}
The first claim follows from \eqref{lem-highrank-3}, and the second claim is by \eqref{lem-highrank-4} and the assumption that $K_0$ is finite.\qed

\subsection{Proof of Lemma~\ref{lem:SBMvsDCBM} }\label{app:lem:SBMvsDCBM}
We focus on the setting that the true model is1-community DCMM with $\Omega= \theta\theta'$ but the assumed model is a 1-community SBM.
When $K=1$, by definition, $\widetilde{\Omega}=\tilde{\rho}{\bf 1}_n{\bf 1}_n'$, where $\tilde{\rho}=n^{-2}{\bf 1}_n'\Omega{\bf 1}_n=n^{-2}\|\theta\|_1^2=\bar{\theta}^2$. It follows that
\[
\Omega - \widetilde{\Omega} = \theta\theta' - \bar{\theta}^2{\bf 1}_n{\bf 1}_n'. 
\]
By direct calculations, 
\begin{align*}
(\Omega - \widetilde{\Omega})^2 &= (\theta\theta' - \bar{\theta}^2{\bf 1}_n{\bf 1}_n')(\theta\theta' - \bar{\theta}^2{\bf 1}_n{\bf 1}_n')\cr
&= \|\theta\|^2\theta\theta' - n\bar{\theta}^3 ({\bf 1}_n\theta' +\theta {\bf 1}_n') + n\bar{\theta}^4 {\bf 1}_n{\bf 1}_n', 
\end{align*}
and 
\begin{align*}
(\Omega - &\widetilde{\Omega})^3 =  (\Omega - \widetilde{\Omega})^2(\theta\theta' - \bar{\theta}^2{\bf 1}_n{\bf 1}_n')\cr
&=  \|\theta\|^4\theta\theta' - n\bar{\theta}^3 \bigl(\|\theta\|^2{\bf 1}_n\theta' + n\bar{\theta} \theta\theta'\bigr) + n^2\bar{\theta}^5 {\bf 1}_n\theta'\cr
&\qquad - n\bar{\theta}^3 \|\theta\|^2\theta{\bf 1}_n' + n\bar{\theta}^5 \bigl(n\bar{\theta}{\bf 1}_n{\bf 1}_n' + n \theta {\bf 1}_n'\bigr) - n^2\bar{\theta}^6 {\bf 1}_n{\bf 1}_n'\cr
&= (\|\theta\|^4-n^2\bar{\theta}^4)\cdot \theta\theta' - n\bar{\theta}^3(\|\theta\|^2-n\bar{\theta}^2)\cdot ({\bf 1}_n\theta' + \theta {\bf 1}_n)'. 
\end{align*}
It follows that 
\begin{align*}
\mathrm{trace}([\Omega-\widetilde{\Omega}^3]) &= (\|\theta\|^4-n^2\bar{\theta}^4)\|\theta\|^2 - 2n\bar{\theta}^3(\|\theta\|^2-n\bar{\theta}^2)n\bar{\theta}\cr
&=(\|\theta\|^2-n\bar{\theta}^2)\bigl[(\|\theta\|^2+n\bar{\theta}^2)\|\theta\|^2 - 2n^2\bar{\theta}^4\bigr]\cr
&= (\|\theta\|^2-n\bar{\theta}^2)\bigl[ (\|\theta\|^2-n\bar{\theta}^2)\|\theta\|^2 +2n\bar{\theta}^2(\|\theta\|^2-n\bar{\theta}^2)\bigr]\cr
&= (\|\theta\|^2-n\bar{\theta}^2)^2(\|\theta\|^2+2n\bar{\theta}^2). 
\end{align*}
Therefore, the SNR is equal to 
\[
\frac{(\|\theta\|^2-n\bar{\theta}^2)^2(\|\theta\|^2+2n\bar{\theta}^2)}{\sqrt{6\|\theta\|^6}} \geq  \frac{\|\theta\|^3}{\sqrt{6}}v^2, \qquad v:= \frac{\|\theta\|^2-n\bar{\theta}^2}{\|\theta\|^2}. 
\]
The proof is complete.

\subsection{Proof of Lemma~\ref{lem:DCBMvsDCMM}}\label{app:lem:DCBMvsDCMM}
For simplicity, we write $b\equiv b_n$ and $a\equiv a_n$. We also write
\[
P=\begin{bmatrix}1 & b\\b& 1\end{bmatrix}, \qquad \pi_i=
\begin{cases}
(1,0)', & i\in L_1, \\
(1-a, a)', & i\in L_2, \\
(0,1)', & i\in L_3,\\
(a, 1-a), & i\in L_4,
\end{cases}\qquad h_i = 
\begin{cases}
(1,0)', & i\in L_1\cup L_2, \\
(0,1)', & i\in L_3\cup L_4. 
\end{cases}
\]
where $L_1, L_2, L_3, L_4$ denote four groups and $h_i$'s are the memberships by mistakenly assuming DCBM with $K=2$ communities. 
The trick here is that we want to re-write $\Omega$ as a DCBM with $K=4$. By computing $\pi_i'P\pi_j$ for $4\times 4$ different cases, we obtain a matrix
{\small
\beq \label{newP}
P^* = 
\begin{bmatrix}
1 & (1-a)+ab& b & a+(1-a)b\\
(1-a)+ab&[(1-a)^2+a^2]+2a(1-a)b &a+(1-a)b& 2a(1-a)+[a^2+(1-a)^2]b\\
b&a+(1-a)b&1& (1-a)+ab\\
a+(1-a)b&2a(1-a)+[a^2+(1-a)^2]b&(1-a)+ab& [(1-a)^2+a^2]+2a(1-a)b
\end{bmatrix}. 
\eeq
}
Define $\Pi^*$ by 
\[
\pi^*_i=
\begin{cases}
(1,0,0,0)', & i\in L_1, \\
(0,1,0,0)', & i\in L_2, \\
(0,0,1,0)', & i\in L_3,\\
(0,0,0,1)', & i\in L_4,
\end{cases}
\]
Then, we can re-write
\[
\Omega=\Theta\Pi^*P^*(\Pi^*)'\Theta. 
\]
For notation simplicity, we remove superscript $*$ in the following derivations.  

By \eqref{tOmega-form-2}-\eqref{tP}, $\widetilde{\Omega} = \Theta \Pi\widetilde{P}\Pi'\Theta$, 
where
\[
\widetilde{P}:= \diag(Pg)Q[\diag(Q'GPg)]^{-1}Q'GPGQ[\diag(Q'GPg)]^{-1}Q'\diag(Pg).
\]
Here with a little abuse of notation, we write $g: = \Vert \theta\Vert^{-1} \Pi'\Theta {\bf 1}_n$. 
However, $Q$ is different now:
\[
Q' = \begin{bmatrix} 1 & 1 \\ & & 1 & 1\end{bmatrix}\;\;\in\;\;\mathbb{R}^{2\times 4}. 
\]
Let $g$, $x$ and $y$ be the same as before. It is seen that
\[
\diag(Q'GPg) = \diag\bigl( y_1+y_2,\; y_3+y_4\bigr). 
\]
In addition, for any matrix $\Delta=\diag(\delta_1,\delta_2)$, 
\[
Q\Delta Q'  =
\begin{bmatrix}
\delta_1{\bf 1}_2{\bf 1}_2'\\
& \delta_2{\bf 1}_2{\bf 1}_2' \end{bmatrix}
\quad\in\quad\mathbb{R}^{4\times 4}. 
\] 
It follows that
\[
Q[\diag(Q'GPg)]^{-1}Q' = 
\begin{bmatrix}
\frac{1}{y_1+y_2}{\bf 1}_2{\bf 1}_2' \\
&  \frac{1}{y_3+y_4}{\bf 1}_2{\bf 1}_2'\end{bmatrix}. 
\]
Write $g=(\bar{g}', \tilde{g})'$, with $\bar{g}\in\mathbb{R}^2$ containing the first two coordinates and $\tilde{g}\in\mathbb{R}^2$ containing the first two coordinates. Define $\bar{x}$ and $\tilde{x}$ similarly. Also, note that $y=x\circ g$. It follows that 
\[
 \diag(Pg)Q[\diag(Q'GPg)]^{-1}Q' G
=
\begin{bmatrix}
\frac{1}{\bar{x}'\bar{g}}\bar{x}\bar{g}' & \\
&  \frac{1}{\tilde{x}'\tilde{g}} \tilde{x}\tilde{g}'\end{bmatrix}. 
\]
Let the block devision of $P$ be the same as in \eqref{block}. We have
\begin{align*}
\widetilde{P} & = \begin{bmatrix}
\frac{1}{\bar{x}'\bar{g}}\bar{x}\bar{g}' & \\
&  \frac{1}{\tilde{x}'\tilde{g}} \tilde{x}\tilde{g}'\end{bmatrix}
\begin{bmatrix}
P_0 & Z\\
Z' &  P_1 \end{bmatrix}\begin{bmatrix}
\frac{1}{\bar{x}'\bar{g}}\bar{g}\bar{x}' & \\
&  \frac{1}{\tilde{x}'\tilde{g}} \tilde{g}\tilde{x}'\end{bmatrix}=\begin{bmatrix}
\frac{\bar{g}'P_0\bar{g}}{(\bar{x}'\bar{g})^2}\bar{x}\bar{x}' & \frac{\bar{g}'Z\tilde{g}}{(\bar{x}'\bar{g})(\tilde{x}'\tilde{g})}\bar{x}\tilde{x}'  \\
\frac{\bar{g}'Z\tilde{g}}{(\bar{x}'\bar{g})(\tilde{x}'\tilde{g})}\tilde{x}\bar{x}'  &  \frac{\tilde{g}'P_1\tilde{g}}{(\tilde{x}'\tilde{g})^2} \tilde{x}\tilde{x}'\end{bmatrix}. 
\end{align*}
In this example, we have exact symmetry: $P_0=P_1$, $\bar{g}=\tilde{g}\propto {\bf 1}_2$, and $\bar{x}=\tilde{x}\propto {\bf 1}_2$. It follows that
\[
\widetilde{P}  = \begin{bmatrix}
\frac{{\bf 1}_2'P_0{\bf 1}_2}{4}{\bf 1}_2{\bf 1}_2' & \frac{{\bf 1}_2'Z{\bf 1}_2}{4}{\bf 1}_2{\bf 1}_2'   \\
 \frac{{\bf 1}_2'Z{\bf 1}_2}{4}{\bf 1}_2{\bf 1}_2'  &  \frac{{\bf 1}_2'P_0{\bf 1}_2}{4}{\bf 1}_2{\bf 1}_2'  \end{bmatrix}. 
\]
Combining it with \eqref{block} gives
\[
P -\widetilde{P} = 
\begin{bmatrix}
P_0 - \frac{{\bf 1}_2'P_0{\bf 1}_2}{4}{\bf 1}_2{\bf 1}_2' & Z - \frac{{\bf 1}_2'Z{\bf 1}_2}{4}{\bf 1}_2{\bf 1}_2'   \\
Z' - \frac{{\bf 1}_2'Z{\bf 1}_2}{4}{\bf 1}_2{\bf 1}_2'  &  P_0 - \frac{{\bf 1}_2'P_0{\bf 1}_2}{4}{\bf 1}_2{\bf 1}_2'  \end{bmatrix}.
\]
By \eqref{newP},
\begin{align*}
& P_0 = \begin{bmatrix}
1 & (1-a)+ab& \\
(1-a)+ab&[(1-a)^2+a^2]+2a(1-a)b
\end{bmatrix},\cr
& \frac{{\bf 1}_2'P_0{\bf 1}_2}{4}= 1-a+ab +\frac{1}{2}a^2(1-b).
\end{align*}
Therefore,
\begin{align*}
 P_0 &- \frac{{\bf 1}_2'P_0{\bf 1}_2}{4}{\bf 1}_2{\bf 1}_2'  =
\begin{bmatrix}
a - ab - \frac{1}{2}a^2(1-b) & - \frac{1}{2}a^2(1-b) \\
- \frac{1}{2}a^2(1-b) & -a+ab + \frac{3}{2}a^2(1-b) 
\end{bmatrix} \cr
&= a(1-b)
\begin{bmatrix}
1 - a/2 & -a/2\\
-a/2 & -1 +3a/2
\end{bmatrix}= a(1-b)\begin{bmatrix}1 \\ & -1\end{bmatrix} + \frac{a^2(1-b)}{2}\begin{bmatrix}-1 & -1 \\ -1 & 3\end{bmatrix}. 
\end{align*}
Also, by \eqref{block}, 
\begin{align*}
& Z=  
\begin{bmatrix}
b & a+(1-a)b\\
a+(1-a)b& 2a(1-a)+[a^2+(1-a)^2]b
\end{bmatrix}= 
\begin{bmatrix}
b & b+a(1-b)\\
b+a(1-b)& b+2a(1-b)(1-a)
\end{bmatrix}.\cr
&\frac{{\bf 1}_2'Z{\bf 1}_2}{4} = b +\frac{a(1-b)(2-a)}{2}.
\end{align*}
Therefore,
\begin{align*}
 Z - &\frac{{\bf 1}_2'Z{\bf 1}_2}{4}{\bf 1}_2{\bf 1}_2'  =
\begin{bmatrix}
-\frac{a(1-b)(2-a)}{2} &  \frac{a(1-b)a}{2}\\
\frac{a(1-b)a}{2} & a(1-b)(1-\frac{3a}{2})
\end{bmatrix}\cr
&=a(1-b)
\begin{bmatrix}
-1+a/2 & a/2\\
a/2 & 1-3a/2
\end{bmatrix}= a(1-b)\begin{bmatrix}-1 \\ & 1\end{bmatrix} + \frac{a^2(1-b)}{2}\begin{bmatrix}1 & 1 \\ 1 & -3\end{bmatrix}. 
\end{align*}
It follows that
\beq
P-\tilde{P}= 
a(1-b) 
\begin{bmatrix}
1 & & -1\\
& -1 & & 1\\
-1 & & 1\\
& 1&&-1
\end{bmatrix} + \frac{a^2(1-b)}{2}
\begin{bmatrix}
-1 & -1 & 1 & 1\\
-1 & 3 & 1 & -3\\
1 & 1 & -1 & -1\\
1 & -3 &-1 &3
\end{bmatrix}.
\eeq
It is not hard to see that $(1, 0,1,0)'$ and $(0,1,0,1)'$ are two eigenvectors associated with the zero eigenvalue. 
Consider an eigenvector of the form $(\epsilon,1, -\epsilon, -1)$: 
\begin{align*}
(P-\widetilde{P})\begin{bmatrix}\epsilon\\1\\-\epsilon\\-1\end{bmatrix} &= a(1-b) \left(
\begin{bmatrix}
2\epsilon\\ -2\\-2\epsilon\\ 2
\end{bmatrix} + \frac{a}{2}
\begin{bmatrix}
-2\epsilon-2\\ -2\epsilon+6 \\ 2\epsilon+2 \\ 2\epsilon +6
\end{bmatrix}\right)= a(1-b) 
\begin{bmatrix}
(2-a)\epsilon-a\\ -2-a\epsilon+3a \\-(2-a)\epsilon+a\\ 2+a\epsilon+6a
\end{bmatrix}.
\end{align*}
For this to be a valid eigenvector, we need $[(2-a)\epsilon-a]/\epsilon=-2-a\epsilon+3a$. It yields
\[
a\epsilon^2+4(1-a)\epsilon-a=0. 
\]
It has two solutions 
\[
\epsilon_1=\frac{-2(1-a)+ \sqrt{4(1-a)^2+a^2}}{a}, \qquad \epsilon_2=\frac{-2(1-a)- \sqrt{4(1-a)^2+a^2}}{a}.
\]
The corresponding eigenvalues are
\begin{align*}
&\lambda_1=a(1-b)(3a-2-a\epsilon_1) =a(1-b)[ a-\sqrt{4(1-a)^2+a^2}],
\\
 &\lambda_2=a(1-b)[ a+\sqrt{4(1-a)^2+a^2}]. 
\end{align*}
It follows that
\[
\mathrm{trace}([P-\widetilde{P}]^3) = \lambda_1^3+\lambda_2^3=8a^4(1-b)^3[a^2+3(1-a)^2].
\]

Furthermore, since $\Pi'\Theta^2 \Pi = \Vert \theta\Vert^2 I_4/4 $ in this example, it can be easily derived that
\begin{align*}
{\rm SNR}_{n,3}(\Omega) = \Big(\frac{ \Vert \theta\Vert^2 }{4}\Big)^3 \frac{ \mathrm{trace}([P-\widetilde{P}]^3)}{\sqrt{6 C_{n,3}}} \geq C^{-1}a^4(1-b)^3 \Vert\theta\Vert^3.
\end{align*}
The proof is complete.

\subsection{Theoretical details for Theorem~\ref{thm:power1-mainpaper}}\label{app:thm:power1-mainpaper}

\subsubsection{Regularity conditions for Theorem~\ref{thm:power1-mainpaper}}\label{app:thm:power1-mainpaper:cond}

Given any integers $1\leq L\leq K$, let ${\cal U}_L(K)$ be the collection of $\Omega$ from a DCBM with $K$ communities such that 
\begin{enumerate}
\item[(i)]$|\lambda_{L}(PG)|-|\lambda_{L+1}(PG)|\geq c_2\beta_n$, where $G=\|\theta\|^{-2}\Pi'\Theta^2\Pi$ and $\beta_n\|\theta\|/\sqrt{\log(n)}\to\infty$.

\item[(ii)] $\lambda_{\min}(\Pi'\Theta\Pi)\geq c_5\Vert \theta\Vert_1$.

\item[(iii)] $\Omega$ satisfies Condition~\ref{cond:afmSCORE}(a) and Condition~\ref{cond:MSCORE}(b)-(c) of the main paper.
\end{enumerate}

\subsubsection{A useful lemma (Lemma~\ref{lem:power-DCBM})}

Fixing $K_0 \geq 1$, under the setting of  Theorem~\ref{thm:power1-mainpaper}, we assume the true model is a DCBM with $K = (K_0+1)$ communities, 
but we mis-specify it as a DCBM with $K_0$ communities.  Let $\Omega = \Theta \Pi P \Pi' \Theta$ be the Bernoulli probability matrix for the true model. We can cast this as the problem of testing   $K = K_0+1$ versus $K = K_0$.  The setting is closely related to the problem of estimating $K$ (e.g., \cite{EstK} and Section \ref{subsec:DCBM}). Let $T_n(\widehat{\Omega}^{{\rm DCBM}})$ be the GoF-SCORE metric 
for DCBM (e.g., Section~\ref{subsec:DCBM}).  For any $\alpha\in (0,1)$, consider the GoF-SCORE test that rejects the null if $|T_n(\widehat{\Omega}^{\mathrm{DCBM}})| \geq z_{\alpha}$, where $\mathbb{P}(|N(0,1)| \geq z_{\alpha/2}) = \alpha$.  Now, under the assumed model (a DCBM with $K = K_0$), we apply SCORE and use it to estimate $\Omega$, both assuming that there are only $K_0$ communities.  
Following \cite{EstK}, we can show that SCORE 
has the so-called Non-Splitting Property (NSP): except for a small probability, each resultant cluster by SCORE is either a true community or the merge of two true communities. 
Without loss of generality, we can focus on the event where SCORE keeps the first $(K_0-1)$ communities the same but merges last two communities. Over this event, $\widehat{\Omega} = \mathbb{M}(A)$ for a mapping $\mathbb{M}$. 
Let $\widetilde{\Omega} = \mathbb{M}(\Omega)$.  Recall that  $\Omega = \Theta \Pi P \Pi' \Theta$.  Partition $P$ by $P=\bigl[\begin{smallmatrix}P_0&Z\\Z'&P_1\end{smallmatrix}\bigr]$ where $P_1 \in \mathbb{R}^{2\times 2}$ and $Z\in\mathbb{R}^{(K_0-1)\times 2}$. Define $(K_0+1)\times 1$ vectors $\tilde{s}=\Vert \theta\Vert^{-2}\Pi'\Theta^2{\bf 1}_n$, $\tilde{q}=\Vert\theta\Vert_1^{-1}\Pi'\Theta{\bf 1}_n$, $\tilde{x}=P\tilde{q}$, and let  $s, q, x \in\mathbb{R}^{2}$ be the respective sub-vectors restricted to last two entries.  
%Lemma \ref{lem:power-DCBM} is proved in Section~\ref{app:lem:power-DCBM}, and in Section~\ref{app:lem:power-DCBM2} we provide a a special case as an example for Lemma~\ref{lem:power-DCBM}.  
%%%%%%%%%%%
%%%%%%%%%%%
%%%%%%%%
\begin{lemma} \label{lem:power-DCBM}
With the same notations as above, it holds that 
$\widetilde{\Omega}=\Theta\Pi \widetilde{P}\Pi'\Theta$, and $$\mathrm{trace}([\Omega - \widetilde{\Omega}]^3) =\Vert \theta\Vert^6\cdot \mathrm{trace}\bigl([(P-\widetilde{P})\diag(\tilde{s})]^3\bigr),$$ 
where $\mathrm{rank}(\widetilde{P}) = K_0$ and 
\beq 
P - \widetilde{P} = 
\begin{bmatrix}
{\bf 0}_{(K-2)\times (K-2)} & Z( I_2 - \frac{1}{x'q} qx') \\ 
( I_2 - \frac{1}{x'q} xq')Z' &  P_1 - \frac{q'P_1q}{(x'q)^2} xx' \end{bmatrix}. 
\eeq
\end{lemma} 
The proof of Lemma~\ref{lem:power-DCBM} is deferred to Section~\ref{app:lem:power-DCBM}.

\subsubsection{Proof of Theorem~\ref{thm:power1-mainpaper}} \label{supp:susecpower1}
Recall our statistics $T_n(\widehat \Omega) = U_{n,3}(\widehat \Omega)/\sqrt{6C_{n,3}}$ with 
\begin{align*}
\widehat \Omega =  {\rm diag} (A{\bf 1}_n) H [{\rm diag}(H'A{\bf 1}_n)]^{-1} H'AH  [{\rm diag}(H'A{\bf 1}_n)]^{-1} H'{\rm diag} (A{\bf 1}_n), \qquad H = \Pi.
\end{align*}
 Under null, $T_n(\widehat\Omega) \to N(0,1)$. Consequently, for the type I error
 \begin{align*}
 \mathbb{P} (|T_n(\widehat\Omega)|>\alpha_n|\text{ true model is }\Omega_0)  = 1- 2\Phi (-\alpha_n/2) + o(1) = o(1)
 \end{align*} 
if $\alpha_n \gg 1$.

Under alternative $\Omega = \Omega_1$ where $K= K_0+1$, according to \cite{EstK}, $H$ satisfies the Non-splitting Property. Specifically, there exist finite many configurations of $H$ such that with probability $1- o(1)$, $\Pi$ equals to one of them. Therefore, without loss of generality, we assume the last two communities merge. It gives 
\beq
H = \Pi Q, \qquad\mbox{where}\quad  Q' =
\begin{bmatrix}
1&  \\ 
 &\ddots &  \\
 & &1& 1
\end{bmatrix} \; \in \mathbb{R}^{(K-1)\times K}. 
\eeq
In this case, $T_n (\widehat \Omega) - T_n ( \Omega)$ may not be negligible. We analyze its behavior below. Write $\Delta = \widehat{\Omega} - \Omega = (\widetilde \Omega - \Omega ) + (\widehat{\Omega}- \widetilde{\Omega}) =: \widetilde \Delta + \widehat\Delta$ with
$$\widetilde \Omega ={\rm diag} (\Omega {\bf 1}_n) H [{\rm diag}(H' \Omega {\bf 1}_n)]^{-1} H'\Omega H  [{\rm diag}(H'\Omega {\bf 1}_n)]^{-1} H'{\rm diag} (\Omega {\bf 1}_n).$$
 Notice that when $H = \Pi$, $\widetilde \Delta = 0 $; in the case that $H = \Pi Q$, $\widetilde \Delta \neq 0$. We employ  Lemma~\ref{lem:diffU} and write $U_{n,3}(\widehat \Omega) -  U_{n,3}(\Omega) = f(W_1, \Delta)$.  In particular, the analysis of $f(W_1, \widehat \Delta)$ is the same as that  of $f(W_1, \Delta)$ in  the null case.  What remains is to study the additional terms in $ f(W_1, \Delta) -  f(W_1, \widehat \Delta)$ that contain $\widetilde \Delta$. By elementary derivations, it is easy to obtain 
the additional terms containing $\widetilde{\Delta}$ can be classified into following groups (with the coefficients igonored)
 \begin{itemize}
 \item [(1)] ${\rm tr} (\widetilde{\Delta}\circ \widetilde{\Delta} \circ \widetilde{\Delta})$, ${\rm tr} (\widetilde{\Delta}\circ \widetilde{\Delta} \circ \widehat{\Delta})$, ${\rm tr} (\widetilde{\Delta}\circ \widehat{\Delta} \circ \widehat{\Delta})$, ${\rm tr} (W_1 \circ W_1 \circ \widetilde{\Delta})$, ${\rm tr} (W_1\circ \widetilde{\Delta} \circ \widetilde{\Delta})$, ${\rm tr} (W_1\circ \widetilde{\Delta} \circ \widehat{\Delta})$, $ {\rm tr} (\widetilde{\Delta} \circ \widetilde{\Delta} \widehat{\Delta}) $,  $ {\rm tr} (\widetilde{\Delta} \circ  \widehat{\Delta}^2) $, $ {\rm tr} (\widehat{\Delta} \circ \widetilde{\Delta} \widehat{\Delta}) $, ${\rm tr}(\widetilde{\Delta} \widehat{\Delta}^2)$, $ {\rm tr} ( W_1\circ \widetilde{\Delta} \widehat{\Delta}) $;
 \item [(2)]  $ {\rm tr} (\widehat{\Delta} \circ \widetilde{\Delta}^2) $, $ {\rm tr} (\widetilde{\Delta} \circ W_1 \widetilde{\Delta}) $, $ {\rm tr} (\widetilde{\Delta} \circ W_1 \widehat{\Delta}) $, $ {\rm tr} (\widehat{\Delta} \circ W_1 \widetilde{\Delta}) $, $ {\rm tr} ( W_1\circ W_1\widetilde{\Delta} ) $, ${\rm tr} (W_1\widetilde{\Delta}\widehat{\Delta})$ ; 
 \item [(3)] ${\rm tr} (\widetilde{\Delta}^3)$, ${\rm tr} (\widetilde{\Delta}^2\widehat{\Delta})$, $ {\rm tr} (\widetilde{\Delta} \circ \widetilde{\Delta}^2) $,  ${\rm tr} (W_1\widetilde{\Delta}^2)$, ${\rm tr} (W_1\circ \widetilde{\Delta}^2)$, ${\rm tr} (\widetilde{\Delta} \circ W_1^2) - {\rm tr} (W_1^2\widetilde{\Delta})$.
 \end{itemize}
 
 In the sequel, we first claim that all the terms in (1) and (2) are of order $o_\mathbb{P}(\Vert\theta \Vert^3)$. Recall the proofs in Section~\ref{supp:subsec_DCBM}, where $\Delta$ is equivalent to $\widetilde \Delta$ here.  We have shown that 
 \begin{align*}
 \Vert \widehat{\Delta} \Vert = o_{\mathbb{P}} (1).
 \end{align*}
For $\widetilde \Delta$, we have the crude bounds that 
\begin{align*}
\Vert  \widetilde \Delta \Vert = \Vert P - \widetilde P\Vert \cdot \Vert \Pi' \Theta^2 \Pi\Vert  \lesssim \Vert \theta\Vert^2\qquad \text{and } \qquad  \big|\widetilde \Delta(i,i)\big|\lesssim \theta_i^2 \text{ for all $1\leq i \leq n$}.
\end{align*}
Further with $\diag(W_1) = \diag(\Omega)$ and Lemma~\ref{lem:hada}, we can bound the terms in the first line of (1) by 
\begin{align*}
\Vert \widetilde \Delta\Vert \cdot \Big( \Vert \diag(\widetilde \Delta)\Vert + \Vert \diag(\widehat \Delta)\Vert  + \Vert \diag(W_1)\Vert\Big)^2 \lesssim \Vert \theta\Vert^2.
\end{align*}
For the left  terms in (1), we have 
\begin{align*}
 &\big|{\rm tr} (\widetilde{\Delta} \circ \widetilde{\Delta} \widehat{\Delta}) \big| + \big| {\rm tr} (\widetilde{\Delta} \circ  \widehat{\Delta}^2)\big| + \big| {\rm tr} ( W_1\circ \widetilde{\Delta} \widehat{\Delta}) \big| 
\\
&\hspace{1cm}\leq \big(\Vert \diag(\widetilde \Delta)\Vert + \Vert \diag(W_1)\Vert \big) \Vert \widehat \Delta\Vert  \big( \Vert \widetilde \Delta\Vert  + \Vert \widehat \Delta\Vert \big)\lesssim \Vert \theta\Vert^2, \notag \\
& \big| {\rm tr} (\widehat{\Delta} \circ \widetilde{\Delta} \widehat{\Delta}) \big| + 
  \big| {\rm tr}(\widetilde{\Delta} \widehat{\Delta}^2) \big| \lesssim 2\Vert \widetilde\Delta \Vert \Vert \widehat \Delta \Vert^2  \lesssim \Vert \theta\Vert^2. 
\end{align*}

Next, we bound the terms in (2).  Using Lemma~\ref{lem:hada}, we derive 
\begin{align*}
& \big| {\rm tr} (\widetilde{\Delta} \circ W_1 \widetilde{\Delta}) \big| + \big|{\rm tr} (\widehat{\Delta} \circ W_1 \widetilde{\Delta}) \big| + \big| {\rm tr} (W_1 \circ W_1 \widetilde{\Delta})  \big|  + \big| {\rm tr} (W_1\widetilde{\Delta}\widehat{\Delta})\big| \notag\\
&\lesssim \Vert W_1 \widetilde\Delta \Vert \Big( \Vert {\rm diag}(\widetilde \Delta) \Vert + \Vert {\rm diag}(\widehat  \Delta )\Vert  + \Vert {\rm diag} (W_1) \Vert + \Vert \widehat \Delta\Vert  \Big)\notag\\
& \lesssim \Vert W_1 \widetilde\Delta \Vert = \Vert \Pi' \Theta W_1 \Theta \Pi (P- \widetilde P) \Vert.
\end{align*}
Employing the first estimate in the third line in Lemma~\ref{lem:1.7} and the crude bound $\Vert P - \widetilde{P} \Vert \lesssim 1$, we finally have the RHS above is of order  $O_{\mathbb{P}} (\Vert \theta\Vert^2)$. For the other terms in (2), we have the following derivations.
\begin{align*}
&\big|{\rm tr} (\widehat{\Delta} \circ \widetilde{\Delta}^2) \big| \lesssim \Vert \diag(\widetilde{\Delta}^2) \widehat{\Delta}\Vert  \lesssim \Vert \diag(\widetilde{\Delta}^2)\Vert = \max_{i} \theta_i^2 |\pi_i' (P-\widetilde P) \Pi' \Theta^2\Pi (P-\widetilde P) \pi_i| \lesssim \Vert\theta\Vert^2 , \notag\\
& \big| {\rm tr} (\widetilde{\Delta} \circ W_1 \widehat{\Delta})\big|   \lesssim \Vert \diag(\widetilde{\Delta})\Vert \Vert W_1 \widehat{\Delta}\Vert \lesssim \Vert \theta\Vert^2 \, .
\end{align*}
Here in the last step of the first equation above, we used the fact that $\Vert \pi_i' (P-\widetilde P) \Vert = O(1) $ and $\Vert \Pi'\Theta^2 \Pi\Vert \lesssim \Vert \theta\Vert^2$; for the last equation, we applied (b) of Lemma~\ref{lem:Delta} in which $\Delta$ is equivalent to  our $\widetilde \Delta$ here. 

To proceed, we analyze the terms in (3). Notably, ${\rm tr} (\widetilde \Delta^3) = {\rm tr} \big( [\widetilde \Omega - \Omega]^3 \big) $. Recall the assumption in Theorem~\ref{thm:power1-mainpaper} that $\big|  {\rm tr} \big( [\widetilde \Omega - \Omega]^3 \big)\big| \gg \Vert P - \widetilde P \Vert \cdot \Vert \theta \Vert^4$. We aim to bound the terms in (3) except for ${\rm tr} (\widetilde \Delta^3)$   by $\Vert P - \widetilde P \Vert \cdot \Vert \theta \Vert^4$ up to some constant. Towards this, by introducing $\tau_n:= \Vert P - \widetilde P\Vert$, we deduce 
\begin{align*}
 &
 {\rm tr} (\widetilde{\Delta}^2\widehat{\Delta}) \lesssim \Vert\widetilde{\Delta} \Vert^2 \Vert\widehat{\Delta} \Vert\lesssim   \Vert P - \widetilde P\Vert^2 \Vert \theta\Vert^4 \lesssim \Vert P - \widetilde P \Vert \cdot \Vert \theta \Vert^4 ,\notag\\
 &
 {\rm tr} (\widetilde{\Delta} \circ \widetilde{\Delta}^2) \lesssim \Vert \diag(\widetilde \Delta) \Vert \Vert \widetilde \Delta\Vert^2 \lesssim \Vert P - \widetilde P\Vert^2 \Vert \theta\Vert^4 \lesssim \Vert P - \widetilde P \Vert \cdot \Vert \theta \Vert^4,\notag\\
 & 
 {\rm tr} (W_1 \circ \widetilde{\Delta}^2)  \lesssim \Vert {\rm diag}(  \Omega ) \Vert \Vert \widetilde \Delta\Vert^2 \lesssim  \Vert P - \widetilde P\Vert^2 \Vert \theta\Vert^4 \lesssim \Vert P - \widetilde P \Vert \cdot \Vert \theta \Vert^4\, . 
 \end{align*}
 For ${\rm tr} (W_1 \widetilde{\Delta}^2) $, using Lemma~\ref{lem:1.7}, it is not hard to derive 
 \begin{align*}
{\rm tr} (W_1 \widetilde{\Delta}^2) \lesssim \Vert \Pi'\Theta W_1 \Theta \Pi\Vert \Vert P - \widetilde P\Vert   \Vert \Pi' \Theta^2 \Pi\Vert \lesssim  \Vert P - \widetilde P \Vert \cdot \Vert \theta \Vert^4 \,.
 \end{align*}
 Lastly, we estimate ${\rm tr} (\widetilde{\Delta} \circ W_1^2) - {\rm tr} (W_1^2\widetilde{\Delta})$. Notice that $\mathbb{E} W_1^2 = {\rm diag} (\Omega {\bf 1_n} )-{\rm diag} (\Omega) $. It follows that 
 \begin{align*}
 {\rm tr} (\widetilde{\Delta} \circ \mathbb{E} W_1^2) - {\rm tr} (\mathbb{E}W_1^2\widetilde{\Delta}) = {\rm tr} (\widetilde{\Delta} \mathbb{E} W_1^2) - {\rm tr} (\mathbb{E}W_1^2\widetilde{\Delta}) =0\,. 
 \end{align*}
 It remains to bound ${\rm tr} (\widetilde{\Delta} \circ (W_1^2- \mathbb E W_1^2) ) - {\rm tr} ( (W_1^2- \mathbb E W_1^2)\widetilde{\Delta})$.
 \begin{align*}
 &{\rm tr} (\widetilde{\Delta} \circ (W_1^2- \mathbb E W_1^2) )  - {\rm tr} ( (W_1^2- \mathbb E W_1^2)\widetilde{\Delta}) \notag\\
 &= \sum_{i=1}^n \widetilde \Delta_{ii} \sum_{j\neq i} (W_{ij}^2 - \mathbb{E} W_{ij}^2)  - \sum_{i,j, t} \widetilde \Delta_{ij}[ W_1(i,t) W_1(t,j) - \mathbb EW_1(i,t) W_1(t,j) ]\notag\\
 &  = -   \sum_{i\neq j, t}   \widetilde \Delta_{ij}[ W_1(i,t) W_1(t,j) - \mathbb EW_1(i,t) W_1(t,j) ]
 = O_{\mathbb{P}}(\Vert P - \widetilde P \Vert \Vert \theta\Vert_3^3 \Vert \theta\Vert \sqrt{\log n}\,  ),
 \end{align*}
 where to obtain the last step, we computed the order of the variance of the sum and applied Chebyshev's inequality.
 Since $\Vert \theta\Vert ^3 \sqrt{\log n}\leq \Vert \theta\Vert^2 \theta_{\max} \sqrt{\log n} \lesssim \Vert \theta\Vert^2$, we then conclude that 
 \begin{align*}
 {\rm tr} (\widetilde{\Delta} \circ W_1^2) - {\rm tr} (W_1^2\widetilde{\Delta}) = O_{\mathbb{P}}(\Vert P - \widetilde P \Vert  \Vert \theta\Vert^3 \sqrt{\log n}\,  ).
 \end{align*}
 
 Combining the analysis of all the terms in (1) - (3) together, we therefore conclude that 
 \begin{align*}
 U_{n,3} (\widehat \Omega ) - U_{n,3} (\Omega) =  {\rm tr} \big( [ \Omega - \widetilde \Omega]^3 \big) + o_{\mathbb{P}} (\Vert \theta\Vert^3) +  O_{\mathbb{P}} (\Vert P - \widetilde P \Vert \cdot \Vert \theta\Vert^4).
 \end{align*}
Note that $C_{n,3} = {\rm tr} (\Omega^3) (1+ o(1))\asymp \Vert \theta\Vert^6 $. We also recall ${\rm SNR}_{n,3}(\Omega ) =  {\rm tr} \big( [ \Omega - \widetilde \Omega]^3 \big)/\sqrt{6{\rm tr} (\Omega^3) }$. If ${\rm SNR}_{n,3}(\Omega )\geq \gamma_n \gg \sqrt{\log (n)}$, then it follows from the condition $\big|  {\rm tr} \big( [\widetilde \Omega - \Omega]^3 \big)\big| \gg \Vert P - \widetilde P \Vert \cdot \Vert \theta \Vert^4$ that 
\begin{align} \label{eq:2023090601}
T_n(\widehat{\Omega} ) &= T_n(\Omega ) + \frac{{\rm tr} \big( [ \Omega - \widetilde \Omega]^3 \big) }{\sqrt{6C_{n,3}}} (1+ o_{\mathbb{P}}(1)) + o(1)  \notag\\
& = T_n(\Omega ) + {\rm SNR}_{n,3}(\Omega )(1+ o_{\mathbb{P}}(1)) + o(1) \gg \sqrt{\log (n)}
\end{align}
with probability $1- o(1)$. In summary, let $\alpha_n = \sqrt{\log (n)}$, it yields that 
\begin{align*}
 \mathbb{P} \Big(|T_n(\widehat\Omega)|>\alpha_n\big| \text{ true model is } \Omega_0\Big)  +  \mathbb{P} \Big(|T_n(\widehat\Omega) \leq \alpha_n\big| \text{ true model is } \Omega_1\Big)  =  o(1)\,. 
\end{align*}

We now turn to show the claim for the case $\gamma_n\to 0$. In particular, we focus on the case $Z=0^{(K_0-1)\times 2}$ where $Z$ is the $(K_0-1)\times 2$ upper right block in $P$. Under this settings, by some elementary computations, it is easy to obtain 
\begin{align*}
{\rm tr} \big( [\Omega - \widetilde \Omega]^3\big) \asymp (1-b)^3\Vert \theta\Vert^6 
\end{align*}
and further 
\begin{align*}
{\rm SNR}_{n,3}(\Omega) \asymp (1-b)^3\Vert \theta\Vert^3 \geq C \big(|\lambda_{K_0+1}(\Omega)|/ \sqrt{\lambda_1(\Omega)}\, \big)^3\,. 
\end{align*}
Then, if ${\rm SNR}_{n,3}\to 0$, it implies immediately that $|\lambda_{K_0+1}(\Omega)|/ \sqrt{\lambda_1(\Omega)}\to 0$ which is the region \cite{EstK} studied. Particularly,  \cite{EstK}
studied a least-favorable setting for testing $K = K_0$ and $K = K_0 +1$ and proved a lower bound when $|\lambda_{K_0+1}(\Omega)|/ \sqrt{\lambda_1(\Omega)}\to 0$. In equation (3.8) of  \cite{EstK},  letting $m=1$ and $\beta= 0$, we observe that their least-favorable setting (under proper scaling which does not affect the analysis) happens to fit our setting. Consequently, their lower bound can be extended to our case. We conclude the proof and refer readers to \cite{EstK} for more details.

\subsubsection{Proof of Lemma~\ref{lem:power-DCBM}}
\label{app:lem:power-DCBM}

We first prove Lemma~\ref{lem:power-DCBM} and discuss a special case for Lemma~\ref{lem:power-DCBM} in the end. 
Under the assumptions in Lemma~\ref{lem:power-DCBM}, it is seen that $H$ satisfies the Non-Splitting Property (NSP). Without loss of generality, we assume in Lemma~\ref{lem:power-DCBM} that the last two communities are merged. Then, 
\beq
H = \Pi Q, \qquad\mbox{where}\quad  Q'=
\begin{bmatrix}
1&  \\ 
 &\ddots &  \\
 & &1& 1
\end{bmatrix} \; \in \mathbb{R}^{(K-1)\times K}. 
\eeq
Our GoF-DCBM algorithm gives that 
\begin{align*}
\widetilde \Omega  = {\rm diag} (\Omega{\bf 1}_n) H [{\rm diag}(H'\Omega{\bf 1}_n)]^{-1} H'\Omega H  [{\rm diag}(H'\Omega{\bf 1}_n)]^{-1} H'{\rm diag} (\Omega{\bf 1}_n).
\end{align*}
 With a little abuse of notation, throughout this subsection, we write $G=\Vert \theta\Vert_1^{-1}\Pi'\Theta\Pi$. We also notice that $\tilde{q}=\Vert \theta\Vert_1^{-1} \Pi'\Theta {\bf 1}_n$. Since $\Omega=\Theta\Pi P\Pi'\Theta$ and $H=\Pi Q$, it follows that 
\[ 
\diag(\Omega{\bf 1}_n) = \Theta \diag(\Pi P \tilde{q}), \quad H'\Omega H =\Vert \theta\Vert_1^2 Q'GPGQ,  \quad \diag(H'\Omega {\bf 1}_n)=\Vert \theta\Vert_1^2 \diag(Q'GP \tilde{q}). 
\]

We plug them into the formula of $\widetilde{\Omega}$ to get
\beq \label{tOmega-form-1}
 \widetilde{\Omega} = \Theta \diag(\Pi P \tilde{q}) \Pi M\Pi' \diag(\Pi P \tilde{q})\Theta, 
\eeq
where 
\[
M:= Q[\diag(Q'GP \tilde{q})]^{-1}Q'GPGQ[\diag(Q'GP \tilde{q})]^{-1}Q'
\]
To simplify \eqref{tOmega-form-1}, we note that the $n$-dimensional vector $\Pi P \tilde{q}$ takes only $K$ distinct values: When $\pi_i=e_k$, the $i$th entry of $\Pi P \tilde{q}$ is equal to the $k$th entry of $P \tilde{q}$. Consequently, when $\pi_i=e_k$, 
the $i$th row of $\Theta \diag(\Pi P  \tilde{q}) \Pi$ is equal to $\theta_i\cdot (P \tilde{q})_k\cdot \pi_i$. It follows that 
\beq \label{key1}
\Theta \diag(\Pi P  \tilde{q}) \Pi  = \Theta \Pi\diag(P  \tilde{q}). 
\eeq
Combining \eqref{tOmega-form-1}-\eqref{key1} gives
\beq \label{tOmega-form-2}
\widetilde{\Omega} = \Theta \Pi\widetilde{P}\Pi'\Theta, 
\eeq
where
\beq \label{tP}
\widetilde{P}:= \diag(P \tilde{q})Q[\diag(Q'GP \tilde{q})]^{-1}Q'GPGQ[\diag(Q'GP \tilde{q} )]^{-1}Q'\diag(P \tilde{q}). 
\eeq

It remains to study $\widetilde{P}$. Since $\Pi'{\bf 1}_K={\bf 1}_n$ and $G$ is a diagonal matrix, we have 
\[
 \tilde{q}= \Pi'\Theta {\bf 1}_n=\Pi'\Theta\Pi {\bf 1}_K= G{\bf 1}_K=\diag(G).
\]  
It follows that $\diag(Q'GP \tilde{q})=\diag(Q'GPG{\bf 1}_K)$. Recall the notation $\tilde{x} = P \tilde{q}$. Introduce the short-hand notations
\beq \label{y}
X=\diag(\tilde{x}), \quad  \tilde{y} = GP \tilde{q}=  \tilde{q}\circ \tilde{x},\quad\mbox{and}\quad Y=\diag(\tilde{y}).  
\eeq
From the expression of $Q'$, we immediately have
\[
\diag(Q'GP \tilde{q}) = \diag\bigl( \tilde{y}_1,\; \ldots,\; \tilde{y}_{K-2},\; \tilde{y}_{K-1}+ \tilde{y}_K\bigr). 
\]
We note that for any matrix $\Delta=\diag(\delta_1,\ldots,\delta_{K-1})$, 
\[
Q\Delta Q'  =
\begin{bmatrix}
\diag(\delta_{1:(K-2)})\\
&  \delta_{K-1} {\bf 1}_2{\bf 1}_2'\end{bmatrix}
\quad\in\quad\mathbb{R}^{K\times K},
\] 
where we denote by $\delta_{1:(K-2)} = (\delta_1, \cdots, \delta_{K-2})'$. 
Combining the above gives
\[
Q[\diag(Q'GP \tilde{q}|)]^{-1}Q' = 
\begin{bmatrix}
Y_{1:(K-2)}^{-1}\\
&  \frac{1}{\tilde{y}_{K-1}+ \tilde{y}_K} {\bf 1}_2{\bf 1}_2'\end{bmatrix}. 
\]
by denoting $Y_{1:(K-2)} = {\rm diag}(\tilde{y}_{1:(K-2)})$.
Recall $(\tilde{x}, X, y, Y)$ in \eqref{y}. With these notations, we write $\diag(P  \tilde{q})=\diag(\tilde{x}) = X$. We also recall ${x}=(x_1, x_2) = (\tilde{x}_{K-1}, \tilde{x}_K)\in\mathbb{R}^2$ and ${q}=(q_1, q_2)= ( \tilde{q}_{K-1},  \tilde{q}_{K})\in\mathbb{R}^2$. Furthermore, we write ${y}= (y_1, y_2)=(\tilde{y}_{K-1}, \tilde{y}_K)\in\mathbb{R}^2$.
It is seen that 
\[
 \diag(P \tilde{q})Q[\diag(Q'GP  \tilde{q})]^{-1}Q' G
=
\begin{bmatrix}
(XY^{-1}G)_{1:(K-2)} & \\
&  \frac{1}{{\bf 1}_2'{y}} {x}{q}'\end{bmatrix}. 
\]
The matrix $XY^{-1}G$ is diagonal. By \eqref{y}, $Y=GX$. This implies that $XY^{-1}G$ is equal to the identify matrix. In addition, ${\bf 1}_2' {y}= {x}' {q}$. Combining these gives
\beq \label{key2}
 \diag(P  \tilde{q})Q[\diag(Q'GP  \tilde{q})]^{-1}Q' G
=
\begin{bmatrix}
I_{K-2} & \\
&  \frac{1}{{x}'{g}} {x}{q}'\end{bmatrix}. 
\eeq
Plug this into the formula of $\widetilde P$ in (\ref{tP}) and
recall the block-division of $P$  as
\beq \label{block}
P = \begin{bmatrix}
P_0 & Z\\
Z' & P_1
\end{bmatrix}, \qquad \mbox{where}\quad P_0\in\mathbb{R}^{(K-2)\times (K-2)},\;  Z\in\mathbb{R}^{(K-2)\times 2},\; P_1\in\mathbb{R}^{2\times 2},
\eeq
it follows that 
\begin{align*}
P - \widetilde{P} = \left[
\begin{array}{cc}
{\bf 0}_{(K-2)\times (K-2)} & Z(I_2 - \frac{1}{{x}'{q}} {q}{x}') \\ [0.2cm]
(I_2 - \frac{1}{{x}'{q}} {x}{q} ')  Z' & P_1 - \frac{{q}'P_1{q}}{({x}'{q})^2} {x}{x}'
\end{array}
\right].
\end{align*}
Furthermore, using $\tilde{s} = \Vert \theta\Vert^{-2}\Pi'\Theta^2 {\bf 1}_n$, it is elementary to have
\begin{align*}
{\rm trace} \big([\Omega - \widetilde \Omega]^3\big)  =\Vert \theta\Vert^6 \cdot  {\rm trace} \big( [(P- \widetilde P)\diag(\tilde s) ]^3\big)
\end{align*}
with $(P- \widetilde P)$ defined in the equation above. This proved Lemma~\ref{lem:power-DCBM}.

\subsubsection{An example for Lemma~\ref{lem:power-DCBM}}
\label{app:lem:power-DCBM2}

We consider the explicit formula of ${\rm trace} \big([\Omega - \widetilde \Omega]^3\big)$ for a special case, which is stated in the following corollary. We introduce some notations first. Denote $P = \begin{pmatrix} 1& b\\ b&1\end{pmatrix}$ and $Z= [ \beta_1, \beta_2]$. Recall that $s= (s_1, s_2)'$, $q = (q_1, q_2)'$. We also define $a = (q_1 - q_2)^2/ (q_1+ q_2)^2$.

\begin{cor}\label{cor}
Consider a special case where $x\propto {\bf 1}_2$. Then, 
 \begin{align*}
  {\rm trace} \big([\Omega - \widetilde \Omega]^3\big)
 & =\Vert \theta\Vert^6 \cdot \frac{3(1-b)\Vert \gamma\Vert^2q_1^2q_2^2}{2(q_1+q_2)^2}\Big[   \Big( \frac{s_1}{q_1}+ \frac{s_2 }{q_2}\Big)^2 -a \Big( \frac{s_1}{q_1} - \frac{s_2 }{q_2}\Big)^2 \Big] \notag\\
 & \quad + \Vert \theta\Vert^6 \cdot \frac{(1-b)^3}{8}  \Big[(1-a)^3 ( s^3_1 + s^3_2)  + 3(1+a)^2(1-a)s_1s_2(s_1 + s_2)\Big] 
 \end{align*}
 where $\gamma=[ \diag(\tilde{s}_1, \cdots, \tilde{s}_{K-2})]^{1/2} (\beta_1- \beta_2)$. Moreover, if $b<1$, $|\lambda_{\min}(P)|\geq c_0|\lambda_{\min}(P_1)|$, and $\min\{\frac{q_1}{q_2},\frac{q_2}{q_1}\}\geq c_0$, for a constant $c_0\in (0,1)$, then $\mathrm{SNR}_{n,3}(\Omega)\geq C^{-1} \bigl(|\lambda_{K_0+1}(\Omega)|/\sqrt{\lambda_1(\Omega)}\bigr)^3$. 
\end{cor}

\begin{proof}[\bf Proof of Corollary~\ref{cor}]

 Write $S:= {\rm diag}(\sqrt{\tilde s_1}, \cdots, \sqrt{\tilde s_{K}}\, )$, by which we can also  represent ${\rm trace} \big([\Omega - \widetilde \Omega]^3\big)   = \Vert \theta\Vert^6\cdot  {\rm trace} \big( [S(P- \widetilde P)S ]^3\big)$. With a little abuse of notations, throughout this subsection, we denote $S_{1:K-2}:= {\rm diag}(\sqrt{\tilde s_1}, \cdots, \sqrt{ \tilde s_{K-2}}\, )$ and $S_{K-1:K}:= {\rm diag}(\sqrt{ \tilde s_{K-1}}, \sqrt{\tilde s_{K}}\, )$. It follows that 
\begin{align*}
S(P- \widetilde{P} )S = \left[
\begin{array}{cc}
{\bf 0}_{(K-2)\times (K-2)} & S_{1:K-2} Z(I_2 - \frac{1}{{x}'{q}} {q} {x}') S_{K-1:K} \\ [0.2cm]
S_{K-1:K} (I_2 - \frac{1}{{x}'{q}} {x}{q} ')  Z' S_{1:K-2}   & S_{K-1:K}\big( P_1 - \frac{{q}'P_1{q}}{({x}'{q})^2} {x}{x}' \big)S_{K-1:K}
\end{array}
\right].
\end{align*}
For $M = \begin{pmatrix}
  0 & A\\ 
  A' & B
\end{pmatrix}$ with symmetric $B$,  it can be checked by some elementary computations that  
\begin{align*}
{\rm tr} (M^3) = 3{\rm tr} (BA'A) + {\rm tr} (B^3).
\end{align*}
As a result, we have 
\begin{align*}
&{\rm tr} \big((S(P- \widetilde{P} )S)^3\big)  \\
& = 3 {\rm tr} \Big( S_{K-1:K}\big( P_1 - \frac{{q}'P_1{q}}{({x}'{q})^2} {x}{x}' \big)S^2_{K-1:K} (I_2 - \frac{1}{{x}'{q}} {x}{q} ')  Z' S^2_{1:K-2}Z(I_2 - \frac{1}{{x}'{q}} {q} {x}') S_{K-1:K} \Big) \notag\\
& \quad +  {\rm tr} \Big( \big(S_{K-1:K}\big( P_1 - \frac{{q}'P_1{q}}{({x}'{q})^2} {x}{x}' \big)S_{K-1:K}\big)^3\Big) \notag\\
&=: T_1 + T_2.
\end{align*}
Consider the special case ${x}\propto {\bf 1}_2$, i.e., $x_{K-1} = e_{K-1}'P  \tilde{q}=  e_{K}'P \tilde{q}= x_{K}$. Recall that $P_1= \begin{pmatrix} 1& b\\ b& 1\end{pmatrix}$. It yields that 
\begin{align*}
P_1 - \frac{{q}'P_1{q}}{({x}'{q})^2} {x}{x}'  & = \begin{pmatrix} 1& b\\ b& 1\end{pmatrix} -  \frac{q^2_1+ q^2_2 + 2b q_1q_2}{(q_1+ q_2)^2}\begin{pmatrix} 1& 1\\ 1& 1\end{pmatrix}
%& = \frac{1-b}{2} \begin{pmatrix} \frac{4g_{K-1}g_K}{(g_{K-1}+ g_K)^2}& \frac{-2(g^2_{K-1} +g^2_K)}{(g_{K-1}+ g_K)^2} \\[0.2cm] \frac{-2(g^2_{K-1} +g^2_K)}{(g_{K-1}+ g_K)^2} & \frac{4g_{K-1}g_K}{(g_{K-1}+ g_K)^2} \end{pmatrix} \notag\\
 =  \frac{1-b}{2}\begin{pmatrix} 1-a& -1-a\\ -1-a& 1-a\end{pmatrix} ,
\end{align*}
where we recall the notation $a = (q_1 - q_2)^2/ (q_1+ q_2)^2$. And we can also derive
\begin{align*}
I_2 -  \frac{1}{{x}'{q}} {q} {x}' = I_2 - \frac{1}{q_1+ q_2}  \begin{pmatrix} q_1& q_1\\ q_2& q_2\end{pmatrix} 
= \frac{1}{q_1+ q_2}  \begin{pmatrix} 1\\ -1 \end{pmatrix}  \begin{pmatrix} q_2 &- q_1\end{pmatrix}. 
\end{align*}
By the notation that $Z = ( \beta_1, \beta_2)$, we further have 
\begin{align*}
Z(I_2 - \frac{1}{{x}'{q}} {q} {x}') =\Big( \frac{q_{2}}{q_1 + q_2} (\beta_1-\beta_2),  -\frac{q_1}{q_1+ q_2} (\beta_1-\beta_2) \Big)=: (\gamma_1, \gamma_2).
\end{align*}
%where $\gamma_1 = \frac{g_{K}}{g_{K-1} + g_K} (\beta_1-\beta_2)$ and $\gamma_2 = -\frac{g_{K-1}}{g_{K-1} + g_K} (\beta_1-\beta_2)  = -\frac{g_{K-1}}{g_K} \gamma_1$.
As a result, 
\begin{align*}
T_1 &=3 {\rm tr}\left[ \frac{1-b}{2} \diag(s) \begin{pmatrix} 1-a& -1-a\\ -1-a& 1-a\end{pmatrix} \diag(s) \begin{pmatrix} \gamma_1'\\ \gamma_2' \end{pmatrix} S^2_{1:K-2} (\gamma_1, \gamma_2) \right] \notag\\
& = \frac{3(1-b)}{2} \Big[ (1-a)  \big( s_1^2\gamma_1' S^2_{1:K-2} \gamma_1+s^2_{2}\gamma_2'S^2_{1:K-2} \gamma_2\big) - 2(1+a) s_1s_2 \big(\gamma_2' S^2_{1:K-2}\gamma_1 \big) \Big]\notag\\
& = \frac{3(1-b)\Vert \gamma\Vert^2q_1^2q_2^2}{2(q_1+q_2)^2}\Big[ (1-a)  \Big( \frac{s^2_1}{q_1^2}+ \frac{s^2_2 }{q_2^2}\Big) +2(1+a)\frac{s_1}{q_1}\frac{s_2}{q_2} \Big],
\end{align*}
where we used the facts that $S^2_{1:K-2} = \diag(\tilde{s}_{1:(K-2)})$, $S^2_{(K-1):K} = \diag(\tilde{s}_{K-1}, \tilde s_{K}) =\diag(s)$ and the short-hand notation $\gamma=[ \diag(\tilde{s}_{1:(K-2)})]^{1/2} (\beta_1- \beta_2)$. In addition,
\begin{align*}
T_2 & =  \frac{(1-b)^3}{8} {\rm tr} \left[ \begin{pmatrix} (1-a) s_1 & -(1+a)\sqrt{s_1s_2} \\ -(1+a) \sqrt{s_1s_2}& (1-a)s_2 \end{pmatrix} ^3\right] \notag\\
& =  \frac{(1-b)^3}{8}  \Big[(1-a)^3 ( s^3_1 + s^3_2)  + 3(1+a)^2(1-a)s_1s_2(s_1 + s_2)\Big].
\end{align*}
Consequently, combining these two terms gives
 \begin{align*}
   {\rm tr} \big((S(P- \widetilde{P} )S)^3\big) 
 & =\frac{3(1-b)\Vert \gamma\Vert^2q_1^2q_2^2}{2(q_1+q_2)^2}\Big[   \Big( \frac{s_1}{q_1}+ \frac{s_2 }{q_2}\Big)^2 -a \Big( \frac{s_1}{q_1} - \frac{s_2 }{q_2}\Big)^2 \Big] \notag\\
 & \quad + \frac{(1-b)^3}{8}  \Big[(1-a)^3 ( s^3_1 + s^3_2)  + 3(1+a)^2(1-a)s_1s_2(s_1 + s_2)\Big],
 \end{align*}
 which proves the first claim. 
 
Next, we prove the second claim.  If $b<1$ and $|\lambda_{\min}(P)| \geq c_0|\lambda_{\min} (P_1)|$, it follows that $1-b  = |\lambda_{\min} (P_1)|\asymp |\lambda_{\min}(P)| $.  Further with the condition that $\min\{ q_1/q_2, q_2/q_1\} \geq c_0$, it can be seen that $a<1$ and $ \Big( \frac{s_1}{q_1}+ \frac{s_2 }{q_2}\Big)^2 -a \Big( \frac{s_1}{q_1} - \frac{s_2 }{q_2}\Big)^2 >0$. As a result,
\begin{align*}
{\rm trace} \big([\Omega - \widetilde \Omega]^3\big)  &=   \Vert \theta\Vert^6\cdot  {\rm tr} \big((S(P- \widetilde{P} )S)^3\big) \notag\\
& \geq    \Vert \theta\Vert^6 \frac{(1-b)^3}{8}\Big[(1-a)^3 ( s^3_1 + s^3_2)  + 3(1+a)^2(1-a)s_1s_2(s_1 + s_2)\Big] \notag\\
&\geq C  \Vert \theta\Vert^6 |\lambda_{\min}(P)|^3 \geq C |\lambda_{k_0+1}(\Omega)|^3.
\end{align*}
Plugging this into the definition ${\rm SNR}_{n,3} (\Omega) ={\rm trace} \big([\Omega - \widetilde \Omega]^3\big) / \sqrt{6 \cdot{\rm trace}(\Omega^3)} $ gives that 
\begin{align*}
{\rm SNR}_{n,3} (\Omega)\geq C^{-1} \big(|\lambda_{k_0+1}(\Omega)|/ \sqrt{\lambda_1(\Omega)}\, \big)^3
\end{align*}
following from the fact that ${\rm trace}(\Omega^3) \leq K \lambda_1^3(\Omega)$.
This completes the proof. 

\end{proof}

\subsection{Extension of  Lemma~\ref{lem:DCBMvsDCMM} to general $K$. }
\label{app:power:2}
Lemma~\ref{lem:DCBMvsDCMM}  discusses the setting for $2$ communities.  Throughout this section, we consider the case of general $K$ communities. 
Suppose the  true model is a DCMM with $K$ communities ($K$ is known),  but we {\it misspecify} it as a DCBM with $K$ communities.  We study the power of detecting model misspecification. 

Same as before, let $\Omega = \Theta \Pi P \Pi' \Theta$ be the Bernoulli probability matrix of the 
true DCMM. For $1\leq k\leq K$, let $\lambda_k$ be the $k$th largest eigenvalue (in magnitude) of $\Omega$, and let $\xi_k\in\mathbb{R}^n$ be the associated eigenvector. Let $R=[\diag(\xi_1)]^{-1}[\xi_2,\ldots,\xi_K]$. Given any partition $\{1,2,\ldots,n\}=\cup_{k=1}^K S_k$, let $\mathrm{RSS}(S_{1:K})=n^{-1}\sum_{k=1}^K\sum_{i\in S_k}\|r_i-\bar{r}_k\|^2$, where $S_{1:K}$ is a short-hand notation of the partition, $r_i$ is the $i$th row of $R$, and $\bar{r}_k=|S_k|^{-1}\sum_{i\in S_k}r_i$. Let $\widetilde{S}_{1:K}$ be the partition that minimizes $\mathrm{RSS}(S_{1:K})$. Let $\widetilde{\Pi}=[\tilde{\pi}_1',\tilde{\pi}_2',\ldots,\tilde{\pi}_n']\in\mathbb{R}^{n\times K}$ be such that $\tilde{\pi}_i=e_k$ if $i\in \widetilde{S}_k$.  
We define a quantity $\chi_n:=\min_{S_{1:K}\neq \widetilde{S}_{1:K}}\{\mathrm{RSS}(S_{1:K})\}- \mathrm{RSS}(\widetilde{S}_{1:K})$, 
where $S_{1:K}\neq \widetilde{S}_{1:K}$ means the two partitions are different up to any permutation of $S_1,\ldots,S_K$. 
Let  $\zeta=\|\theta\|_1^{-1}P\Pi'\Theta{\bf 1}_n$. Define $\widetilde{\theta}\in\mathbb{R}^{n}$ by $\widetilde{\theta}_i=\theta_i\frac{\|\theta\|_1(\pi_i'\zeta)}{\sum_{j\in S_k}\theta_j(\pi_j'\zeta)}$. Write $\widetilde{P}=\frac{1}{\|\theta\|_1^2}\widetilde{\Pi}'\Omega\widetilde{\Pi}$, $\widetilde{\Theta}=\diag(\widetilde{\theta})$, $G=\frac{1}{\|\theta\|^{2}}\Pi'\Theta^2\Pi$, $\widetilde{G}=\frac{1}{\|\theta\|^2}\widetilde{\Pi}'\widetilde{\Theta}^2\widetilde{\Pi}$, and $\Gamma=\frac{1}{\|\theta\|^2}G^{-\frac12}\Pi'\Theta\widetilde{\Theta}\widetilde{\Pi}'\widetilde{G}^{-\frac12}$. 
For any $B, M_1,M_2\in\mathbb{R}^{K\times K}$, let $
f(M_1, M_2, B) = \mathrm{trace}(M_1^3-M_2^3)+3\mathrm{trace}(M_1BM_2B'M_1-M_2B'M_1BM_2)$. 
Same as before, let $\kappa(\cdot)$ denote the conditioning number of a matrix. 
Lemma~\ref{lem2:power-DCMM} is proved in Section~\ref{newnew}. 

\begin{lemma} \label{lem2:power-DCMM}
Consider a DCMM model with the same notations above. Suppose Condition~\ref{cond:afmSCORE} and Condition~\ref{cond:MSCORE} (b)-(c) hold, and $\kappa(\Pi'\Theta\widetilde{\Pi})\leq C$ up to a column-wise permutation of $\widetilde{\Pi}$. If $\chi_n \geq (\delta_n\beta_n\|\theta\|)^{-1}\log(n)$, then $\mathbb{P}(\widehat{\Pi} = \widetilde{\Pi}) = O(n^{-3})$, up to a column-wise permutation of $\widetilde{\Pi}$. Furthermore, on the event $\widehat{\Pi}=\widetilde{\Pi}$, $\widehat{\Omega}=\mathbb{M}(A)$ for a mapping $\mathbb{M}$. Let $\widetilde{\Omega}=\mathbb{M}(\Omega)$. This matrix satisfies that 
$\widetilde{\Omega}=\widetilde{\Theta}\widetilde{\Pi} \widetilde{P}\widetilde{\Pi}'\widetilde{\Theta}$ and $\mathrm{trace}([\Omega - \widetilde{\Omega}]^3)=\|\theta\|^6\, f\bigl(G^{\frac12}PG^{\frac12},\widetilde{G}^{\frac12}\widetilde{P}\widetilde{G}^{\frac12},\Gamma\bigr)$.
\end{lemma} 

Using the expressions in Lemma~\ref{lem2:power-DCMM}, we can prove the following theorem
on the power of testing DCMM against DCBM with a general $K$: 
%%%%%%%%
%%%%%%%%%%
%%%%%%%%%%%%%
\begin{thm} \label{thm:power2} 
Fix $K\geq 2$ and suppose the true model is a DCMM such that the conditions in Lemma~\ref{lem2:power-DCMM} hold and $f\bigl(G^{\frac12}PG^{\frac12},\widetilde{G}^{\frac12}\widetilde{P}\widetilde{G}^{\frac12},\Gamma\bigr)\gg \|\theta\|^{-2}$. If $\mathrm{SNR}_{n,3}(\Omega)\to\infty$, then $T_n(\widehat{\Omega}^{\mathrm{DCBM}})\to\infty$, and for any fixed $\alpha\in (0,1)$, the power of the level-$\alpha$ GoF-SCORE test tends to $1$. 
\end{thm}

\subsubsection{Proof of Lemma \ref{lem2:power-DCMM}}\label{newnew}
We start with claiming $\widehat \Pi = \widetilde \Pi$ with high probability when $\chi_n\geq (\delta_n \beta_n \Vert \theta\Vert )^{-1}\log(n) $.  Recall the partition $\cup_{k=1}^K S_k $ and the residual square sum ${\rm RSS}(S_{1:K}) = n^{-1} \sum_{k=1}^K \sum_{i\in S_k} \Vert r_i - \bar{r}_k\Vert^2$. Similarly, we define  $\widehat{\rm RSS}(S_{1:K}) $ for observed $\hat{r}_i's$ by
\begin{align*}
\widehat{\rm RSS}(S_{1:K}) = \frac 1n \sum_{k=1}^K \sum_{i\in S_k} \Vert \hat r_i  - \hat{\bar{r}}_k \Vert^2
\end{align*}
 where $\hat{\bar{r}}_k : = |S_k|^{-1} \sum_{i\in S_k} \hat r_i$. We write $\varepsilon_n: = (\delta_n \beta_n\Vert \theta \Vert)^{-1} \sqrt{\log (n)}$ for short. By SCORE, we have $ \max_{i}\Vert \hat r_i - r_i \Vert \lesssim \varepsilon_n = o(1)$. This also yields that 
 \begin{align*}
\Vert   \hat{\bar{r}}_k  - \bar{r}_k  \Vert  =\Big\Vert  \frac{1}{|S_k|} \sum_{i\in S_k} ( \hat r_i - r_i  )\Big\Vert \leq  \frac{1}{|S_k|} \sum_{i\in S_k}  \Vert \hat r_i - r_i  \Vert \lesssim \varepsilon_n\, . 
 \end{align*}
Based on these two upper bounds, we can derive
\begin{align*}
\big| \widehat{\rm RSS}(S_{1:K})   - {\rm RSS}(S_{1:K}) \big| \lesssim  \frac 1n \sum_{k=1}^K \sum_{i\in S_k} \big(2 \Vert r_i - \bar r_k\Vert + \varepsilon_n \big)\varepsilon_n  \lesssim \varepsilon_n  {\rm RSS}(S_{1:K})  +  \varepsilon_n^2 \lesssim \varepsilon_n
 \end{align*}
with probability $1- o(n^{-3})$ for all possible clustering $S_{1:K}$. This, together with the conditions that $\chi_n = \min_{S_{1:K}\neq \widetilde S_{1:K}} \{{\rm RSS}(S_{1:K})  \} -  {\rm RSS}( \widetilde S_{1:K})  \gg \varepsilon_n$, leads to 
\begin{align*}
\widehat{\rm RSS}(S_{1:K}) &\geq {\rm RSS}(S_{1:K})  - C\varepsilon_n \geq {\rm RSS}(\widetilde S_{1:K})  + \chi_n - C\varepsilon_n  \notag\\
&\geq \widehat{\rm RSS}(\widetilde S_{1:K}) + \chi_n - 2C\varepsilon_n> \widehat{\rm RSS}(\widetilde S_{1:K})
\end{align*}
for any $S_{1:K}\neq \widetilde S_{1:K}$. Consequently, $\widetilde \Pi = \widehat \Pi$ with high probability up to a column-wise permutation. 

Next, on the event $\widehat \Pi = \widetilde \Pi$, it follows directly 
\begin{align*}
\widehat \Omega =  \mathbb{M}(A)=\diag(A{\bf 1}_n)  \widetilde \Pi [\diag( \widetilde \Pi 'A {\bf 1}_n)]^{-1} \widetilde \Pi'A  \widetilde \Pi [\diag( \widetilde \Pi'A {\bf 1}_n)]^{-1} \widetilde \Pi'\diag(A{\bf 1}_n).
\end{align*}

In the sequel, we prove the last claim. 
Since $\widetilde{\Omega}=\mathbb{M}(\Omega)$, it follows that 
\beq
\widetilde{\Omega} = \diag(\Omega{\bf 1}_n)\widetilde{\Pi}[\diag(\widetilde{\Pi}'\Omega {\bf 1}_n)]^{-1}\widetilde{\Pi}'\Omega \widetilde{\Pi} [\diag(\widetilde{\Pi}'\Omega {\bf 1}_n)]^{-1}\widetilde{\Pi}'\diag(\Omega{\bf 1}_n). \notag
\eeq
Note that $\Omega=\Theta\Pi P\Pi'\Theta$. Recall that $\zeta:= \Vert \theta\Vert_1^{-1} \Pi' \Theta {\bf 1}_n$.  Let $J:=\|\theta\|_1^{-1}\Pi'\Theta\widetilde{\Pi}$. By direct calculations, 
\begin{align*}
\diag(\Omega{\bf 1}_n)= \Vert \theta\Vert_1 \cdot \Theta\diag(\Pi P\zeta ), \quad  \widetilde{\Pi}'\Omega \widetilde{\Pi} =\Vert \theta\Vert_1^2 \cdot  J'PJ, \quad  \diag(\widetilde{\Pi}'\Omega {\bf 1}_n)=\Vert \theta\Vert_1^2 \cdot\diag(J'P\zeta).
\end{align*}
It follows that
\beq \label{tildeOmega-1}
\widetilde{\Omega}=\|\theta\|_1^2\cdot \Theta\diag(\Pi P\zeta)\widetilde{\Pi}[\diag(J'P\zeta)]^{-1}J'PJ[\diag(J'P\zeta)]^{-1}\widetilde{\Pi}'\diag(\Pi P\zeta)\Theta. 
\eeq
Since each row of $\widetilde{\Pi}$ only takes values in $\{e_1,e_2,\ldots,e_K\}$, it follows that  
\[
\widetilde{\Pi}[\diag(v)]^{-1}=[\diag(\widetilde{\Pi}v)]^{-1}\widetilde{\Pi}, \qquad\mbox{for any }v\in\mathbb{R}^K. 
\] 
Therefore, 
\begin{align} \label{temp3}
\diag(\Pi P\zeta)\widetilde{\Pi}[\diag(J'P\zeta)]^{-1} &= \diag(\Pi P\zeta)[\diag(\widetilde{\Pi}J'P\zeta)]^{-1} \widetilde{\Pi}\cr
&=  \diag(\Pi P\zeta)[\diag(\widetilde{\Pi}\widetilde{\Pi}'\Theta \Pi P\zeta)]^{-1} \widetilde{\Pi}. 
\end{align}
Write
\[
\widetilde{\Theta}:=\|\theta\|_1\cdot\Theta \diag(\Pi P\zeta)[\diag(\widetilde{\Pi}\widetilde{\Pi}'\Theta \Pi P\zeta)]^{-1}. 
\]
Then, \eqref{tildeOmega-1} can be re-written as
\beq \label{tildeOmega-2}
\widetilde{\Omega} = \widetilde{\Theta}\widetilde{\Pi}J'PJ\widetilde{\Pi}'\widetilde{\Theta}. 
\eeq
And if $i\in S_k$, 
\begin{align*}
\widetilde \Theta(i,i) = \Vert \theta\Vert_1 \cdot \frac{ \theta_i \pi_i' P \zeta}{ \tilde\pi_i' \widetilde \Pi' \Theta \Pi P \zeta } = \theta_i  \cdot \frac{ \Vert \theta\Vert_1\pi_i' P \zeta}{ \sum_{j\in S_k} \theta_j  (\pi' P \zeta) } . 
\end{align*}
Further let 
\[
\widetilde{P}=J'PJ = \|\theta\|_1^{-2}\widetilde{\Pi}\Theta\Pi P\Pi'\Theta\widetilde{\Pi}=\|\theta\|_1^{-2}\widetilde{\Pi}\Omega\widetilde{\Pi}',
\]
we arrive at $\widetilde \Omega = \widetilde{\Theta}\widetilde{\Pi}\widetilde P \widetilde{\Pi}'\widetilde{\Theta} $. Lastly, we show the representation of ${\rm trace} ([\Omega - \widetilde \Omega]^3)$ below. 

We re-write
\[
\Omega - \widetilde{\Omega} = \begin{bmatrix}\Theta\Pi & \widetilde{\Theta}\widetilde{\Pi}J' \end{bmatrix} \begin{bmatrix} P \\ & -P\end{bmatrix}\begin{bmatrix}\Pi'\Theta\\J\widetilde{\Pi}'\widetilde{\Theta}\end{bmatrix}:=ABA'. 
\]
Note that $\mathrm{trace}((ABA')^3)=\mathrm{trace}[B(A'A)B(A'A)B(A'A)]$, where
\beq\label{temp2}
B = \begin{bmatrix} P \\ & -P\end{bmatrix}, \qquad A'A = \begin{bmatrix} \Pi'\Theta^2\Pi & \Pi'\Theta\widetilde{\Theta}\widetilde{\Pi}J' \\J\widetilde{\Pi}'\widetilde{\Theta}\Theta\Pi & J\widetilde{\Pi}'\widetilde{\Theta}^2\widetilde{\Pi} J'\end{bmatrix}:=\begin{bmatrix} M_{11} & M_{12} \\M_{12}' & M_{22}\end{bmatrix}. 
\eeq
By elementary calculations,
%\begin{align*}
%& B(A'A)B(A'A)B \cr
%=\;\; & \begin{bmatrix} PM_{11}PM_{11}P - PM_{12}PM_{12}'P & -PM_{11}PM_{12}P+PM_{12}PM_{22}P\\
%-PM_{12}'PM_{11}P+PM_{22}PM_{12}'P & - PM_{12}'PM_{12}P + PM_{22}PM_{22}P \end{bmatrix}. 
%\end{align*}
the top left block of $B(A'A)B(A'A)B(A'A)$ is
\[
PM_{11}PM_{11}PM_{11}-PM_{12}PM'_{12}PM_{11} - PM_{11}PM_{12}PM'_{12} + PM_{12}PM_{22}PM'_{12}, 
\]
and the bottom right block of $B(A'A)B(A'A)B(A'A)$ is 
\[
- PM_{12}'PM_{11}PM_{12} + PM_{22}PM'_{12}PM_{12} + PM'_{12}PM_{12}PM_{22} - PM_{22}PM_{22}PM_{22}. 
\]
It follows that
\begin{align} \label{trace-1}
\mathrm{trace}([\Omega-\widetilde{\Omega}]^3) &= \mathrm{trace}(PM_{11}PM_{11}PM_{11}) - \mathrm{trace}(PM_{22}PM_{22}PM_{22}) \cr
&\quad -3\mathrm{trace}(PM_{11}PM_{12}PM_{12}') +    3\mathrm{trace}(PM_{22}PM'_{12}PM_{12}). 
\end{align}
We plug in the expressions of $M_{11}$, $M_{12}$ and $M_{22}$ in \eqref{temp2}. 
It follows that 
\begin{align*}
PM_{11}PM_{11}PM_{11} & = P\Pi'\Theta^2\Pi P\Pi'\Theta^2\Pi P\Pi'\Theta^2\Pi, \cr
PM_{11}PM_{12}PM_{12}' &= P\Pi'\Theta^2\Pi P \Pi'\Theta\widetilde{\Theta}\widetilde{\Pi}J'P J\widetilde{\Pi}'\widetilde{\Theta} \Theta \Pi ,\cr
PM_{22}PM_{12}'PM_{12} &=  P J\widetilde{\Pi}'\widetilde{\Theta}^2\widetilde{\Pi} J' P J\widetilde{\Pi}'\widetilde{\Theta} \Theta \Pi P \Pi'\Theta\widetilde{\Theta}\widetilde{\Pi}J',\cr
PM_{22}'PM_{22}PM_{22} & = P J\widetilde{\Pi}'\widetilde{\Theta}^2\widetilde{\Pi} J' P J\widetilde{\Pi}'\widetilde{\Theta}^2\widetilde{\Pi} J' PJ\widetilde{\Pi}'\widetilde{\Theta}^2\widetilde{\Pi} J'.
\end{align*}
%First, we calculate $\mathrm{trace}(PM_{22}'PM_{22}P)$. Note that 
%\begin{align*}
%\widetilde{Q}&:= \widetilde{\Pi}\widetilde{\Pi}'\widetilde{\Theta}^2\widetilde{\Pi}\widetilde{\Pi}' \cr
%&= 
%\begin{bmatrix}
%{\bf 1}_{S_1}{\bf 1}_{S_1}'\\
%&\ddots \\
%&& {\bf 1}_{S_K}{\bf 1}_{S_K}'
%\end{bmatrix} \begin{bmatrix}
%\widetilde{\Theta}^2_{S_1}\\
%&\ddots \\
%&&\widetilde{\Theta}^2_{S_K}
%\end{bmatrix}
%\begin{bmatrix}
%{\bf 1}_{S_1}{\bf 1}_{S_1}'\\
%&\ddots \\
%&& {\bf 1}_{S_K}{\bf 1}_{S_K}'
%\end{bmatrix}\cr
%&= \begin{bmatrix}
%\|\widetilde{\theta}_{S_1}\|^2{\bf 1}_{S_1}{\bf 1}_{S_1}'\\
%&\ddots \\
%&& \|\widetilde{\theta}_{S_K}\|^2{\bf 1}_{S_K}{\bf 1}_{S_K}'
%\end{bmatrix}.
%\end{align*}
%It follows that
%\begin{align}
%\mathrm{trace}(PM_{22}'PM_{22}P) = \|\theta\|_1^{-4}\mathrm{trace}(P\Pi'\Theta\widetilde{Q}\Theta\Pi P\Pi'\Theta\widetilde{Q}\Theta\Pi P) = 
%\end{align}
Recall the shorthand notations
\begin{align*}
G=\Vert \theta\Vert^{-2}\Pi'\Theta^2\Pi,\qquad 
\widetilde{G}=\Vert \theta\Vert^{-2}\widetilde{\Pi}'\widetilde{\Theta}^2\widetilde{\Pi},\qquad
\Gamma=\Vert \theta\Vert^{-2} G^{-\frac12}\Pi'\Theta\widetilde{\Theta}\widetilde{\Pi}\widetilde{G}^{\frac12}.
\end{align*}
We plug the above into \eqref{trace-1} to get
\begin{align} \label{trace-2}
\mathrm{trace}([\Omega-\widetilde{\Omega}]^3) &= \Vert \theta\Vert^6 \Big[ \mathrm{trace}\bigl(PG PG PG\bigr) - \mathrm{trace}\bigl(PJ\widetilde{G}J'PJ\widetilde{G}J'PJ\widetilde{G}J'\bigr)\cr
&\quad -3\mathrm{trace}\bigl(PG P G^{\frac12}\Gamma\widetilde{G}^{\frac12} J'P J\widetilde{G}^{\frac12} \Gamma'G^{\frac12}\bigr) \cr
&\quad + 3\mathrm{trace}\bigl(P J\widetilde{G}J' P J\widetilde{G}^{\frac12}\Gamma'G^{\frac12} P G^{\frac12}\Gamma\widetilde{G}^{\frac12} J' \bigr) \Big]\cr 
&= \Vert \theta\Vert^6 \Big[ \mathrm{trace}\Bigl(\bigl[G^{\frac12} PG^{\frac12}\bigr]^3\Bigr) - \mathrm{trace}\Bigl(\bigl[\widetilde{G}^{\frac12}\widetilde P\widetilde{G}^{\frac12}\bigr]^3\Bigr)\cr
&\quad -3\mathrm{trace}\Bigl([G^{\frac12} PG^{\frac12}]\Gamma [\widetilde{G}^{\frac12} \widetilde P \widetilde{G}^{\frac12}]\Gamma' [G^{\frac12} PG^{\frac12}]\Bigr) \cr
&\quad + 3\mathrm{trace}\Bigl([\widetilde{G}^{\frac12} \widetilde P\widetilde{G}^{\frac12}]\Gamma'[G^{\frac12}PG^{\frac12}]\Gamma [\widetilde{G}^{\frac12} \widetilde P\widetilde{G}^{\frac12}]\Bigr) \Big] \notag\\
& = \Vert \theta\Vert^6 f(G^{\frac12} PG^{\frac12}, \widetilde{G}^{\frac12} \widetilde P\widetilde{G}^{\frac12}, \Gamma ).
\end{align}
This finishes our proofs.

%%%Proof of Theorem 3.7
\subsubsection{Proof of Theorem~\ref{thm:power2} }

   We first recall that 
   \beq
\widetilde{\Omega} = \diag(\Omega{\bf 1}_n)\widetilde{\Pi}[\diag(\widetilde{\Pi}'\Omega {\bf 1}_n)]^{-1}\widetilde{\Pi}'\Omega \widetilde{\Pi} [\diag(\widetilde{\Pi}'\Omega {\bf 1}_n)]^{-1}\widetilde{\Pi}'\diag(\Omega{\bf 1}_n). \notag
\eeq   
\[
\widetilde{\Theta}:=\|\theta\|_1\cdot\Theta \diag(\Pi P\zeta)[\diag(\widetilde{\Pi}\widetilde{\Pi}'\Theta \Pi P\zeta)]^{-1}. 
\]
   Similarly to the proof of Theorem \ref{thm:power1-mainpaper} in Section~\ref{supp:susecpower1}, we need to bound those additional terms containing $\widetilde \Delta = \widetilde \Omega - \Omega$ which share the same form (1)-(3) in  Section~\ref{supp:susecpower1}. Note that 
 \[
  \widetilde{\Omega} - \Omega= \begin{bmatrix}\Theta\Pi & \widetilde{\Theta}\widetilde{\Pi}J' \end{bmatrix} \begin{bmatrix} -P \\ & P\end{bmatrix}\begin{bmatrix}\Pi'\Theta\\J\widetilde{\Pi}'\widetilde{\Theta}\end{bmatrix} 
\]  
  where $J:=\|\theta\|_1^{-1}\Pi'\Theta\widetilde{\Pi}$ and $\widetilde{\Theta}:=\|\theta\|_1\cdot\Theta \diag(\Pi P\zeta)[\diag(\widetilde{\Pi}\widetilde{\Pi}'\Theta \Pi P\zeta)]^{-1}$. We write for short $D: = \|\theta\|_1\cdot \diag(\Pi P\zeta)[\diag(\widetilde{\Pi}\widetilde{\Pi}'\Theta \Pi P\zeta)]^{-1}$. It follows that 
  \begin{align*}
   \widetilde{\Omega} - \Omega= \Theta \begin{bmatrix}\Pi & D\widetilde{\Pi}J' \end{bmatrix} \begin{bmatrix} -P \\ & P\end{bmatrix}\begin{bmatrix}\Pi'\\J\widetilde{\Pi}'D\end{bmatrix}  \Theta  =: \Theta \overline \Pi \, \overline P \, \overline \Pi ' \Theta.
  \end{align*} 
  
We now claim that all the diagonal entries of  $D$ are of order $O(1)$ based on the condition $\kappa(\Pi'\Theta \widetilde \Pi) = O(1)$. To see this, we first notice that by Cauchy-Schwarz inequality, the largest singular value of $\Pi'\Theta \widetilde \Pi$ satisfies 
\begin{align*}
\sigma_1(\Pi' \Theta \widetilde \Pi)\geq K^{-\frac 12} \Vert \Pi' \Theta \widetilde \Pi\Vert_{F} \geq K^{-3/2} {\bf 1}_K' \Pi' \Theta \widetilde \Pi{\bf 1}_K \asymp \Vert \theta\Vert_1, 
\end{align*}
and $\sigma_K(\Pi'\Theta \widetilde \Pi)\geq C \Vert \theta\Vert_1 $ following from $\kappa(\Pi'\Theta \widetilde \Pi) = O(1)$. By $\Vert \Pi'\Theta \widetilde \Pi\Vert_{\max} < \Vert \theta\Vert_1$,
we have  that for all $1\leq k \leq K$, 
\begin{align*}
\zeta(k) = \Vert \theta\Vert_1^{-1} e_k' \Pi' \Theta {\bf 1}_n =  \Vert \theta\Vert_1^{-1} e_k' \Pi' \Theta \widetilde \Pi {\bf 1}_K \geq \Vert \theta\Vert_1^{-2} \Vert e_k' \Pi' \Theta \widetilde \Pi \Vert^2 \geq  \Vert \theta\Vert_1^{-2} \sigma_K^2(\Pi'\Theta \widetilde \Pi) \geq C,
\end{align*}
and trivially $\zeta(k) = \Vert \theta\Vert_1^{-1} e_k' \Pi' \Theta \widetilde \Pi {\bf 1}_K \leq K $.  Consequently, there exist  some constants $0<c<C$ such that $c<\pi_i' P \zeta<C$ for all $1\leq i \leq n$ and $c<e_k' P\zeta<C$ for all $1\leq k \leq K$. Furthermore, for any $1\leq k\leq K$, 
\begin{align*}
e_k' \widetilde{\Pi}'\Theta \Pi P\zeta \geq ce_k' \widetilde{\Pi}'\Theta \Pi {\bf 1}_K \geq c' \Vert \theta\Vert_1  \quad \text{and } \quad e_k' \widetilde{\Pi}'\Theta \Pi P\zeta \leq C e_k' \widetilde{\Pi}'\Theta \Pi {\bf 1}_K \leq C' \Vert \theta\Vert_1. 
\end{align*}
We therefore obtain that all the diagonal entries in $ \diag(\Pi P\zeta)$ and $\Vert \theta\Vert_1[\diag(\widetilde{\Pi}\widetilde{\Pi}'\Theta \Pi P\zeta)]^{-1}$ are lower and upper bounded by some some positive constants, which concludes that $\Vert D\Vert_{\max} \asymp 1$ and $\kappa(D) \leq C$ for some constant $C>0$. It is also worth noticing that $\Vert J\Vert_{\max} =O(1)$.  As a result, each entry of $\overline \Pi$ is of order $O(1)$ and $\Vert \overline \Pi' \Theta^2 \overline \Pi \Vert \asymp \Vert \overline \Pi' \Theta^2 \overline \Pi \Vert_{\max} \lesssim \Vert \theta\Vert^2$.
We thus can derive 
   \begin{align*}
   \Vert \widetilde \Delta \Vert  \lesssim \Vert \overline{P} \Vert \Vert \overline{\Pi}' \Theta^2 \overline{\Pi} \Vert \lesssim \Vert  \theta\Vert^2  \quad  \text{ for } \quad  \big|\widetilde \Delta(i,j)\big|\lesssim \theta_i\theta_j \quad \text{ for all $1\leq i,j  \leq n$}.
   \end{align*}
 Based on the above bounds,   
similarly to  the analysis in Section~\ref{supp:susecpower1},  all the arguments there also hold in the  current proof by adapting $\Vert P - \widetilde P\Vert $ there by $\Vert \overline P\Vert = O(1) $.
We therefore omit the details and directly conclude that 
 \begin{align*}
 U_{n,3} (\widehat{\Omega} ) - U_{n,3} (\Omega) =  {\rm tr} \big( [ \Omega - \widetilde \Omega]^3 \big) + o_{\mathbb{P}} (\Vert \theta\Vert^3) +  O_{\mathbb{P}} (\Vert \overline{P}  \Vert \cdot \Vert \theta\Vert^4).
 \end{align*}
 
  Under the condition that $f(G^{\frac12} PG^{\frac12}, \widetilde{G}^{\frac12} \widetilde P\widetilde{G}^{\frac12}, \Gamma )\gg \Vert \theta\Vert^{-2}$,  we further deduce that ${\rm tr} \big( [ \Omega - \widetilde \Omega]^3 \big) \gg \Vert \theta\Vert^4$. This, together with the trivial bound $\Vert \overline{P} \Vert = O(1)$,  implies that
 \begin{align} \label{eq:20230901201}
T_n(\widehat{\Omega})&= T_n(\Omega ) + \frac{{\rm tr} \big( [ \Omega - \widetilde \Omega]^3 \big) }{\sqrt{6C_{n,3}}} (1+ o_{\mathbb{P}}(1)) + o(1)  \notag\\
& = T_n(\Omega ) + {\rm SNR}_{n,3}(\Omega )(1+ o_{\mathbb{P}}(1)) + o(1) .
\end{align}
 Consequently, if ${\rm SNR}_{n,3}(\Omega) \to \infty$, $T_n(\widehat{\Omega}^{\rm DCBM})\to \infty$, and the claim follows immediately.

\spacingset{0}
\footnotesize 

\bibliographystyle{chicago}
\bibliography{RefGoF}

\end{document}